\documentclass[11pt]{amsart}
\usepackage{amsmath}
\usepackage{amssymb}
\usepackage{amsthm}
\usepackage{amsfonts}
\usepackage{latexsym}
\usepackage{graphicx}
\usepackage{hyperref}
\usepackage{enumerate}
\usepackage[all]{xy}
\usepackage{chngcntr}
\usepackage{xcolor}
\usepackage{tikz-cd}
\usepackage{tikz}
\usepackage{multicol}
\usepackage{adjustbox}

\newtheorem{thm}{Theorem}[section]
\newtheorem{cor}[thm]{Corollary}
\newtheorem{lem}[thm]{Lemma}
\newtheorem{prop}[thm]{Proposition}

\newtheorem{thmintro}{Theorem}

\theoremstyle{definition}

\newtheorem{rem}[thm]{Remark}

\newcommand{\iref}[1]{\S \ref{#1}, p. \pageref{#1}}

\newcommand{\Z}{\mathbb Z}
\newcommand{\Q}{\mathbb Q}
\newcommand{\R}{\mathbb R}
\newcommand{\C}{\mathbb C}
\newcommand{\F}{\mathbb F}
\newcommand{\mf}{\mathfrak}
\newcommand{\mc}{\mathcal}
\newcommand{\mb}{\mathbf}
\newcommand{\mh}{\mathbb}

\def\cH{{\mathcal H}}
\newcommand{\mr}{\mathrm}
\newcommand{\ind}{\mathrm{ind}}
\newcommand{\enuma}[1]{\begin{enumerate}[\textup{(}a\textup{)}] {#1} \end{enumerate}}
\newcommand{\Fr}{\mathrm{Frob}}
\newcommand{\Sc}{\mathrm{sc}}
\newcommand{\ad}{\mathrm{ad}}
\newcommand{\cusp}{\mathrm{cusp}}
\newcommand{\nr}{\mathrm{nr}}

\newcommand{\cpt}{\mathrm{cpt}}
\newcommand{\Rep}{\mathrm{Rep}}
\newcommand{\Res}{\mathrm{Res}}
\newcommand{\af}{\mathrm{aff}}

\newcommand{\der}{\mathrm{der}}

\newcommand{\matje}[4]{\left(\begin{smallmatrix} #1 & #2 \\ 
#3 & #4 \end{smallmatrix}\right)}

\newcommand{\Mod}{\mathrm{Mod}}
\newcommand{\Hom}{\mathrm{Hom}}
\newcommand{\End}{\mathrm{End}}

\newcommand{\isom}{\xrightarrow{\,\sim\,}}

\newcommand{\Ad}{\mathrm{Ad}}

\newcommand{\ff}{\mathfrak{f}} 
\newcommand{\Irr}{\mathrm{Irr}}
\newcommand{\SL}{{\mathrm{SL}}}
\newcommand{\GL}{{\mathrm{GL}}}
\newcommand{\PGL}{{\mathrm{PGL}}}
\newcommand{\SU}{{\mathrm{SU}}}
\newcommand{\fl}{\mathrm{fl}}
\newcommand{\II}{\operatorname{I}}
\newcommand{\Xo}{\mathfrak{X}^0}

\numberwithin{equation}{section}

\begin{document}

\title[Non-singular depth-zero representations]{Hecke algebras and 
local Langlands correspondence for non-singular depth-zero representations}
\author{Maarten Solleveld}
\address{Institute for Mathematics, Astrophysics and Particle Physics, 
Radboud Universiteit Nijmegen, Nijmegen,
The Netherlands}
\email{m.solleveld@science.ru.nl}

\author{Yujie Xu}
\address{Department of Mathematics, 
Columbia University, 
New York, NY, USA}
\email{xu.yujie@columbia.edu}

\date{\today}

\begin{abstract}
Let $G$ be a connected reductive group over a non-archimedean local field.~We say that 
an irreducible depth-zero (complex) $G$-representation is non-singular 
if its cuspidal support is non-singular.~We establish a local Langlands correspondence 
for all such representations.~We obtain it as a specialization from a categorical version:
an equivalence between the category of finite-length non-singular depth-zero 
$G$-representations and the category of finite-length right modules of
a direct sum of twisted affine Hecke algebras constructed from Langlands
parameters. We also show that our LLC and our equivalence of categories have
several nice properties, for example compatibility with parabolic induction and
with twists by depth-zero characters.
\end{abstract}

\maketitle

\thispagestyle{empty}
\vspace{-5mm}

\tableofcontents

\section{Introduction}

\textbf{Overview and main results.} \\
Let $F$ be a non-archimedean local field and $G$ a connected reductive algebraic group over $F$. 
Let $G^{\vee}$ be the group of $\C$-points of the reductive group whose root datum is dual to 
that of $G$. Let $\mb W_F$ be the Weil group of $F$. As a vast generalization of local class field
theory, the classical explicit local Langlands conjecture, first proposed in the 1960s \cite{Bor},
predicts a surjective map from the ``group side'', which consists of irreducible smooth
representations of $G(F)$ up to isomorphism, to the ``Galois side'', which consists of 
``$L$-parameters'', i.e.~continuous homomorphisms $\varphi\colon \mb W_F\times\SL_2(\C) \to
G^{\vee}\rtimes \mb W_F$. This conjectural surjective map oftentimes has non-singleton fibres, 
called \textit{L-packets}, which are expected to be always finite. When $G$ is a torus, the 
local Langlands conjecture recovers local class field theory. Both tori and $\GL_n$ famously 
have singleton $L$-packets.

In order to formulate a conjectural bijection (or an equivalence of categories) for more general
reductive groups, partially driven by aesthetics, many mathematicians such as Deligne, Vogan,
Lusztig etc.~proposed a refined form of the local Langlands conjecture (see for example 
\cite{Vog} and \cite{ABPS1} for a more detailed exposition), which takes into account the 
non-singleton nature of $L$-packets, and probes further into the \textit{internal structure} 
of the $L$-packets,
parametrized by \textit{enhancements} of the $L$-parameters. The refined local Langlands 
conjecture considers \textit{enhanced $L$-parameters} on the Galois side, which consist of 
$L$-parameters $\varphi$ together with an irreducible representation of a certain component 
group attached to $\varphi$ (see \S\ref{subsec:Lparam-prelim} for more details). 

In this article, we establish the explicit refined local Langlands conjecture for a large 
class of representations. In this overview, we first survey some known results in the 
literature, then highlight the new advancements to the field added by our current article. 

On the group side, i.e.~in the smooth complex representation theory of $p$-adic groups, 
depth-zero representations play a pivotal role. On the one hand, it is 
expected that most representations of higher depth can be reduced in some sense to 
depth-zero representations; on the other hand, experts have long postulated that almost all 
possible technical difficulties (and new phenomena!) already arise at depth zero. In the 
groundbreaking work \cite{DR}, DeBacker and Reeder constructed depth-zero \textit{regular}
supercuspidal $L$-packets, where the condition of ``regularity'' on a supercuspidal 
representation can be very roughly (and perhaps rather inaccurately) thought of as the 
character $\theta$ (in Deligne--Lusztig's $R_T^{\theta}$) being ``in general position'', 
a notion which goes all the way back to \cite{DeLu}. The results of \cite{DR} were later 
generalized from depth-zero to arbitrary depth in \cite{Kal2}, and the assumption of 
re\-gu\-larity was later relaxed to non-singularity in \cite{Kal3}.
To venture beyond the realm of \textit{non-singular supercuspidals}, one necessarily needs to 
enlist the theory of Hecke algebras: (i) either one would like to consider \textit{singular
supercuspidals}--terminology first due to \cite{Aubert-Xu-G2}, which are supercuspidals 
whose $L$-packets \textit{mix} supercuspidals and non-supercuspidals 
and whose study necessarily require a careful analysis of their Bernstein block Hecke algebras; 
(ii) or one would like to consider \textit{non-singular non-supercuspidals}, which are 
$G$-representations whose supercuspidal supports are non-singular. 

Hecke algebra techniques have proven particularly powerful in attacking the local Langlands 
conjecture, as can be seen in \cite{Aubert-Xu-Hecke-algebra, Aubert-Xu-G2, SolLLCunip, SolQS}. 
This is in part due to the fact that Hecke algebras naturally show up on the Galois side 
of the conjectural local Langlands correspondence (LLC). More precisely, as shown in \cite{AMS1} 
(see also \S\ref{subsec:HeckeL-prelim}) the enhanced $L$-parameters admit a natural 
decomposition, \`a la Bernstein,  according to their cuspidal supports, and each such Bernstein 
component on the Galois side is parametrized by the irreducible representations of a certain 
Hecke algebra \cite{AMS3} (see also \S\ref{subsec:HeckeL-prelim}). 

In this article, we generalize the aforementioned literature and construct a local Langlands
correspondence for all depth-zero $G$-representations with non-singular supercuspidal 
support\footnote{For a precise definition, see \S\ref{sec:DL}.}. In 
\cite{Aubert-Xu-Hecke-algebra}, an axiomatic setup for constructing a bijective local Langlands
correspondence was proposed, which can be combined with an analysis of Hecke algebras to obtain
stronger results. In this article, we verify these requirements for all non-singular 
depth-zero Bernstein blocks.

Our first main result is a bijection between
\begin{itemize}
\item the set $\Irr^0(G)_{ns}$ of irreducible non-singular depth-zero
$G$-representations (up to isomorphism); and 
\item the set $\Phi^0_e (G)_{ns}$ of non-singular enhanced L-parameters
for $G$ which are trivial on the wild inertia subgroup of the Weil group $\mb W_F$.
\end{itemize}
Here (and throughout the paper) $G$ should be viewed as a rigid inner twist
of a quasi-split $F$-group. 

\begin{thmintro}\label{thm:B} \textup{(all results in \S \ref{sec:LLC})} \\
There exists a bijection    
\begin{equation}\label{eqn:ThmB-intro-LLC-eqn}
\begin{array}{ccc}
\Irr^0 (G)_{ns} & \longleftrightarrow & \Phi^0_e (G)_{ns} \\
\pi & \mapsto & (\varphi_\pi, \rho_\pi) \\
\pi (\varphi,\rho) & \text{\reflectbox{$\mapsto$}} & (\varphi,\rho)
\end{array}
\end{equation}
such that:
\enuma{
\item The map $\Irr^0 (G)_{ns} \mapsto \Phi (G) : \pi \mapsto \varphi_\pi$
is canonical.
\item The bijection is equivariant for the natural actions of the depth-zero subgroup 
of $H^1 (\mb W_F,Z(G^\vee))$ and the associated group of depth-zero characters of $G$.
\item The central character of $\pi$ is equal to the character of $Z(G)$
canonically determined by $\varphi_\pi$.
\item $\pi$ is tempered if and only if $\varphi_\pi$ is bounded.
\item $\pi$ is essentially square-integrable if and only if $\varphi_\pi$ is 
discrete.
\item Our LLC \eqref{eqn:ThmB-intro-LLC-eqn}, its version for supercuspidal representations 
of Levi subgroups of $G$ and the cuspidal support maps form a commutative diagram
\[
\begin{array}{ccc}
\Irr^0 (G)_{ns} & \longleftrightarrow & \Phi_e^0 (G)_{ns} \\
\downarrow \mr{Sc} & & \downarrow \mr{Sc} \\
\bigsqcup\nolimits_L \, \Irr^0_\cusp (L)_{ns} / W(G,L) & \longleftrightarrow & 
\bigsqcup\nolimits_L \, \Phi^0_\cusp (L)_{ns} / W(G^\vee,L^\vee)^{\mb W_F}
\end{array} .
\]
Here $W(G,L) = N_G (L) / L,\; W(G^\vee,L^\vee) = N_{G^\vee}(L^\vee) / L^\vee$ and
$L$ runs through a set of representatives for the $G$-conjugacy classes
of Levi subgroups of $G$.
\item Our LLC \eqref{eqn:ThmB-intro-LLC-eqn} is compatible with parabolic induction. 
Suppose that $P = M U$ is a parabolic subgroup of $G$, with Levi factor $M$. 
Let $(\varphi,\rho^M) \in \Phi^0_e (M)_{ns}$ be bounded. Let $\pi_0 (S_\varphi^{+})$ and 
$\pi_0 (S_\varphi^{M+})$ be the component groups for $\varphi$ as object of, respectively,
$\Phi (G)$ and $\Phi (M)$. Then
\[
\II_P^G \big( \pi^M (\varphi,\rho^M)\big) \cong \bigoplus\nolimits_\rho \, 
\Hom_{S_\varphi^{M+}} (\rho, \rho^M) \otimes \pi (\varphi, \rho) ,
\]
where the sum runs through all $\rho \in \Irr \big( \pi_0 (S_\varphi^+) \big)$ such that
Sc$(\varphi,\rho)$ is $G^\vee$-conjugate to Sc$(\varphi,\rho^M)$.
\item Our LLC \eqref{eqn:ThmB-intro-LLC-eqn} is compatible with the Langlands classification. 
Suppose that $(\varphi,
\rho) \in \varphi_e^0 (G)_{ns}$ and $\varphi = z \varphi_b$ with $\varphi_b \in \Phi (M)$ 
bounded and $z \in \Hom (M,\R_{>0})$ strictly positive with respect to 
$P = MU$. Then $\II_P^G (z \otimes \pi^M (\varphi_b,\rho))$ is a standard 
$G$-representation and $\pi (\varphi,\rho)$ is its unique irreducible quotient.
\item The $p$-adic Kazhdan--Lusztig conjecture holds for $\Rep^0 (G)_{ns}$.
}
\end{thmintro}

For any progenerator $\Pi_{\mf s}$ (e.g.~from a type), the category $\Rep (G)_{\mf s}$
is naturally equivalent to the category of right modules for $\End_G (\Pi_{\mf s})$. 
By \cite{Mor1,Mor3}, $\End_G (\Pi_{\mf s})$ is rather close to an affine Hecke algebra, 
while its irreducible modules have been studied extensively in \cite{SolEnd}. 
The Bernstein blocks $\Rep (G)_{\mf s}$ altogether make up the category of non-singular
depth-zero $G$-representations $\Rep^0 (G)_{ns}$. We indicate its full subcategory of 
finite-length representations by a subscript ``fl''.

On the Galois side, the set $\Phi^0_e (G)_{ns}$ decomposes naturally as a disjoint union 
of Bernstein components $\Phi_e (G)^{\mf s^\vee}$ \cite{AMS1}, indexed by a 
finite set $\mf B^\vee (G)^0_{ns}$. 
To every such Bernstein component $\Phi_e (G)^{\mf s^\vee}$, one can associate a certain 
twisted affine Hecke algebra $\cH (\mf s^\vee, q_F^{1/2})$ (see
\cite{AMS3}\footnote{Compared to \cite{AMS3}, we specialized an indeterminate
$q$-parameter to $q_F^{1/2} = |k_F|^{1/2}$.}), which is constructed in terms of 
the geometry of the complex
variety of Langlands parameters underlying $\Phi_e (G)^{\mf s^\vee}$,
and whose irreducible modules\footnote{In this paper, modules of a Hecke algebra are 
by default right modules.} are parametrized canonically by
$\Phi_e (G)^{\mf s^\vee}$. Such an algebra $\cH (\mf s^\vee,q_F^{1/2})$
can be compared with $\End_G (\Pi_{\mf s})$ for an appropriate inertial
equivalence class $\mf s$ for $\Rep (G)$. 
Our second main result is the following. 

\begin{thmintro}\label{thm:A} \textup{(Theorem \ref{thm:10.7})} \\
There exists an equivalence of categories 
\[
\Rep^0_\fl (G)_{ns} \; \cong \; \bigoplus\nolimits_{\mf s^\vee \in \mf B^\vee (G)^0_{ns}} \,
\Mod_\fl \text{ - }\cH (\mf s^\vee, q_F^{1/2}) ,
\]
which is compatible with parabolic induction and restriction and with twists
by depth-zero characters.
\end{thmintro}

There seem to be obstructions to generalizing this equivalence 
to categories of representations of arbitrary length, due to certain 2-cocycles 
in the Hecke algebras from \cite{Mor1} on the cuspidal level. 
On the other hand, for some special cases of groups and representations, 
an equivalence of categories of the form
\[
\Rep (G)_{\mf s} \; \cong \; \Mod \text{ - }\mc H (\mf s^\vee,q_F^{1/2})
\]
is known. See \cite{AMS3} for inner forms of $\GL_n (F)$, \cite{AMS4} for pure inner 
forms of quasi-split classical groups, \cite{SolLLCunip,SolRamif} for unipotent 
representations, \cite{Aubert-Xu-G2} for $\mathrm{G}_2 (F)$, \cite{Suzuki-Xu} for 
$\mathrm{GSp}_4(F)$, and \cite{SolQS} for principal series representations. 

Theorem \ref{thm:A} is in the spirit of recent geometric 
and categorical versions of a local Langlands correspondence
\cite{Hel,Zhu,BZCHN,FaSc}, where the objects on the Galois (or spectral)
side are equivariant coherent sheaves on stacks of Langlands 
parameters, and one must pass to derived categories on both sides of the
(conjectural) correspondence to formulate the conjecture. The construction of 
$\cH (\mf s^\vee,q_F^{1/2})$ in \cite{AMS3} strongly suggests that its
modules are related to such equivariant coherent sheaves, but it has proven
difficult to make that precise. 

In Section \ref{sec:DL}, we prove new results on Deligne--Lusztig packets of supercuspidal
$L$-representations, which in the end show that they behave well with respect to conjugation
by $N_G (L)$. In \S\ref{sec:L0}, we conduct a closer analysis on the representations of the 
component groups of supercuspidal L-parameters for $L$, related to conjugation by 
$N_{G^\vee}(L^\vee)$. On both sides of the LLC, it involves checking that certain extensions of
groups split equivariantly (see \S \ref{par:2.split} and \S \ref{par:3.split} for details).
Using this, we are able to (in \S \ref{sec:LLCcusp}) even prove new results about the LLC 
on the cuspidal level from \cite{Kal3}.

\begin{thmintro}\label{thm:C} \textup{(See Theorem \ref{thm:8.2})} \\
Identify $\big( Z(L^\vee)^{\mb I_F} \big)_{\mb W_F}^{\; \circ}$ with the set of Langlands
parameters for the group of unramified characters $\mf{X}_\nr (L)$.
In the LLC for non-singular supercuspidal $L$-representations, 
the choices can be made so that the bijection
\[
\Irr^0_\cusp (L)_{ns} \longleftrightarrow \Phi^0_\cusp (L)_{ns}
\]
is equivariant for the natural actions of
\[
W(G,L) \ltimes \mathfrak{X}_{\nr} (L) \cong 
W(G^\vee,L^\vee )^{\mb W_F} \ltimes \big( Z(L^\vee)^{\mb I_F} \big)_{\mb W_F}^{\; \circ} .
\]
\end{thmintro}

\textbf{Outline and remarks of strategy.}  \\
Theorem \ref{thm:C} provides in particular a bijection between:
\begin{itemize}
\item the set of inertial equivalence classes $\mf s = [L,\tau]_G$ for $\Rep (G)$, 
such that $\tau \in \Irr^0_\cusp (L)_{ns}$ for some Levi subgroup $L \subset G$,
\item the set of inertial equivalence classes 
$\mf s^\vee = \big( Z(L^\vee)^{\mb I_F} \big)_{\mb W_F}^{\; \circ} \cdot (\varphi_L,\rho_L)$ 
for $\Phi_e (G)$, such that $(\varphi_L,\rho_L) \in \Phi^0_\cusp (L)_{ns}$.
\end{itemize}
We will denote this bijection simply by
\begin{equation}\label{eq:6}
\mf s \longleftrightarrow \mf s^\vee .    
\end{equation}
It allows us to pass freely between the set of Bernstein components $\Irr (G)_{\mf s}$
of $\Irr^0 (G)_{ns}$ and the set of Bernstein components $\Phi_e (G)^{\mf s^\vee}$ of
$\Phi_e^0 (G)_{ns}$. 

Sections \ref{sec:IW}--\ref{sec:HeckeG} study Hecke algebras for $p$-adic groups. 
These sections are logically independent from Sections \ref{sec:DL}--\ref{sec:LLCcusp}.
For a non-singular depth-zero Bernstein block $\Rep (G)_{\mf s}$, the work of Morris 
\cite{Mor1,Mor3} provides us with a type $(\hat P_\ff, \hat \sigma)$, where $\hat P_\ff$ 
denotes the pointwise stabilizer of a facet $\ff$ in the Bruhat--Tits building of $G$. 

Extending results of Morris, we show that $\cH (G,\hat P_\ff, \hat \sigma)$ is a 
crossed product of an affine Hecke algebra and a twisted group algebra (see Theorem 
\ref{thm:3.1}). Since we already understand the set of cuspidal supports for $\Rep (G)_{\mf s}$, 
we only need to further consider two aspects of $\cH (G,\hat P_\ff, \hat \sigma)$: 
the $q$-parameters of the simple reflections from the associated finite root system 
$R_{\hat \sigma}$, and the 2-cocycle by which the group algebra has been twisted. 

Let $\Pi (L,T,\theta)$ be a Deligne--Lusztig packet (see \eqref{eq:7.8}) containing a 
representation in the set of cuspidal supports for $\Rep (G)_{\mf s}$. We show in 
Proposition \ref{prop:4.8} that the $q$-parameters of $\cH (G,\hat P_\ff, \hat \sigma)$ are 
equal to the $q$-parameters of a Hecke algebra for suitable principal series representations 
of a quasi-split reductive subgroup $G_{\hat \sigma} \subset G$ (see \eqref{eqn:defn-G_sigma}) 
with $T$ as minimal Levi subgroup. The argument runs mainly via similar Hecke algebras for 
finite reductive groups. These $q$-parameters for $G_{\hat \sigma}$ can be computed
explicitly from $(T,\theta)$ \cite{SolParam}.

The comparison between Hecke algebras on the $p$-adic side and on the Galois side of 
the Langlands correspondence is done in Sections
\ref{sec:HeckeL}--\ref{sec:equiv}.
On the Galois side, the twisted affine Hecke algebra $\cH (\mf s^\vee, q_F^{1/2})$
involves a finite root system $R_{\mf s^\vee}$ and $q$-parameters, defined in completely
different terms from complex algebraic geometry. Fortunately, these parameters can also 
be reduced to the case of $(G_{\hat \sigma},T,\theta)$, already studied in \cite{SolQS}. 
In \eqref{eq:9.13}, we establish a canonical isomorphism of root systems 
\begin{equation}\label{eq:7}
R_{\hat \sigma} \cong R_{\mf s^\vee} ,    
\end{equation}
and we show that the $q$-parameters on both sides agree.
The further comparison of the Hecke algebras 
is more difficult. Recall that Bernstein associated a finite group
\begin{equation}\label{eqn:finite-group-attached-to-Bernstein-component}
W_{\mf s} := \mr{Stab}_{W(G,L)} (\Rep (L)_{\mf s})
\end{equation}
to $\Rep (G)_{\mf s}$. Similarly, one can associate a finite group
\[
W_{\mf s^\vee} := 
\mr{Stab}_{W(G^\vee,L^\vee)^{\mb W_F}} \big( \Phi_e (L)^{\mf s^\vee} \big)
\] 
to $\Phi_e (G)^{\mf s^\vee}$.~By Theorem \ref{thm:C} there is a canonical isomorphism 
(see also \cite{Aubert-Xu-Hecke-algebra}) 
\begin{equation}\label{eq:8}
W_{\mf s} \, \cong \, W_{\mf s^\vee}.     
\end{equation}
Let $\Gamma_{\mf s}$ be the stabilizer in $W_{\mf s}$ of the set of positive 
roots in $R_{\hat \sigma}$, and define $\Gamma_{\mf s^\vee} \subset W_{\mf s^\vee}$
analogously. Using \eqref{eq:7}, we can decompose \eqref{eq:8} as
\begin{equation}\label{eq:9}
W_{\mf s} = W(R_{\hat \sigma}) \rtimes \Gamma_{\mf s} ,\;
W_{\mf s^\vee} = W(R_{\mf s^\vee}) \rtimes \Gamma_{\mf s^\vee} ,\;
W(R_{\hat \sigma}) \cong W(R_{\mf s^\vee}) ,\; \text{and}\;
\Gamma_{\mf s} \cong \Gamma_{\mf s^\vee} .
\end{equation}
The algebra $\cH (\mf s^\vee, q_F^{1/2})$ can be written as 
\[
\cH (\mf s^\vee, q_F^{1/2})^\circ \rtimes 
\C [\Gamma_{\mf s^\vee}, \natural_{\mf s^\vee}] ,
\]
where $\cH (\mf s^\vee, q_F^{1/2})^\circ$ is an affine Hecke algebra and
$\natural_{\mf s^\vee}$ is a 2-cocycle of $\Gamma_{\mf s^\vee}$. While  
$\cH (\mf s^\vee, q_F^{1/2})^\circ$ is canonically isomorphic to a subalgebra of 
$\cH (G,\hat P_\ff ,\hat \sigma)$
(see Lemma \ref{lem:3.7} and Proposition \ref{prop:9.4}), $\Gamma_{\mf s}$ appears
only indirectly in $\cH (G,\hat P_\ff ,\hat \sigma)$. One can instead replace
$\cH (G,\hat P_\ff ,\hat \sigma)$ with $\End_G (\Pi_{\mf s})$, where $\Pi_{\mf s}$
is a canonical progenerator of $\Rep (G)_{\mf s}$ constructed by Bernstein. The
algebras $\cH (G,\hat P_\ff ,\hat \sigma)$ and $\End_G (\Pi_{\mf s})$ are Morita
equivalent, but for various reasons it is easier to work with the latter \cite{SolEnd}. 
In general, however, $\End_G (\Pi_{\mf s})$ still does not contain a twisted group algebra 
of $\Gamma_{\mf s}$. To introduce at least a subgroup of $\Gamma_{\mf s}$ into the picture, 
we localize $\End_G (\Pi_{\mf s})$ with respect to suitable sets of characters of its centre. 
Theorem \ref{thm:C} provides an isomorphism
\begin{equation}
Z \big( \End_G (\Pi_{\mf s}) \big) \cong \mc O \big( \Irr (L)_{\mf s}
\big)^{W_{\mf s}} \cong \mc O \big( \Phi_e (L)^{\mf s^\vee} \big)^{W_{\mf s^\vee}} 
\cong Z \big( \cH (\mf s^\vee, q_F^{1/2}) \big),
\end{equation}
so we can localize $\cH (\mf s^\vee, q_F^{1/2})$ with respect to the corresponding
set of central characters. (This localization technique does not work well for 
representations of infinite length, so from here on we restrict to finite length modules.) 
In Proposition \ref{prop:9.7}, we show that Theorem \ref{thm:C} and \eqref{eq:9} induce an
algebra isomorphism of the form
\begin{equation}\label{eq:10}
\text{localized version of } \End_G (\Pi_{\mf s}) \; \cong \;
\text{localized version of } \cH (\mf s^\vee, q_F^{1/2}) .
\end{equation}
In fact, both sides of \eqref{eq:10} can be described in terms of twisted graded Hecke 
algebras. On the right-hand side of \eqref{eq:10}, the twist is given by the restriction of 
$\natural_{\mf s^\vee}$ to a subgroup of $W_{\mf s^\vee}$. The twisted graded Hecke algebra 
on the left-hand side of \eqref{eq:10} involves a 2-cocycle of a subgroup of $W_{\mf s}$ as in 
\cite[Proposition 7.3]{SolEnd}. The comparison of the 2-cocycles on both sides of \eqref{eq:10} 
is indeed the most difficult step of the paper. It is finally achieved in Theorem \ref{thm:9.6},
using the technical ingredients we established in Appendices \ref{sec:A} and \ref{sec:B}.
Combining cases of \eqref{eq:10} gives equivalences of categories
\begin{equation}\label{eq:11}
\Rep_\fl (G)_{\mf s} \;\cong\; \Mod_\fl \text{ - }\End_G (\Pi_{\mf s}) \;\cong\;
\Mod_\fl \text{ - }\cH (\mf s^\vee, q_F^{1/2});
\end{equation}
see Theorem \ref{thm:10.3}.~Using \eqref{eq:6}, one deduces Theorem 
\ref{thm:A}.~We then obtain our bijective LLC using the parametrization
of $\Irr \text{-}\cH (\mf s^\vee, q_F^{1/2})$ from \cite[Theorem 3.18]{AMS3}, which
concerns left $\cH (\mf s^\vee, q_F^{1/2})$-modules whereas in Theorem \ref{thm:10.4}
we translate to right modules of \eqref{eq:11}.
Finally, we prove the list of properties of our LLC in \S\ref{par:prop}.\\

\textbf{Open problems and outlook.}\\
Clearly it would be desirable to make our LLC for non-singular depth-zero 
representations canonical (including the enhancements). To this end, the input
would have to include a Whittaker datum for the quasi-split inner form of $G$.
However, this is not enough, even at the cuspidal level. At the moment, our LLC,
or that from \cite{Kal2,Kal3}, is not specified uniquely by a Whittaker datum; 
more requirements would be needed. This would possibly involve character 
formulas and endoscopy, as in \cite{FiKaSp}, in combination with a better 
understanding of the traces of the representations in question.

In another direction, one could try to make our LLC functorial with respect to
homomorphisms $f : \mc H \to \mc G$ of reductive $F$-groups such that both
ker $f$ and coker $f$ are commutative. The desired outcome was already 
conjectured in \cite{Bor,SolFunct}, and has been proven in the cuspidal cases in 
\cite{BoMe}. This would require some alignment between the local Langlands 
correspondences for $\Irr^0 (G)_{ns}$ and $\Irr^0 (H)_{ns}$, which would 
render them more canonical.

A local Langlands correspondence for non-singular supercuspidal representations
of positive depth was established simultaneously with the one in depth zero 
\cite{Kal2,Kal3}, for groups $G$ that split over a tamely ramified extension of
$F$. Types for Bernstein blocks of non-singular representations of such groups
are known from \cite{KiYu}. Recently it was shown \cite{AFMO1,AFMO2} that
the Hecke algebras from these Kim--Yu types are isomorphic to Hecke algebras
from depth zero types, as in \cite{Mor1,Mor3}. In view of these developments,
it is reasonable to expect that our LLC can be generalized to non-singular
representations of arbitrary depth.

In a similar manner, one expects that the methods developed in this paper will be
useful for the study of arbitrary depth-zero representations.\\

\textbf{Acknowledgements.}\\
We thank Anne-Marie Aubert for some helpful comments. Y.X.~was partially supported 
by the U.S. National Science Foundation under Award No.~2202677.

\section{Deligne--Lusztig packets for \texorpdfstring{$p$}{p}-adic groups}
\label{sec:DL}

Let $F$\label{i:1} be a non-archimedean local field with residue field $k_F$ \label{i:2}.
Let $\mc L$\label{i:51} be an $F$-Levi subgroup of a larger connected reductive 
$F$-group $\mc G$\label{i:3}. We write \label{i:4} $G = \mc G (F), L = \mc L (F)$ etc.  
Let $\mc L_\ad$\label{i:10} be the adjoint group of $\mc L$, 
and let $\mc B (\mc L_\ad,F) = \mc B (L)$\label{i:11} be the semisimple Bruhat--Tits 
building of $L$. Let $Z(\mc L)$\label{i:12} be the centre of $\mc L$, and
$Z^\circ (\mc L)$\label{i:13} its neutral component. We write $Z^\circ (L) =
Z^\circ (\mc L)(F)$, and let $X_* (Z^\circ (L))$\label{i:14} be its lattice of 
$F$-rational cocharacters. Recall that the Bruhat--Tits building 
$\mc B (\mc L,F) = \mc B (L)$\label{i:15} is the Cartesian product of 
$\mc B (L_\ad)$ and $X_* (Z^\circ (L)) \otimes_\Z \R$.

Let $\mc T$\label{i:16} be an elliptic maximal $F$-torus in $\mc L$ which contains
a maximal unramified $F$-torus of $\mc L$. Let $\ff_L$\label{i:17}
be the facet of $\mc B (\mc L,F)$ corresponding to $\mc T (F)$. Recall that every facet of 
$\mc B (\mc L,F)$ is the Cartesian product of a facet in $\mc B (\mc L_\ad,F)$ and 
$X_* (Z^\circ (L)) \otimes_\Z \R$. We fix an embedding 
$\mc B (\mc L,F) \hookrightarrow \mc B (\mc G,F)$ 
that is admissible in the sense of \cite[Chapter 14]{KaPr}. We choose a facet 
$\ff$\label{i:18} of $\mc B (\mc G,F)$ that is open in $\ff_L$.

Let $P_\ff = G_{\ff,0} \subset G$\label{i:19} be the parahoric subgroup associated to $\ff$, with 
pro-unipotent radical denoted by \label{i:Uf} $G_{\ff,0+}$.~Then $P_\ff / G_{\ff,0+}$ can be viewed as 
the $k_F$-points of a connected reductive group.~More precisely, by \cite[\S 5.2]{BrTi2}, 
there is a model $\mc P_\ff^\circ$ of $\mc G$ over the ring of
integers $\mf o_F$\label{i:20}, such that $P_\ff = \mc P_\ff^\circ (\mf o_F)$. Then 
$\mc G^\circ_\ff (k_F) := P_\ff / G_{\ff,0+}$\label{i:21} 
is the maximal reductive quotient of $\mc P_\ff^\circ (k_F)$.~Let \label{i:22}$\hat P_\ff$
be the pointwise stabilizer of $\ff$ in $G$, it contains $P_\ff$ with finite index. 
Since $P_\ff$ is a characteristic subgroup of $\hat P_\ff$, these two have the same 
normalizer in $G$, i.e.~we have\label{i:23}
\begin{equation}
G_\ff:=\mr{Stab}_G (\ff)=N_G (P_\ff) = N_G (\hat P_\ff).
\end{equation}
By \cite[Remark 8.3.4 and \S 9.2.5]{KaPr}, there exists an $\mf o_F$-group scheme 
$\mc P_\ff$, which is locally of finite type but not always affine, such that 
$\mc P_\ff (\mf o_F) = G_\ff$.
It gives rise to a $k_F$-group scheme $\mc G_\ff$\label{i:24} satisfying 
$\mc G_\ff (k_F) = G_\ff / G_{\ff,0+}$. 
This contains $\hat P_\ff / G_{\ff,0+}$ as the group of $k_F$-rational points of a 
(possibly disconnected) reductive subgroup $\hat{\mc G}_\ff \subset \mc G_\ff$.
Similar notations will be used for $\mc L$, but they only depend on the larger
facet $\ff_L$. We shall write $P_{L,\ff} := L \cap P_\ff$ 
instead of $P_{\ff_L}$.\label{i:25}

In $G_\ff = \mc P_\ff (\mf o_F)$ we also have $T = \mc T (F)$ as the $\mf o_F$-points
of a subgroup scheme of $\mc P_\ff$. In this way $\mc T$ can be viewed as an
$\mf o_F$-group scheme. The $\mf o_F$-torus \label{i:26}
$\mc T_\ff := \mc T \cap \mc G_\ff^\circ$ is (considered
over $F$) a maximal unramified torus in $\mc L$ 
and in $\mc G$. Since $\mc G$ becomes quasi-split over an unramified extension of $F$,
$Z_{\mc G}(\mc T_\ff)$ is a maximal torus of $\mc G$, and thus it must be $\mc T$.
By the ellipticity of $\mc T$, the maximal $F$-split subtorus $\mc T_s$\label{i:95} 
of $\mc T$ is contained in $Z(\mc L)^\circ$, thus we have 
\begin{equation}\label{eq:4.13}
\mc L = Z_{\mc G}(Z(\mc L)^\circ) = Z_{\mc G}(\mc T_s) .    
\end{equation}
A character of $T = \mc T (F)$ is said to have depth zero if it is trivial on
$\ker (\mc T_\ff (\mf o_F) \to \mc T_\ff (k_F))$. By the construction of $\mc P_\ff$, 
this kernel equals $\ker (\mc T (\mf o_F) \to \mc T (k_F))$.
Consider a depth-zero character $\theta$\label{i:27} of $T$, or equivalently a character 
$\theta$ of $\mc T (k_F)$.~Throughout this section, we assume that $\theta$ 
is $F$-non-singular for $(\mc T,\mc L)$ in the sense of \cite[Definition 3.1.1]{Kal3}. 
It means that, for any unramified extension $E/F$ and any coroot 
$\alpha^\vee$ of $(\mc L (E),\mc T (E))$, the character 
\[
\theta \circ (\text{norm map for } E/F \text{ on } \mc T ) \circ \alpha^\vee : 
E^\times \to \C^\times
\]
is nontrivial on $\mf o_E^\times$. As mentioned in \cite[3.1.4]{Kal3}, 
$\theta_\ff := \theta |_{\mc T_\ff (k_F)}$\label{i:28} is non-singular for 
$(\mc T_\ff (k_F), \mc L_\ff^\circ (k_F))$ in the sense of \cite[Definition 5.15]{DeLu}, 
that is, $\theta_\ff$ is not orthogonal to any coroot of $(\mc L_\ff^\circ, \mc T_\ff)$.
Compared to \cite{Kal2,Kal3}, we do not require that $\mc T$ splits over a tamely ramified
extension of $F$.

From the data $(\mc L_\ff, \mc T, \theta)$, one can build a Deligne--Lusztig representation 
$\mc R_{\mc T (k_F)}^{\mc L_\ff (k_F)} (\theta)$\label{i:29} of $\mc L_\ff (k_F)$ 
(see for example \cite{DeLu} and \cite[\S 2]{Kal3}), in the same way as for the connected 
group $\mc G_\ff^\circ (k_F)$. 
It is a virtual representation of $\mc L_\ff (k_F)$, but $\pm \mc R_{\mc T (k_F)}^{
\mc L_\ff (k_F)} (\theta)$\label{i:30} is an actual representation for a 
suitable sign $\pm$. By \cite[Corollary 2.6.2]{Kal3}, $\pm \mc R_{\mc T (k_F)}^{
\mc L_\ff (k_F)} (\theta)$ is a quotient of
\[
\ind_{\mc L_\ff^\circ (k_F)}^{\mc L_\ff (k_F)} 
\big( \pm \mc R_{\mc T_\ff (k_F)}^{
\mc L_\ff^\circ (k_F)} \theta_\ff \big) \cong
\pm \mc R_{\mc T (k_F)}^{\mc L_\ff (k_F)}
\big( \ind_{\mc T_\ff (k_F)}^{\mc T (k_F)} \theta_\ff \big) .
\]
Moreover $\pm \mc R_{\mc T (k_F)}^{
\mc L_\ff (k_F)} (\theta)$ is a representation of $\mc L_\ff (k_F) \times
\mc T (k_F)$, where $\mc L_\ff (k_F)$ acts from the left and $\mc T (k_F)$ acts from the
right via the character $\theta$. The action of 
$Z(\mc L_\ff)(k_F) \subset \mc T (k_F)$ is the same from the left and from the right, therefore,
\begin{equation}\label{eq:7.64}
Z(\mc L_\ff)(k_F) \text{ acts on } \pm \mc R_{\mc T (k_F)}^{\mc L_\ff (k_F)} (\theta)
\text{ via } \theta |_{Z(\mc L_\ff)(k_F)}.
\end{equation}
We define the \textit{Deligne--Lusztig packet} 
\[
\Pi \big(\mc L_\ff (k_F), \mc T (k_F), \theta\big) \subset \Irr (\mc L_\ff (k_F))
\]
as the set of irreducible constituents of $\pm \mc R_{\mc T (k_F)}^{\mc L_\ff (k_F)} (\theta)$.
Let $N_{\mc L_\ff (k_F)} (\mc T)_\theta$\label{i:31} be the stabilizer of $\theta$ in 
\label{i:9} $N_{\mc L_\ff (k_F)} (\mc T)$. Let $\Irr (N_{\mc L_\ff (k_F)} (\mc T)_\theta, \theta)$ 
be the set of irreducible representations of $N_{\mc L_\ff (k_F)} (\mc T)_\theta$ whose
restriction to $\mc T (k_F)$ contains $\theta$. The group $N_{\mc L_\ff (k_F)} (\mc T)_\theta$
acts on $\pm \mc R_{\mc T (k_F)}^{\mc L_\ff (k_F)} (\theta)$ by $\mc L_\ff (k_F)$-intertwiners,
constructed in \cite[(2.18)]{Kal3}. First, canonical $\mc L_\ff (k_F)$-intertwining operators 
are exhibited, by geometric means. These respect the multiplication in 
$N_{\mc L_\ff (k_F)} (\mc T)_\theta$ only up to a scalars, and to combine them into an 
actual representation the geometric intertwining operators are normalized 
by the choice of a ``coherent splitting" \cite[Definition 2.4.9]{Kal3} which we indicate by
$\epsilon$. By \cite[Theorem 2.7.7.1]{Kal3}, there is a bijection
\begin{equation}\label{eq:7.30} 
\begin{array}{ccl}
\Irr \big(N_{\mc L_\ff (k_F)} (\mc T)_\theta, \theta\big) &
\to & \Pi \big(\mc L_\ff (k_F), \mc T (k_F), \theta\big) \\
\rho & \mapsto & \big( \rho \otimes \pm \mc R_{\mc T (k_F)}^{\mc L_\ff (k_F)} 
(\theta)^\epsilon \big)^{N_{\mc L_\ff (k_F)} (\mc T)_\theta}
\end{array}.
\end{equation}
Let $\Rep (L)$\label{i:34} be the category of smooth $L$-representations on complex vector
spaces. We recall that a $L$-representation $(\pi,V)$ has depth zero if it is generated
by the union, over all facets $\ff$ of $\mc B (\mc L,F)$, of the subspaces 
$V^{\pi (G_{\ff,0+})}$. We denote the full subcategory of $\Rep (L)$ formed by depth-zero
representations by $\Rep^0 (L)$. Let $\Irr (L)$\label{i:36} be the set of irreducible 
$L$-representations in $\Rep (L)$ (up to isomorphism). 
For the $p$-adic group $L$, we define the \textit{Deligne--Lusztig packet} \label{i:32}
\begin{equation}\label{eq:7.8}
\Pi (L,T,\theta) := \big\{ \ind_{L_\ff}^L (\sigma') : \sigma' \text{ is a constituent of }
\mr{inf}_{\mc L_f (k_F)}^{L_\ff} \big( \pm \mc R_{\mc T (k_F)}^{\mc L_\ff (k_F)} (\theta) 
\big) \big\} 
\end{equation}
in $\Irr (L)$.~More precisely, \label{i:33} 
\[
\Pi (L,T,\theta) \; \subset \;\Irr^0 (L):= \{ \pi \in \Irr (L) : \pi \text{ has depth zero} \}.
\]
By definition, a supercuspidal $L$-representation of depth zero is non-singular if
and only if it belongs to one of the packets $\Pi (L,T,\theta)$.
By \eqref{eq:7.64}, we know that
\begin{equation}\label{eq:7.65}
\text{every } \pi \in \Pi (L,T,\theta) \text{ admits the central character }
\theta|_{Z(L)}.
\end{equation}
By \cite[Proposition 6.6]{MoPr2}, we have a bijection
\[
\ind_{L_\ff}^L \mr{inf}_{\mc L_f (k_F)}^{L_\ff} : 
\Pi (\mc L_\ff (k_F), \mc T (k_F), \theta) \to \Pi (L,T,\theta) .
\]
Let $\Irr \big(N_L (T)_\theta,\theta\big)$\label{i:37} be the set of irreducible representations 
of $N_L (T)$ whose restriction to $T$ contains $\theta$ (or equivalently, on which $T$ 
acts via the character $\theta$).~As explained in \cite[\S 2.7 and \S 3.3]{Kal3}, 
there is a bijection\label{i:63}
\begin{equation}\label{eq:7.21}
\Irr (N_L (T)_\theta,\theta) 
\to 
\Pi (L,T,\theta),\quad
\rho 
\mapsto 
\big( \rho \otimes \kappa_{(T,\theta)}^{L,\epsilon} \big)^{N_L (T)_\theta} 
=: \kappa_{T,\theta,\rho}^{L,\epsilon},
\end{equation}
where $\kappa_{(T,\theta)}^{L,\epsilon}:=\mr{ind}_{L_\ff}^L \mr{inf}^{L_\ff}_{\mc L_\ff (k_F)}
\big( \pm \mc R_{\mc T (k_F)}^{\mc L_\ff (k_F)} (\theta)^\epsilon \big)$\label{i:38} is an 
$L \times N_L (T)_\theta$-representation. The action of $N_L (T)_\theta$ factors through 
$N_{\mc L_\ff (k_F)}(\mc T)_\theta$ and is induced from the action on 
$\pm \mc R_{\mc T (k_F)}^{\mc L_\ff (k_F)} (\theta)^\epsilon$ in \eqref{eq:7.30}.

\begin{lem}\label{lem:7.13}
The supercuspidal $L$-representation $\kappa_{T,\theta,\rho}^{L,\epsilon}$, as defined in 
\eqref{eq:7.21}, is tempered if and only if $\theta$ is unitary. 
\end{lem}
\begin{proof}
Any irreducible supercuspidal representation is tempered if and only if its central 
character is unitary. Since $T$ is a maximal torus of $L$, it contains $Z(L)$. We denote
the maximal compact subgroup of a torus over $F$ by a subscript cpt. Since
$T$ is elliptic, $Z^\circ (L) / Z^\circ (L)_\cpt$ is a finite-index subgroup of
$T / T_\cpt$.\label{i:cpt} Hence $\theta$ is unitary if and only if $\theta |_{Z(L)}$ is
unitary. The constructions of 
$\pm \mc R_{\mc T (k_F)}^{\mc L_\ff (k_F)} (\theta)$ 
and 
$\kappa_{(T,\theta)}^{L,\epsilon}$ 
show that they admit central character $\theta |_{Z(L)}$. Hence so does   
$\kappa_{T,\theta,\rho}^{L,\epsilon}$.
\end{proof}

If $X$ is any set with an $N_G (L)$-action, then the group\label{i:39} 
$W(G,L) := N_G (L) / L$ 
acts naturally on the set of $L$-orbits in $X$. Let $N_G (L,T)$\label{i:NGLT} be the largest 
subgroup of $G$ that normalizes both $L$ and $T$. The $W(G,L)$-stabilizer 
of the $L$-conjugacy class of $(T,\theta)$ can be expressed as\label{i:40}
\begin{equation}\label{eq:7.9}
W(G,L)_{(T,\theta)} \cong N_G (L,T)_\theta / N_L (T)_\theta . 
\end{equation}
The actions of $N_G (L,T)_\theta$ on the sets in \eqref{eq:7.21} are trivial on 
$N_L (T)_\theta$, thus by \eqref{eq:7.9}, they factor through $W(G,L)_{(T,\theta)}$. 
We note that $W(G,L)_{(T,\theta)}$ is a quotient of the stabilizer of $\theta$ in 
\label{i:41} $W(N_G (L),T) = N_G (L,T) / T$.

For characters of $L$, there are several reasonable notions of ``depth-zero''. It is not a priori obvious which one is the most appropriate, but fortunately they all coincide by
\cite[Theorem 1.4]{SoXu2}. Let $L_\Sc = \mc L_\Sc (F)$ be the simply 
connected cover of the derived group $L_\der = \mc L_\der (F)$. We abbreviate the cokernel of the canonical map $L_\Sc \to L$ as 
$L / L_\Sc$, and we consider the following group of characters: \label{i:Xo}
\[
\Xo (L) = \{ \chi : L / L_\Sc \to \C^\times \mid \chi |_T \text{ has depth zero
for all maximal tori } T \subset L \}.
\]
We showed in \cite[Theorem 3.4]{SoXu2}that $\Xo (L)$ is equal to the group of characters 
of $L$ that are trivial on the image of $L_\Sc \to L$ and on $L_{\ff_L,0+}$ for every 
facet $\ff_L$ of $\mc B (\mc L,F)$. An advantage of this notion of ``depth-zero'' for 
characters is that tensoring representations by elements of $\Xo (L)$ stabilizes 
$\Rep^0 (L)$.

Recall that a character $L \to \C^\times$ is called unramified if it is trivial on every 
compact subgroup of $L$. The group of unramified characters $\mathfrak{X}_{\nr} (L)$ \label{i:42} 
is essential for defining Bernstein blocks in $\Rep (L)$. For any maximal torus $T \subset L$,
the pro-$p$ radical $T_{0+}$ of the unique parahoric subgroup of $T$ is compact. Hence 
every unramified character $\chi : L \to \C^\times$ has $T_{0+}$ in its kernel, so
$\chi |_T$ has depth zero and $\chi \in \Xo (L)$.

More precisely, $\mf{X}_\nr (L)$ is a connected component of $\Xo (L)$. Every 
character of $L$ also defines a character of any inner form $L'$ of $L$. In this way, 
the groups $\mf{X}_\nr (L)$ and $\Xo (L)$ can be identified with their versions for $L'$.

The group $\Xo (L)$ acts on $\Irr (L)$ by tensoring, and this action preserves 
the set $\Irr^0 (L)$ of irreducible depth-zero representations of $L$. 
For $\chi \in \Xo (L)$, we have 
\[
N_L (T)_{\chi \otimes \theta} = N_L (T)_\theta \quad \text{and} \quad 
W(L,T)_{\chi \otimes \theta} = W (L,T)_\theta;
\]
similarly for other analogous sub-quotients of $L$. We define
$N_L (T)_{\Xo (L) \theta}$ (resp. $N_G (L,T)_{\Xo (L) \theta}$) to be the stabilizer of 
$\Xo (L) \otimes \theta$ in  $N_L (T)$ (resp.~$N_G (L,T)$). Set \label{i:43}
\begin{equation*}
W(L,T)_{\Xo (L) \theta} = N_L (T)_{\Xo (L) \theta} / T\; \text{and}\; 
W(N_G (L),T)_{\Xo (L) \theta} = N_G (L,T)_{\Xo (L) \theta} / T.
\end{equation*} 
Likewise, let $W(G,L)_{(T,\Xo (L) \theta)}$ be the 
stabilizer of $L \cdot (T,\Xo (L) \theta)$ in $W(G,L)$, which is isomorphic to 
$W(N_G (L),T)_{\Xo (L) \theta} / W(L,T)_{\Xo (L) \theta}$.
Notice that $N_G (L,T)_{\Xo (L) \theta}$ normalizes $N_L (T)_\theta$ and that 
\[
n \cdot \Irr \big(N_L (T)_\theta, \theta\big) = \Irr \big(N_L (T)_\theta, n \cdot 
\theta\big) \quad \text{for any } n \in N_G (L,T)_{\Xo (L) \theta}.
\]
Tensoring a representation with a character does not change its space of
self-intertwiners, so in \eqref{eq:7.21} we can
\begin{equation}\label{eq:7.4}
\text{pick the same coherent splitting } \epsilon \text{ for all } 
\theta' \in \Xo (L) \theta.
\end{equation}

\begin{prop}\label{prop:7.11}
Under \eqref{eq:7.4}, the collection of bijections \eqref{eq:7.21} for all 
$\theta' \in \Xo (L) \theta$ is $W(G,L)_{(T,\Xo (L) \theta)}$-equivariant. 
In particular, the bijection \eqref{eq:7.21} is $W(G,L)_{(T,\theta)}$-equivariant.
\end{prop}
\begin{proof}
Let $\mc U$ be the unipotent radical of a Borel subgroup of $\mc L_\ff^\circ$ containing 
$\mc T_\ff$.~The representation $\pm \mc R_{\mc T (k_F)}^{\mc L_\ff (k_F)} (\theta)^\epsilon$ 
is defined on the vector space $H^{d_\mc U}_c (Y_{\mc U}^{\mc L_\ff}, \overline{\Q}_\ell 
)_\theta$, which arises from the variety 
\[
Y_{\mc U}^{\mc L_\ff} := 
\{ l \mc U \in \mc L_\ff / \mc U : l^{-1} \Fr (l) \in \mc U \cdot \Fr \mc U \},
\] 
see \cite[\S 2.6]{Kal3}. It is viewed as a complex representation via a fixed field isomorphism
$\overline{\Q}_\ell \cong \C$. For $g \in N_G (L,T)_{\Xo (L) \theta}$, the map 
$l \mc U \mapsto g l \mc U g^{-1}$ induces a linear bijection
\begin{equation}\label{eqn:cohom-bij}
H^{d_\mc U}_c (Y_\mc U^{\mc L_\ff}, \overline{\Q}_\ell )_\theta \xrightarrow{\sim}
H^{d_\mc U}_c (Y_{g \mc U g^{-1}}^{\mc L_\ff}, \overline{\Q}_\ell )_{g \cdot \theta}.
\end{equation}
As in \cite[(2.18)]{Kal3}, we compose this with $\epsilon \Psi^{\mc L_\ff}_{\mc U,g \mc U g^{-1}}$
to land in $H^{d_\mc U}_c (Y_\mc U^{\mc L_\ff}, \overline{\Q}_\ell )_{g \cdot \theta}$. 
Here $\Psi^{\mc L_\ff}_{\mc U,g \mc U g^{-1}}$ is obtained from a canonical geometric
construction \cite[(2.17)]{Kal3} and it is normalized by means of a ``coherent splitting"
$\epsilon$. In the process, the $N_{\mc L_\ff (k_F)} (\mc T)_\theta$-action from 
\cite[(2.18)]{Kal3} is precomposed with conjugation by $g$, which we indicate by 
a superscript $\epsilon \circ g^{-1}$. This allows us to define an isomorphism of 
$\mc L_\ff (k_F) \times N_{\mc L_\ff (k_F)} (\mc T)_\theta$-representations
\begin{equation}\label{eq:7.28}
g \cdot \pm \mc R_{\mc T (k_F)}^{\mc L_\ff (k_F)} (\theta)^\epsilon \; \isom \;
\pm \mc R_{\mc T (k_F)}^{\mc L_\ff (k_F)} (g \cdot \theta)^{\epsilon \circ g^{-1}},
\end{equation}
which is canonical once $\epsilon$ has been chosen.
Now by \eqref{eq:7.21}, we have
\begin{equation}\label{eq:7.26}
g \cdot \kappa_{(T,\theta,\rho)}^{L,\epsilon} \cong \big( \rho \otimes 
\kappa_{(T, g \cdot \theta)}^{L,\epsilon \circ g^{-1}} \big)^{N_L (T)_\theta} =
\big( g \cdot \rho \otimes \kappa_{(T,g \cdot \theta)}^{L,\epsilon} \big)^{N_L (T)_\theta} =
\kappa_{(T,g \cdot \theta,g \cdot \rho)}^{L,\epsilon}. \qedhere
\end{equation} 
\end{proof}

The collection of bijections considered in Proposition \ref{prop:7.11} is also 
$\Xo (L)$-equivariant. Namely, by \cite[Theorem 2.7.7]{Kal3}
\begin{equation}\label{eq:7.2}
\chi \otimes \kappa_{(T,\theta,\rho)}^{L,\epsilon} \cong \kappa_{(T,\chi \otimes 
\theta,\chi \otimes \rho)}^{L,\epsilon} \qquad \chi \in \Xo (L).
\end{equation}
Let $\sigma$\label{i:107} be a constituent of the Deligne--Lusztig representation
$\pm R_{\mc T_\ff (k_F)}^{\mc L_\ff^\circ (k_F)} (\theta_\ff)$, which decomposes
into mutually inequivalent subrepresentations:
\[
\pm R_{\mc T_\ff (k_F)}^{\mc L_\ff^\circ (k_F)} (\theta_\ff) =
\bigoplus\nolimits_{\chi \in \Irr (\Omega_{\theta_\ff})} \sigma_\chi \qquad
\Omega_{\theta_\ff} = W(\mc L_\ff^\circ ,\mc T_\ff)(k_F)_{\theta_\ff} .
\]
By definition, these representations $\sigma_\chi$ form a Deligne--Lusztig packet for 
$\mc L_\ff^\circ (k_F)$. 
By inflation, we can also regard $\sigma$ and the $\sigma_\chi$ as representations
of $P_{L,\ff}$. Via the types $(P_{L,\ff},\sigma_\chi)$, they give rise to the category
\begin{equation}\label{eqn:RepL-subscript-PLf-sigma-chi}
\Rep (L)_{(\mc T_\ff,\theta_\ff)} = \bigoplus\nolimits_{\chi \in 
\Irr (\Omega_{\theta_\ff})} \Rep (L)_{(P_{L,\ff},\sigma_\chi)} .
\end{equation}
Consider the set 
$W(G,L)_{(\mc T_\ff,\theta_\ff)} := \big\{ w \in W(G,L) : w \cdot 
\Rep (L)_{(P_{L,\ff},\sigma)} \subset \Rep (L)_{(\mc T_\ff,\theta_\ff)} \big\}$.

\begin{lem}\label{lem:7.1}
$W(G,L)_{(\mc T_\ff,\theta_\ff)}$ is a group, isomorphic to
$N_G (P_{L,\ff},\mc T_\ff)_{\theta_\ff} / N_{L_\ff} (\mc T_\ff)_{\theta_\ff}$.
\end{lem}
\begin{proof}
The irreducible representations in $\Rep (L)_{(\mc T_\ff,\theta_\ff)}$ are 
$\ind_{L_\ff}^L (\tilde \sigma_\chi)$, where $\tilde \sigma_\chi$ is an extension of
$\sigma_\chi$ to $L_\ff$ and $\chi \in \Irr (\Omega_{\theta_\ff})$. More precisely:
\begin{equation}\label{eq:7.7}
\Irr (L)_{(\mc T_\ff,\theta_\ff)} = 
\bigcup\nolimits_{\chi \in \Irr (\Omega_{\theta_\ff})} \Irr (L)_{(P_{L,\ff},\sigma_\chi)} .
\end{equation}
The set $W(G,L)_{(\mc T_\ff,\theta_\ff)}$ sends $\Irr (L)_{(P_{L,\ff},\sigma)}$
to $\Irr (L)_{(\mc T_\ff,\theta_\ff)}$ and 
\[
w \cdot (P_{L,\ff},\sigma) = (w P_{L,\ff} w^{-1}, w \cdot \sigma).
\]
By the essential uniqueness of depth-zero types for supercuspidal representations 
\cite[Theorem 5.2]{MoPr1}, $(P_{L,\ff},\sigma)$ is uniquely determined up to 
$L$-conjugacy. Hence we can find a representative for $w$ in $N_G (P_{L,\ff})$.~Then
$w \cdot \sigma$ must be one of the $\sigma_\chi$, thus $w$ stabilizes the Deligne--Lusztig
series associated to $(\mc T_\ff,\theta_\ff)$.~Here $(\mc T_\ff,\theta_\ff)$ is
unique up to $\mc L_\ff^\circ (k_F)$-conjugacy, so we can even represent $w$ by an 
element of $N_G (P_{L,\ff}, \mc T_\ff)_{\theta_\ff}$. Conversely, every element 
of $N_G (P_{L,\ff}, \mc T_\ff)_{\theta_\ff}$ represents a class in 
$W(G,L)_{(\mc T_\ff,\theta_\ff)}$. Thus the natural group homomorphism 
$N_G (P_{L,\ff}, \mc T_\ff)_{\theta_\ff} \to W(G,L)$ 
has image $W(G,L)_{(\mc T_\ff,\theta_\ff)}$ and kernel 
$L \cap N_G (P_{L,\ff}, \mc T_\ff)_{\theta_\ff} = 
N_{L_\ff} (\mc T_\ff)_{\theta_\ff}$. 
\end{proof}

\begin{lem}\label{lem:7.6}
\enuma{
\item The stabilizer of $\Pi (L,T,\theta)$ inside $W(G,L)$ equals $W(G,L)_{(T,\theta)}$. It is a 
subgroup of $W(G,L)_{(\mc T_\ff,\theta_\ff)}$, where $\theta_\ff = \theta |_{\mc T_\ff (k_F)}$.
\item If $\Pi (L,T,\theta)$ contains a representation fixed by $W(G,L)_{(\mc T_\ff,
\theta_\ff)}$, then $W(G,L)_{(T,\theta)}$ equals $W(G,L)_{(\mc T_\ff,\theta_\ff)}$.
}
\end{lem}
\begin{proof}
(a) The construction of $\Pi (L,T,\theta)$ implies that any element of $W(G,L)$ that 
stabilizes the $L$-conjugacy class of $(T,\theta)$ also stabilizes $\Pi (L,T,\theta)$. 
For an arbitrary Deligne--Lusztig packet $\Pi (L,\tilde T, \tilde \theta)$, we claim that 
\begin{itemize}
\item[(i)] $\Pi (L,\tilde T, \tilde \theta)=\Pi (L,T,\theta)$ if $(T,\theta)$ and 
$(\tilde T, \tilde \theta)$ are $L$-conjugate;
\item[(ii)] $\Pi (L,\tilde T, \tilde \theta)$ is disjoint from $\Pi (L,T,\theta)$ if $(T,\theta)$ 
and $(\tilde T, \tilde \theta)$ are not $L$-conjugate.
\end{itemize}
Indeed, by \cite[Proposition 2.6.11]{Kal3}, this holds for the group $\mc L_\ff (k_F)$ instead
of for $L$. This statement transfers to $L_\ff$ by inflation of representations.~Consider
any $\pi = \ind_{L_\ff}^L (\sigma') \in \Pi (L,T,\theta)$ as in \eqref{eq:7.8}. By 
\cite[\S 6]{MoPr2}, $\pi$ determines the $L$-conjugacy class of $(L_\ff,\sigma')$. Hence
the validity of (i) and (ii) extends from $L_\ff$ to $L$.

Consequently, any element of $W(G,L)$ that sends a member of $\Pi(L,T,\theta)$ into
$\Pi (L,T,\theta)$ must stabilize the $L$-conjugacy class of $(T,\theta)$. Then it also
stabilizes the $L$-conjugacy class of $(\mc T_\ff, \theta_\ff)$, and thus it belongs to
$W(G,L)_{(\mc T_\ff,\theta_\ff)}$. 

(b) For any $w \in W(G,L)_{(\mc T_\ff,\theta_\ff)}$, by the definition of these 
Deligne--Lusztig packets, we have 
$w \cdot \Pi (L,T,\theta) = \Pi (L,w T w^{-1}, w \cdot \theta)$, which contains an element of $\Pi (L,T,\theta)$ by assumption. By (i), we know that 
$(w T w^{-1}, w \cdot \theta)$ is $L$-conjugate to $(T,\theta)$, which implies that
$\Pi (L,w T w^{-1}, w \cdot \theta)$ equals $\Pi (L,T,\theta)$.
\end{proof}

The Weyl group $W(\mc L, \mc T)$\label{i:5} has the structure of a finite $F$-group, 
such that $N_L (T) / T$ is a subgroup of $W(\mc L, \mc T)(F)$.
Recall from \cite[Lemma 3.2.1]{Kal3} that $W(\mc L,\mc T)(F)_\theta$ is abelian.

\begin{lem}\label{lem:7.7}
The subgroup $W(\mc L,\mc T)(F)_\theta$ of $W(N_{\mc G}(\mc L),\mc T)(F)_\theta$ is central. 
\end{lem}
\begin{proof}
Recall that $\mc T_\ff = \mc T \cap \mc G_\ff^\circ$. 
Let $\mc T_{\ff,\ad}$ be the image of $\mc T_\ff$ in $\mc L^\circ_{\ff,\ad}$. 
By the proofs of \cite[Lemmas 2.2.1 and 3.2.1]{Kal3}, we have an embedding
\begin{equation}\label{eq:7.1}
W(\mc L,\mc T)(F)_\theta\hookrightarrow 
\Irr \big( \mr{coker} ( \mc T_\ff (k_F) \to \mc T_{\ff,\ad} (k_F) )\big) .
\end{equation}
This embedding is natural, and is in particular $W(\mc G_\ff ,\mc T_\ff) (k_F)_{\theta_\ff}
$-equivariant. One can easily see that 
the action of $W(\mc G_\ff ,\mc T_\ff) (k_F)_{\theta_\ff}$ 
on 
$\mc T_\ff$ 
lifts to the action of $W(\mc G,\mc T)$ 
on $\mc T$, which adjusts $\mc T$ by elements of $\Z R (\mc G,\mc T) \otimes_\Z F_s^\times$ for a separable 
closure $F_s$\label{i:44} of $F$. Thus the action 
of $W(\mc G_\ff ,\mc T_\ff) (k_F)_{\theta_\ff}$ on $\mc T_{\ff,\ad}$ 
only adjusts $\mc T_{\ff,\ad}(k_F)$ by elements of 
\begin{equation}\label{eq:7.3}
\big( \Z R (\mc G,\mc T) \cap X_* (\mc T_{\ff,\ad}) \otimes_\Z \overline{k}_F \big)^\Fr, 
\end{equation}
where $\Fr$\label{i:45} denotes the Frobenius automorphism of $\overline{k_F}/k_F$.~Since 
$\Z R (\mc G,\mc T) \cap X_* (\mc T_{\ff,\ad})$ is contained in $X_* (\mc T_\ff)$, all elements of
\eqref{eq:7.3} come from $\mc T_\ff (k_F)$. Hence $W(\mc G_\ff ,\mc T_\ff) (k_F)_{\theta_\ff}$
acts trivially on $\mr{coker}(\mc T_\ff (k_F) \to \mc T_{\ff,\ad} (k_F))$. Via the embedding
\eqref{eq:7.1}, the conjugation action of $W(\mc G_\ff ,\mc T_\ff) (k_F)_{\theta_\ff}$ on
$W(\mc L_\ff^\circ, \mc T_\ff) (k_F)_{\theta_\ff}$ is trivial.
\end{proof}
Now we study the structure of $W(\mc G,\mc T)(F)$ in greater detail. \label{i:46}

\begin{lem}\label{lem:7.14}
\enuma{
\item We have $W(\mc G,\mc T)(F) = W(N_{\mc G}(\mc L),\mc T)(F)$.
}
In the following, assume moreover that $\mc G (F)$ is quasi-split. Recall that $\ff_L$ is the 
facet in $\mc B (\mc L,F)$ containing $\ff$.
\enuma{ \setcounter{enumi}{1}
\item By replacing $\mc T$ within its stable conjugacy class for $\mc L$, we can achieve 
that there exists a point $y \in \ff_L$ whose image in $\mc B (\mc G_\ad, F)$ 
is a special vertex. 
\item Let $\mc G_y^\circ$ be the connected reductive $k_F$-group associated to the
facet $y$ of $\mc B (\mc G_\ad,F)$, and define $\mc L_y^\circ$ analogously. We have 
\[
W(\mc G,\mc T)(F) \cong W(\mc G_y^\circ, \mc T_\ff)(k_F) \cong 
W(N_{\mc G_y^\circ}(\mc L_y^\circ), \mc T_\ff)(k_F).
\]
}
\end{lem}
\begin{proof}
(a) As already noted in \eqref{eq:4.13}, we have $\mc L = Z_{\mc G}(\mc T_s)$,
where $\mc T_s$ denotes the maximal $F$-split subtorus of $\mc T$. Every element 
$w \in W (\mc G,\mc T)(F)$ normalizes $\mc T_s$, thus also normalizes $\mc L$. Hence
$w \in W(N_{\mc G}(\mc L),\mc T)(F)$. 

(b) By \cite[Lemma 3.4.12]{Kal2}, we can achieve (by changing $\mc T$ within its
stable conjugacy class) that the image of $\ff_L$ in $\mc B (\mc L_\ad ,F)$
is a special vertex. Thus for every $y_L \in \ff_L$, the root system
$R(\mc L_{y_L}^\circ (k_F), \mc T_\ff (k_F))$ equals $R(\mc L(F), \mc T_\ff (F))$.
Take a basis $\Delta_{\mc L}$ of $R(\mc L(F), \mc T_\ff (F))$, and extend it to a basis
$\Delta$ of $R(\mc G (F), \mc T_\ff (F))$. By the linear independence of 
$\Delta \setminus \Delta_{\mc L}$ in the character lattice of $Z(\mc L)^\circ$, for every 
$\alpha \in \Delta \setminus \Delta_{\mc L}$, we can translate $y_L$ 
by an element $t \in X_* (Z (\mc L)^\circ) \otimes_\Z \R$, such that $y := y_L + t$
lies in a wall of $\mc B (\mc G,F)$ in the direction $\alpha$. Then $y$ is special. 

(c) Part (b) implies the first isomorphism of Weyl groups. The second isomorphism follows
from this and part (a). (One can also rephrase the proof of part (a) so that it applies
to the $k_F$-group $\mc G_y^\circ$.)
\end{proof}

We warn the reader that the $k_F$-group $\mc G_y^\circ$ from Lemma \ref{lem:7.14} is in
general bigger than $\mc G_\ff^\circ$. Via  the isomorphism 
$\mc T_\ff (k_F) \cong X_* (\mc T_\ff )_{\Fr_F}$ 
from \cite[(5.2.3)]{DeLu}, we can view $\theta_\ff$ as a character of $X_* (\mc T_\ff)$.
The values of $\theta_\ff$ belong to group of roots of unity in $\C^\times$, which is 
isomorphic to $\Q / \Z$. Hence $\theta_\ff$ can also be viewed as an element of
$X^* (\mc T_\ff) \otimes_\Z \Q / \Z$, where $X^* (\mc T_\ff)$\label{i:47} denotes the character 
lattice of $\mc T_\ff$. Then the action of $W(\mc G_y^\circ,\mc T_\ff)$ on $X_* (\mc T_\ff)$ 
(or the action on $X^* (\mc T_\ff)$) gives rise to a $k_F$-group  
$W(\mc G_y^\circ, \mc T_\ff)_{\theta_\ff}$, satisfying 
$W(\mc G_y^\circ, \mc T_\ff)_{\theta_\ff}(k_F) = W(\mc G_y^\circ, \mc T_\ff)(k_F)_{\theta_\ff}$. 
From \cite[p.~131]{DeLu}, one can deduce the following about the structure of 
$W(\mc G_y^\circ, \mc T_\ff)_{\theta_\ff}$:
\begin{enumerate}[(i)]
\item It has a normal subgroup $W(\mc G_y^\circ,\mc T_\ff)^\circ_{\theta_\ff}$, generated by
the reflections whose coroot in $X_* (\mc T_\ff)$ is orthogonal to $\theta_\ff$. It is equal to the 
$W(\mc G_y^\circ, \mc T_\ff)$-stabilizer of an extension of $\theta_\ff$ to a group in which
$\mc G_y^\circ (k_F)$ is ``regularly embedded'' in the sense of \cite[\S 1.7]{GeMa}.
\item It admits a decomposition $W(\mc G_y^\circ,\mc T_\ff)_{\theta_\ff} = 
W(\mc G_y^\circ,\mc T_\ff)^\circ_{\theta_\ff} \rtimes \Gamma$, 
where $\Gamma$ is the stabilizer of a set of positive roots for the Weyl group
$W(\mc G_y^\circ,\mc T_\ff)^\circ_{\theta_\ff}$.
\item There exists a point $\tilde \theta_\ff$ in the fundamental alcove for the action of
the affine Weyl group of $R(\mc G_y^\circ,\mc T_\ff)$ on $X^* (\mc T_\ff) \otimes_\Z \Q$,
such that 
\begin{equation}\label{eq:7.44}
W(\mc G_y^\circ,\mc T_\ff)_{\theta_\ff} \cong
\big( W(\mc G_y^\circ,\mc T_\ff) \ltimes X^* (\mc T_\ff) \big)_{\tilde \theta_\ff} .
\end{equation}
\item The previous item implies that $\Gamma \cong W(\mc G_y^\circ,\mc T_\ff)_{\theta_\ff} /
W(\mc G_y^\circ,\mc T_\ff)^\circ_{\theta_\ff}$ is isomorphic to a subgroup of
$X^* (\mc T_\ff) / \Z R (\mc G_y^\circ,\mc T_\ff)$.  
\end{enumerate}
By part (iv) the group $\Gamma$ tends to be very small, provided that $\mc G_y^\circ$ is semisimple.
We will use that later, to analyse $W(\mc G_y^\circ, \mc T_\ff)_{\theta_\ff}$.\\
We also have a version of Lemma \ref{lem:7.7} in this context.

\begin{lem}\label{lem:7.15}
Assume that we are in the setting of Lemma \ref{lem:7.14}.b--c.
The group $W(\mc L_y^\circ,\mc T_\ff)_{\theta_\ff}$ is central in
$W(N_{\mc G_y^\circ}(\mc L_y^\circ),\mc T_\ff)_{\theta_\ff}$.
\end{lem}
\begin{proof}
By the non-singularity of $\theta_\ff$, the intersection of
$W(\mc L_y^\circ,\mc T_\ff)_{\theta_\ff}$ and $W(\mc G_y^\circ,\mc T_\ff)^\circ_{\theta_\ff}$
is trivial. Thus (iv) above provides an embedding 
\begin{equation}\label{eq:7.43}
W(\mc L_y^\circ,\mc T_\ff)_{\theta_\ff} \hookrightarrow 
X^* (\mc T_\ff) / \Z R (\mc G_y^\circ,\mc T_\ff) .
\end{equation}
The construction of this embedding via (iii) shows that it is 
$W(N_{\mc G_y^\circ}(\mc L_y^\circ),\mc T_\ff)_{\theta_\ff}$-equivariant. Since a reflection
$s_\alpha \in W(\mc G_y^\circ,\mc T_\ff)$ translates every element of $X^* (\mc T_\ff)$ by 
a multiple of $\alpha$, the action of $W(\mc G_y^\circ,\mc T_\ff)$ on $X^* (\mc T_\ff)$ 
only adjusts the latter by elements of $\Z R (\mc G_y^\circ,\mc T_\ff)$. In particular, the 
action of $W(\mc G_y^\circ,\mc T_\ff)_{\theta_\ff}$ on \eqref{eq:7.43} is trivial.
\end{proof}

\subsection{Embeddings of tori and extensions} \  
\label{par:2.split}

Given the $F$-torus $\mc T$, there are various ways to embed it in a rigid
inner twist of $\mc L$. Starting from one embedding one obtains first a stable 
conjugacy class of embeddings, and next by composition with inner twists a 
larger collection of embeddings, which are then called admissible. We fix a finite
central $F$-subgroup \label{i:7} $\mc Z$ of $\mc L$. The equivalence
classes of such admissible embeddings $j$\label{i:48} can be parametrized by 
a cohomology set $H^1 (\mc E, \mc Z \to \mc T)$ 
\cite[\S 4.4]{Kal3} and \cite[\S 3]{Dil}. Here the symbol $\mc E$ denotes a certain
gerbe whose precise definition is not important to us.
This parametrization requires the choice of a standard admissible embedding 
of $F$-groups $j_0 : \mc T \to \mc L^\flat$\label{i:49}, such that $\mc L^\flat$\label{i:50} 
is a quasi-split rigid inner twist of $\mc L$ and $j_0 \mc T (F)$ fixes an absolutely 
special point in the reduced Bruhat--Tits building of $\mc L^\flat$ over $F$.

We now compare Weyl groups associated to different admissible embeddings. 
Any two such admissible embeddings, say $j : \mc T \to \mc L$ and $j' : \mc T \to \mc L'$, 
correspond to the same conjugacy class of embeddings of Langlands dual groups. 
In fact that is another characterization of our class of admissible embeddings
\cite[\S 5.1]{Kal2}. Hence 
\begin{equation}\label{eq:7.10}
W (\mc L, j \mc T)(F) \cong W(\mc L', j' \mc T)(F) . 
\end{equation}
Let $\mc G'$ be a rigid inner twist of $\mc G$ containing $\mc L'$ as an
$F$-Levi subgroup.~By \cite[Proposition 3.1]{ABPS1}, $W(G,L)$
is naturally isomorphic to $W(G',L')$. Similar to \eqref{eq:7.10}, 
\[
W(G,L)_{(jT,\theta)} \cong W(\mc G,\mc L)(F)_{(j T,\theta)} \cong 
W(\mc G',\mc L')(F)_{(j' T,\theta)} \cong W(G',L')_{(j' T,\theta)} .
\]
We fix $j_0$ as above, and we abbreviate\label{i:52} 
$j_0 \mc T = \mc T^\flat, j_0 T = T^\flat$. 
The character $\theta \circ j_0^{-1}$ of $T^\flat$ will still be denoted $\theta$.
By \cite[\S 4.5]{Kal3}, there is an isomorphism
\begin{equation}\label{eq:7.11}
W (L^\flat,T^\flat) \cong W (\mc L^\flat, \mc T^\flat)(F) .    
\end{equation}
We warn the reader that this need not hold for another embedding $j$. It turns out that
\eqref{eq:7.11} can be generalized to a setting with $\mc G$. Let $\mc G^\flat$\label{i:53}
be a rigid inner twist of $\mc G$ containing $\mc L^\flat$ as an $F$-Levi subgroup,
thus in particular $G^\flat$ is quasi-split.

\begin{lem}\label{lem:7.8}
There is a natural isomorphism 
$W( N_{G^\flat}(L^\flat), T^\flat) \cong W( N_{\mc G^\flat}(\mc L^\flat), \mc T^\flat)(F)$. 
\end{lem}
\begin{proof}
First we note the following isomorphisms
\begin{equation}\label{eq:7.12}
W( N_{G^\flat}(L^\flat), T^\flat) / W(L^\flat, T^\flat) \cong 
N_{G^\flat}(L^\flat, T^\flat) / N_{L^\flat} (T^\flat) \cong W(G^\flat,L^\flat)_{T^\flat},
\end{equation}
where the subscript $T^\flat$ means that the $L^\flat$-conjugacy class of $T^\flat$ is
stabilized. Since $L^\flat$ is quasi-split, the natural maps 
\begin{equation}
N_{G^\flat}(L^\flat) \to W(G^\flat,L^\flat) \to (N_{\mc G^\flat}(\mc L^\flat)/ \mc L^\flat)(F) =
W(\mc G^\flat,\mc L^\flat)^{\mb W_F} 
\end{equation}
are surjective. Hence \eqref{eq:7.12} is equal to $W(\mc G^\flat, \mc L^\flat)(F)_{T^\flat}$.
Clearly, any element of this group also stabilizes the $\mc L^\flat$-conjugacy class of 
$\mc T^\flat$. On the other hand, if an element of $W(\mc G^\flat, \mc L^\flat)$ stabilizes
the $\mc L^\flat$-conjugacy class of $\mc T^\flat$, then a suitable representative of
that element normalizes $T^\flat$. Therefore \eqref{eq:7.12} is naturally isomorphic to
\begin{equation}\label{eq:7.13}
W(\mc G^\flat, \mc L^\flat)_{\mc T^\flat}(F) \cong
\big( W(N_{\mc G^\flat}(\mc L^\flat), \mc T^\flat) / W(\mc L^\flat, \mc T^\flat) \big) (F),    
\end{equation}
which has a subgroup 
$W(N_{\mc G^\flat}(\mc L^\flat), \mc T^\flat)(F) / W(\mc L^\flat, \mc T^\flat) (F)$. 
The natural isomorphism from the left hand side of \eqref{eq:7.12} to the
right hand side of \eqref{eq:7.13} factors through \eqref{eq:7.13}, thus this subgroup 
is equal to \eqref{eq:7.13}.~Combined with \eqref{eq:7.11}, it implies that the
natural map 
$W(N_{G^\flat}(L^\flat), T^\flat) \to W(N_{\mc G^\flat}(\mc L^\flat), \mc T^\flat)(F)$ 
is an isomorphism.
\end{proof}

We now consider the case of an embedding $j : \mc T \to \mc L$ admissible in 
the sense of \cite[\S~4.4]{Kal3}, such that $\mc L$ is not necessarily quasi-split.
Suppose $j$ corresponds to \label{i:54}
\begin{equation}
[x] = \mr{inv}(j,j_0) \in H^1 (\mc E, \mc Z \to \mc T).
\end{equation} 
Then $x$ defines a form of $N_{\mc L^\flat}(\mc T^\flat)$, with the property that
\begin{equation}\label{eq:7.18}
W(L,jT) \cong  W(\mc L^\flat, \mc T^\flat)_x (F) = W(\mc L^\flat, \mc T^\flat)(F)_{[x]} .
\end{equation}
Similar to \cite[\S 4.5]{Kal3}, we construct the extension
\begin{equation}\label{eq:7.19}
1 \to T \to N_L (j T)_\theta = N_{\mc L^\flat}(\mc T^\flat)_{x}(F) \to W(L,jT)_\theta \to 1 .
\end{equation}
By \eqref{eq:7.18} and pushout along $\theta$, we obtain a central extension\label{i:55}
\begin{equation}\label{eq:7.20}
1 \to \C^\times \to \mc E_\theta^{[x]} \to W(\mc L^\flat, \mc T^\flat)(F)_{[x],\theta} \to 1.
\end{equation}
The set $\Irr (\mc E_\theta^{[x]},\mr{id})$\label{i:56} of irreducible representations of 
$\mc E_\theta^{[x]}$ on which $\C^\times$ acts as $z \mapsto z$ is naturally in bijection with 
$\Irr ( N_L (j T)_\theta ,\theta )$. Via \eqref{eq:7.10} and Proposition \ref{prop:7.11}, 
it matches with $\Pi (L,jT,\theta)$. We can rephrase \eqref{eq:7.20} as 
$\mc E_\theta^{[x]} = N_L (jT)_\theta \times_{jT,\theta} \C^\times$. 
For $\chi \in \Xo (L)$, the bijection
\[
N_L (jT)_{\chi \otimes \theta} \times_{jT,\chi \otimes \theta} \C^\times \to 
N_L (jT)_\theta \times_{jT,\theta} \C^\times : (n,c) \mapsto (n,\chi (n) c) 
\]
induces a canonical isomorphism of extensions
\begin{equation}\label{eq:8.14}
    \begin{tikzcd}
        1 \arrow[]{r}{} & \C^\times \arrow[]{r}{}\arrow[equals]{d}{} & 
        \mc E_{\chi \otimes \theta}^{[x]} \arrow[]{r}{}\arrow[]{d}{\cong} & 
        W(L,jT)_{\chi \otimes \theta} \arrow[]{r}{}\arrow[equals]{d}{} & 1 \\
1  \arrow[]{r}{} & \C^\times \arrow[]{r}{} & \mc E_{\theta}^{[x]} \arrow[]{r}{} & 
W(L,jT)_{\theta} \arrow[]{r}{} & 1 
    \end{tikzcd}.
\end{equation}
In this way the extensions $\mc E_{\theta'}^{[x]}$ for varying 
$\theta' \in \Xo (L) \otimes \theta$ are naturally isomorphic. 
The conjugation action of $N_G (L,jT)$ on \eqref{eq:7.18} descends to an action of
$N_G (L,jT)_{\Xo (L) \theta}$ on \eqref{eq:7.19}. The quotient 
\begin{multline}\label{eq:7.27}
N_G (L,jT)_{\Xo (L) \theta} / jT = W(N_G (L), jT)_{\Xo (L) \theta} \cong \\
W(N_{\mc G^\flat}(\mc L^\flat), \mc T^\flat)_x (F)_{\Xo (L) \theta} \cong 
W(N_{\mc G^\flat}(\mc L^\flat), \mc T^\flat) (F)_{[x],\Xo (L) \theta} 
\end{multline}
acts on the family of extensions \eqref{eq:7.20}, with $g \in N_G (L,jT)_{\Xo (L) \theta}$
sending $\mc E_\theta^{[x]}$ to $\mc E_{g \cdot \theta}^{[x]}$. Only the subgroup 
$W(N_G (L), jT)_\theta \cong W(N_{\mc G^\flat}(\mc L^\flat), \mc T^\flat) (F)_{[x],\theta}$ 
stabilizes $\mc E_\theta^{[x]}$ and  $\Irr ( N_L (j T)_\theta ,\theta ) \cong 
\Irr (\mc E_\theta^{[x]},\mr{id})$. Consider now the extension
\begin{equation}\label{eq:7.15}
1 \to T \to N_{L^\flat} (T^\flat)_\theta \to W(L^\flat, T^\flat)_\theta = 
W(\mc L^\flat, \mc T^\flat)(F)_\theta \to 1, 
\end{equation}
which gives the following central extension via pushout along $\theta$: \label{i:57}
\begin{equation}\label{eq:7.16}
1 \to \C^\times \to \mc E_\theta^0 \to W(\mc L^\flat, \mc T^\flat)(F)_\theta \to 1 .   
\end{equation}
The extensions $\mc E_{\theta'}^0$, for varying $\theta' \in \Xo (L) \otimes \theta$, 
are naturally identified by a variation on \eqref{eq:8.14}.~The group 
$N_{G^\flat}(L^\flat, T^\flat)_{\Xo (L) \theta}$ acts on \eqref{eq:7.15}. Its subgroup
$N_{G^\flat}(L^\flat, T^\flat)_\theta$ also acts on \eqref{eq:7.16}, and that action 
factors through $W (N_{G^\flat}(L^\flat), T^\flat)_{\theta}$.
This gives an action of $N_{G^\flat}(L^\flat, T^\flat)_\theta$ on 
\[
\Irr \big( N_{L^\flat}(T^\flat)_\theta ,\theta \big) \cong 
\Irr \big( \mc E_\theta^0, \mr{id} \big),
\]
which by Lemma \ref{lem:7.6} (a) factors through 
$N_{G^\flat}(L^\flat, T^\flat)_\theta / N_{L^\flat}(T^\flat)_\theta \cong 
W(G^\flat,L^\flat)_{(T^\flat, \theta)}$. 
Pulling back \eqref{eq:7.15} and \eqref{eq:7.16} along
$W(\mc L^\flat,\mc T^\flat)(F)_{[x],\theta} \to W(\mc L^\flat, \mc T^\flat )(F)_\theta$ gives
\begin{align}\label{eq:7.22}
& 1 \to T \to N_{L^\flat} (T^\flat)_{[x],\theta} \to W(\mc L^\flat, \mc T^\flat)(F)_{[x],\theta} \to 1 \\
& 1 \to \C^\times \to \mc E_\theta^{0,[x]} \to W(\mc L^\flat, \mc T^\flat)(F)_{[x],\theta} \to 1 .   
\label{eq:7.23}
\end{align}
In general, $\mc E_\theta^{0,[x]}$ is not isomorphic to $\mc E_\theta^{[x]}$, and the 
difference can be measured by yet another extension, i.e.~similar to \cite[\S 8.1]{Kal4},
we consider
\begin{equation}\label{eq:7.24}
1 \to T \to (\mc T^\flat \rtimes W(\mc L^\flat,\mc T^\flat))_x (F)_\theta \to 
W(\mc L^\flat, \mc T^\flat)(F)_{[x],\theta} \to 1 ,
\end{equation}
where $x$ determines a form of the algebraic group $\mc T^\flat \rtimes W(\mc L^\flat,\mc T^\flat)$.
By pushout along $\theta$, we produce a central extension\label{i:58}
\begin{equation}\label{eq:7.25}
1 \to \C^\times \to \mc E_\theta^{\ltimes [x]} \to W(\mc L^\flat, \mc T^\flat)(F)_{[x],\theta} \to 1.
\end{equation}
Reasoning as in \eqref{eq:7.20} and \eqref{eq:7.16}, the extensions $\mc E_{\theta'}^{\ltimes [x]}$
with $\theta' \in \Xo (L) \otimes \theta$ are naturally identified, and this family
of extensions is endowed with a conjugation action of 
$W(N_{\mc G^\flat}(\mc L^\flat), \mc T^\flat)(F)_{[x],\Xo (L) \theta}$.

\begin{lem}\label{lem:7.10}
\enuma{
\item The extension \eqref{eq:7.19} is the Baer sum of the extensions \eqref{eq:7.22}
and \eqref{eq:7.24}.
\item The extension \eqref{eq:7.20} is the Baer sum of \eqref{eq:7.23} and \eqref{eq:7.25},
as extensions of 
$W(N_{\mc G^\flat}(\mc L^\flat), \mc T^\flat)(F)_{[x],\theta}$-groups.
}
\end{lem}
\begin{proof}
(a) In the proof of \cite[Proposition 8.2]{Kal4}, setwise splittings of \eqref{eq:7.22} 
and \eqref{eq:7.24} are chosen. A setwise splitting of \eqref{eq:7.19} can then be obtained
essentially as the product of these two splittings. It follows that the 2-cocycle classifying
\eqref{eq:7.19} is the sum of the 2-cocycles classifying \eqref{eq:7.22} and \eqref{eq:7.24},
which means that \eqref{eq:7.19} is isomorphic to the Baer sum of the other two extensions.

(b) As the difference between the extensions in part (a) and those in part (b) is given by pushout 
along $\theta$ in all three cases, the isomorphism here is a direct consequence of part (a). The
construction of this Baer sum takes place in the category of groups with an action of
$W(N_{\mc G^\flat}(\mc L^\flat), \mc T^\flat)(F)_{[x],\theta}$, with the 
actions given above of this lemma. 
\end{proof}

Lemma \ref{lem:7.10}, combined with the next proposition, shows that $\mc E_\theta^{[x]}$ is
isomorphic to $\mc E_\theta^{\ltimes [x]}$ as $W(N_{\mc G^\flat}(\mc L^\flat), 
\mc T^\flat)(F)_{[x],\Xo (L) \theta}$-groups. 

\begin{prop}\label{prop:7.9}
The family of extensions $\mc E_{\theta'}^0$, with $\theta' \in \Xo (L) \otimes \theta$, admits 
an $N_{G^\flat}(L^\flat,T^\flat)_{\Xo (L) \theta}$-equivariant splitting.~In 
particular, $\theta \in \Irr (T^\flat)$ extends to a 
$W(G^\flat,L^\flat)_{(T^\flat,\theta)}$-stable character of $N_{L^\flat}(T^\flat)_\theta$.
\end{prop}
\begin{proof}
First we reduce to the case of finite reductive groups. Let $P_y^\flat \subset G^\flat$ be 
the parahoric subgroup associated to the special vertex $y$ from Lemma \ref{lem:7.14} (b). 
Similar to the extension \eqref{eq:7.15}, we consider the extension 
\begin{equation}\label{eq:7.39}
1 \to P_y^\flat \cap T^\flat \to P_y^\flat \cap N_{G^\flat}(L^\flat,T^\flat)_{\theta_\ff} \to 
W(N_{G^\flat}(L^\flat), T^\flat)_{\theta_\ff} \to 1 .
\end{equation}
By Lemma \ref{lem:7.14} (c), pullback of \eqref{eq:7.39} along 
$W(L^\flat, T^\flat)_{\theta} \to W(N_{G^\flat}(L^\flat), T^\flat)_{\theta_\ff}$, followed by pushout 
along $\theta_\ff : P_y \cap T^\flat \to \C^\times$, recovers the extension \eqref{eq:7.16}.
Since $\theta_\ff$ has depth zero, the pushout of \eqref{eq:7.39} along $\theta_\ff$ can
also be obtained from
\begin{equation}\label{eq:7.40}
1 \to \mc T_\ff (k_F) \to N_{\mc G_y^\circ}(\mc L_y^\circ,\mc T_\ff)(k_F)_{\theta_\ff} 
\to W(N_{\mc G_y^\circ}(\mc L_y^\circ), \mc T_\ff)(k_F)_{\theta_\ff} \to 1 .
\end{equation}
The image of $W(N_{G^\flat}(L^\flat), T^\flat)_{\theta_\ff}$ in $W(\mc G_y^\circ, \mc T_\ff)(k_F)$ 
is contained in $W(N_{\mc G_y^\circ}(\mc L_y^\circ), \mc T_\ff)(k_F)_{\theta_\ff}$. If we can
establish a $W(\mc G_y^\circ, \mc T_\ff)(k_F)$-equivariant splitting of \eqref{eq:7.40}, then 
we can extend it $W(N_{G^\flat}(L^\flat), T^\flat)_{\Xo (L) \theta_\ff}$-equivariantly to the
versions of \eqref{eq:7.40} for other $\theta_\ff$.

Thus it suffices to construct an 
$N_{\mc G_y^\circ}(\mc L_y^\circ,\mc T_\ff)(k_F)_{\theta_\ff}$-equivariant setwise splitting of
\begin{equation}\label{eq:7.41} 
1 \to \mc T_\ff (k_F) \to N_{\mc L_y^\circ}(\mc T_\ff)(k_F)_{\theta_\ff} 
\to W(\mc L_y^\circ, \mc T_\ff)(k_F)_{\theta_\ff} \to 1,     
\end{equation}
which becomes a group homomorphism in the following pushout along $\theta_\ff$:
\begin{equation}\label{eq:7.42}
1 \to \C^\times \to \mc E_{\theta_\ff}^0 \to W(\mc L_y^\circ, \mc T_\ff)(k_F)_{\theta_\ff} \to 1 .
\end{equation}
The existence of such a splitting was shown in \cite[Lemma 4.5.6 and Corollary 4.5.7]{Kal3}; 
it remains to prove its $N_{\mc G_y^\circ}(\mc L_y^\circ,\mc T_\ff)(k_F)_{\theta_\ff}$-equivariance. 

We denote the derived group of $\mc G$ by $\mc G_\der$\label{i:59}, and its simply
connected cover by $\mc G_\Sc$\label{i:60}. For disconnected algebraic groups, we define
the derived and simply connected analogues by first passing to the neutral component of 
the group.
Let $\mc T_{\ff,c}$ be the preimage of $\mc T_\ff$ in $\mc G_{y,\Sc}$. Then $\theta_\ff$ 
can be pulled back to a character $\theta_{\ff,c}$ of $\mc T_{\ff,c}(k_F)$. The preimage 
$\mc L_{y,c}$ of $\mc L_y^\circ$ in $\mc G_{y,\Sc}$ is a Levi subgroup of the latter, so the
derived subgroup of $\mc L_{y,c}$ is simply connected and we denote it by $\mc L_{y,\Sc}$.
We write $\mc T_{\ff,\Sc} = \mc T_{\ff,c} \cap \mc L_{\ff,\Sc}$. By pushout along $\theta_{\ff,c} : 
\mc T_{\ff,c} (k_F) \to \C^\times$ and pullback along 
$W(\mc L_y^\circ, \mc T_\ff)(k_F)_{\theta_\ff} \to 
W(\mc L_{y,c}, \mc T_{\ff,c})(k_F)_{\theta_{\ff,c}}$ of the extension
\begin{equation}\label{eq:7.17}
1 \to \mc T_{\ff,c}(k_F) \to N_{\mc L_{y,c}}(\mc T_{\ff,c})(k_F)_{\theta_{\ff,c}} 
\to W(\mc L_{y,c}, \mc T_{\ff,c})_{\theta_{\ff,c}} \to 1,    
\end{equation} 
one can obtain \eqref{eq:7.42}. 
The group $N_{\mc G_y^\circ}(\mc L_y^\circ,\mc T_\ff)(k_F)_{\theta_\ff}$ still acts by
conjugation on \eqref{eq:7.17}, and we need to keep track of equivariance for that action. This 
may be replaced by the conjugation action of 
$N_{\mc G_{y,\Sc}}(\mc L_{\ff,c},\mc T_{\ff,c})(k_F)_{\theta_{\ff,c}}$, because the
latter group has a larger image in $\mc G_{y,\der}(k_F)$. Therefore we may assume without loss 
of generality that $\mc G_y^\circ$ is simply connected. Then \eqref{eq:7.40} decomposes as a 
direct product of the analogous extensions for the $k_F$-simple factors of $\mc G_y^\circ$,
thus we may assume without loss of generality that $\mc G_y^\circ$ is in addition $k_F$-simple. 
By passing to a finite field extension of $k_F$, we can make $\mc G_y^\circ$ absolutely simple. 

\textit{From now on, we assume that $\mc G_y^\circ$ is absolutely simple and simply connected.}\\
By Lemma \ref{lem:7.7}, $W(\mc L_y^\circ,\mc T_\ff)(k_F)_{\theta_\ff}$ 
commutes with $N_{\mc G_y^\circ} (\mc L_y^\circ,\mc T_\ff)(k_F)_{\theta_\ff} / 
\mc T_\ff (k_F)$ in the group $W (\mc G_y^\circ, \mc T_\ff)(k_F)$. 
Choose a pinning of $\mc G_y^\circ$, associated to a maximal torus of $\mc L_y^\circ$ and 
stable under the action of $\Fr$ used to define $\mc G_y^\circ$ as $k_F$-group. Then 
the Levi subgroup $\mc L_y^\circ$ is $\Fr$-stable, and $\Fr$ stabilizes the pinning of 
$\mc L_y^\circ$ obtained by restriction from $\mc G_y^\circ$. We can write 
$N_{\mc G_y^\circ}(\mc L_y^\circ) = \mc L_y^\circ \rtimes \Gamma$, 
where $\Gamma$ is a finite group of pinning-preserving automorphisms of $\mc L_y^\circ$. 
For $\gamma \in \Gamma \subset N_{\mc G_y^\circ}(\mc L_y^\circ)$, the element 
$\Fr \gamma \Fr^{-1}$ preserves the pinning and thus again lies in $\Gamma$. In other words, 
$\Fr$ acts on the subgroup $\Gamma$ of $N_{\mc G_y^\circ}(\mc L_y^\circ)$, and that action 
coincides with the $\Fr$-action
on $N_{\mc G_y^\circ}(\mc L_y^\circ) / \mc L_y^\circ \cong \Gamma$. This implies
\[
N_{\mc G_y^\circ}(\mc L_y^\circ)(k_F) = N_{\mc G_y^\circ}(\mc L_y^\circ)^\Fr =
(\mc L_y^\circ \rtimes \Gamma)^\Fr = \mc L_y^{\circ,\Fr} \rtimes \Gamma^\Fr =
\mc L_y^\circ (k_F) \rtimes \Gamma^\Fr .
\]
Let $\mc L_{y,i}$\label{i:64} be an almost direct factor of $\mc L_y^\circ$, coming from one 
$N_{\mc G_y^\circ}(\mc L_y^\circ)(k_F) \times \langle \Fr \rangle$-orbit of simple factors
of $\mc L_y^\circ$. The group $\mc L_{y,\der}^\circ = \mc L_{y,\Sc}$ is simply connected,
thus is equal to the direct product of these $\mc L_{y,i}$'s. The group $\mc L_y^\circ$ is an
almost direct product of $\mc L_{y,\der}^\circ$ and $Z(\mc L_y^\circ)^\circ$. For
each $i$, let $\Gamma_i$ be the image of $\Gamma$ in the automorphism group of 
$\mc L_{y,i}$. We write \label{i:65} 
$\mc L_{y,i}^+ := \mc L_{y,i} \rtimes \Gamma_i$. 
Similarly, we can define 
$\Gamma_Z \subset \mr{Aut}(Z_{\mc G_y^\circ}(\mc L_{y,\der}^\circ)^\circ)$ and write
\[
Z_{\mc G_y^\circ}(\mc L_{y,\der}^\circ)^+ := 
Z_{\mc G_y^\circ}(\mc L_{y,\der}^\circ)^\circ \rtimes \Gamma_Z . 
\]
Then there is a natural embedding 
\begin{equation}\label{eq:7.45}
W(N_{\mc G_y^\circ}(\mc L_y^\circ), \mc T_\ff) \hookrightarrow
W \big( Z_{\mc G_y^\circ}(\mc L_{y,\der}^\circ)^+, Z(\mc L_y^\circ)^\circ \big) \times
\prod\nolimits_i \, W (\mc L_{y,i}^+, \mc T_{\ff,i}) ,
\end{equation}
where \label{i:61}$\mc T_{\ff,i} := \mc T_\ff \cap \mc L_{y,i}$. We define 
$W (\mc L_{y,i}^+, \mc T_{\ff,i})_{\theta_\ff}$ as the image of 
$W(N_{\mc G_y^\circ}(\mc L_y^\circ), \mc T_\ff)_{\theta_\ff}$ under \eqref{eq:7.45} followed
by projection onto the $i$-th coordinate. This group satisfies
\begin{equation}\label{eq:7.51}
W(\mc T_\ff \mc L_{y,i}, \mc T_\ff)_{\theta_\ff} \; \subset \; 
W (\mc L_{y,i}^+, \mc T_{\ff,i})_{\theta_\ff} \; \subset \;
W (\mc L_{y,i}^+, \mc T_{\ff,i})_{\theta_{\ff,i}} ,
\end{equation}
where $\theta_{\ff,i} := \theta_\ff |_{\mc T_{\ff,i} (k_F)}$. In general, both of the above inclusions 
can be proper. The upshot of the construction is that Lemma \ref{lem:7.15} still applies.
For every $i$, we can construct an extension similar to \eqref{eq:7.40} as follows:
\begin{equation}\label{eq:7.46}
1 \to \mc T_{\ff,i}(k_F) \to N_{\mc L_{y,i}^+}(\mc T_{\ff,i})(k_F)_{\theta_\ff} \to
W (\mc L_{y,i}^+, \mc T_{\ff,i})(k_F)_{\theta_\ff} \to 1,
\end{equation}
where the subscript $\theta_\ff$ is defined as above and it guarantees the exactness
of the above sequence \eqref{eq:7.46}. The extension \eqref{eq:7.42} can be recovered 
from the extensions \eqref{eq:7.46} for all $i$. Since the action of 
$N_{\mc G_y^\circ} (\mc L_y^\circ,\mc T_{\ff})(k_F)_{\theta_\ff}$ still shows up in
\eqref{eq:7.42}, it suffices to construct splittings of
\begin{equation}\label{eq:7.47}
1 \to \mc T_{\ff,i}(k_F) \to N_{\mc L_{y,i}}(\mc T_{\ff,i})(k_F)_{\theta_\ff} \to
W (\mc L_{y,i}, \mc T_{\ff,i})(k_F)_{\theta_\ff} \to 1,
\end{equation}
which are invariant for the conjugation action of 
$W (\mc L_{y,i}^+, \mc T_{\ff,i})(k_F)_{\theta_\ff}$ and become group homomorphisms in the following
pushout along $\theta_{\ff,i}$:
\begin{equation}\label{eq:7.48}
1 \to \C^\times \to \mc E_{\theta_{\ff,i}}^0 \to  
W (\mc L_{y,i}, \mc T_{\ff,i})(k_F)_{\theta_\ff} \to 1 .
\end{equation}
The existence of such a splitting was shown in \cite[Lemma 4.5.6 and Corollary 4.5.7]{Kal3}; 
it remains to show the invariance. The $N_{\mc L_{y,i}^+}(\mc T_{\ff,i})(k_F)_{\theta_\ff}$-invariants 
in the pushout of \eqref{eq:7.47} are canonically isomorphic to the invariants in the 
version of \eqref{eq:7.47} for one $k_F$-simple factor of $\mc L_{y,i}$ except with invariants 
under a subgroup of $N_{\mc L_{y,i}^+}(\mc T_{\ff,i})(k_F)_{\theta_\ff}$. Therefore we 
may assume without loss of generality that in \eqref{eq:7.47}, $\mc L_{y,i}$ is $k_F$-simple and 
simply connected.

Furthermore, $\mc L_{y,i}$ is the scalar restriction of a simple group $\mc L'_{y,i}$ 
over a field extension $k'$ of $k_F$. We may replace without loss of generality $k_F$ by $k'$ and $L_{y,i}$ by 
$\mc L'_{y,i}$. Thus we can reduce to the case where $\mc L_{y,i}$ is absolutely simple and 
simply connected. Lemma \ref{lem:7.15} still applies, and shows that 
\begin{equation}\label{eq:7.49}
W (\mc L_{y,i}, \mc T_{\ff,i})(k_F)_{\theta_\ff} \text{ is central in } 
W (\mc L_{y,i}^+, \mc T_{\ff,i})(k_F)_{\theta_\ff} .
\end{equation}
Since $W(\mc G_y^\circ,\mc T_\ff)$ is a Weyl group containing 
$W(\mc L_y^\circ,\mc T_\ff)$, and 
$N_{\mc G_y^\circ} (\mc L_y^\circ,\mc T_\ff)(k_F)_{\theta_\ff} / \mc T_\ff (k_F)$ normalizes 
$W(\mc L_y^\circ,\mc T_\ff)$, the actions of $N_{\mc L_{y,i}^+}(\mc T_{\ff,i})(k_F)_{\theta_\ff}$
on \eqref{eq:7.47} and \eqref{eq:7.48} are constructed from: 
\begin{itemize}
\item conjugation by elements of $N_{\mc L_{y,i}}(\mc T_{\ff,i})(k_F)$; 
\item  for each Dynkin diagram automorphism $W(\mc L_{y,i},\mc T_{\ff,i})$, at most one 
coset of $N_{\mc L_{y,i}}(\mc T_{\ff,i})(k_F)$. 
\end{itemize}
We now check case-by-case.\\

\noindent \textit{Case I. $N_{\mc L_{y,i}^+}(\mc T_{\ff,i})(k_F)_{\theta_\ff}$ acts on 
$\mc T_{\ff,i}$ as conjugation by elements of $W(\mc L_{y,i},\mc T_{\ff,i})$.} \\
This holds whenever $W(\mc L_y^\circ,\mc T_\ff)$ has type $A_1, B_n, C_n, E_7, E_8, F_4$ or $G_2$, 
because then the only Dynkin diagram automorphism is the identity. Then the action of\\
$N_{\mc L_{y,i}^+}(\mc T_{\ff,i})(k_F)_{\theta_\ff}$ can be viewed as coming from elements of 
$N_{\mc L_{y,i}}(\mc T_{\ff,i})(k_F)_{\theta_\ff}$. 
Hence any splitting of \eqref{eq:7.48}, as in \cite[\S 4.5]{Kal3}, is 
$N_{\mc L_{y,i}^+}(\mc T_{\ff,i})(k_F)_{\theta_\ff}$-invariant.\\

\noindent \textit{Case II. $W(\mc L_{y,i},\mc T_{\ff,i})$ has type $E_6$.} \\
If $W (\mc L_{y,i},\mc T_{\ff,i})(k_F)_{\theta_\ff}$ is trivial, there is nothing to prove. 
Otherwise, we have \\
$W (\mc L_{y,i},\mc T_{\ff,i})(k_F)_{\theta_\ff} \cong \Z / 3 \Z$. 
However, its image in the automorphism group of the affine Dynkin diagram of 
$W (\mc L_{y,i},\mc T_{\ff,i})$ does not commute with the nontrivial diagram 
automorphism $\tau$ of $E_6$, thus by \eqref{eq:7.49}, $\tau$ does not play a role in the picture. We conclude as in Case I.\\

\noindent\noindent \textit{Case III. $W(\mc L_{y,i},\mc T_{\ff,i})$ has type 
$A_{n-1}$ with $n > 2$.} \\
If $\mc L_{y,i}$ is an outer form of a split group, then $W(\mc L_{y,i},\mc T_{\ff,i})^\Fr$ 
becomes a Weyl group of type $B_{n-1}$ or $C_{n-1}$. The associated root system does not admit nontrivial
diagram automorphisms, thus we reduce back to case I. Therefore we may assume without loss of generality that
$\mc L_{y,i}$ is split, and thus it is isomorphic to $\SL_n$. 
By a change of coordinates for $\SL_n$, we may assume without loss of generality that $\mc T_{\ff,i}$ is the diagonal 
torus in $\mc L_{y,i}$, with $\Fr$-action given by the $q_F$-th power map composed with 
conjugation by an elliptic element $F_A \in W(\mc L_{y,i},\mc T_{\ff,i})$. With respect
to these coordinates, this is also how $\Fr$ acts on $\SL_n$. Since elliptic elements in $S_n$ are $n$-cycles, we may assume that $F_A = (1 \, 2 \ldots n)$. 
Note that $\mc T_{\ff,i}$ splits over the degree $n$ extension $k'$ of $k_F$ 

By \eqref{eq:7.44}, one can classify the possibilities for 
$\theta_{\ff,i}$ in terms of points of a fundamental alcove. The non-singularity of 
$\theta_\ff$ and \eqref{eq:7.44} force that $W (\mc L_{y,i},\mc T_{\ff,i})_{\theta_\ff}$ is 
either trivial or comes from the barycentre of an alcove, in which case we have
\[
W (\mc L_{y,i},\mc T_{\ff,i})_{\theta_{\ff,i}} = \langle (1\,2 \ldots n) \rangle .
\] 
More precisely, $\GL_1 (k')$ admits an order $n$ character represented by $\zeta_n \in
k'$, and $\theta_{\ff,i}$ can be represented by 
$\mr{diag}(1,\zeta_n,\ldots, \zeta_n^{-1}) \in \PGL_n (k')$. 
However, $W (\mc L_{y,i},\mc T_{\ff,i})_{\theta_{\ff,i}}$ does not commute with the 
nontrivial diagram automorphism $\tau$ of $S_n$; from their actions on the affine Dynkin 
diagram of type $A_{n-1}$, we see that $\tau$ acts by inversion. By \eqref{eq:7.49}, 
$W (\mc L_{y,i},\mc T_{\ff,i})_{\theta_{\ff}}$ can only contain elements of 
$W (\mc L_{y,i},\mc T_{\ff,i})_{\theta_{\ff,i}}$ of order $\leq 2$. If 
$W (\mc L_{y,i},\mc T_{\ff,i})_{\theta_{\ff}}$ is trivial, then there is nothing to show;
thus we may assume that $n$ is even and that 
$W (\mc L_{y,i},\mc T_{\ff,i})_{\theta_{\ff}} = \langle (1 \, 2 \ldots n)^{n/2} \rangle$. 
The group $\mc L_{y,i}^+ := \mc L_{y,i} \rtimes \Gamma$ is equal to $\SL_n \rtimes \langle -\top
\rangle$, where $-\top$ denotes the inverse transpose automorphism, because the 
nontrivial element of $\Gamma$ acts by $-\top$ composed with conjugation by an element of
$\SL_n$ (because $n$ is even). One can check that 
$W (\mc L_{y,i}^+,\mc T_{\ff,i})_{\theta_{\ff}} = \langle w_1, w_2 \rangle \cong
(\Z / 2 \Z)^2$, where 
$w_1 = (1 \, 2 \ldots n)^{n/2}$ and $w_2 = -\top \circ (2 \, n) (3 \, n\text{ - }1) \cdots (n/2 \, n/2 \!+\! 2)$. 
However, the element $w_2$ does not commute with $\Fr$ in $W(\mc L_{y,i}^+,\mc T_{\ff,i})$.
Hence we have
\[
N_{\mc L_{y,i}^+}(\mc T_{\ff,i})(k_F)_{\theta_\ff} = 
N_{\mc L_{y,i}}(\mc T_{\ff,i})(k_F)_{\theta_\ff} .
\]
\noindent\textit{Case IV. $W(\mc L_{y,i},\mc T_{\ff,i})$ has type $D_n$ with $n>4$.} \\
Now $W(\mc L_{y,i},\mc T_{\ff,i})_{\theta_{\ff,i}}$ embeds in $\Z / 4\Z$ (for $n$ odd) or in 
$\Z / 2 \Z \times \Z / 2\Z$ (for $n$ even).~If we are not in Case I, the action of 
$N_{\mc L_{y,i}^+}(\mc T_{\ff,i})_{\theta_{\ff,i}}$ on the Dynkin diagram $D_n$ uses the
nontrivial automorphism $\epsilon_n$.~We realize $\mc L_{y,i}$ as a spin group on a vector 
space of dimension $2n$, and $\mc T_{\ff,i}$ as the diagonal torus.~Then $\epsilon_n$ 
becomes the reflection in the $n$-th coordinate of $\mc T_{\ff,i}$. It only fixes one 
nontrivial element of $X^* (\mc T_{\ff,i}) / \Z R (\mc L_{y,i},\mc T_{\ff,i})$, thus 
$W(\mc L_{y,i},\mc T_{\ff,i})_{\theta_{\ff,i}}$ has order two. 

Again by \eqref{eq:7.44}, we have a classification of the possible $\theta_{\ff,i}$ via points in 
a fundamental alcove.~Via conjugation, we can reduce to the following situation: the character
$\theta_{\ff,i}$ of $\mc T_{\ff,i} (k_F)$ has trivial restriction to the first coordinate and
quadratic restriction to the $n$-th coordinate, while the restrictions to the other coordinates 
(as well as their inverses) differ and have higher order. Then 
$W(\mc L_{y,i}^+,\mc T_{\ff,i})_{\theta_{\ff,i}} = \langle \epsilon_1 , \epsilon_n \rangle$ 
and $W(\mc L_{y,i},\mc T_{\ff,i})_{\theta_{\ff,i}} = \langle \epsilon_1 \epsilon_n \rangle$. 
Recall that \eqref{eq:7.47} has a splitting, say it sends $\epsilon_1 \epsilon_n$ to $s_1 s_n$ 
with $s_1, s_n \in \mc L_{y,i}(k_F)$ representing reflections in these coordinates and 
$s_1^2 = s_n^{-2} \in Z(\mc L_{y,i})(k_F)$. Then $\epsilon_n$ acts as conjugation by $s_n$ on
$\mc L_y^\circ$, and fixes $s_1 s_n$. \\

\noindent\textit{Case V. $W(\mc L_{y,i},\mc T_{\ff,i})$ has type $D_4$.} \\
In the automorphism group of the affine Dynkin diagram of $D_4$, the order-three automorphisms 
of $D_4$ do not commute with (the image of) any nontrivial element of 
$W(\mc L_{y,i},\mc T_{\ff,i})_{\theta_{\ff,i}} \subset \Z / 2 \Z \times \Z / 2\Z$. If the 
action of $N_{\mc L_{y,i}^+} (\mc T_{\ff,i})(k_F)_{\theta_{\ff,i}}$ on the Dynkin diagram
$D_4$ includes such exceptional automorphism, then $W(\mc L_y^\circ,\mc T_\ff)_{\theta_\ff}$ 
is trivial and there is nothing to show.~Otherwise either we are in Case I, or the action of 
$N_{\mc L_{y,i}^+} (\mc T_{\ff,i})(k_F)_{\theta_{\ff,i}}$
uses exactly one nontrivial diagram automorphism $\psi$, of order two.~But $\psi$ is 
conjugate to the standard Dynkin diagram automorphism $\epsilon_n$ by another diagram 
automorphism $\tau$ of $D_4$.~Conjugating everything by this $\tau$ brings us back to Case IV, 
which works just as well for $n = 4$ with these simplifications.
\end{proof}

\section{Supercuspidal \texorpdfstring{$L$}{L}-parameters of depth zero}
\label{sec:L0}
\subsection{Preliminaries}\label{subsec:Lparam-prelim} \

Let $G^\vee = \mc G^\vee (\C)$\label{i:67} be the complex dual group of $G$,
endowed with an action of the Weil group \label{i:6} $\mb W_F \subset \mr{Gal}(F_s/F)$
that stabilizes a pinning. The Langlands dual group of $G$ is \label{i:68}
${}^L G = {}^L \mc G := G^\vee \rtimes \mb W_F$. 
Consider Langlands parameters \label{i:77}
\[
\varphi : \mb W_F \times \SL_2 (\C) \to {}^L G = G^\vee \rtimes \mb W_F,
\]
such that $\varphi |_{\SL_2 (\C)} : \SL_2 (\C) \to G^\vee$ is an algebraic homomorphism, 
$\varphi |_{\mb W_F}$ is a continuous homomorphism preserving the projections onto $\mb W_F$,
and $\varphi (\mb W_F)$ consists of semisimple elements. Let  \label{i:69}
$\mb P_F \subset \mb I_F \subset \mb W_F$ be the wild inertia and inertia subgroups of $\mb W_F$. Let $\Fr_F$\label{i:88} 
be a geometric Frobenius element of $\mb W_F$. A Langlands parameter has \textit{depth zero}
if $\varphi (w) = w$ for all $w \in \mb P_F$. 
For any $w \in \mb W_F$ and $x \in \SL_2 (\C)$, we write 
$\varphi (w,x) = \varphi (x) \varphi_0 (w) w$, where $\varphi_0 : \mb W_F \to G^\vee$ is a 1-cocycle. Since $\mb P_F$ is normal in $\mb W_F$, we have
\[
w p w^{-1} = \varphi (w p w^{-1}) = \varphi (w) \varphi (p) \varphi (w)^{-1} = 
\varphi_0 (w) w p w^{-1} \varphi_0 (w)^{-1}
\]
for any depth-zero $L$-parameter $\varphi$. Since $w p w^{-1}$ runs through all of $\mb P_F$, we have 
$\varphi_0 (w) \in Z_{G^\vee}(\mb P_F)$ 
for all $w \in \mb W_F$. 
Since $\varphi (\SL_2 (\C))$ commutes with $\varphi (\mb P_F) = \mb P_F$, it lies in 
$Z_{G^\vee}(\mb P_F)$ as well.~Therefore, any depth-zero Langlands parameter 
$\varphi$ gives rise to a 1-cocycle.
$\varphi_0 : \mb W_F / \mb P_F \times \SL_2 (\C) \to Z_{G^\vee}(\mb P_F)$.

We abbreviate $M^\vee := Z_{G^\vee}(\mb P_F)$.
This group is reductive (because $\mb P_F$ acts on $G^\vee$ via a finite quotient) but not 
necessarily connected.~Although $L$-parameters are usually considered up to $G^\vee$-conjugacy, our 
depth-zero condition is not preserved under $G^\vee$-conjugation, therefore we need to make some 
adjustments. We denote the set of $M^\vee$-conjugacy classes of depth-zero $L$-parameters for $G$ by 
$\Phi^0 (G)$, which injects into the set $\varphi (G)$\label{i:70} of $G^\vee$-conjugacy classes of 
$L$-parameters for $G$. 

Consider $G$ as a rigid inner twist of its quasi-split inner form $G^\flat$, with respect to a 
chosen finite subgroup $\mc Z \subset Z(\mc G)$ as in \cite{Kal1,Dil}. More precisely, this means 
that $G$ is equipped with more information, which can be packaged into a character $\zeta_G$ of a 
certain group $\pi_0 (Z (\bar G^\vee)^+ )$. Here $\bar{\mc G} := \mc G / \mc Z$ has complex dual group 
$\bar G^\vee$, and $Z(\bar G^\vee)^+$ is the preimage of $Z(G^\vee)^{\mb W_F}$ in $Z(\bar G^\vee)$. 
The group $\bar G^\vee$ is a central extension of $G^\vee$, which gives rise to a conjugation 
action on $G^\vee \rtimes \mb W_F$. Its associated version of the centralizer of $\varphi$ is
\label{i:72}
\[
S_\varphi^+ := Z_{{\bar G}^\vee} \big(\varphi (\mb W_F \times \SL_2 (\C)) \big) =
\text{ preimage of } Z_{G^\vee}(\varphi) \text{ in } \bar G^\vee.
\]
An \textit{enhancement} of $\varphi$ is an irreducible representation $\rho$ of 
$\pi_0 (S_\varphi^+)$, and it is called $G$-relevant if $\rho|_{Z(\bar G^\vee)^+}$ is 
$\zeta_G$-isotypic.~The group $\bar G^\vee$ acts naturally on the set of enhanced $L$-parameters 
for $G$, and this action factors through $G^\vee$.~The group $M^\vee := Z_{G^\vee}(\mb P_F)$ 
does the same if we restrict to depth-zero enhanced $L$-parameters.~Let $\Phi^0_e (G)$ be the 
set of $M^\vee$-orbits of $G$-relevant enhanced depth-zero $L$-parameters.\label{i:71} It is a 
subset of the set $\Phi_e (G)$ of $G^\vee$-orbits of $G$-relevant enhanced $L$-parameters. 

A Langlands parameter $\varphi$ is called \textit{supercuspidal} if it is discrete and trivial on 
$\SL_2 (\C)$. It is expected that this condition should be equivalent to the $L$-packet 
$\Pi_\varphi$ consisting entirely of supercuspidal representations (of various inner twists of $G)$. 
This expectation is a special case of \cite[Conjecture 7.8]{AMS1}.  

\begin{lem}\label{lem:6.1}
Every supercuspidal depth-zero $L$-parameter for $G$ gives rise to the following objects, which 
are canonical up to $M^\vee$-conjugacy:
\begin{itemize}
\item an $L$-group ${}^L T$ with an embedding ${}^L j : {}^L T \to {}^L G$, such that ${}^L j (T^\vee)$ 
is a maximal torus of $G^\vee$ and ${}^L j (\mb P_F)$ equals $\mb P_F \subset {}^L G$;
\item an $F$-torus $\mc T$ and a non-singular depth-zero character $\theta$ of $T$, such that
$\mc T / Z(\mc G)$ is elliptic and $\varphi$ is equal to the composition of ${}^L j$ with the 
$L$-parameter of $\theta$.
\end{itemize}
\end{lem}
\begin{proof}
Consider a depth-zero supercuspidal $L$-parameter $\varphi$ for $G$.~By \cite[Lemma 4.1.3.2]{Kal3},
$Z_{G^\vee}(\varphi (\mb I_F))^\circ$ is a torus. (The torally wild assumption in \cite{Kal3} 
is not needed in the proof.) Note that $Z_{G^\vee}(\varphi (\mb I_F)) \subset M^\vee$ because
$\varphi (\mb P_F) = \mb P_F$. By the proof of \cite[Lemma 5.2.2.2]{Kal2}, we have that 
$T_M^\vee := Z_{M^{\vee,\circ}}\big( Z_{G^\vee} (\varphi (\mb I_F))^\circ \big)$ 
is a maximal torus of $M^\vee$, normalized by $\varphi (\mb W_F)$ and contained in a Borel subgroup 
of $M^\vee$ normalized by $\varphi (\mb I_F)$.~Upon conjugating $\varphi$ by a suitable element of 
$M^\vee$, we may assume without loss of generality that $T_M^\vee$ is contained in a $\mb W_F$-stable 
Borel subgroup of $M^\vee$, s.t.~
\begin{equation}\label{eq:6.20}
\varphi (\mb W_F) \subset N_{M^\vee}(T_M^\vee) \rtimes \mb W_F .
\end{equation}
Now, Ad$(\varphi)$ gives an action of $\mb W_F / \mb P_F$ on $T_M^\vee$, and the $F$-torus
$\mc T_M$ dual to $T_M^\vee$ (with this $\mb W_F$-action) is tamely ramified. 
Since $\varphi$ is discrete, 
$Z_{T_M^\vee}(\varphi (\mb W_F)) / Z(G^\vee)^{\mb W_F}$ is finite and $\mc T_M / Z(\mc G)$ is elliptic. 
By \cite[Remark 4.1.5]{Kal3}, there exist canonical tamely ramified $\chi$-data for
$R(M^\vee,T_M^\vee)$. As in \cite{LaSh} and \cite[\S 6.1]{Kal19a}, these $\chi$-data yield an embedding 
${}^L T_M / \mb P_F \hookrightarrow M^\vee \rtimes \mb W_F / \mb P_F$, 
whose $M^\vee$-conjugacy class is canonical. It inflates to an $L$-embedding
\begin{equation}\label{eq:6.30}
{}^L j_M : {}^L T_M \hookrightarrow M^\vee \rtimes \mb W_F \: \subset \: {}^L G,
\end{equation}
such that ${}^L j_M (\mb W_F)$ stabilizes a pinning of $M^\vee$. By \cite[Lemma 4.1.11]{Kal3},
we may assume that
\begin{equation}\label{eq:6.28}
{}^L j_M (1,x) = (1,x) \text{ for all } x \in \mb I_F.    
\end{equation}
For any embedding $j_M : \mc T_M \to \mc G$ whose dual L-homomorphism is ($G^\vee$-conjugate to)
${}^L j_M$, we have 
\begin{equation}\label{eq:6.17}
j_M (T_M) \text{ is a maximal tamely ramified torus of } G.
\end{equation}
In particular, $Z_{\mc G}(j_M (\mc T_M))$ is a maximal $F$-torus of $\mc G$. Hence \label{i:73}
$T^\vee := Z_{G^\vee}(T_M^\vee)$ is a maximal torus of $G^\vee$, and it is stable under Ad$({}^L j_M ({}^L T_M))$-action 
because $T_M^\vee$ is. We denote ${}^L j_M$ with target ${}^L G$ by ${}^L j$\label{i:74}. 
Then \label{i:75} 
\[
{}^L T := T^\vee \rtimes {}^L j (\mb W_F) 
\]
is the $L$-group of an $F$-torus $\mc T$. Since $\varphi$ is discrete,
$Z_{T^\vee}(\varphi (\mb W_F)) / Z(G^\vee)^{\mb W_F}$ is finite and 
$\mc T / Z(\mc G)$ is elliptic. 

Note that ${}^L j (\mb W_F)$ normalizes a Borel subgroup of $G^\vee$, i.e.~the group
generated by $T^\vee = Z_{G^\vee}(T_M^\vee)$ and the root subgroups $U_{\alpha^\vee}$ for 
$\alpha^\vee \in R (G^\vee,T^\vee)$ such that $\alpha^\vee |_{T_M^\vee}$ is positive with 
respect to the ${}^L j_M (\mb W_F)$-stable Borel subgroup of $M^\vee$ from \eqref{eq:6.20}.
The same arguments as for ${}^L j_M$ ensure that ${}^L j (\mb W_F)$ stabilizes a pinning 
of $G^\vee$.

By construction, $\varphi$ factors as ${}^L j_M \circ \varphi_{T_M}$, where 
$\varphi_{T_M} : \mb W_F \to {}^L T_M$.~Via the LLC for tori \cite{Lan2,Yu}, 
$\varphi_{T_M}$ yields a character $\theta_{T_M}$ of $T_M$. The LLC 
for tori preserves depth zero \cite[Proposition 1.3]{SoXu2}, so $\theta_{T_M}$ has depth zero. 
We can also write that 
$\varphi$ factors as ${}^L j \circ \varphi_T$, where $\varphi_T \in \Phi^0 (T)$.~Then 
$\varphi_T$\label{i:76} determines a depth-zero character $\theta$ of $T$ that extends 
$\theta_{T_M}$.~Since $Z_{G^\vee}(\varphi (\mb I_F))^\circ$ is a torus, by \cite[Lemma 4.1.10]{Kal3}, 
$\theta_{T_M}$ is $F$-nonsingular.~By \eqref{eq:6.17}, this implies that
$\theta$ is $F$-nonsingular as well.~More precisely, for any embedding $j : T \to G$ associated to 
${}^L j$, we see that $\theta$ determines a non-singular depth-zero character $j_* \theta$ of $j(T)$. 
\end{proof}

Note that from the data $(T,{}^L j,\theta)$ in Lemma \ref{lem:6.1}, we can recover
$\varphi = {}^L j \circ \varphi_T$. The following result was established in arbitrary depth 
by Kaletha, we formulate it explicitly here because in depth zero fewer assumptions are necessary.

\begin{lem}\label{lem:6.6}
There is a canonical bijection between the supercuspidal part of $\Phi^0 (G)$ and 
$M^\vee$-conjugacy classes of data $(T,{}^L j,\theta)$ as in Lemma \ref{lem:6.1}.
\end{lem}
\begin{proof}
The argument for bijectivity is given in \cite[Proposition 5.2.7]{Kal1} 
and \cite[Proposition 4.1.8]{Kal3}. With Lemma \ref{lem:6.1} at hand can apply these proofs 
in the special case of depth-zero $L$-parameters, then the tame ramification assumption on 
$T$ and $G$ in \cite{Kal1,Kal3} is not needed.
\end{proof}

\subsection{Extensions related to enhancements of L-parameters} \
\label{par:3.split}

We now fix a Levi subgroup $L$ of $G$ and consider depth-zero supercuspidal L-parameters 
for $L$.~To view $L$ as a rigid inner twist of its quasi-split inner form, we use the same 
$\mc Z$ as for $G$.~In this setup, we obtain the set of supercuspidal parameters in $\Phi_e^0 (L)$, 
which carries a natural action of a group analogous to $W(G,L) = N_G (L)/L$, 
i.e.~by \cite[Proposition 3.1]{ABPS1}, there is a canonical isomorphism
\begin{equation}\label{eq:6.21}
N_G (L)/L \cong N_{G^\vee}(L^\vee \rtimes \mb W_F) / L^\vee .
\end{equation}
We write $W(G^\vee,L^\vee) := N_{G^\vee}(L^\vee) / L^\vee$.~It is easy to see there
is a natural isomorphism\label{i:78}
\begin{equation}\label{eq:6.22}
N_{G^\vee}(L^\vee \rtimes \mb W_F) / L^\vee \cong W(G^\vee,L^\vee)^{\mb W_F} .
\end{equation}
The subgroup $W (M^\vee,L^\vee)^{\mb W_F} = N_{M^\vee}(L^\vee \rtimes \mb W_F) / (M^\vee \cap L^\vee)$ acts by conjugation on $\Phi^0 (L)$.~This action extends to $\Phi^0_e (L)$ in the following way.
Let $(\varphi,\rho)$ represent an element of $\Phi^0_e (L)$, and let $m \in M^\vee$ 
represent an element of $W (M^\vee,L^\vee)^{\mb W_F}$.~Then 
\begin{equation}\label{eq:6.19}
m \cdot (\varphi,\rho) = \left(m \varphi m^{-1}, \rho \circ \mr{Ad}(m)^{-1}\right) .
\end{equation}
If we allow arbitrary (enhanced) Langlands parameters for $L$, then the entire group
$W(G^\vee,L^\vee)^{\mb W_F}$ acts in this way.~Denote the stabilizer of $\varphi \in \Phi^0 (L)$ 
in $W (G^\vee,L^\vee)^{\mb W_F}$, or equivalently in $W (M^\vee,L^\vee)^{\mb W_F}$, by 
$W (G^\vee,L^\vee)^{\mb W_F}_\varphi$.~It acts naturally on $\Irr (\pi_0 (S_\varphi^+))$ and we have
\begin{equation}\label{eq:6.2}
W(G^\vee,L^\vee)^{\mb W_F}_\varphi \cong 
\big( N_{G^\vee}(L^\vee \rtimes \mb W_F) \cap Z_{G^\vee}(\varphi) \big) \big/ Z_{L^\vee}(\varphi) .
\end{equation}
We write\label{i:79}
\[
W(N_{G^\vee}(L^\vee),T^\vee )^{\mb W_F} = \big( N_{G^\vee}(L^\vee,T^\vee) /T^\vee 
\big)^{\mb W_F} \cong N_{G^\vee} (L^\vee, T^\vee \rtimes \mb W_F) / T^\vee. 
\]
Since ${}^L T$ is determined by $\varphi$ and is contained in ${}^L L$, \eqref{eq:6.2} is equal to 
\begin{equation}\label{eq:6.3} 
\big( N_{G^\vee}(L^\vee, {}^L T) \cap Z_{G^\vee}(\varphi) \big) \big/  Z_{L^\vee}(\varphi) \cong 
W(N_{G^\vee}(L^\vee),T^\vee )^{\mb W_F}_{\varphi_T} / W(L^\vee,T^\vee )^{\mb W_F}_{\varphi_T}.
\end{equation}
Here $\mb W_F$ acts via ${}^L j$, and the group $W(L^\vee,T^\vee)^{\mb W_F}$ acts naturally 
on $\Phi^0 (T)$. By the functoriality of the LLC for tori \cite{Yu}, this action satisfies
\begin{equation}\label{eq:6.5}
W (L^\vee,T^\vee)^{\mb W_F}_{\varphi_T} \cong W(\mc L,\mc T)(F)_\theta .
\end{equation}
By \cite[Lemma 3.2.1]{Kal3}, the group $W (L^\vee,T^\vee)^{\mb W_F}_{\varphi_T}$ is abelian.~Hence the conjugation action of $W(N_{G^\vee} (L^\vee), T^\vee)^{\mb W_F}_{\varphi_T}$ on 
$W(L^\vee,T^\vee)^{\mb W_F}_{\varphi_T}$ descends via \eqref{eq:6.3} to an action of
$W(G^\vee,L^\vee)^{\mb W_F}_\varphi$.

\begin{lem}\label{lem:6.2}
The subgroup $W (L^\vee,T^\vee)^{\mb W_F}_{\varphi_T}$ of 
$W(N_{G^\vee} (L^\vee), T^\vee)^{\mb W_F}_{\varphi_T}$ is central.
\end{lem}
\begin{proof}
Via \eqref{eq:6.5} and the similar isomorphism
\begin{equation}\label{eq:6.24}
W (N_{G^\vee}(L^\vee),T^\vee )^{\mb W_F}_{\varphi_T} \cong W(N_{\mc G} (\mc L),\mc T) (F)_\theta ,
\end{equation}
the desired statement is equivalent to Lemma \ref{lem:7.7}.
\end{proof}

Write $\overline{\mc L} := \mc L / \mc Z$ and $\overline{\mc T} := \mc T / \mc Z$.~Let 
$\overline{T}^{\vee,+} = Z_{{\overline T}^\vee} (\varphi_T)$\label{i:80} be the 
preimage of $T^{\vee,\mb W_F}$ in ${\overline T}^\vee$.~By \cite[Corollary 4.3.4]{Kal3} and 
\eqref{eq:6.5}, there is a functorial exact sequence
\begin{equation}\label{eq:6.4}
1 \to \overline T^{\vee,+} \to S_\varphi^+ \to W (L^\vee,T^\vee)^{\mb W_F}_{\varphi_T} \to 1 .
\end{equation}
We note that the group $N_{G^\vee}(L^\vee \rtimes \mb W_F, T^\vee) \cap Z_{G^\vee}(\varphi)$ from 
\eqref{eq:6.3} acts naturally on \eqref{eq:6.4}.~Moreover,
\eqref{eq:6.4} implies that $\overline T^{\vee,+}$ is an abelian normal subgroup of $S_\varphi^+$,
such that $W (L^\vee,T^\vee)^{\mb W_F}_{\varphi_T}$ acts naturally on 
$\Irr (\overline T^{\vee,+})$.~For \label{i:81}$\eta \in \Irr \big( \pi_0 (\overline T^{\vee,+}) \big)$, 
let $\Irr (S_\varphi^+,\eta)$\label{i:82} be the 
set of irreducible representations $\rho$ of $S_\varphi^+$ whose restriction to $\overline T^{\vee,+}$ 
contains $\eta$.~In fact, any such $\rho$ contains the entire 
$W (L^\vee,T^\vee)^{\mb W_F}_{\varphi_T}$-orbit $[\eta]$ of $\eta$.~By \eqref{eq:6.3}, 
$W (G^\vee,L^\vee)^{\mb W_F}_\varphi$ acts naturally on the set of 
$W (L^\vee,T^\vee)^{\mb W_F}_{\varphi_T}$-orbits in $\Irr \big( \pi_0 (\overline{T}^{\vee,+}) \big)$.
In particular, the stabilizer $W (G^\vee,L^\vee)^{\mb W_F}_{\varphi,[\eta]}$ of $[\eta]$ 
is well-defined.~Similar to \eqref{eq:6.2} and \eqref{eq:6.3}, set 
\begin{equation}\label{eq:6.18}
\begin{array}{lll}
W(G^\vee,L^\vee)^{\mb W_F}_{\varphi,\eta} & := & 
\big( N_{G^\vee}(L^\vee,{}^L T)_\eta \cap Z_{G^\vee}(\varphi) \big) / Z_{L^\vee}(\varphi)_\eta,\\ 
W(N_{G^\vee}(L^\vee), T^\vee )^{\mb W_F}_{\varphi_T,\eta} & := & 
\big( N_{G^\vee}(L^\vee,{}^L T)_\eta \cap Z_{G^\vee}(\varphi) \big) / T^{\vee,\mb W_F}.
\end{array}
\end{equation}
The group $W(G^\vee,L^\vee)^{\mb W_F}_{\varphi,\eta}$ embeds naturally in 
$W (G^\vee,L^\vee)^{\mb W_F}_\varphi$, which gives an isomorphism 
$W(G^\vee,L^\vee)^{\mb W_F}_{\varphi,\eta} \cong W (G^\vee,L^\vee)^{\mb W_F}_{\varphi,[\eta]}$. 

Similar to $\Xo (L)$, we consider \label{i:Xov}
\begin{equation}
\Xo (L^\vee) := 
\begin{Bmatrix}\psi \in H^1 ( \mb W_F, Z(L^\vee)) : \psi \text{ has depth
zero in } H^1 (\mb W_F,T^\vee) \\
\text{for every maximal torus } T \subset L \end{Bmatrix}.
\end{equation}
The groups $\Xo (L^\vee)$ and $H^1 ( \mb W_F, Z(L^\vee))$ act naturally on $\Phi (L)$ as 
\begin{equation}\label{eq:6.27}
(z \cdot \varphi)( \gamma, A) = z (\gamma) \varphi (\gamma,A) \text{ for } 
\varphi \in \Phi (L), z \in \Xo (L^\vee), \gamma \in \mb W_F, A \in SL_2 (\C). 
\end{equation}
This action \eqref{eq:6.27} stabilizes $\Phi^0 (L)$ and we have $S_{z \varphi}^+ = S_\varphi^+$. 
Moreover, it extends to an action on $\Phi_e (L)$ and on $\Phi_e^0 (L)$ that acts trivially on 
enhancements. The group\label{i:84}
\begin{equation}\label{eq:6.6}
\tilde N_\varphi := N_{G^\vee}(L^\vee,{}^L T) \cap 
\mr{Stab}_{G^\vee} ( \Xo (L^\vee) \varphi)
\end{equation}
acts naturally on all terms of \eqref{eq:6.4}.
Restricting \eqref{eq:6.4} to the stabilizers of $\eta$ gives the following extension of
$\tilde N_{\varphi,\eta}$-groups 
\begin{equation}\label{eq:6.7} 
1 \to \overline T^{\vee,+} \to (S_\varphi^+)_\eta \to 
W (L^\vee,T^\vee)^{\mb W_F}_{\varphi_T,\eta} \to 1 .
\end{equation}
Pushout of \eqref{eq:6.7} along $\eta$ gives a central extension \label{i:85}
\begin{equation}\label{eq:6.8}
1 \to \C^\times \to \mc E_\eta^{\varphi_T} \to W (L^\vee,T^\vee)^{\mb W_F}_{\varphi_T,\eta} \to 1,
\end{equation}
where the $\tilde N_{\varphi,\eta}$-action from \eqref{eq:6.7} factors via 
\begin{equation}\label{eq:6.23}
W (N_{G^\vee} (L^\vee), T^\vee)^{\mb W_F}_{\eta, \Xo (L^\vee) \varphi_T} := 
\tilde N_{\varphi,\eta} / T^{\vee,\mb W_F} .
\end{equation}
The significance of \eqref{eq:6.8} is that $\Irr (\mc E_\eta^{\varphi_T},\mr{id})$ is naturally 
in bijection with $\Irr (S_\varphi^+,\eta)$, which will parametrize a part of an L-packet 
(in Proposition \ref{prop:8.7}). Next we express \eqref{eq:6.7} and \eqref{eq:6.8} 
as Baer sums of simpler extensions. Firstly, consider the following split extension 
\begin{equation}\label{eq:6.9}
1 \to T^\vee \to T^\vee \rtimes W (L^\vee,T^\vee)^{\mb W_F}_\eta \to 
W (L^\vee,T^\vee )^{\mb W_F}_\eta \to 1.
\end{equation}
Restricting to $\varphi_T (\mb W_F)$-invariants and then taking preimages in 
${\overline T}^\vee \rtimes W (L^\vee,T^\vee)^{\mb W_F}_\eta$ gives a central extension 
\begin{equation}\label{eq:6.10} 
1 \to \overline T^{\vee,+} \to \big( {\overline T}^\vee \rtimes W (L^\vee,T^\vee)^{\mb W_F}_\eta 
\big)^{\varphi_T (\mb W_F)} \to W (L^\vee,T^\vee )^{\mb W_F}_{\eta,\varphi_T} \to 1, 
\end{equation}
whose push out along $\eta$ gives us an extension 
\label{i:86}
\begin{equation}\label{eq:6.11}
1 \to \C^\times \to \mc E_\eta^{\rtimes \varphi_T} \to 
W (L^\vee,T^\vee )^{\mb W_F}_{\eta,\varphi_T} \to 1 .
\end{equation}
For any $z \in \Xo (L^\vee)$, the groups $\varphi_T (\mb W_F)$ and 
$z \varphi_T (\mb W_F)$ centralize the same elements of $L^\vee$.
Hence $\tilde N_{\varphi,\eta}$ acts on \eqref{eq:6.10} via its 
quotient $W ( N_{G^\vee}(L^\vee),T^\vee )^{\mb W_F}_{\eta,\Xo (L^\vee) \varphi_T}$,
and that descends to an action on \eqref{eq:6.11}. 

Secondly, consider the extension
\begin{equation}\label{eq:6.12} 
1 \to T^\vee \to N_{L^\vee} (T^\vee) \to W(L^\vee,T^\vee) \to 1 ,
\end{equation}
endowed with the $\mb W_F$-action from ${}^L j : \mb W_F \to {}^L L$.~By \cite[Lemma 4.5.3]{Kal3}, 
it remains exact upon taking $\mb W_F$-invariants.~(For use in Appendix \ref{sec:B} 
we remark that this still 
works if we replace $N_{L^\vee}(T^\vee)$ by $N_{G^\vee}(L^\vee,T^\vee)$ in \eqref{eq:6.12},
by the same argument.)~Next taking preimages in ${\overline G}^\vee$ and pullback along 
$W (L^\vee,T^\vee)^{\mb W_F}_{\varphi_T} \to W(L^\vee,T^\vee)^{\mb W_F}$ give
\begin{equation}\label{eq:6.13}
1 \to \overline T^{\vee,+} \to N_{{\overline L}^\vee}({\overline T}^\vee)^+_{\varphi_T}
\to W (L^\vee,T^\vee)^{\mb W_F}_{\varphi_T} \to 1 .  
\end{equation}
Then we pull back along $W (L^\vee,T^\vee)^{\mb W_F}_{\eta,\varphi_T} \to 
W(L^\vee,T^\vee)^{\mb W_F}_{\varphi_T}$ to obtain the extension
\begin{equation}\label{eq:6.14}
1 \to \overline T^{\vee,+} \to (N_{{\overline L}^\vee}({\overline T}^\vee)^+)_{\varphi_T,\eta} 
\to W (L^\vee,T^\vee)^{\mb W_F}_{\eta,\varphi_T} \to 1 . 
\end{equation}
Pushout along $\eta$ gives a central extension
\begin{equation}\label{eq:6.15}
1 \to \C^\times \to \mc E_\eta^{0,\varphi_T} \to W (L^\vee,T^\vee )^{\mb W_F}_{\eta,\varphi_T} \to 1 .    
\end{equation}
Again the group $\tilde N_{\varphi,\eta}$ from \eqref{eq:6.6} acts naturally on 
\eqref{eq:6.12}--\eqref{eq:6.15}, which induces an action of 
$W (N_{G^\vee}(L^\vee),T^\vee)^{\mb W_F}_{\eta, \Xo (L^\vee) \varphi_T}$ on \eqref{eq:6.15}.

\begin{lem}\label{lem:6.3}
\enuma{
\item The extension \eqref{eq:6.7} is isomorphic to the Baer sum of \eqref{eq:6.10} and 
\eqref{eq:6.14}, as extensions of $\tilde N_{\varphi,\eta}$-groups.
\item The extension \eqref{eq:6.8} is isomorphic to the Baer sum of \eqref{eq:6.11} and 
\eqref{eq:6.15}, as extensions of 
$W (N_{G^\vee}(L^\vee),T^\vee)^{\mb W_F}_{\eta, \Xo (L^\vee) \varphi_T}$-groups.
}
\end{lem}
\begin{proof}
(a) The following is shown in the proof of \cite[Proposition 8.2]{Kal4}. One has setwise 
splittings of \eqref{eq:6.10} and \eqref{eq:6.14}, from which one constructs 2-cocycles in
$Z^2 (W(L^\vee,T^\vee)^{\mb W_F}_{\eta,\phi_T})$ that classify these two extensions. Then 
the product of these 2-cocycles classifies the extension \eqref{eq:6.7}. Translating back
from 2-cocycles to extensions establishes the desired isomorphism. To these arguments we only 
have to add that they all take place in the category of $\tilde N_{\varphi,\eta}$-groups.

(b) This is a direct consequence of part (a) and the earlier observations that, upon pushout 
along $\eta$, the $\tilde N_{\varphi,\eta}$-actions on the said extensions factor through
\eqref{eq:6.23}.
\end{proof}

We shall also need the following technical result similar to Proposition \ref{prop:7.9}.
Although all the extensions $\mc E_\eta^{0,\varphi'_T}$ with $\varphi'_T \in \Xo (L^\vee)
\varphi_T$ are naturally isomorphic, we need to distinguish them for this purpose. \label{i:87}

\begin{prop}\label{prop:6.4}
The family of extensions $\mc E_\eta^{0,\varphi'_T}$ with 
$\varphi'_T \in \Xo (L^\vee) \varphi_T$ admits a 
$W (N_{G^\vee}(L^\vee),T^\vee)^{\mb W_F}_{\eta, \Xo (L^\vee) \varphi_T}$-equivariant splitting.  
\end{prop}
\begin{proof}
It suffices to show that \eqref{eq:6.15} admits a 
$W (N_{G^\vee}(L^\vee),T^\vee)^{\mb W_F}_{\eta,\varphi_T}$-equivariant splitting, for then
the other required splittings are obtained by conjugating with elements of 
$W (N_{G^\vee}(L^\vee),T^\vee)^{\mb W_F}_{\eta, \Xo (L^\vee) \varphi_T}$.
By Lemma \ref{lem:7.14} (c), we have
\begin{equation}\label{eq:6.36}
W(G^\vee,T^\vee)^{\mb W_F} \cong W(\mc G,\mc T)(F) \cong W(\mc G_y^\circ,\mc T_\ff)(k_F) .
\end{equation}
Since the group $\mc G_y^\circ$ and its maximal torus $\mc T_\ff$ split over an 
unramified extension of $F$, \eqref{eq:6.36} shows that we can realize $W(G^\vee,T^\vee)^{\mb W_F}$ in $W(G^{\vee,\mb I_F,\circ}, T^{\vee,\mb I_F,\circ})^{\mb W_F}$.
(Recall from \eqref{eq:6.28} that $\mb I_F$ acts in the same way on $T^\vee$ and on 
$G^\vee$.)~Furthermore, as in the proof of Proposition \ref{prop:7.9}, we may replace $\theta$ by its
restriction $\theta_\ff$ to $\mc T_\ff (\mf o_F)$, which factors via $\mc T_\ff (k_F)$.~This 
makes the stabilizer of $\theta$ bigger but preserves the non-singularity.~On the dual side, 
it means that we may replace $\varphi$ by $\varphi |_{\mb I_F}$, which has image in 
$T^{\vee,\mb I_F,\circ} \times \mb I_F$.~If $\imath_F \in \mb I_F$ is a topological generator of
$\mb I_F / \mb P_F$, then $\varphi (\imath_F)$, i.e.~the semisimple parameter of $\theta_\ff$, and 
$\varphi (\mb I_F)$ have the same centralizer in $G^{\vee,\mb I_F,\circ}$. 
These replacements for $G^\vee, T^\vee$ and $\varphi$ tell us that it suffices to prove the
proposition assuming $\mb I_F$ acts trivially on $G^\vee$, and with the centralizer of
$\varphi (\imath_F) \in T^\vee$ instead of the centralizer of $\varphi$.

Let $T_\Sc$ be the preimage of $T$ in $L_\Sc$. The element 
\[
\eta \in \Irr \big( \pi_0 (\overline T^{\vee,+}) \big) \cong H^1 (\mc E, \mc Z \to \mc T)
\]
can be constructed as an invariant of $(j,j_0)$ as in \cite[\S 4.4]{Kal3}. Since the embeddings 
$j,j_0 : \mc T \to \mc L$ lift to $\mc T_\Sc \to \mc L_\Sc$, the element $\eta$ lifts to 
\[
\eta_\Sc \in H^1 (F, T_\Sc) \cong \Irr \big( \pi_0 (T_{\Sc}^{\vee,\mb W_F}) \big). 
\]
This means that 
\begin{equation}\label{eq:6.32}
\eta : \pi_0 (\overline T^{\vee,+}) \to \C^\times \quad \text{factors through} \quad  
\eta_\Sc : \pi_0 (T_{\Sc}^{\vee,\mb W_F}) \to \C^\times.
\end{equation}
As in \cite[Corollary 4.5.5]{Kal3}, \eqref{eq:6.15} can be obtained from 
\begin{equation}\label{eq:6.16} 
1 \to T_{\Sc}^{\vee,\mb W_F} \to N_{L_{\Sc}^\vee} (T_{\Sc}^\vee)^{\mb W_F}_{\varphi_T} 
\to W (L_{\Sc}^\vee,T_{\Sc}^\vee)^{\mb W_F}_{\varphi_T} \to 1 
\end{equation}
via pullback along $W (L^\vee,T^\vee)^{\mb W_F}_{\varphi_T,\eta} \to W (L_{\Sc}^\vee,
T_{\Sc}^\vee)^{\mb W_F}_{\varphi_T}$ and pushout along $\eta_\Sc$.~Since 
${L_{\Sc}}^\vee = {L^{\vee}}_\ad = L^\vee / Z(L^\vee)$, we may replace all relevant subgroups 
of $N_{G^\vee}(L^\vee)$ by their image in 
$N_{G^\vee}(L^\vee) / Z(L^\vee)$.~Then it suffices to find an $\tilde N_\varphi$-equivariant 
setwise splitting of \eqref{eq:6.16}, which becomes a group homomorphism upon pushout along 
$\eta_\Sc$. The existence of such a splitting was shown in \cite[Lemma 4.5.4]{Kal3}; it remains 
to show $N_{G^\vee}(L^\vee, {}^L T)_\varphi$-equivariance. 

As an intermediate step in this reduction process, we can divide out $Z(G^\vee)$, such that
$G^\vee$ is of adjoint type. Then $G^\vee$ is a direct product of simple groups, permuted
by $\mb W_F$, and the extension \eqref{eq:6.14} decomposes accordingly. Therefore we may
(and will) assume the $G^\vee$ is simple and of adjoint type.

Let $L_i^\vee$ be a direct factor of $L_{\Sc}^\vee$, which is the product of all simple 
factors of $L_{\Sc}^\vee$ in one $\tilde N_\varphi \rtimes \mb W_F$-orbit, and write
$T_i^\vee := T_{\Sc}^\vee \cap L_i^\vee$. Then the decomposition
\begin{equation}\label{eq:6.44}
W(L_{\Sc}^\vee, T_{\Sc}^\vee) = \prod\nolimits_i \, W(L_i^\vee, T_i^\vee)
\end{equation}
is preserved by $N_{G^\vee}(L^\vee, {}^L T)_\varphi \rtimes \mb W_F$.~Let 
$W(L_i^\vee, T_i^\vee)_{\varphi_T}$ be the image of $W(L_{\Sc}^\vee, T_{\Sc}^\vee)_{\varphi_T}$ 
in $W(L_i^\vee, T_i^\vee)$ via projection onto the $i$-th coordinate in \eqref{eq:6.44}. 
Similar to \eqref{eq:7.51}, there are inclusions
\[
W ({T_\Sc}^\vee L_i^\vee, {T_\Sc}^\vee)_{\varphi_T} \;
\subset \; W(L_i^\vee, T_i^\vee)_{\varphi_T} \; \subset \; 
W(L_i^\vee, T_i^\vee)_{Z({L_\Sc}^\vee) \varphi_T}.
\]
The extension \eqref{eq:6.16} embeds in a direct product of analogous extensions 
\begin{equation}\label{eq:6.45}
1 \to (T_i^\vee)^{\mb W_F} \to N_{L_i^\vee} (T_i^\vee)^{\mb W_F}_{\varphi_T} 
\to W (L_i^\vee,T_i^\vee)^{\mb W_F}_{\varphi_T} \to 1 
\end{equation}
for the various $i$.~Hence it suffices to consider one such extension.~The 
$N_{G^\vee}(L^\vee, {}^L T)_\varphi$-invariants in the pushout of \eqref{eq:6.45} along $\eta_\Sc$ 
are canonically isomorphic to the invariants in the analogue for one $F$-simple factor of $L$ 
except for invariants with respect to a subgroup of $\tilde N_\varphi$.~Therefore, it 
suffices to prove the proposition when $L$ is $F$-simple and simply connected.

Now $L^\vee$ is a direct product of simple factors, and $\mb W_F$ permutes these factors
transitively.~We may replace $L^\vee$ by one of its simple factors and $\mb W_F$ by the 
stabilizer of that simple factor, because this replacement preserves the group of 
$\mb W_F$-invariants. Hence we may assume without loss of generality that $L^\vee$ is simple 
and adjoint. Recall that by the simplifications at the start of the proof we are in a setting 
where $G^\vee \rtimes \langle \Fr_F \rangle$ is dual to a connected finite reductive group. 

By the proof of \cite[Lemma 4.5.4]{Kal3}, we know that $W (L^\vee,T^\vee)^{\mb W_F}_{\varphi_T}$ 
is cyclic or isomorphic to the Klein four group, and by Lemma \ref{lem:6.2} we know that 
$W (L^\vee,T^\vee)^{\mb W_F}_{\varphi_T}$ commutes with $N_{G^\vee}(L^\vee, {}^L T)_\varphi$ in 
$W (G^\vee,T^\vee)^{\mb W_F}$. 

Since $W(G^\vee,T^\vee)$ is a Weyl group containing the Weyl group $W(L^\vee,T^\vee)$ and 
$N_{G^\vee}(L^\vee, {}^L T)_\varphi$ normalizes $W(L^\vee,T^\vee)$, the action of 
$N_{G^\vee}(L^\vee, {}^L T)_\varphi$ on \eqref{eq:6.15} comes from conjugation by elements 
of $W(L^\vee,T^\vee)$ and Dynkin diagram automorphisms of $W(L^\vee,T^\vee)$. Thus we can 
conclude with a case-by-case check. 
This is entirely analogous to the cases I--V in the proof of Proposition \ref{prop:7.9}.
\end{proof}

\section{An LLC for non-singular depth-zero supercuspidal representations}
\label{sec:LLCcusp}

In this section, we first recall the LLC for depth-zero supercuspidal $L$-parameters from 
\cite{DR,Kal3}; then we prove further functorial properties of this LLC.

Consider a supercuspidal $L$-parameter $\varphi \in \Phi^0 (L)$, and factor it as 
${}^L j \circ \varphi_T$ as in Lemma \ref{lem:6.1}.~Fix a Whittaker datum for the 
quasi-split inner twist $L^\flat$ of $L$, which by 
\cite[Lemma 4.2.1]{Kal3} determines an embedding $j_0 : \mc T \to \mc L^\flat$. We also fix 
$\eta \in \Irr \big( \pi_0 (\bar{T}^{\vee,+}) \big)$. Recall the natural isomorphism 
\begin{equation}\label{eq:8.15}
\Irr \big( \pi_0 (\bar{T}^{\vee,+}) \big) \cong H^1 (\mc E, \mc Z \to \mc T) 
\end{equation}
from \cite[Corollary 7.11]{Dil}.~As in \cite[\S 4.2]{Kal3}, these data determine
a rigid inner twist $\mc L'$ of $\mc L$ and an embedding $j : \mc T \to \mc L'$ of $F$-groups, 
such that the invariant $(j,j_0)$ equals $\eta$. Then $j(T)$ is a maximal torus
of $L'$ and $j(T) / Z(L')$ is elliptic.~By Lemma \ref{lem:6.1}, $\varphi_T$
corresponds to a character $\theta$ of $T$, which can also be viewed as a character of $j(T)$.

The torus $j(T)$ determines a unique vertex in the Bruhat--Tits building 
$\mc B (\mc L'_\ad,F)$, as follows. By \eqref{eq:6.17}, $j(\mc T)$ contains a unique maximal 
tamely ramified torus $j(\mc T_M)$ of $\mc L'$.~Let $E$ denote a finite tamely ramified 
extension of $F$ inside $F_s$ such that $\mc T_M$ splits.~Then $j(\mc T_M (E))$ is a maximal 
$E$-split torus in $\mc L' (E)$, so it determines an apartment 
$\mh A_{j (\mc T (E))} = \mh A_{\mc T_M (E)}$ of $\mc B (\mc L'_\ad,E)$. 
Since $\mc T$ is $F$-elliptic, $\mh A_{\mc T_M (E)}^{\mr{Gal}(E/F)}$ consists of just one 
point. By \cite[Lemma 3.4.3]{Kal2}, it is also a vertex of $\mc B (\mc L'_\ad,F)$. In other
words, we can associate the same vertex of $\mc B (\mc L'_\ad, F)$ to $j(T)$ as to $j(T_M)$. 

This vertex gives a unique minimal facet $\ff$ in $\mc B (\mc L', F)$, stabilized by
$j(\mc T)$.~In particular, $j(T) \subset L'_\ff$, and moreover $j(T)$ gives rise to a subgroup
scheme of $\mc L'_\ff$.~In fact, $j(\mc T)$ and $j(\mc T_M)$ determine the same subgroup scheme 
of $\mc L'_\ff (k_F)$, because $\mc T_M$ contains the maximal unramified subtorus of $\mc T$. 
Therefore the discussions from \cite[\S 3]{Kal3} about maximal tamely ramified tori and their
images in $\mc L'_\ff$ apply to $\mc T_M$ and carry over to $\mc T$. 
We define, in the notation from \eqref{eq:7.8}\label{i:89}, 
\begin{equation}\label{eq:8.1}
\Pi_{\varphi,\eta} := \Pi (L', j(T), \theta) \subset \Irr (L').
\end{equation} 
We emphasize that, given $\varphi, \eta$ and a Whittaker datum for $L^\flat$, the construction 
of the Deligne--Lusztig packet $\Pi_{\varphi,\eta}$ is natural, and in particular independent
of the choice of $\varphi$ in its equivalence class.
Let $\eta$ run through $H^1 (\mc E, \mc Z \to \mc T) / W(\mc L,\mc T)(F)_\theta$.~Then $j$ 
runs through all $W(\mc L,\mc T)(F)_\theta$-equivalence classes of embeddings 
$\mc T \to \mc L$.~We define the \textit{compound L-packet of $\varphi$} as \label{i:90}
$\Pi_\varphi := \bigsqcup\nolimits_\eta \, \Pi_{\varphi,\eta}$. Compared to $\Pi_{\phi,\eta}$
the dependence on $\eta$ and $j_0$ has disappeared, so $\Pi_\varphi$ depends
naturally on $\varphi \in \Phi^0 (L)$. It is a set of irreducible representations of various 
rigid inner twists $L'$ of $L$, i.e.~\label{i:91}
\[
\Pi_\varphi = \bigsqcup\nolimits_{L'} \, \Pi_\varphi (L') 
,\quad\text{where } \Pi_\varphi (L') = \Pi_\varphi \cap \Irr (L').
\]
Recall that the local Langlands correspondence for tori from \cite{Yu} matches unitary characters 
with bounded L-parameters.~Together with Lemma \ref{lem:7.13}, this implies that
\begin{equation}\label{eq:8.27}
\text{if } \varphi \text{ is bounded, then } \Pi_\varphi 
\text{ consists entirely of tempered representations.}    
\end{equation}
Conversely, if $\varphi$ is not bounded, then $\theta$ is not unitary, and thus by Lemma 
\ref{lem:7.13}, $\Pi_\varphi$ does not contain any tempered representation.

To make the naturality of $\Pi_\varphi$ more concrete, consider an $F$-automorphism $\gamma$ 
of $\mc L$. Let ${}^L \gamma := \gamma^\vee \rtimes \mr{id}_{\mb W_F}$ be an associated 
L-automorphism of ${}^L L$ (which means that the actions of $\gamma$ and $\gamma^\vee$ 
on the absolute root datum of $\mc L$ are dual). 

\begin{lem}\label{lem:8.3}
The assignments $\varphi \mapsto \Pi_\varphi$ and $(\varphi,\eta) \mapsto \Pi_{\varphi,\eta}$ 
intertwine the action of ${}^L \gamma$ with the action of $\gamma$, i.e.~we have 
$\gamma \cdot \Pi_\varphi = \Pi_{{}^L \gamma \circ \varphi}$ 
and $\gamma \cdot \Pi_{\varphi,\eta} = \Pi_{{}^L \gamma \circ \varphi, {}^L \gamma \cdot \eta}$. 
\end{lem}
\begin{proof}
The set of rigid inner twists of $L$ is parametrized by
\[
\Irr \big( \pi_0 (Z(\bar L^\vee)^+) \big) \cong H^1 (\mc E, \mc Z \to \mc L).
\]
This natural isomorphism intertwines the actions of ${}^L \gamma$ and $\gamma$,
thus the parametrization of rigid inner twists is also equivariant under these actions. 
It is clear from definition \eqref{eq:7.8} that 
\begin{equation}\label{eq:8.5}
\gamma \cdot \Pi_{\varphi,\eta} = \gamma \cdot \Pi (L', j(T), \theta) = 
\Pi ( \gamma (L'), \gamma j (T), \gamma \cdot \theta).     
\end{equation}
The LLC for tori is functorial \cite{Yu}, so intertwines the actions of $\gamma$ and ${}^L \gamma$. 
Hence the L-parameter of $(\gamma j (T), \gamma \cdot \theta)$ is ${}^L \gamma \circ j^L \circ
\varphi_T = {}^L \gamma \circ \varphi$, and the right-hand side of \eqref{eq:8.5} equals 
$\Pi_{{}^L \gamma \circ \varphi, {}^L \gamma \cdot \eta}$. 
Now we combine \eqref{eq:8.5} for all possible $j : \mc T \to \mc L$,
or equivalently for all $\eta \in H^1 (\mc E, \mc Z \to \mc T)$, and obtain the desired
$\gamma \cdot \Pi_\varphi = \Pi_{{}^L \gamma \circ \varphi}$.
\end{proof}

Recall that Langlands \cite{Lan2} defined a natural homomorphism
\begin{equation}\label{eq:8.3}
H^1 (\mb W_F, Z(G^\vee)) \to \Hom (G/G_\Sc ,\C^\times) :\: \psi \mapsto \chi_\psi .
\end{equation}
In \cite[Theorem 3.1]{SoXu2}, we showed that \eqref{eq:8.3} is an isomorphism of topological groups.
In \eqref{eq:8.3}, the left hand side acts naturally on $\Phi_e (G)$ by \eqref{eq:6.27}, while the 
right hand side acts naturally on $\Rep (G)$ by tensoring. In general, it is expected that a local
Langlands correspondence is equivariant with respect to these actions of the groups in 
\eqref{eq:8.3}.
By \cite[Lemma 6]{SoXu2}, \eqref{eq:8.3} restricts to an isomorphism
\begin{equation}\label{eq:8.4}
\Xo (G^\vee) \isom \Xo (G) .    
\end{equation}
As we noted before, it is clear from the definitions that the actions of 
$\Xo (G^\vee)$ and $\Xo (G)$ stabilize the depth-zero parts of $\Phi_e (G)$ 
and $\Rep (G)$. Similar to Lemma \ref{lem:8.3}, it follows immediately from \eqref{eq:7.2} that
\begin{equation}
\Pi_{\psi \cdot \varphi,\eta} = \chi_\psi \otimes \Pi_{\varphi, \eta} := \{ \chi_\psi \otimes \pi : \pi \in \Pi_{\varphi,\eta} \} \qquad
\psi \in \Xo (G^\vee) .
\end{equation}
We now analyze the parametrization of $\Pi_\varphi$ in more detail.~For reasons that will become 
clear in later paragraphs, we assume that $\mc L$ is an $F$-Levi subgroup of a larger reductive 
$F$-group $\mc G$.~For the sake of compatibility, we require that the component groups for 
$\Phi (G)$ and $\Phi (L)$ are constructed with respect to the same finite central subgroup 
$\mc Z \subset Z(\mc G)$.~This implies that our rigid inner twists of $G$ and of $L$ are parametrized 
by $\Irr (Z(\bar G^\vee)^+)$ and $\Irr (Z(\bar L^\vee)^+)$, respectively. By \cite[Lemma 1.1]{Art}, 
\[
Z(L^\vee)^{\mb W_F} = Z(G^\vee)^{\mb W_F} Z(L^\vee)^{\mb W_F,\circ}.
\]
Via the coverings of complex reductive groups dual to $\mc G \to \mc G / \mc Z$
and $\mc L \to \mc L / \mc Z$, this becomes 
$Z(\bar L^\vee)^+ = Z(\bar G^\vee)^+ Z(\bar L^\vee)^{+,\circ}$.~This gives a short exact sequence
\begin{equation}\label{eq:8.8}
1 \to \big( Z(\bar G^\vee)^+ \cap Z(\bar L^\vee)^{+,\circ} \big) / Z(\bar G^\vee)^+ 
\to \pi_0 (Z(\bar G^\vee)^+) \to \pi_0 (Z(\bar L^\vee)^+) \to 1 .
\end{equation}
A similar argument as \cite[Lemma 0.4.9]{KMSW} and \cite[Lemma 6.6]{AMS1}
shows:
\begin{lem}\label{lem:8.4}
\enuma{
\item The character $\zeta_G$ of $\Irr (Z(\bar G^\vee)^+)$ is equal to the pullback of 
$\zeta_L \in \Irr (Z(\bar L^\vee)^+)$ along \eqref{eq:8.8}.
\item An $F$-Levi subgroup $\mc L$ of $\mc G$ is relevant for a rigid inner twist $G'$
of $G$ if and only if $\ker (\zeta_{G'})$ contains 
$\big( Z(\bar G^\vee)^+ \cap Z(\bar L^\vee)^{+,\circ} \big) / Z(\bar G^\vee)^+$.
}
\end{lem}

Recall that the group $W(G,L) = N_G (L) / L$ acts naturally on $\Irr (L)$, stabilizing the
subset of non-singular depth-zero supercuspidal representations of $L$. In \eqref{eq:6.19},
we specified how 
\[
W(G,L) \cong N_{G^\vee}(L^\vee \rtimes \mb W_F) / L^\vee = W(G^\vee,L^\vee)^{\mb W_F} 
\]
acts naturally on $\Phi^0_e (L)$.~By Lemma \ref{lem:8.4}, $W(G^\vee,L^\vee)^{\mb W_F}$ fixes 
the characters $\zeta_G$ and $\zeta_L$.~Recall from \eqref{eq:7.21} and Proposition 
\ref{prop:7.11} that there is a $W(G,L)_{(jT,\theta)}$-equivariant bijection
\begin{equation}\label{eq:8.20}
\Irr (N_L (jT)_\theta, \theta) \to \Pi (L,T,\theta).
\end{equation}
Furthermore, by \eqref{eq:7.19}--\eqref{eq:7.20}, we obtain a canonical bijection 
\begin{equation}\label{eq:8.16}
\Irr (N_{G'}(jT)_\theta, \theta) \longleftrightarrow \Irr (\mc E_\theta^{[x]},\mr{id}).
\end{equation}
On the other hand, the enhanced L-parameters for $\Pi_{\varphi,\eta}$ are given by $\varphi$ 
enhanced with elements of $\Irr (S_\varphi^+, \eta)$. By Clifford theory, the canonical map
\begin{equation}\label{eq:8.17}
\mr{ind}_{(S_\varphi^+)_\eta}^{S_\varphi^+} : 
\Irr ( (S_\varphi^+)_\eta, \eta) \rightarrow \Irr (S_\varphi^+, \eta)
\end{equation}
is bijective. By \eqref{eq:6.7}--\eqref{eq:6.8}, we obtained a natural bijection 
\begin{equation}\label{eq:8.18}
\Irr \big( (S_\varphi^+)_\eta, \eta \big) \longleftrightarrow 
\Irr (\mc E_\eta^{\varphi_T},\mr{id}) .
\end{equation}
Thus a desired internal parametrization of $\Pi_\varphi$ by $\Irr (S_\varphi^+)$ 
should include a comparison between $\mc E_\theta^{[x]}$ and $\mc E_\eta^{\varphi_T}$. 
Now, consider the extensions from \eqref{eq:7.20}, \eqref{eq:7.23}, \eqref{eq:7.25}, 
\eqref{eq:6.8}, \eqref{eq:6.11} and \eqref{eq:6.15}:
\begin{equation}\label{eq:8.19}
\begin{array}{lllll@{\qquad}lllll}
\C^\times & \hookrightarrow & \mc E_\theta^{0,[x]} & \twoheadrightarrow
& W(\mc L^\flat,\mc T^\flat) (F)_{[x],\theta} & \C^\times & \hookrightarrow & \mc E_\eta^{0,\varphi_T} 
& \twoheadrightarrow & W(\mc L^\vee,T^\vee)^{\mb W_F}_{\eta,\varphi_T} \\
\C^\times & \hookrightarrow & \mc E_\theta^{[x]} & \twoheadrightarrow
& W(\mc L^\flat,\mc T^\flat) (F)_{[x],\theta} & \C^\times & \hookrightarrow & 
\mc E_\eta^{\varphi_T} 
& \twoheadrightarrow & W(\mc L^\vee,T^\vee)^{\mb W_F}_{\eta,\varphi_T} \\
\C^\times & \hookrightarrow & \mc E_\theta^{\rtimes [x]} & \twoheadrightarrow
& W(\mc L^\flat,\mc T^\flat) (F)_{[x],\theta} & \C^\times & \hookrightarrow & 
\mc E_\eta^{\rtimes \varphi_T} 
& \twoheadrightarrow & W(\mc L^\vee,T^\vee)^{\mb W_F}_{\eta,\varphi_T} 
\end{array}
\end{equation}
Recall from Lemmas \ref{lem:6.3} and \ref{lem:7.10} that the extensions in the middle
rows of \eqref{eq:8.19} are the Baer sums of those above and below them. By \eqref{eq:6.21} and
\eqref{eq:6.22}, there is a natural isomorphism 
$W(\mc L^\flat, \mc T^\flat)(F) \cong W(L^\vee,T^\vee)^{\mb W_F}$, which restricts to an isomorphism
\begin{equation}\label{eq:8.21}
W(\mc L^\flat, \mc T^\flat)(F)_{[x],\theta} \cong W(L^\vee,T^\vee)^{\mb W_F}_{\eta,\varphi_T}.    
\end{equation}
From the left-hand side of \eqref{eq:8.19} we obtain three families of extensions, by letting
$\theta$ vary over $\Xo (L) \theta'$ for some $\theta'$. We showed in \S\ref{sec:DL} that the group 
\[
W(N_{\mc G^\flat}(\mc L^\flat), \mc T^\flat) (F)_{[x],\Xo (L) \theta} \cong 
W(N_G (L), jT)_{\Xo (L) \theta}
\]
acts naturally on these three families of extensions on the left-hand side of \eqref{eq:8.19}.
On the other hand, from the right-hand side of \eqref{eq:8.19} we obtain three families of
extensions, by letting $\varphi_T$ run over $\Xo (L^\vee) \varphi'_T$ for some $\varphi'_T$.
We showed in \S\ref{sec:L0} that  $W(N_{G^\vee}(L^\vee),T^\vee )^{\mb W_F}_{\eta,\Xo (L^\vee) 
\varphi_T}$ acts canonically on these three families extensions on the right-hand side of 
\eqref{eq:8.19}. From \eqref{eq:6.2}, \eqref{eq:8.21} and the natural isomorphism
$\Xo (L) \cong \Xo (L^\vee)$ from \eqref{eq:8.4}, we obtain a natural isomorphism
\begin{equation}\label{eq:8.22}
W(N_{\mc G^\flat}(\mc L^\flat), \mc T^\flat) (F)_{[x],\Xo (L) \theta} \cong 
W(N_{G^\vee}(L^\vee),T^\vee )^{\mb W_F}_{\eta,\Xo (L^\vee) \varphi_T} .
\end{equation}
Thus it makes sense to say that \eqref{eq:8.22} acts canonically on all extensions in \eqref{eq:8.19}.
Recall that we constructed $\mc E_\theta^{0,[x]}$ in \eqref{eq:7.23}, and $\mc E_\eta^{0,\varphi_T}$ in \eqref{eq:6.15}.

\begin{lem}\label{lem:8.5}
There exists a $W(N_{G^\vee}(L^\vee),T^\vee )^{\mb W_F}_{\eta,\Xo (L^\vee) \varphi_T}
$-equivariant family of group isomorphisms 
\begin{equation}\label{eqn:isom-in-lem-8.5}
\zeta^0 : \mc E_{\chi \otimes \theta}^{0,[x]} \xrightarrow{\sim} \mc E_\eta^{0,\chi \varphi_T} 
\qquad \chi \in \Xo (L) \cong \Xo (L^\vee)
\end{equation}
\end{lem}
\begin{proof}
By canonicity of \eqref{eq:8.22},
this follows from Propositions \ref{prop:7.9} and \ref{prop:6.4}.
\end{proof}
By \cite[Proposition 8.1]{Kal4}, there exists a canonical isomorphism of extensions
\begin{equation}\label{eq:8.23}
    \begin{tikzcd}
        1 \arrow[]{r}{} & \C^\times \arrow[]{r}{} \arrow[equals]{d}{}& 
        \mc E_\theta^{\rtimes [x]}\arrow[]{d}{\zeta^\rtimes} \arrow[]{r}{} &
W(\mc L^\flat, \mc T^\flat) (F)_{[x],\theta} \arrow[]{r}{}\arrow[]{d}{\cong} & 1 \\
1 \arrow[]{r}{} & \C^\times \arrow[]{r}{} & \mc E_\eta^{\rtimes \varphi_T} \arrow[]{r}{} &
W(L^\vee, T^\vee)^{\mb W_F}_{\eta, \varphi_T} \arrow[]{r}{} & 1 
    \end{tikzcd}
\end{equation}
Canonicity ensures that it is equivariant for the natural actions of \eqref{eq:8.22}.
Combined with the Baer sum expressions for the extensions in \eqref{eq:8.19}, 
we have the following.

\begin{lem}\label{lem:8.6}
Every choice of a $\zeta^0$ in \eqref{eqn:isom-in-lem-8.5} 
gives rise to a family of isomorphisms
\[
B(\zeta^0, \zeta^\rtimes) : \mc E_{\chi \otimes \theta}^{[x]} \xrightarrow{\sim} \mc E_\eta^{\chi \varphi_T}
\qquad \chi \in \Xo (L) \cong \Xo (L^\vee),
\]
which is the Baer sum of $\zeta^0$ and $\zeta^\rtimes$ (hence our notation $B(\zeta^0, \zeta^\rtimes)$). This family of isomorphisms is 
$W(N_{G^\vee}(L^\vee),T^\vee )^{\mb W_F}_{\eta,\Xo (L^\vee) \varphi_T}$-equivariant.
\end{lem}
By \eqref{eq:8.20}--\eqref{eq:8.18} and Lemma \ref{lem:8.6}, we 
get $W(N_{G^\vee}(L^\vee),T^\vee )^{\mb W_F}_{\eta,\varphi_T}$-equivariant bijections
\begin{equation}\label{eq:8.2}
\Irr (S_\varphi^+,\eta) \longrightarrow \Irr (\mc E_\eta^{\varphi_T},\mr{id}) 
\underset{B(\zeta^0,\zeta^\rtimes)}{\longrightarrow} \Irr (\mc E_\theta^{[x]}, \mr{id}) 
\longrightarrow \Pi (L,T,\theta) .
\end{equation}
They combine to the following $W(N_{G^\vee}(L^\vee),T^\vee )^{\mb W_F}_{\eta,\Xo
(L^\vee) \varphi_T}$-equivariant bijection
\begin{equation}\label{eq:8.26}
\bigcup_{\varphi' \in \Xo (L^\vee) \varphi_T} \Irr (S_{\varphi'}^+,\eta) 
\longleftrightarrow \bigcup_{\theta' \in \Xo (L) \theta} \Pi (L,T,\theta' ) .
\end{equation}

\begin{prop}\label{prop:8.7}
For all $\eta, [x]$ as above, fix 
$W(N_{G^\vee}(L^\vee),T^\vee)^{\mb W_F}_{\Xo (L^\vee) \varphi_T}$-equivariant 
choices of $\zeta^0$ in Lemma \ref{lem:8.6} and of coherent splittings $\epsilon$ as in 
\eqref{eq:7.21}.~Then \eqref{eq:8.2} and \eqref{eq:8.26} for these $\eta, [x]$ combine to a 
$W(N_{G^\vee}(L^\vee),T^\vee)^{\mb W_F}_{\varphi_T}$-equivariant bijection 
$\Pi_\varphi \longleftrightarrow \Irr (S_\varphi^+)$, 
and a $W(N_{G^\vee}(L^\vee),T^\vee)^{\mb W_F}_{\Xo (L^\vee) \varphi_T}$-equivariant bijection
\begin{equation}\label{eqn:union-Lpackets-prop8.7}
\bigcup_{\varphi' \in \Xo (L^\vee) \varphi_T} \Pi_{\varphi'} \longleftrightarrow
\bigcup_{\varphi' \in \Xo (L^\vee) \varphi_T} \Irr (S_{\varphi'}^+).
\end{equation}
Under this bijection, tempered representations correspond to bounded enhanced L-parameters.
\end{prop}
\begin{proof}
By the construction of $\mc E_\eta^{\varphi_T}$ in \eqref{eq:6.7} and \eqref{eq:6.8}, 
$\Irr (S_\varphi^+)$ is the union of the sets $\Irr (S_\varphi^+,\eta)$, where $\eta$ runs 
through $W(L^\vee,T^\vee)^{\mb W_F}_{\varphi_T}$-equivalence classes.~By definition, 
$\Pi_\varphi$ is the union of the corresponding packets $\Pi_{\varphi,\eta} = 
\Pi (L,T,\theta)$.~Thus \eqref{eq:8.2} and \eqref{eq:8.26} combine to give the desired 
bijections. Recall from earlier that every single bijection \eqref{eq:8.2} is 
$W(N_{G^\vee}(L^\vee),T^\vee )^{\mb W_F}_{\eta,\varphi_T}$-equivariant.
The $W(N_{G^\vee}(L^\vee),T^\vee)^{\mb W_F}_{\Xo (L^\vee) \varphi_T}$-equivariance
of the choices in the construction ensures that the collection of bijections 
$\Pi_\varphi \longleftrightarrow \Irr (S_\varphi^+)$ is also 
$W(N_{G^\vee}(L^\vee),T^\vee)^{\mb W_F}_{\Xo (L^\vee) \varphi_T}$-equivariant, 
and does not depend on the choices of $\eta, [x]$ within their 
$W(N_{G^\vee}(L^\vee),T^\vee)^{\mb W_F}_{\Xo (L^\vee) \varphi_T}$-equivalence classes.\\
The correspondence between temperedness and boundedness follows from \eqref{eq:8.27}.
\end{proof}

Since $\mathfrak{X}_{\nr} (L) \cong \big( Z (L^\vee)^{\mb I_F} \big)_{\mb W_F}^{\; \circ}$ is
contained in $\Xo (L) \cong \Xo (L^\vee)$, the union of the L-packets in 
\eqref{eqn:union-Lpackets-prop8.7} forms a collection of Bernstein components in $\Irr^0 (L')$,
for rigid inner twists $L'$ of $L$. Similarly, the collection of enhanced L-parameters for \eqref{eqn:union-Lpackets-prop8.7} forms a union of Bernstein components in
$\Phi_e^0 (L')$ for the same $L'$.~Proposition \ref{prop:8.7}, combined with the main 
result of \cite{Kal3}, gives a bijection
\begin{equation}\label{eq:8.7}
\Phi^0_\cusp (L)_{ns} := \begin{Bmatrix} (\varphi,\rho) \in \Phi^0_e (L): \\ \varphi \text{ supercuspidal} \end{Bmatrix}  \longleftrightarrow 
\begin{Bmatrix} \pi \in \Irr^0 (L): \pi \text{ is non-}\\ 
\text{singular supercuspidal}\end{Bmatrix}=:\Irr^0_\cusp (L)_{ns}.
\end{equation}
We will write instances of \eqref{eq:8.7} as
\[
(\varphi,\rho) \mapsto \pi (\varphi,\rho) \qquad \text{or} \qquad \pi \mapsto (\varphi_\pi, \rho_\pi) .
\]
Recall from \cite[p.20--23]{Lan} and \cite[\S10.1]{Bor} that every $\varphi \in \Phi (L)$
determines a character $\chi_\varphi$\label{i:92} of $Z(L)$ constructed as follows. One first
embeds $\mc L$ into a connected reductive $F$-group $\tilde{\mc L}$ satisfying
$\mc L_\der = \tilde{\mc L}_\der$, such that $Z(\tilde{\mc L})$
is connected.~Then one lifts $\varphi$ to an L-parameter $\tilde \varphi$ for
$\tilde L = \tilde{\mc L}(F)$.~The natural projection ${}^L \tilde{\mc L} \to
{}^L Z(\tilde{\mc L})$ produces an L-parameter $\tilde{\varphi}_z$ for 
$Z(\tilde L) = Z(\tilde{\mc L})(F)$, and via the local Langlands correspondence
for tori, $\tilde{\varphi}_z$ uniquely determines a character $\chi_{\tilde \varphi}$ of 
$Z (\tilde L)$. Then $\chi_\varphi$ is given by restricting $\chi_{\tilde \varphi}$
to $Z(L)$. By \cite[p. 23]{Lan}, $\chi_\varphi$ does not depend on the choices made above. 

\begin{lem}\label{lem:8.8}
In \eqref{eq:8.7}, the $Z(L)$-character of $\pi (\varphi,\rho)$ 
is precisely $\chi_\varphi$.
\end{lem}
\begin{proof}
Since all the admissible embeddings $j : \mc T \to \mc L$ are $\mc L$-conjugate, the
preimage of $Z(\mc L)$ under $j$ does not depend on the choice of $j$.~We may denote
it by $Z_{\mc T}(\mc L)$.~Then any $j$ restricts to the same bijective embedding
$j : Z_{\mc T}(\mc L) \to Z(\mc L)$.~By \eqref{eq:7.65}, every 
$\pi \in \Pi (L,jT,\theta) \subset \Pi_\varphi (L)$
admits the central character $\theta |_{Z(L)} = \theta|_{Z_{\mc T}(\mc L)(F)}$, and
by construction $\varphi_T$ is the L-parameter of $\theta$. 

Now we follow the procedure in \cite{Lan} recalled above.~There is a unique
maximal torus $\widetilde{j \mc T}$ of $\tilde{\mc L}$ containing $j \mc T$.~By 
functoriality of the LLC for tori \cite{Yu}, $\tilde \varphi$ determines a character of 
$\widetilde{j \mc T}(F)$ that extends $\theta$.~Hence $\tilde \varphi_z$ corresponds to a 
character of $Z(\tilde L)$ that extends $\theta_{Z(L)}$, and $\chi_\varphi  = \theta_{Z(L)}$.
\end{proof}
However, the bijections in \eqref{eq:8.2}, Proposition \ref{prop:8.7} and \eqref{eq:8.7} 
are not entirely canonical, because they depend on choices of isomorphisms between two
extensions in \cite[\S 4.5]{Kal3}. In \eqref{eq:8.2}, one can adjust the
bijection by tensoring one side with a character of $W(\mc L,\mc T)(F)_{[x],\theta} \cong W(L^\vee,T^\vee)^{\mb W_F}_{\varphi_T,\eta}$.~This corresponds to changing the coherent 
splitting $\epsilon$; see \cite[Definition 2.7.6]{Kal3}. Proposition \ref{prop:8.7} shows that 
there are many ways to make the choices for \eqref{eq:8.7} so that the LLC becomes 
$W(N_{G^\vee}(L^\vee),T^\vee)^{\mb W_F}_{\Xo (L^\vee) \varphi_T}$-equivariant. If one is 
willing to work with more L-packets (in one $W(G^\vee,L^\vee)^{\mb W_F}$-orbit) at once, then 
one can even make \eqref{eq:8.7} $W(G^\vee,L^\vee)^{\mb W_F}$-equivariant. 
In principle, the choices for $\varphi$'s in different $W(G^\vee,L^\vee)^{\mb W_F}$-orbits are
independent. But, of course, we want to align them in a nice way. 

\begin{thm}\label{thm:8.1}
Suppose that on every $\Xo (L)$-orbit of the datum $(jT,\theta)$, we choose the 
same $\zeta^0$ from Lemma \ref{lem:8.6} and the same $\epsilon$ from \eqref{eq:7.21} and 
\eqref{eq:7.9}.~The LLC from Proposition \ref{prop:8.7} is equivariant with respect to 
$\Xo (L)$ and 
$W(N_{G^\vee}(L^\vee),T^\vee)^{\mb W_F}_{\Xo (L^\vee) \varphi_T}$.
\end{thm}
\begin{proof}
In \cite[start of \S 2.7 and Fact 2.7.2]{Kal3}, we can take 
$N_G (L,jT)_{\Xo (L) \theta} / L_{\ff,0+}$ in the role of $\Gamma$ and 
$W(N_{G^\vee}(L^\vee),T^\vee)^{\mb W_F}_{\Xo (L^\vee) \varphi_T}$ in the role
of $\overline \Gamma$, and the proof goes through. Thus this collection of $\epsilon$'s (or 
rather the ensuing actions of $N_L (jT)_\theta$ on $\kappa_{T,\theta}^{L,\epsilon}$) are
equivariant for $W(N_{G^\vee}(L^\vee),T^\vee)^{\mb W_F}_{\Xo (L^\vee) \varphi_T}$.~By 
Lemma \ref{lem:8.6}, the chosen $\zeta^0$'s form a 
$W(N_{G^\vee}(L^\vee),T^\vee)^{\mb W_F}_{\Xo (L^\vee) \varphi_T}$-equivariant
family. Hence Proposition \ref{prop:8.7} applies.

Suppose that $(\varphi,\rho') \in \Phi_e (L)$ corresponds to 
$\kappa_{(jT,\theta,\rho)}^{L,\epsilon}$ via Proposition \ref{prop:8.7}. Take 
$z \in \Xo (L^\vee)$ with image $\chi \in \Xo (L)$ \eqref{eq:8.4}. Then 
\[
S_{z \varphi}^+ = S_\varphi^+ ,\quad (S_{z \varphi}^+)_\eta = (S_\varphi^+)_\eta 
\quad \text{and} \quad \mc E_\eta^{z \varphi_T} = \mc E_\eta^{\varphi_T}.
\]
This gives a family $\{ (z \phi,\rho') : z \in \Xo (L^\vee) \}$ in
$\Phi_e (L)$. We need to figure out the corresponding family of $L$-representations. Via 
Lemma \ref{lem:8.6}, $\rho'$ is translated into $\rho \in \Irr (\mc E_\theta^{[x]},\mr{id})$.
Since $\chi$ is a character of the entire group $L$, we have 
\begin{equation}\label{eq:8.13}
\begin{aligned}
& N_{\mc L_\ff (k_F)} (j \mc T)_\theta = N_{\mc L_\ff (k_F)} (j \mc T)_{\chi \otimes \theta} 
,\text{ and }\\
& \Irr (N_{\mc L_\ff (k_F)} (j \mc T)_\theta, \chi \otimes \theta) = 
\{ \chi \otimes \rho : \rho \in \Irr (N_{\mc L_\ff (k_F)} (j \mc T)_\theta, \theta) \} .
\end{aligned}    
\end{equation}
The condition on $\zeta^0$ in the theorem means that $(z\phi,\rho')$ corresponds to 
$\chi \otimes \theta \in \Irr (jT)$ and to the representation $\chi \otimes \rho \in 
\Irr \big( \mc E_{\chi \otimes \theta}^{[x]}, \mr{id} \big)$ obtained from $\rho$ via the
natural isomorphsm $\mc E^{[x]}_{\chi \otimes \theta} \cong \mc E^{[x]}_\theta$ from
\eqref{eq:8.14}.~We can also view $\chi \otimes \rho$ as an element of 
$\Irr (N_L (jT)_{\chi \otimes \theta}, \chi \otimes \theta)$, then it is obtained
from $\rho \in \Irr (N_L (jT)_\theta$ by tensoring with $\chi$.~By \cite[Theorem 2.7.7.3]{Kal3} 
(which uses the condition on $\epsilon$), we have
\begin{equation}\label{eq:8.12}
\chi \otimes \kappa_{(jT,\theta,\rho)}^{L,\epsilon} \cong 
\kappa_{(jT,\chi \otimes \theta,\chi \otimes \rho)}^{L,\epsilon} .
\end{equation}
Thus $(z\phi,\rho')$ corresponds to $\chi \otimes \kappa_{(jT,\theta,\rho)}^{L,\epsilon}$
in Proposition \ref{prop:8.7}.
\end{proof}

The equivariance in Theorem \ref{thm:8.1} can be upgraded when we consider all non-singular 
depth-zero supercuspidal representations and all enhanced supercuspidal $L$-parameters 
in Proposition \ref{prop:8.7}, as follows.       

\begin{thm}\label{thm:8.2}
The choices in the LLC \eqref{eq:8.7} can be made such that, for all non-singular supercuspidal 
depth-zero representations of $L$, \eqref{eq:8.7} is equivariant with respect to 
\[
W(G,L) \ltimes \Xo (L) \cong W(G^\vee, L^\vee)^{\mb W_F} \ltimes \Xo (L^\vee).
\]
\end{thm}    
\begin{proof}
For a given supercuspidal parameter $\varphi \in \Phi^0 (L)$, consider the set
\begin{equation}\label{eq:8.25}
\big\{ (\varphi',\rho) \in \Phi_e^0 (L) : 
\varphi' \in \Xo (L^\vee) \varphi, \rho \in \Irr (S_{\varphi'}^+) \big\}
\end{equation}
from Proposition \ref{prop:8.7}. The $W(G^\vee,L^\vee)^{\mb W_F}$-stabilizer 
$W(G^\vee,L^\vee)^{\mb W_F}_{\Xo (L^\vee) \varphi}$ of \eqref{eq:8.25} consists 
precisely of the elements that come from 
$W(N_{G^\vee}(L^\vee),T^\vee)^{\mb W_F}_{\Xo (L^\vee) \varphi_T}$. 
By construction, other elements of $W(G^\vee,L^\vee)^{\mb W_F}$ do not map any element of 
\eqref{eq:8.25} to an element of \eqref{eq:8.25}. We apply the 
$W(G^\vee,L^\vee)^{\mb W_F}_{\Xo (L^\vee) \varphi} \ltimes \Xo (L^\vee)
$-equivariant LLC from Theorem \ref{thm:8.1} to \eqref{eq:8.25}, and we extend it
$W(G^\vee,L^\vee)^{\mb W_F} \ltimes \Xo (L^\vee)$-equivariantly to the 
$W(G^\vee,L^\vee)^{\mb W_F}$-orbit of \eqref{eq:8.25} and to 
$\bigcup\limits_{\varphi' \in \Xo (L^\vee) \varphi} W(G,L) \cdot \Pi_{\varphi'}(L)$.~Next, 
we let $\varphi$ run through a set of representatives for
\[
\{ \varphi' \in \Phi^0 (L) : \varphi' \text{ supercuspidal} \} \big/ 
W(G^\vee, L^\vee)^{\mb W_F} \ltimes \Xo (L^\vee),
\]
and we carry out the above steps for all those $\varphi$.
\end{proof}
\begin{rem}
The group of unramified characters of $L$ is contained in $\Xo (L)$, so
Theorem \ref{thm:8.2} also holds with $\mf{X}_\nr (L)$ instead of $\Xo (L)$.

Ideally, the LLC from Theorem \ref{thm:8.2} should be equivariant with respect to all 
$F$-automorphisms of $\mc L$, as conjectured in \cite[Conjecture 2]{SolFunct} and  
\cite[Conjecture 2.12]{Kal5}. Unfortunately, this seems to be out of reach at the time of writing. 
\end{rem}

 \section{Some subquotients of the Iwahori--Weyl group}
\label{sec:IW}
The following sections \S \ref{sec:Morris}--\S\ref{sec:HeckeL} treat Hecke
algebras for non-supercuspidal representations of $G$, and do not rely on the previous sections. We now slightly adjust the earlier setup.
Let $\mc S$\label{i:93} be a maximal $F$-split torus in $\mc G$. Let
$R(G,S)$\label{i:94} be the root system of $(G,S)$.
Let \label{i:102}$\mh A_S := X_* (\mc S) \otimes_\Z \R$ be the apartment of $\mc B (\mc G,F)$ associated 
to $\mc S$. The walls of $\mh A_S$ determine an affine root system $\Sigma$, and the map 
that sends an affine root to its linear part is a canonical surjection 
\label{i:96}$D : \Sigma \to R(G,S)$.

Let $C_0$ be a chamber in $\mh A_S$ whose closure contains $0$. 
Let $\Delta_\af$\label{i:97} be the set of simple affine roots in $\Sigma$ determined 
by $C_0$. The associated set of simple affine reflections $S_\af$\label{i:98} 
generates an affine Weyl group $W_\af$.\label{i:99} The standard Iwahori subgroup 
of $G$ is $P_{C_0}$, and the Iwahori--Weyl group of $(G,S)$ is\label{i:100}
\begin{equation}\label{eq:2.31}
W := N_G (S) / (N_G (S) \cap P_{C_0}) \cong Z_G (S) / (Z_G (S) \cap P_{C_0}) \rtimes W(G,S) .
\end{equation}
Note that it acts on $\mh A_S = X_* (S) \otimes_\Z \R = X_* (Z_G (S)) \otimes_\Z \R \cong 
Z_G (S) / Z_G (S)_\cpt \otimes_\Z \R$, with $Z_G (S) / (Z_G (S) \cap P_{C_0})$ acting by translations, and $W(G,S)$ as the stabilizer
of a chosen special vertex of $C_0$.~The kernel of this action is the finite subgroup 
$Z_G (S)_\cpt / Z_{P_{C_0}}(S)$, where the subscript cpt denotes the (unique) maximal 
compact subgroup.~Furthermore, $W$ contains $W_\af$ as the subgroup supported on the kernel of the Kottwitz 
homomorphism for $G$. The group \label{i:101}
$\Omega := \{ w \in W : w (C_0) = C_0\}$ forms a complement to $W_\af$, and we have
\begin{equation}\label{eq:2.10}
W = W_\af \rtimes \Omega .
\end{equation}
Let $\ff$ be a facet in $\mc B (\mc G,F)$. Since $G$ acts transitively on the set 
of chambers of $\mc B (\mc G,F)$, we may assume without loss of generality that $\ff$ is
contained in the closure of $C_0$. Let $\Sigma_\ff$ be the set of affine roots that vanish 
on $\ff$, and let \label{i:116}$J := \Delta_\af \cap \Sigma_\ff$ be its subset of simple
affine roots. Its associated set of affine reflections $\{ s_j : j \in J\}$ generates 
a finite Weyl group $W_J$, which can be identified with the Weyl group of the 
$k_F$-group $\mc G_\ff^\circ (k_F)$ with respect to the torus $\mc S (k_F)$.

Let $R^c_\ff$\footnote{The superscript $c$ will become self-explanatory in the next paragraph.} 
be the set of roots for $(G,S)$ that are constant on $\ff$,
a parabolic root subsystem of $R(G,S)$. Let $\mc L$ be the Levi $F$-subgroup 
of $\mc G$ determined by $\mc S$ and $R^c_\ff$. By \cite[Theorem 2.1]{Mor3}, 
$P_{L,\ff} := P_\ff \cap L$ is a maximal parahoric subgroup of $L$ (associated to a 
facet $\mf f_L \supset \ff$) and we have
\begin{equation}\label{eq:2.13}
\begin{array}{ccccc}
\hat P_{L,\ff} / P_{L,\ff} & = & 
(\hat{P}_\ff \cap L) / (P_\ff \cap L)
& \cong & \hat P_\ff / P_\ff ,\\
P_{L,\ff} / L_{\ff,0+} & = & (P_\ff \cap L) / (G_{\ff,0+} \cap L) & 
\cong & P_\ff / G_{\ff,0+} .
\end{array} 
\end{equation}
Let $R_\ff$ be the image of $\Sigma_\ff$ in $R (G,S)$. Its closure 
$(\Q R_\ff) \cap R (G,S)$ is precisely $R^c_\ff$. Although $R^c_\ff$ and
$R_\ff$ have the same rank, it is quite possible that they have different Weyl groups. 
We write
\begin{equation}\label{eq:2.6}
\Omega_\ff = \{ \omega \in \Omega : \omega (\ff) = \ff \} =
\{ \omega \in \Omega : P_\ff \omega P_\ff \subset G_\ff \}
\cong G_\ff / P_\ff .
\end{equation}
Since $P_\ff$ and $\hat P_\ff$ depend only on $\ff$, they have the same normalizer in
$G$, i.e.~the set-theoretic stabilizer $G_\ff$ of $\ff$. Let $\Omega_\ff^0 \cong \hat P_{L,\ff} / P_{L,\ff} = \hat P_\ff / P_\ff$ be the point-wise stabilizer of $\ff$ in $\Omega_\ff$. By \eqref{eq:2.6}, we have
\begin{equation}\label{eq:2.7}
G_\ff / P_\ff \cong \Omega_\ff \quad \text{and} \quad
G_\ff / \hat P_\ff \cong \Omega_\ff / \Omega_\ff^\circ .
\end{equation}

\begin{lem}\label{lem:3.2}
\enuma{
\item The group $\Omega_\ff = G_\ff / P_\ff$ is abelian and finitely generated. 
\item Suppose moreover that $\ff$ is a minimal facet in $\mc B (\mc G,F)$, or equivalently that
$P_\ff$ is maximal parahoric subgroup of $G$.~The group $G_\ff / \hat P_\ff$ is abelian 
and isomorphic to a lattice in $X_* (Z^\circ (G)) \otimes_\Z \R$.~In particular, it is 
free of the same rank as $X_* (Z^\circ (G))$.
} 
\end{lem}
\begin{proof}
(a) Recall the Kottwitz homomorphism $\kappa$ for $G$, which takes values in a subquotient of
the algebraic character group of $Z(G^\vee)$.
Since $\ker \kappa\cap G_\ff=P_\ff$ (see for example \cite[Propositions 7.6.4 and 11.5.4]{KaPr}), 
we have $G_\ff / P_\ff \cong \kappa (G_\ff)$. This is a subquotient of 
$X^* (Z(G^\vee))$, hence is abelian and finitely generated.

(b) The group under consideration is a quotient of $G_\ff / P_\ff$, 
thus by part (a) it is abelian and finitely generated. Note that $L = G$ by the minimality 
of $\ff$, thus $\hat P_{L,\ff} = \hat P_\ff$. 
For any $x \in \ff$, the $X_* (Z^\circ (G)) \otimes_\Z \R$-orbit of $x$ equals $\ff$, and 
$\hat P_\ff$ equals the stabilizer of $x$ in $G$. We define a map
$t : G_\ff / \hat P_\ff \to X_* (Z(G)) \otimes_\Z \R$ by $g \cdot x = x + t(g)$, where the addition takes place in $\mh A_S$.~Since translations by
$X_* (Z^\circ (G)) \otimes_\Z \R$ commute with the action of $G$ on $\mc B (\mc G,F)$, we can compute
\[
x + t(gg') = g g' \cdot x = g \cdot (x + t(g')) = x + t(g) + t (g') .
\]  
This shows that $t$ is a group homomorphism, and by definition its kernel is trivial. Hence $t$
provides an isomorphism between $G_\ff / \hat P_\ff$ and a subgroup of $X_* (Z(G)) 
\otimes_\Z \R$. The latter is a real vector space, so all its subgroups are torsion-free. On 
$Z(G) \subset G_\ff$, the map $t$ boils down to the quotient map
\[
Z^\circ (G) \to Z^\circ (G) / Z^\circ (G)_\cpt \cong X_* (Z^\circ (G)) .
\]
On the other hand, the group of translations $t(G)$ of $X_* (Z^\circ (G)) \otimes_\Z \R$,
which arises from the action of $G$, contains $X_* (Z^\circ (G))$ with finite index, 
thus it is a lattice in $X_* (Z^\circ (G)) \otimes_\Z \R$. Now we have inclusions 
$X_* (Z^\circ (G)) \subset t(G_\ff) \subset t(G)$, where the outer sides are lattices in $X_* (Z^\circ (G)) \otimes_\Z \R$, thus the group in the 
middle is as well.
\end{proof}
The group $N_W (W_J) / W_J$ is isomorphic to\label{i:103}
\begin{equation}
N_W (J) := \{ w \in N_W (W_J) : w (J) = J \} .
\end{equation}
Note that $\Omega_\ff \subset N_W (J)$. The group $N_W (J)$ naturally contains an affine 
Weyl group $W_\af (J)$, obtained in the following way. For $\alpha \in \Delta_\af \setminus J$, 
the reflections in the roots $J \cup \{\alpha\}$ generate a finite Weyl group 
$W_{J \cup \alpha} \subset W$. Its longest element $w_{J \cup \alpha}$ satisfies
\[
w_{J \cup \alpha} (J) \; \subset \; w_{J \cup \alpha} (J \cup \{\alpha\}) = -J \cup \{-\alpha\}. 
\]
Suppose that $w_{J \cup \alpha} (J) = -J$ and let $w_J$ be the longest element of $W_J$. 
Then \label{i:104}
\begin{equation}\label{eqn:defn-v(alpha,J)}
v (\alpha,J) := w_{J \cup \alpha} w_J = w_J w_{J \cup \alpha}
\end{equation}
has order two. Such involutions are called an $R$-elements in \cite[\S 2.6]{Mor1}. 
Let $\Delta_{\ff,\af}$\label{i:105}
be the set of $\alpha \in \Delta_\af \setminus J$ for which there
exists an $\alpha'$ from the same simple factor of $G$, such that both $v(\alpha,J)$ and
$v(\alpha',J)$ are $R$-elements. The group\label{i:175}
\begin{equation}
\Omega (J) = \{ w \in N_W (J) : w (\Delta_{\ff,\af}) = \Delta_{\ff,\af} \}
\end{equation}
contains $\Omega_\ff$. By \cite[Corollary 2.8 and \S 7]{Mor1}, the set $S_{\ff,\af} := \{ v(\alpha,J) : \alpha \in \Delta_{\ff, \af} \}$ 
\label{i:106}generates an affine Weyl group $W_\af (J)$ in $N_W (J)$ and we have 
\begin{equation}\label{eq:2.4}
N_W (J) = W_\af (J) \rtimes \Omega (J).
\end{equation}
The inverse image of $N_W (J)$ in $N_G (S)$ stabilizes the facet of $\mc B (\mc L,F)$ 
containing $\ff$, and it normalizes $L$ and $P_{L,\ff}$. By \eqref{eq:2.13}, this induces 
an action of $N_W (J)$ on 
\[
\mc G^\circ_\ff (k_F) \cong P_{L,\ff} / L_{\ff,0+} \cong P_\ff / G_{\ff,0+} .
\]
Let $\sigma$ be an irreducible cuspidal representation of $\mc G_\ff^\circ (k_F)$, also 
viewed as a representation of $P_\ff$ by inflation. 
As in \cite[\S 4.16]{Mor1}, we define\label{i:108}
\begin{equation}\label{eq:2.11}
W(J,\sigma) = \{ w \in N_W (J) : w \cdot \sigma \cong \sigma \} .
\end{equation}
For any of the groups $G_\ff, \hat P_\ff, \Omega_\ff, \Omega_\ff^0$, we add a subscript
$\sigma$ to indicate the subgroup that stabilizes $\sigma$.\label{i:174}
Then, by \eqref{eq:2.7}, we have
\begin{equation}\label{eq:2.8}
G_{\ff,\sigma} / P_\ff \cong \Omega_{\ff,\sigma} \quad \text{and} \quad 
G_{\ff,\sigma} / \hat P_{\ff,\sigma} \cong \Omega_{\ff,\sigma} / \Omega_{\ff,\sigma}^0 
\cong G_{\ff,\sigma} \hat P_\ff / \hat P_\ff .
\end{equation}
If moreover $\ff$ is a minimal facet, then Lemma \ref{lem:3.2} (b) applies equally well to 
$G_{\ff,\sigma} / \hat P_{\ff,\sigma}$.

Recall that $\Omega_\ff^0$ is the point-wise stabilizer of $\ff$ in $\Omega_\ff$.
\begin{lem}\label{lem:3.5}
$\Omega_\ff^0$ is a central subgroup of $N_W (J)$, which intersects the 
commutator subgroup of $N_W (J)$ only in $\{1\}$. The same holds with $W (J,\sigma)$
instead of $N_W (J)$.
\end{lem}
\begin{proof}
The first claim is shown in \cite[(38)--(39)]{SolLLCunip}. By \eqref{eq:2.10} and the
commutativity of $\Omega_\ff$ as in Lemma \ref{lem:3.2} (a), the commutator subgroup of $N_W (J)$
is contained in $W_\af$. By \eqref{eq:2.10}, $\Omega_\ff \subset \Omega$ 
intersects $W_\af$ trivially. Hence the intersection of $\Omega_\ff$ with
the commutator subgroup of $N_W (J)$ or of $W (J,\sigma)$ is just the identity.
\end{proof}

For $\alpha$ such that $v (\alpha,J) \in S_{\ff,\af} \cap W(J,\sigma)$, by 
\cite[Proposition 6.9]{Mor1}, one obtains a number $p_\alpha \in \Z_{\geq 1}$, which we will 
denote instead by $q_{\sigma,\alpha}$\label{i:109}. We set
\begin{align}
\nonumber & S_{\ff,\af, \sigma} := \{ v(\alpha,J) \in S_{\ff,\af} \cap W(J,\sigma) : 
q_{\sigma,\alpha} > 1 \} ,\\
\label{eq:2.12} & \Delta_{\ff,\af,\sigma} := 
\{ \alpha \in \Delta_{\ff,\af} : s_\alpha \in S_{\ff,\af,\sigma} \} ,\\
\nonumber & \Omega (J,\sigma) := \{ w \in W(J,\sigma) : w (\Delta_{\ff,\af,\sigma}) = 
\Delta_{\ff,\af,\sigma} \} .
\end{align}
Here $S_{\ff,\af,\sigma}$ is the set of simple reflections in an affine Coxeter group 
$W_\af (J,\sigma)$\label{i:110}. It is known from \cite[\S 7]{Mor1} that
\begin{equation}\label{eq:2.3}
W(J,\sigma) = W_\af (J,\sigma) \rtimes \Omega (J,\sigma) .
\end{equation}
We warn the reader that $\Omega (J,\sigma)$ need not be contained in $\Omega (J)$. For any of the
above groups, a subscript $L$ means that they are constructed from $L$ instead of $G$.
In particular we have the Iwahori--Weyl group $W_L$ of $L$ and likewise $W_L (J,\sigma)$.

By \cite[Theorem 6.11]{MoPr2} or \cite[Theorem 4.5]{Mor3}, $(P_\ff, \sigma)$ is a type 
in the sense of Bushnell--Kutzko, for a sum of
finitely many Bernstein blocks in $\Rep (G)$, say $\Rep (G)_{(P_\ff,\sigma)}$\label{i:177}
Moreover every Bernstein block consisting of depth-zero representations arises in this way.

\begin{lem}\label{lem:2.2}
\enuma{
\item The category $\Rep (L)_{(P_{L,\ff},\sigma)}$ determines $(P_{L,\ff},\sigma)$ up
to $L$-conjugacy.
\item Let $W(G,L)_\sigma$ be the stabilizer of $\Rep (L)_{(P_{L,\ff},\sigma)}$ in
$N_G (L) / L$. The natural map $W(J,\sigma) / W_L (J,\sigma) \to W(G,L)_\sigma$ is
an isomorphism.
}
\end{lem}
\begin{proof}
(a) Let $\Irr (L)_{(P_{L,\ff},\sigma)}$ be the set of irreducible objects in 
$\Rep (L)_{(P_{L,\ff},\sigma)}$. Since $\ff$ becomes a vertex in $\mc B (\mc L_\ad, F)$,
all these irreducible $L$-representations $\omega$ are supercuspidal and have depth zero
\cite[\S 6]{MoPr2}. Each such $\omega$ has $(P_{L,\ff},\sigma)$ as unrefined minimal
K-type in the sense of \cite{MoPr1,MoPr2}. By \cite[Theorem 5.2]{MoPr2}, $\omega$
determines $(P_{L,\ff},\sigma)$ up to $L$-conjugacy.

(b) Since $J$ and $S$ determine $L$ and $Z_G (S) \subset L$, every element of 
\[
N_W (J) \subset N_G (S) / (Z_G (S) \cap P_{C_0})
\] 
normalizes $L$. The natural map $W(J,\sigma) \to N_G (L) / L$ has kernel 
\[
W_L (J,\sigma) = (W(J,\sigma) \cap N_L (S)) / (Z_G (S) \cap P_{C_0}) .
\] 
By definition, $W(J,\sigma)$ stabilizes $(P_{L,\ff},\sigma)$, so it stabilizes 
$\Rep (L)_{(P_{L,\ff},\sigma)}$. Thus we obtain an injection
\begin{equation}\label{eq:2.9}
W(J,\sigma) / W_L (J,\sigma) \hookrightarrow W(G,L)_\sigma .
\end{equation}
Conversely, let $w \in W(G,L)_\sigma$. By part (a), $w$ stabilizes the $L$-conjugacy 
class of $(P_{L,\ff},\sigma)$. Hence we can represent $w$ by an element 
$n \in N_G (L, P_{L,\ff})$ that stabilizes $\sigma$. Then $n (Z_G (S) \cap P_{C_0})
\in W(J,\sigma)$, thus $w$ lies in the image of \eqref{eq:2.9}.
\end{proof}

\section{\texorpdfstring{$q$}{q}-parameters for Hecke algebras}
\label{sec:Morris}
Consider the Hecke algebra\label{i:128}
\[
\cH (G,P_\ff,\sigma) = \{ f : G \to \End_\C (V_\sigma) \mid f (kg k') = \sigma (k)
f(g) \sigma (k') \; \forall g \in G, k,k' \in P_\ff \} .
\]
We note that this algebra is sometimes called $\cH (G,\sigma^\vee)$, for instance in
\cite{BuKu}. By \cite[Theorem 4.5]{Mor3} and \cite[Theorem 4.3]{BuKu}, its category 
of right modules is equivalent to $\Rep (G)_{(P_\ff,\sigma)}$.

\begin{thm}[\cite{Mor1},~Theorem~7.12]\label{thm:2.4} 
The algebra $\mc H (G,P_\ff,\sigma)$ has a basis $\{ T_w : w \in W(J,\sigma) \}$
with $T_w$ supported on $P_\ff w P_\ff$.~There exist a parameter function $q_\sigma : S_{\ff,\af,\sigma} \to \Z_{>1}$, and 
a 2-cocycle $\mu_\sigma$ of $W(J,\sigma)$ which factors through 
$\Omega (J,\sigma) \cong W(J,\sigma) / W_\af (J,\sigma)$, such that 
\begin{equation}
\mc H (G,P_\ff, \sigma) \cong \mc H (W_\af (J,\sigma), q_\sigma) \rtimes 
\C [\Omega (J,\sigma), \mu_\sigma ].
\end{equation}
Here $\mc H (W_\af (J,\sigma), q_\sigma)$ denotes an Iwahori--Hecke algebra,
$\C [\Omega (J,\sigma), \mu_\sigma ]$ is a twisted group algebra and $T_\omega T_{s_\alpha} T_\omega^{-1} = T_{\omega s_\alpha \omega^{-1}}$, where 
 $s_\alpha \in S_{\ff,\af,\sigma}$ and $\omega \in \Omega (J,\sigma)$.
\end{thm}

For background on Iwahori--Hecke algebras and affine Hecke algebras we refer
to \cite{SolSurv}. It is known that $\mu_\sigma$ is sometimes nontrivial, see 
\cite[Proposition 4.4]{HeVi}. All the parameters $q_\sigma (s)$ are powers of
the cardinality $q_F$\label{i:166} of $k_F$.

From now on, let $\sigma$ be a non-singular cuspidal representation of 
$P_\ff/G_{\ff,0+} = \mc G_\ff^\circ (k_F)$. Thus, by definition, $\sigma$ is an 
irreducible constituent of a Deligne--Lusztig representation 
$R_{\mc T_\ff (k_F)}^{\mc G_\ff^\circ (k_F)}(\theta_\ff)$, where $\mc T_\ff$ is an
elliptic maximal $k_F$-torus in $\mc G_\ff^\circ$, and $\theta_\ff$ is a non-singular
character of $\mc T_\ff (k_F)$. (This is slightly more general than in Section \ref{sec:DL},
because $\theta_\ff$ is $k_F$-non-singular but we do not require $F$-non-singularity.)
Since we are dealing with smooth complex $G$-representations,
the values of $\theta_\ff$ must lie in $\C$. On the other hand, the techniques used in 
Deligne--Lusztig theory apply to representations of $\overline{\Q}_\ell$-vector spaces,
where $\ell$ is a prime number different from $p$. We fix an isomorphism $\C \cong 
\overline{\Q}_\ell$, so that we can regard $\theta_\ff$ as taking values in both fields.~Let $\mc G_\ff^\vee$ be the dual group of $\mc G_\ff^\circ$ and let
$s \in \mc G_\ff^\vee (\overline{\F}_p)$ be an element in the semisimple conjugacy
class corresponding to $\theta_\ff$. The non-singularity of $\theta_\ff$ means that 
$Z_{\mc G_\ff^\vee}(s)^\circ$ is a torus. 
Recall that $\Omega_s := \pi_0 \big(Z_{\mc G_\ff^\vee}(s) \big)$ is a finite abelian group. 

In general, a Deligne--Lusztig 
representation is a virtual representation, but in our setting, by \cite[Theorem 8.3]{DeLu}, 
$\pm R_{\mc T_\ff (k_F)}^{\mc G_\ff^\circ (k_F)}(\theta_\ff)$ is an actual representation 
for a suitable sign $\pm$. Moreover, Lusztig \cite[Proposition 5.1]{Lus-discon} showed that 
$\pm R_{\mc T_\ff (k_F)}^{\mc G_\ff^\circ (k_F)}(\theta_\ff)$ is a direct sum of mutually 
inequivalent irreducible cuspidal representations, which can be parametrized as $\sigma_\chi$, 
where $\sigma_1 = \sigma$ and $\chi$ runs through the characters of 
\begin{equation}\label{eqn:Omega-theta-ff}
\Omega_{\theta_\ff} := W(\mc G_\ff^\circ, \mc T_\ff) (k_F)_{\theta_\ff}.
\end{equation}

\begin{prop}\label{prop:2.1}
The full subcategory of $\Rep (\mc G_\ff^\circ (k_F))$ generated by 
$\pm R_{\mc T_\ff (k_F)}^{\mc G_\ff^\circ (k_F)}(\theta_\ff)$ is equivalent to 
$\Rep (\Omega_{\theta_\ff})$.
\end{prop}
\begin{proof}
Let $\Rep_s (\mc G_\ff^\circ (k_F))$ be the category of $\mc G_\ff^\circ (k_F)$-representations
generated by the objects in the geometric Deligne--Lusztig series determined by 
$s \in \mc G_\ff^\vee (\overline{\F}_p)$. Let $\mc H$ be the split reductive 
$\overline{\F}_p$-group dual to $\mc H^\vee := Z_{\mc G_\ff^{\vee}}(s)^\circ$. By the
nonsingularity of $\theta_\ff$ and $s$, both $\mc H$ and $\mc H^{\vee}$ are tori. By 
\cite[Corollary 12.7]{LuYu} \footnote{Strictly speaking,\cite[Corollary 12.7]{LuYu} gives 
\eqref{eq:2.1} with, in addition, fixed cells on both sides, but one can sum over all cells 
for $\mc G_\ff^\circ (k_F)$ to obtain the desired \eqref{eq:2.1} as stated.}, 
there exists a canonical equivalence of categories
\begin{equation}\label{eq:2.1}
\Rep_s (\mc G_\ff^\circ (k_F)) \cong \bigoplus\nolimits_\beta \,
\Rep_1 \big( \mc H (\overline{\F}_p)^{\Fr_\beta} \big)^{\Omega_{s,\beta}} .
\end{equation}
Here $\beta$ runs through a finite set that parametrizes certain $k_F$-forms of $\mc H$, 
which appear in the notation via a Frobenius action $\Fr_\beta$. They correspond to various 
rational Deligne--Lusztig series associated to $s$.~The superscript $\Omega_{s,\beta}$ means
equivariant objects, with respect to a (canonical up to canonical isomorphisms) action of the
subgroup $\Omega_{s,\beta} \subset \Omega_s$ that stabilizes $\beta$.
Since $\mc H$ is a torus, the category $\Rep_1 \big( \mc H (\overline{\F}_p)^{\Fr_\beta} \big)$ 
consists precisely of all multiples $\tau$ of the trivial representation of 
$\mc H (\overline{\F}_p)^{\Fr_\beta}$. An $\Omega_{s,\beta}$-equivariant
structure on such a representation $\tau$ consists of a collection of morphisms
$\tau \to \omega \cdot \tau$ for $\omega$ running through the appropriate extension of 
$\Omega_{s,\beta}$ by $\mc H (\overline{\F}_p)^{\Fr_\beta}$, compatible with the group 
structure of that extension. However, since $\tau \big( \mc H (\overline{\F}_p)^{\Fr_\beta} \big) = \mr{id}$, the extension can be ignored and we only need morphisms
$\tau \to \omega \cdot \tau$ for $\omega \in \Omega_{s,\beta}$. In other words, 
An $\Omega_{s,\beta}$-equivariant structure on $\tau$ 
just means that it is upgraded to an $\Omega_{s,\beta}$-representation.
Thus \eqref{eq:2.1} simplifies to an equivalence of categories
\begin{equation}\label{eq:2.2}
\Rep_s (\mc G_\ff^\circ (k_F)) \cong \bigoplus\nolimits_\beta \, \Rep (\Omega_{s,\beta}) .
\end{equation}
The representation $\pm R_{\mc T_\ff (k_F)}^{\mc G_\ff^\circ (k_F)}(\theta_\ff)$ generates
a unique rational Deligne--Lusztig series in \eqref{eq:2.2}, and the associated $\beta$
satisfies $\Omega_{s,\beta} \cong \Omega_{\theta_\ff}$.
\end{proof}

We now analyze the $q$-parameters in Theorem \ref{thm:2.4}, and express them more
explicitly.~Consider $\alpha \in \Delta_\af \setminus J$. By the observations in 
\cite[\S 3]{Mor1}, there 
exists a unique facet $\mf f_\alpha$ of $\ff$, such that the associated parahoric subgroup
$P_{\mf f_\alpha}$ has set of simple affine roots $J \cup \{\alpha\}$. The group \label{i:111}
$\mc G_{\mf f_\alpha}^\circ (k_F) = P_{\mf f_\alpha} / U_{\mf f_\alpha}$ contains $P_\ff 
/ U_{\mf f_\alpha}$ as a parabolic subgroup $\mc Q_{\ff,\alpha} (k_F)$ with Levi factor 
$\mc G_\ff^\circ (k_F) = P_\ff / G_{\ff,0+}$.
The quotient $G_{\ff,0+} / G_{\mf f_\alpha,0+}$ is isomorphic to $\mc U_\alpha (k_F)$, where
$\mc U_\alpha$ denotes the root subgroup of $\mc G_{\mf f_\alpha}^\circ$ associated to $\alpha$. 
By \cite[\S 6.7--6.9]{Mor1}, the number $q_{\sigma,\alpha}$ from \eqref{eq:2.12}
can be computed (whenever defined) from $\ind_{\mc Q_{\ff,\alpha} (k_F)}^{\mc G_{\mf f_\alpha}^\circ (k_F)} (\sigma)$, thus entirely in terms of connected algebraic groups over finite fields.

Recall from \S \ref{sec:IW} that $\mc S$ defines a maximal $k_F$-split torus in $\mc G_\ff^\circ$, 
and that $\mc T_\ff$ is a maximal $k_F$-torus in $\mc G_\ff^\circ$. By 
\cite[Proposition 5.1.10.b]{BrTi2} and \cite[Lemma 2.3.1]{DeB}, $\mc T_\ff$ can be lifted
to an $\mf o_F$-torus in $\mc P_\ff^\circ$.
We fix one such lift, so that we can view $\mc T_\ff$ as an $\mf o_F$-torus which splits 
over an unramified extension of $F$. Since $\mc T_\ff (\mf o_F) \subset P_\ff$ normalizes 
$P_\ff$ and $P_{\mf f_\alpha}$, hence it normalizes their pro-unipotent radicals 
$G_{\ff,0+}$ and $G_{\mf f_\alpha,0+}$. Thus $\mc T_\ff (\mf o_F)$ normalizes 
$\mc Q_{\ff,\alpha}(k_F) = P_\ff / G_{\mf f_\alpha,0+}$ and its unipotent radical
$\mc U_\alpha (k_F) = G_{\ff,0+} / G_{\mf f_\alpha,0+}$. In particular, 
\begin{equation}\label{eq:4.5}
\alpha \in \Delta_\af \setminus J \text{ can be viewed as a root of } \mc T_\ff ,
\end{equation}
and it is defined over $\mf o_F$ because $\mc U_\alpha$ is defined over $\mf o_F$.
This also shows that $s_\alpha$ can be represented in $N_{\mc G_\ff^\circ (k_F)}(\mc T_\ff)$.

\begin{lem}\label{lem:4.4}
\enuma{
\item $\pm R_{\mc T_\ff (k_F)}^{\mc G_{\mf f_\alpha}^\circ (k_F)} (\theta_\ff)=\bigoplus\limits_{\chi \in \Irr (\Omega_{\theta_\ff})} \ind_{\mc Q_{\ff,\alpha} (k_F)}^{\mc G_{\mf f_\alpha
}^\circ (k_F)} (\sigma_\chi)$.
\item The representations $\ind_{\mc Q_{\ff,\alpha} (k_F)}^{\mc G_{\mf f_\alpha}^\circ (k_F)} 
(\sigma_\chi)$ and $\ind_{\mc Q_{\ff,\alpha} (k_F)}^{\mc G_{\mf f_\alpha}^\circ (k_F)} 
(\sigma_{\chi'})$ with $\chi \neq \chi' \in \Irr (\Omega_{\theta_\ff})$ do not have any 
irreducible subquotients in common.
}
\end{lem}
\begin{proof}
(a) This follows from the description of 
$\pm R_{\mc T_\ff (k_F)}^{\mc G_\ff^\circ (k_F)}(\theta_\ff)$ above \eqref{eqn:Omega-theta-ff}, 
combined with the transitivity of Deligne--Lusztig induction as in \cite[\S 11.5]{DiMi}.

(b) By Frobenius reciprocity and the Mackey formula, $\ind_{\mc Q_{\ff,\alpha} (k_F)}^{
\mc G_{\mf f_\alpha}^\circ (k_F)} (\sigma_\chi)$ and\\ $\ind_{\mc Q_{\ff,\alpha} (k_F)}^{
\mc G_{\mf f_\alpha}^\circ (k_F)
} (\sigma_{\chi'})$ can only have common irreducible constituents if $\sigma_{\chi'}$ is isomorphic
to $\sigma_\chi$ or to $s_\alpha \cdot \sigma_\chi$. We already noted in the paragraph above
\eqref{eqn:Omega-theta-ff} that, by \cite[Proposition 5.1]{Lus-discon},
$\sigma_\chi \not\cong \sigma_{\chi'}$. Thus the only remaining possibility is if $\sigma_{\chi'}$ 
is isomorphic to $s_{\alpha}\cdot\sigma_{\chi}$. 

By Proposition \ref{prop:2.1}, $\pm R_{\mc T_\ff (k_F)}^{\mc G_\ff^\circ (k_F)} 
(\theta_\ff)$ generates a full subcategory $\mc C$ of $\Rep (\mc G_\ff^\circ (k_F))$ that is
equivalent to $\Rep (\Omega_{\theta_\ff})$.~Suppose $s_\alpha$ maps
some element of $\mc C$ into $\mc C$.~Then $s_\alpha \in N_{\mc G_\ff^\circ (k_F)}(\mc T_\ff)$ must
fix $\theta_\ff$, and $s_\alpha$ stabilizes $\mc C$.~For $\omega \in \Omega_{\theta_\ff}$,
$\omega s_\alpha \omega^{-1} \in W(\mc G_{\mf f_\alpha}^\circ, \mc T_\ff)$ is a reflection
associated to a root defined over $k_F$.~By ellipticity of
$\mc T_\ff$ in $\mc G_\ff^\circ$, $\omega s_\alpha \omega^{-1}$ must equal $s_\alpha$.~Hence 
$s_\alpha$ commutes with $\Omega_{\theta_\ff}$, and thus the action of
$s_\alpha$ on $\mc C \cong \Rep (\Omega_{\theta_\ff})$ is trivial.~In particular,
$s_\alpha \cdot \sigma_\chi \cong \sigma_\chi$, which is not isomorphic to $\sigma_{\chi'}$.
\end{proof}

Let $\mc B$ be a Borel subgroup (not necessarily defined over $k_F$) of $\mc G_{\mf 
f_\alpha}^\circ$, such that $\mc T_\ff \, \mc U_\alpha \subset \mc B \subset \mc Q_{\ff,\alpha}$. 
Let $\mc L_\alpha$\label{i:112} be the $k_F$-subgroup of $\mc G_{\mf f_\alpha}^\circ$ generated 
by $\mc T_\ff \cup \mc U_\alpha \cup \mc U_{-\alpha}$. It is a twisted Levi subgroup of 
$\mc G_{\mf f_\alpha}^\circ$.  Let
\begin{equation}
R_{\mc T_\ff \subset \mc B}^{\mc G_{\mf f_\alpha}^\circ} : \Rep (\mc T_\ff (k_F)) \to
\Rep (\mc G_{\mf f_\alpha}^\circ (k_F))
\end{equation}
denote the Deligne--Lusztig induction functor. By transitivity of Deligne--Lusztig induction
\cite[\S 11.5]{DiMi}, there are natural isomorphisms
\begin{equation}\label{eq:4.1}
\begin{aligned}
\ind_{\mc Q_{\ff,\alpha}(k_F)}^{\mc G_{\mf f_\alpha}^\circ (k_F)}
R_{\mc T_\ff (k_F)}^{\mc G_\ff^\circ (k_F)} (\theta_\ff) & = R_{\mc G_\ff^\circ \subset 
\mc Q_{\ff,\alpha}}^{\mc G_{\mf f_\alpha}^\circ} 
R_{\mc T_\ff \subset \mc B \cap \mc G_\ff^\circ}^{\mc G_\ff^\circ} (\theta_\ff) \\
& \cong R_{\mc T_\ff \subset \mc B}^{\mc G_{\mf f_\alpha}^\circ} (\theta_\ff) 
\cong R_{\mc L_\alpha \subset \mc Q_{\ff,\alpha} \mc L_\alpha}^{\mc G_{\mf f_\alpha}^\circ}
R_{\mc T_\ff \subset \mc L_\alpha \cap \mc B}^{\mc L_\alpha} (\theta_\ff) \\
& = R_{\mc L_\alpha \subset \mc Q_{\ff,\alpha} \mc L_\alpha}^{\mc G_{\mf f_\alpha}^\circ}
\ind_{(\mc L_\alpha \cap \mc B)(k_F)}^{\mc L_\alpha (k_F)} (\theta_\ff) .
\end{aligned}
\end{equation}
Notice that $\mc L_\alpha \cap \mc B = \mc T_\ff \ltimes \mc U_\alpha$ is defined over $k_F$.

\begin{lem}\label{lem:4.5}
There are canonical algebra isomorphisms
\begin{align*}
\End_{\mc G_{\mf f_\alpha}^\circ (k_F)} \big( R_{\mc T_\ff \subset \mc B}^{\mc G_{\mf 
f_\alpha}^\circ} \theta_\ff \big) & 
\cong \bigoplus\nolimits_{\chi \in \Irr (\Omega_{\theta_\ff})} 
\End_{\mc G_{\mf f_\alpha}^\circ (k_F)} \big(
\ind_{\mc Q_{\ff,\alpha}(k_F)}^{\mc G_{\mf f_\alpha}^\circ (k_F)} \sigma_\chi \big) \\
& \cong \bigoplus\nolimits_{\chi \in \Irr (\Omega_{\theta_\ff})} \End_{\mc L_\alpha (k_F)}
\big( \ind_{(\mc L_\alpha \cap \mc B)(k_F)}^{\mc L_\alpha (k_F)} \theta_\ff \big) . 
\end{align*}
\end{lem}
\begin{proof}
The first isomorphism follows from Lemma \ref{lem:4.4}. By \eqref{eq:4.1}, we obtain 
\[
\End_{\mc G_{\mf f_\alpha}^\circ (k_F)} \big( R_{\mc T_\ff \subset \mc B}^{\mc G_{\mf 
f_\alpha}^\circ} \theta_\ff \big) \cong \End_{\mc G_{\mf f_\alpha}^\circ (k_F)} \big( 
R_{\mc L_\alpha \subset \mc Q_{\ff,\alpha} \mc L_\alpha}^{\mc G_{\mf f_\alpha}^\circ}
\ind_{(\mc L_\alpha \cap \mc B)(k_F)}^{\mc L_\alpha (k_F)} (\theta_\ff) \big) .
\]
By functoriality of $R_{\mc L_\alpha \subset \mc Q_{\ff,\alpha} \mc L_\alpha}^{
\mc G_{\mf f_\alpha}^\circ}$, the algebra $\End_{\mc L_\alpha (k_F)}
\big( \ind_{(\mc L_\alpha \cap \mc B)(k_F)}^{\mc L_\alpha (k_F)} \theta_\ff \big)$ embeds in
$\End_{\mc G_{\mf f_\alpha}^\circ (k_F)} \big( R_{\mc T_\ff \subset \mc B}^{\mc G_{\mf 
f_\alpha}^\circ} \theta_\ff \big)$. Let $\mr{pr}_\chi$ denote the projection
\[
\End_{\mc G_{\mf f_\alpha}^\circ (k_F)} \big( R_{\mc T_\ff \subset \mc B}^{\mc G_{\mf 
f_\alpha}^\circ} \theta_\ff \big) \to \End_{\mc G_{\mf f_\alpha}^\circ (k_F)} \big( 
\ind_{\mc Q_{\ff,\alpha}(k_F)}^{\mc G_{\mf f_\alpha}^\circ (k_F)} \sigma_\chi \big)
\]
obtained from the first isomorphism in the statement. We have
\[
\mr{pr}_\chi R_{\mc L_\alpha \subset \mc Q_{\ff,\alpha} \mc L_\alpha}^{\mc G_{\mf 
f_\alpha}^\circ} (A) \in \End_{\mc G_{\mf f_\alpha}^\circ (k_F)} \big( \ind_{\mc P_{\ff,
\alpha}(k_F)}^{\mc G_{\mf f_\alpha}^\circ (k_F)} \sigma_\chi \big) \text{ for }
A \in \End_{\mc L_\alpha (k_F)} \big( 
\ind_{(\mc L_\alpha \cap \mc B)(k_F)}^{\mc L_\alpha (k_F)} \theta_\ff \big).
\]
By construction, we have  
\[
R_{\mc L_\alpha \subset \mc Q_{\ff,\alpha} \mc L_\alpha}^{\mc G_{\mf f_\alpha}^\circ} (A) =
\sum\nolimits_{\chi \in \Irr (\Omega_{\theta_\ff})} \mr{pr}_\chi
R_{\mc L_\alpha \subset \mc Q_{\ff,\alpha} \mc L_\alpha}^{\mc G_{\mf f_\alpha}^\circ} (A),
\]
and $\mr{pr}_\chi R_{\mc L_\alpha \subset \mc Q_{\ff,\alpha} 
\mc L_\alpha}^{\mc G_{\mf f_\alpha}^\circ} (A)$ is invertible if $A$ is so.
Since $\mc G_\ff^\circ$ is a maximal Levi subgroup of $\mc G_{\mf f_\alpha}^\circ$ and
$s_\alpha$ stabilizes $\sigma_\chi$, we have $\dim_\C \End_{\mc G_{\mf f_\alpha}^\circ (k_F)} \big( \ind_{\mc Q_{\ff,\alpha}(k_F)}^{
\mc G_{\mf f_\alpha}^\circ (k_F)} \sigma_\chi \big) = 2$.~Comparing dimensions, we see that
$\bigoplus_{\chi \in \Irr (\Omega_{\theta_\ff})} \mr{pr}_\chi R_{\mc L_\alpha \subset 
\mc Q_{\ff,\alpha} \mc L_\alpha}^{\mc G_{\mf f_\alpha}^\circ}$ is an algebra isomorphism
\begin{multline}\label{eq:4.2}
\bigoplus\limits_{\chi \in \Irr (\Omega_{\theta_\ff})} \hspace{-2mm} \End_{\mc L_\alpha (k_F)}
\big( \ind_{(\mc L_\alpha \cap \mc B)(k_F)}^{\mc L_\alpha (k_F)} \theta_\ff \big) \to 
\bigoplus\limits_{\chi \in \Irr (\Omega_{\theta_\ff})} \hspace{-2mm} 
\End_{\mc G_{\mf f_\alpha}^\circ (k_F)} \big( \ind_{\mc Q_{\ff,\alpha}(k_F)}^{
\mc G_{\mf f_\alpha}^\circ (k_F)} \sigma_\chi \big). \qedhere
\end{multline}
\end{proof}

Using Lemma \ref{lem:4.5}, we can simplify the computation of the parameter function 
$q_\sigma$ from Theorem \ref{thm:2.4}. Recall the notations $w_{J\cup \alpha}$ and $v(\alpha,J)$ 
from \eqref{eqn:defn-v(alpha,J)}. \label{i:113}

\begin{lem}\label{lem:4.6}
Let $\alpha \in \Delta_\af \setminus J$ be such that $w_{J \cup \alpha}(J) = -J$. Then:
\enuma{
\item $s_\alpha \cdot \theta_\ff = \theta_\ff$ if and only if 
$v(\alpha,J) \cdot \sigma \cong \sigma$.    
\item Suppose that $v(\alpha,J) \in W(J,\sigma)$.~The parameter $q_{\sigma,\alpha} := 
q_\sigma (v (\alpha,J))$ equals the parameter $q_{\theta,\alpha} := q_\theta (s_\alpha)$ computed 
from $\mc T_\ff (k_F)$, $\theta_\ff$ and $\mc L_\alpha (k_F)$. 
}
\end{lem}
\begin{proof}
(a) By \cite[Corollary 2.3 and Proposition 3.9]{HoLe}, $\End_{\mc G_{\mf f_\alpha}^\circ (k_F)} \big(
\ind_{\mc Q_{\ff,\alpha}(k_F)}^{\mc G_{\mf f_\alpha}^\circ (k_F)} \sigma \big)$ has dimension two if $v(\alpha,J) \cdot \sigma \cong \sigma$ and has dimension one otherwise. 
By Lemma \ref{lem:4.5}, $\End_{\mc G_{\mf f_\alpha}^\circ (k_F)} \big(\ind_{\mc Q_{\ff,\alpha}(k_F)}^{
\mc G_{\mf f_\alpha}^\circ (k_F)} \sigma \big)$ is naturally isomorphic to $\End_{\mc L_\alpha (k_F)}
\big( \ind_{(\mc L_\alpha \cap \mc B)(k_F)}^{\mc L_\alpha (k_F)} \theta_\ff \big)$. Again by \cite{HoLe}, the latter algebra has dimension two if $s_\alpha \cdot \theta_\ff = 
\theta_\ff$ and has dimension one otherwise.

(b) By the construction of $\mc H (W_\af (J,\sigma), q_\sigma)$, 
$T_{v(\alpha,J)}$ satisfies a quadratic relation
\begin{equation}\label{eq:4.3}
(T_{v(\alpha,J)} + 1)(T_{v(\alpha,J)} - q_{\sigma,\alpha}) 
\text{ with } q_{\sigma,\alpha} \in \Z_{\geq 1}.
\end{equation}
We can view $T_{v(\alpha,J)}$ as an element of 
\[
\mc H (P_{\mf f_\alpha} / G_{\mf f_\alpha,0+}, P_\ff/G_{\mf f_\alpha,0+}, \sigma) =
\mc H (\mc G_{\mf f_\alpha}^\circ (k_F), \mc Q_{\ff,\alpha}(k_F), \sigma) \cong
\End_{\mc G_{\mf f_\alpha}^\circ (k_F)} \big( \ind_{\mc Q_{\ff,\alpha}(k_F)}^{\mc 
G_{\mf f_\alpha}^\circ (k_F)} \sigma \big) 
\]
supported on 
\[
(P_{\mf f_\alpha} \setminus P_\ff) / G_{\mf f_\alpha,0+} = 
P_\ff s_\alpha P_\ff / G_{\mf f_\alpha,0+} = \mc Q_{\ff,\alpha}(k_F) s_\alpha
\mc Q_{\ff,\alpha}(k_F) = \mc G_{\mf f_\alpha}^\circ (k_F) \setminus \mc Q_{\ff,\alpha}(k_F).
\]
Via Lemma \ref{lem:4.5}, $T_{v(\alpha,J)}$ corresponds to an element $\mr{pr}_1 R_{\mc L_\alpha \subset \mc Q_{\ff,\alpha} \mc L_\alpha}^{\mc G_{\mf f_\alpha}^\circ} 
(N_{s_\alpha})$, where $N_{s_\alpha} \in \End_{\mc L_\alpha (k_F)}
\big( \ind_{(\mc L_\alpha \cap \mc B)(k_F)}^{\mc L_\alpha (k_F)} \theta_\ff \big)$.~The support condition on $T_{v(\alpha,J)}$ translates via \eqref{eq:4.2} to
\[
\mr{supp}\, N_{s_\alpha} \subset \mc L_\alpha (k_F) \setminus (\mc L_\alpha \cap \mc B)(k_F) =
\mc U_\alpha (k_F) \mc T_\ff (k_F) s_\alpha \mc U_\alpha (k_F) .
\]
The standard basis element $T_{s_\alpha}$ of $\mc H (\mc L_\alpha (k_F), (\mc L_\alpha \cap \mc B)
(k_F), \theta_\ff)$ also has support $\mc U_\alpha (k_F) \mc T_\ff (k_F) s_\alpha \mc U_\alpha (k_F)$,
and satisfies a quadratic relation
\begin{equation}\label{eq:4.4}
(T_{s_\alpha} + 1)(T_{s_\alpha} - q_{\theta,\alpha}) = 0 
\text{ with } q_{\theta,\alpha} \in \R_{\geq 1}.
\end{equation}
The elements of $\mc H (\mc L_\alpha (k_F), (\mc L_\alpha \cap \mc B)(k_F), \theta_\ff)$ with 
support $\mc U_\alpha (k_F) \mc T_\ff (k_F) s_\alpha \mc U_\alpha (k_F)$ form a one-dimensional 
space, so $N_{s_\alpha} = \lambda T_{s_\alpha}$ for some $\lambda \in \C^\times$. Comparing 
\eqref{eq:4.3} and \eqref{eq:4.4}, we deduce that $\lambda = 1$ and $q_{\sigma,\alpha} = 
q_{\theta,\alpha} \geq 1$ or $\lambda = -1$ and $q_{\sigma,\alpha} = q_{\theta,\alpha} = 1$.
\end{proof}

In some cases, the parameters $q_{\sigma,\alpha} = q_{\theta,\alpha}$ automatically reduce to 1.
The next result must have been well-known to experts for a long time, but we could not find 
a reference, so we record it here for later use. 

\begin{prop}\label{prop:4.1}
In the setting of Lemma \ref{lem:4.6}, let $k_\alpha / k_F$ be a finite field extension
over which $\alpha$ splits, and let $N_{k_\alpha/k_F} : \mc T_\ff (k_\alpha) \to \mc T_\ff (k_F)$
be the norm map. Suppose that $s_\alpha (\theta_\ff) = \theta_\ff$ but $\theta_\ff \circ 
N_{k_\alpha/k_F} \circ \alpha^\vee \neq 1$. Then $q_{\theta,\alpha} = 1$.
\end{prop}
\begin{proof}
This is a statement about the reductive $k_F$-group $\mc L_\alpha$. To compute $q_{\theta,\alpha}$,
for instance as in \cite{HoLe}, we only need the derived group $\mc L_{\alpha,\der}(k_F)$.
By the classification of quasi-split rank one semisimple groups, $\mc L_{\alpha,\der}$ is obtained
by restriction of scalars from one of the groups: $\SL_2, \mr{PGL}_2, \SU_3, \mr{PU}_3$. Hence it 
suffices to consider these four groups, over an arbitrary finite field that we still call $k_F$. 

First we look at $\SU_3 (k_\alpha / k_F)$, for a quadratic extension $k_\alpha / k_F$ with
non-trivial field automorphism denoted by $x \mapsto \bar x$. In this case,
\[
\mc T_\ff (k_F) = \{ x \in (k_\alpha^\times)^3 : x_3 = \bar{x_1}^{-1}, x_2 = \bar{x_1} x_1^{-1} \} ,
\]
and $s_\alpha$ exchanges $x_1$ with $x_3$. Projection to the first coordinate gives an isomorphism 
$\mc T_\ff (k_F) \xrightarrow{\sim} k_\alpha^\times$.
Let $\chi \in \Irr (k_\alpha^\times)$ be the character corresponding to $\theta_\ff$ via this
isomorphism. The condition $s_\alpha \theta_\ff = \theta_\ff$ translates to $1 = \chi (x_1) \chi (\bar{x_1})^{_1} = \chi (x_1 \bar{x_1})$. On the other hand, 
\[
N_{k_\alpha / k_F} \circ \alpha^\vee (x_1) = N_{k_\alpha / k_F} ( (x_1,1,x_1^{-1}) ) =
(x_1 \bar{x}_1 , 1, (x_1 \bar{x}_1)^{-1}).
\]
Thus $\theta_\ff \circ N_{k_\alpha/k_F} \circ \alpha^\vee = 1$, and the assumptions of the lemma
are not fulfilled.

For $\mr{PU}_3 (k_\alpha/k_F)$, the maximal torus $\mc T_\ff (k_F)$ is isomorphic to 
$k_\alpha^\times$ via $(x_1,x_2,x_3) \mapsto x_1 x_2^{-1}$. The same arguments as for 
$\mr{SU}_3 (k_\alpha / k_F)$ apply.

Next we consider the split group $\mr{PGL}_2 (k_F)$. We may identify $\mc T_\ff (k_F)$ with
$\{ \matje{x}{0}{0}{1} : x \in k_F^\times \}$. The calculation
\[
1 = \theta_\ff (\matje{x}{0}{0}{1} ) \theta_\ff ( s_\alpha \matje{x}{0}{0}{1}^{-1} ) =
\theta_\ff ( \matje{x}{0}{0}{x^{-1}} ) = \theta_\ff (\alpha^\vee (x)) 
\]
shows that again the assumption of the lemma cannot be fulfilled.

Finally we study $\SL_2 (k_F)$, with its maximal torus 
\[
\mc T_\ff (k_F) = \{ \matje{x}{0}{0}{x^{-1}} : x \in k_F^\times \} \cong k_F^\times.
\]
Writing $\theta_\ff ( \matje{x}{0}{0}{x^{-1}} ) = \chi (x)$, we have
\[
1 = \theta_\ff (  \matje{x}{0}{0}{x^{-1}} ) \theta_\ff ( s_\alpha \matje{x}{0}{0}{x^{-1}} )^{-1} =
\theta_\ff (  \matje{x^2}{0}{0}{x^{-2}} ) = \chi (x^2) .
\]
Since $1 \neq \theta_\ff \circ \alpha^\vee = \chi$, $\chi$ must be the Legendre symbol of
$k_F^\times$, its unique order two character.~Now we really have to compute in 
$\C [\SL_2 (k_F)]$.~To do so, we introduce the following idempotents:  
\[
\begin{array}{lll@{\qquad}lll}
\langle U_\alpha \rangle & = & |k_F|^{-1} \sum\nolimits_{x \in k_F} \, \matje{1}{x}{0}{1}, &
p_\chi & = & 
|k_F^\times|^{-1} \sum\nolimits_{x \in k_F^\times} \, \chi (x) \matje{x}{0}{0}{x^{-1}} \\
\langle U_{-\alpha} \rangle & = & |k_F|^{-1} \sum\nolimits_{x \in k_F} \, \matje{1}{0}{x}{1}, 
& T_e & = & p_\chi \langle U_\alpha \rangle. 
\end{array}
\]
Note that $p_\chi$ commutes with $\langle U_\alpha \rangle ,\langle U_{-\alpha} \rangle$ and
$s_\alpha$, because $\chi^2 = 1$. The operator $T_{s_\alpha}$ is a scalar multiple of
$T'_{s_\alpha} = \langle U_\alpha \rangle s_\alpha p_\chi \langle U_\alpha \rangle$. We compute:
\begin{align*}
T_{s_\alpha}^{'2} & = \langle U_\alpha \rangle s_\alpha p_\chi \langle U_\alpha \rangle 
\langle U_\alpha \rangle p_\chi s_\alpha \langle U_\alpha \rangle 
= \langle U_\alpha \rangle s_\alpha p_\chi \langle U_\alpha \rangle s_\alpha \langle U_\alpha \rangle 
= \langle U_\alpha \rangle p_\chi \langle U_{-\alpha} \rangle \langle U_\alpha \rangle \\
& = \langle U_\alpha \rangle p_\chi |k_F|^{-1} \langle U_\alpha \rangle + 
\sum\nolimits_{x \in k_F^\times} \, \langle U_\alpha \rangle p_\chi |k_F|^{-1} \matje{1}{0}{x}{1} 
\langle U_\alpha \rangle \\
& = |k_F|^{-1} \langle U_\alpha \rangle p_\chi + |k_F|^{-1} \sum\nolimits_{x \in k_F^\times} \,
p_\chi \langle U_\alpha \rangle \matje{1}{-x^{-1}}{0}{1} \matje{1}{0}{x}{1} 
\matje{1}{-x^{-1}}{0}{1} \langle U_\alpha \rangle \\
& = |k_F|^{-1} T_e + |k_F|^{-1} \sum\nolimits_{x \in k_F^\times} \, \langle U_\alpha 
\rangle p_\chi \matje{x^{-1}}{0}{0}{x} s_\alpha \langle U_\alpha \rangle \;=\; |k_F|^{-1} T_e .
\end{align*}
Hence $T_{s\alpha}^2 \in \C T_e$, and the quadratic equation \eqref{eq:4.4} simplifies to
\[
0 = (T_{s_\alpha} + 1)(T_{s_\alpha} - 1) .
\]
That means precisely that $q_{\theta,\alpha} = 1$.
\end{proof}
We denote the linear part of an affine root $\alpha$ by $D \alpha$, and write 
\begin{equation}
\Delta_{\ff,\sigma} := \{ D \alpha : \alpha \in \Delta_{\ff,\af},\,  v(\alpha,J) \in
W(J,\sigma) \} .
\end{equation}
Note that $D (\Delta_{\ff,\af,\sigma}) \subset \Delta_{\ff,\sigma} \subset D (\Delta_{\ff,\af})$.~By \eqref{eq:4.5}, $\Delta_{\ff,\sigma}$ can also be viewed as a set of $\mf o_F$-rational roots 
of $(\mc G,\mc T_\ff)$. We define the $\mf o_F$-torus
\begin{equation}
\mc T_\sigma = \big( \bigcap\nolimits_{\alpha \in \Delta_{\ff,\sigma}} \, \ker 
\alpha |_{\mc T_\ff} \big)^\circ  \quad \subset \mc T_\ff .
\end{equation}
In many cases (but not always), $\mc T_\sigma$ is $F$-anisotropic. Now\label{i:114}
\begin{equation}\label{eqn:defn-G_sigma}
\mc G_\sigma := Z_{\mc G}(\mc T_\sigma)
\end{equation}
is a reductive $F$-subgroup of $\mc G$. More precisely, it is a twisted Levi subgroup that becomes 
an actual Levi subgroup over every field extension of $F$ that splits $\mc T_\sigma$.

\begin{lem}\label{lem:4.9}
The group $\mc G_\sigma$ is $F$-quasisplit and contains the torus 
$\mc T  := Z_{\mc G}(\mc T_\ff)$ as a minimal $F$-Levi subgroup.
\end{lem}
\begin{proof}
By the maximality of $\mc T_\ff$ in $\mc G_\ff^\circ$, $\mc T_\ff$ is a maximal unramified torus 
in $\mc G$.~Since $\mc G$ becomes quasi-split over a maximal unramified extension of $F$, $\mc T$ 
is a maximal $F$-torus in $\mc G$. The maximal $F$-split subtorus $\mc T_s$ of $\mc T_\ff$ is generated by $Z(\mc G) \cap \mc T_s$ 
and the images of the coroots $\alpha^\vee$ with $\alpha \in \Delta_{\ff,\sigma}$ (which are 
defined over $F$), so is contained in $\mc S$. Since $\mc T_\ff$ is the maximal unramified torus 
in $\mc T$, $\mc T_s$ is also the maximal $F$-split subtorus of $\mc T$.
Furthermore, $\mc T_\ff$ is generated by $\mc T_s \cup \mc T_\sigma$, so 
$\mc T = Z_{\mc G} (\mc T_\sigma \mc T_s) = Z_{\mc G_\sigma}(\mc T_s)$. 
Hence the centralizer in $\mc G_\sigma$ of a maximal $F$-split torus containing $\mc T_s$ is
contained in the torus $\mc T$, and is itself a torus. At the same time, that centralizer is
a Levi subgroup of $\mc G_\sigma$, so it is a minimal Levi subgroup and a maximal torus. 
We conclude that the torus $\mc T$ is a minimal $F$-Levi subgroup of $\mc G_\sigma$.
\end{proof}

Since $\mc T_\ff$ is elliptic in $\mc G_\ff^\circ$, and $\mc L$ is an $F$-Levi subgroup
of $\mc G$ minimal for the property $\mc L_\ff^\circ = \mc G_\ff^\circ$, the torus $\mc T_\ff$
is elliptic in $\mc L$. Since $\mc T_\ff$ is the maximal unramified subtorus of $\mc T$, Lemma
\ref{lem:4.9} implies that $\mc T$ is an elliptic maximal torus in $\mc L$ (so we are
back in the setting from \S \ref{sec:DL}). In other words,
$\mc T / Z(\mc L)^\circ$ is $F$-anistropic. Then $\mc T_s$ equals the maximal $F$-split 
subtorus of $Z(\mc L)^\circ$. 

By Lemma \ref{lem:4.9}, there is a unique apartment of 
$\mc B (\mc G_\sigma,F$) associated to $\mc T$ and its maximal $F$-split subtorus 
$\mc T_s \subset \mc S$. We call that apartment $\mh A_T$. From the inclusion
\begin{equation}
\mh A_T = X_* (\mc T_s) \otimes_\Z \R \subset X_* (\mc S) \otimes_\Z \R =: \mh A_S
\end{equation}
and the $W$-invariant metric on $\mh A_S$, we obtain a projection
\begin{equation}\label{eq:4.6}
\mh A_S \to \mh A_T \subset \mc B (\mc G_\sigma,F) .
\end{equation}

\begin{lem}\label{lem:4.7}
The intersection $\mc G_\sigma \cap \mc L$ equals $\mc T = Z_{\mc L}(\mc T_\ff)$.
\end{lem}
\begin{proof}
Since $\sigma$ lies in the series in $\Irr (\mc G_\ff^\circ (k_F))$ parametrized by
$(\mc T_\ff (k_F),\theta_\ff)$, and $\mc G_\ff^\circ (k_F)$ equals $P_{L,\ff} / L_{\ff,0+}$, 
the $\mf o_F$-group $\mc T_\ff$ can be realized in $\mc L$. Then $\mc T_\ff$ is a maximal 
unramified torus of $\mc L$, and $\mc T_\ff Z(\mc L)^\circ$ is an $F$-torus in $\mc L$. Hence 
\begin{equation}\label{eq:4.9}
Z(\mc L)^\circ \subset Z_{\mc G} (\mc T_\ff) = \mc T \subset Z_{\mc G}(\mc T_\sigma) \cap
Z_{\mc G}(Z(\mc L)^\circ) = \mc G_\sigma \cap \mc L.
\end{equation}
Consequently, $\mc G_\sigma \cap \mc L = Z_{\mc G_\sigma}(Z(\mc L)^\circ)$ is an $F$-Levi 
subgroup of $\mc G_\sigma$. By definition, $R(L,S)$ consists of the roots in 
$R(G,S)$ that are constant on $\ff$. Hence $R(\mc L,\mc T_\ff)$ consists of 
roots that are constant on the image of $\ff$ in $\mh A_T$ via \eqref{eq:4.6}. 

For any $\alpha = D \alpha' \in \Delta_{\ff, \sigma}$, the reflection $s_{\alpha'}$ of
$\mh A_S$ stabilizes $\Q J$ and the span of $\ff$. Hence it also stabilizes the orthogonal
complement $\ff^\perp$ of the span of $\ff$ in $\mh A_S$. As $\alpha'$ is not constant
on $\ff$, this is only possible if $\alpha |_{\ff^\perp} = 0$. 
Thus $R(\mc L,\mc T_\ff)$ and $\Delta_{\ff,\sigma}$ are orthogonal: the first is constant 
on the span of $\ff$ while the second has $\ff^\perp$ in its joint kernel. Consequently,
\[
\mc T_\sigma Z(\mc L)^\circ = \big( \bigcap\nolimits_{\alpha \in \Delta_{\ff,\sigma}} \,
\ker \alpha |_{\mc T_\ff} \big)^\circ \: \big( 
\bigcap\nolimits_{\alpha \in R(\mc L,\mc T_\ff)} \, \ker \alpha |_{\mc T_\ff} \big)^\circ
\] 
equals $\mc T_\ff Z(\mc L)^\circ$. From this and \eqref{eq:4.9} we deduce that
\[
\mc G_\sigma \cap \mc L = Z_{\mc G} (\mc T_\sigma) \cap Z_{\mc G}(Z(\mc L)^\circ) =
Z_{\mc G}(\mc T_\sigma Z(\mc L)^\circ) 
\]
equals $Z_{\mc G}(\mc T_\ff Z(\mc L)^\circ) = Z_{\mc G}(\mc T_\ff) \cap \mc L = \mc T$.
\end{proof}
By the definition of $\mc T_\sigma$, we have $R(\mc G_\sigma,\mc T_\ff) = \Q \Delta_{\mf f,\sigma} \cap X^* (\mc T_\ff)$. 
Since $\mc G_\sigma$ is quasi-split and $Z_{\mc G_\sigma}(\mc T_\ff) = \mc T$, this gives $R(\mc G_\sigma,\mc T) = \{ \alpha \in \mc R (\mc G,\mc T) : \alpha |_{\mc T_\ff} \in 
\Q \Delta_{\mf f,\sigma} \}$. Let $P_{G_\sigma,\ff} = G_\sigma \cap P_\ff$ be the parahoric subgroup of $G_\sigma$ 
associated to the image of $\ff$ in $\mh A_T$. Then, similar to \eqref{eq:2.13}, we have $P_{G_\sigma,\ff} / G_{\sigma,\ff,0+} \cong P_{T,\ff} / T_{\ff,0+} \cong \mc T_\ff (k_F)$. In particular, $\theta_\ff$ can be inflated to an irreducible representation of 
$P_{G_\sigma,\ff}$, and we can consider the Hecke algebra 
$\mc H (G_\sigma, P_{G_\sigma,\ff},\theta_\ff)$. The cuspidal support of the Bernstein 
component $\Irr (G_\sigma)_{(P_{G_\sigma,\ff},\theta_\ff)}$ is 
$\Irr (T)_{(P_{T,\ff},\theta_\ff)}$, so $\Rep (G_\sigma)_{(P_{G_\sigma,\ff}, \theta_\ff)}$ 
is a Bernstein block in the principal series of the quasi-split group $G_\sigma$.

\begin{prop}\label{prop:4.8}
\enuma{
\item There exists a canonical bijection between $W(J,\sigma) \cap S_{\ff,\af}$ and 
$W(\emptyset, \theta_\ff) \cap (S_{\ff,\af}$ for $G_\sigma)$, which preserves the 
$q$-parameters in $\Z_{\geq 1}$. 
\item Part (a) induces an isomorphism between the affine root systems of $\mc H (G,P_\ff,\sigma)$ 
and $\mc H (G_\sigma,P_{G_\sigma,\ff}, \theta_\ff)$, such that the parameter functions on
both sides agree. 
\item Part (b) induces an algebra isomorphism $\mc H (W_\af (J,\sigma),q_\sigma) 
\cong \mc H (W_\af (\emptyset, \theta_\ff), q_\theta)$.
}
\end{prop}
\begin{proof}
(a) It is clear from the definitions of $\mc T_\sigma$ and $\mc G_\sigma$ that 
$\Delta_{\ff,\af}$ for $G_\sigma$ is contained in $\Delta_{\ff,\af}$ for $G$. By construction, 
$R(L,S) = R(G,S) \cap \Q D(J)$, and by Lemma \ref{lem:4.7},
\[
R(L,S) \cap R(G_\sigma,S) = R(T,S) = \emptyset .
\]
Hence $\Q D(J) \cap R(G_\sigma,S)=\varnothing$, and the elements of $J$ are constant 
functions on $\mh A_T \cap \mc B (\mc G_{\sigma,\der},F)$. In particular, $\Delta_\af$ 
for $G_\sigma$ is contained in $\Delta_\af \setminus J$ for $G$. 

For $\alpha \in \Delta_{\ff,\af}$ such that $w_{J \cup \alpha} (J) \neq -J$, $s_{\alpha}$ 
does not stabilize $\Q J$, and hence does not stabilize the span of $\ff$ in $\mh A_S$. 
It follows that $s_{\alpha}$ cannot stabilize the image of $\ff$ in $\mh A_T$, and hence 
cannot define an element of $S_{\ff,\af}$ for $G_\sigma$.

For $\alpha \in \Delta_{\ff,\af}$ such that $W_{J \cup \alpha} (J) = -J$, 
the proof of Lemma \ref{lem:4.4} shows that
\begin{align*}
s_\alpha \cdot \theta_\ff = \theta_\ff & \Longleftrightarrow s_\alpha \cdot 
R_{\mc T_\ff (k_F)}^{\mc G_{\ff}^\circ (k_F)} (\theta_\ff) \cong 
R_{\mc T_\ff (k_F)}^{\mc G_\ff^\circ (k_F)} (\theta_\ff)\\
& \Longleftrightarrow v(J,\alpha) \cdot R_{\mc T_\ff (k_F)}^{\mc G_\ff^\circ (k_F)} 
(\theta_\ff) \cong R_{\mc T_\ff (k_F)}^{\mc G_\ff^\circ (k_F)} (\theta_\ff) 
\Longleftrightarrow v(J,\alpha) \cdot \sigma \cong \sigma .
\end{align*}
This provides the required bijection.

(b) The group $W_\af (J) = \langle S_{\ff,\af} \rangle$ is realized in 
\cite[Theorem 2.7]{Mor1} as an affine Weyl group via its action on 
$D(J)^\perp \subset \mh A_S$. On $D(J)^\perp$, $v(\alpha,J)$ coincides with $s_\alpha$, 
so the bijection from part (a) extends to a group isomorphism
\begin{equation}\label{eq:4.11}
\langle W(J,\sigma) \cap S_{\ff,\af} \rangle \cong \langle W(\emptyset,\theta_\ff) \cap 
(S_{\ff,\af} \text{ for } G_\sigma) \rangle .
\end{equation}
The data $\mc T_\ff (k_F)$, $\theta_\ff$ and $\mc L_\alpha (k_F)$ used in Lemma \ref{lem:4.6} (b)
are the same for $\mc G$ and for $\mc G_\sigma$. Hence that lemma implies that 
$q_\sigma (v(\alpha,J))$ equals $q_\theta (s_\alpha)$, where the latter is computed from 
$(G_\sigma,P_{G_\sigma,\ff},\theta_\ff)$.

(c) This is a direct consequence of part (b).
\end{proof}

Proposition \ref{prop:4.8} says that $\mc H (W_\af (J,\sigma),q_\sigma)$ is naturally 
isomorphic to the Iwahori--Hecke algebra from a Bernstein block of principal series 
representations of a quasi-split reductive $p$-adic group. Hecke algebras and the local 
Langlands correspondence for such representations were analyzed in detail in \cite{SolQS}. 

For $\alpha \in \Delta_{\ff,\sigma}$ we write \label{i:115}
$\mc T_\alpha := (\ker \alpha |_{\mc T_\ff})^\circ 
\subset \mc T_\ff$, and $\mc G_\alpha := Z_{\mc G}(\mc T_\alpha)$. Then $\mc G_\alpha$ is a Levi subgroup of $\mc G_\sigma$ containing $\mc T$, so in
particular $G_\alpha = \mc G_\alpha (F)$ is quasi-split. Furthermore, since $\alpha$ 
is defined over $F$, we have
\begin{equation}\label{eq:4.8}
R(\mc G_\alpha,\mc T) = \{ \beta \in R(\mc G,\mc T) : \beta |_{\mc T_\ff} \in \R^\times \alpha \}
= \{ \beta \in R(\mc G,\mc T) : \beta |_{\mc T_s} \in \R^\times \alpha \}.
\end{equation}
The data $\mc T_\ff (k_F)$, $\theta_\ff$ and $\mc L_\alpha (k_F)$ figuring in Lemma \ref{lem:4.6}
can be constructed from $G_\alpha$ equally well as from $G_\sigma$ or $G$. Analogous to
Proposition \ref{prop:4.8}, this gives the following.
\begin{cor}\label{cor:4.2}
The parameter $q_\sigma (v(\alpha,J)) = q_\theta (s_\alpha)$ equals the $q$-parameter for
$s_\alpha$ in $\mc H (G_\alpha, P_{G_\alpha,\ff}, \theta_\ff)$.
\end{cor}

\section{The Hecke algebra of a non-singular depth-zero Bernstein block}
\label{sec:HeckeG}

We continue the conventions from Section \ref{sec:Morris}, in particular, $\sigma$ is a
non-singular cuspidal representation of $\mc G_\ff^\circ (k_F) = P_\ff / G_{\ff,0+}$. 
Let $\hat \sigma$\label{i:129} be an irreducible representation of $\mc G_\ff(k_F) = 
\hat P_\ff / G_{\ff,0+}$ whose restriction to $\mc G_\ff^\circ (k_F)$ contains $\sigma$. 
\label{i:130}

\begin{thm}[\cite{MoPr2,Mor3}]\label{thm:3.6}
The pair $(\hat P_\ff, \hat \sigma)$ is a type for a single Bernstein block 
$\Rep (G)_{(\hat P_\ff,\hat \sigma)} \subset \Rep (G)$. Moreover 
$(\hat P_\ff, \hat \sigma)$ is a cover of the type $(\hat P_{L,\ff},\hat \sigma)$.
\end{thm}

For general results about types and their $G$-covers, we refer the reader to \cite{BuKu}. 
For our use, here we record in particular an equivalence of categories\label{i:141}
\begin{equation}\label{eq:3.4}
\Rep (G)_{(\hat P_\ff,\hat \sigma)} \; \cong \; \Mod \text{ - } \cH (G,\hat P_\ff, \hat \sigma),     
\end{equation}
where $\Mod\text{ - }\mathcal{R}$ denotes the category of \textit{right} $\mathcal{R}$-modules.

Since $\hat P_\ff / P_\ff$ is abelian by Lemma \ref{lem:3.2} (a), every alternative
$\hat \sigma'$ is isomorphic to $\chi \otimes \hat \sigma$ for some (not necessarily unique)
character $\chi$ of $\Omega_\ff^\circ$.
In particular, the multiplicity $m(\hat \sigma,\sigma)$ of $\hat \sigma$ in 
$\ind_{P_\ff}^{\hat P_\ff} (\sigma)$, which equals the multiplicity of 
$\sigma$ in $\hat \sigma$, is independent of the choice of $\hat \sigma$. Therefore, we have 
\begin{equation}\label{eq:3.2}
\ind_{P_\ff}^{\hat P_\ff} (\sigma) \cong \C^{m(\hat \sigma,\sigma)} \otimes 
\bigoplus\limits_{\hat \sigma \text{ contains } \sigma} \hat \sigma 
\quad \text{and} \vspace{-3mm}
\end{equation}
\begin{equation}
\label{eq:5.18} \mc H (\hat P_\ff, P_\ff, \sigma) = \End_{\hat P_\ff} 
\big( \ind_{P_\ff}^{\hat P_\ff} (\sigma) \big) \cong M_{m(\hat \sigma,\sigma)}(\C) \otimes 
\bigoplus\limits_{\hat \sigma \text{ contains } \sigma} \C \, \mr{id}_{V_{\hat \sigma}} .
\end{equation}
Recall moreover from \cite{BuKu} that there are natural algebra isomorphisms
\begin{align}\label{eq:3.3}
\begin{split}
& \mc H (G,P_\ff,\sigma) \cong \End_G \big( \ind_{P_\ff}^G (\sigma) \big) =
\End_G \big( \ind_{\hat P_\ff}^G \ind_{P_\ff}^{\hat P_\ff} (\sigma) \big) \quad \text{and}\\
& \mc H (G,\hat P_\ff,\hat \sigma) \cong 
\End_G \big( \ind_{\hat P_\ff}^G (\hat \sigma) \big) .
\end{split}
\end{align}
We choose a decomposition of $\hat P_\ff$-representations
\begin{equation}\label{eq:5.19}
\ind_{P_\ff}^{\hat P_\ff} (\sigma) = \hat \sigma \oplus \hat \pi.
\end{equation}
By \eqref{eq:3.2} and \eqref{eq:3.3}, we see that \eqref{eq:5.19} induces an algebra embedding 
\begin{equation}\label{eq:5.20}
\mc H (G,\hat P_\ff,\hat \sigma) \hookrightarrow \mc H (G,P_\ff,\sigma),
\end{equation}
as already noted in \cite[p. 150]{Mor3}.
The unit element $T_e$ of $\mc H (G,P_\ff,\sigma)$ can be identified with 
$\sigma : P_\ff \to \End_\C (V_\sigma)$. Let $\hat T_e \in \mc H (\hat P_\ff, P_\ff, \sigma)$
be the minimal idempotent whose kernel is $\hat \pi$ and whose image equals $\hat \sigma$ as in 
\eqref{eq:5.19}. Then $\hat T_e \in \mc H (G,P_\ff,\sigma)$ is the image of the unit element of 
$\mc H (G,\hat P_\ff,\hat \sigma)$ via \eqref{eq:5.20}.

Let $G_{\ff,\hat \sigma}$ be the stabilizer of $\hat \sigma \in \Irr (\hat P_\ff)$ in $G_\ff$,
and let $\Omega (J,\hat \sigma)$\label{i:119} denote its image in
\begin{equation}
\Omega (J,\sigma) \hat P_\ff / \hat P_\ff \cong \Omega (J,\sigma) \Omega_\ff^0 / \Omega_\ff^0
\cong \Omega (J,\sigma) / \Omega_{\ff,\sigma}^0 .
\end{equation}
Note that by Lemma \ref{lem:3.2} (b), $\Omega (J, \hat \sigma)$ is isomorphic to a lattice  
in $X_* (Z(G)) \otimes_\Z \R$.

\begin{thm}\label{thm:3.1}
There are algebra isomorphisms
\[
\cH (G,\hat P_\ff,\hat \sigma) \cong \hat T_e \mc H (G,P_\ff,\sigma) \hat T_e \cong 
\mc H (W_\af (J,\sigma),q_\sigma) \rtimes \C[ \Omega (J,\hat \sigma) ,\mu_{\hat \sigma}] .
\]
The support of $\cH (G,\hat P_\ff,\hat \sigma)$ is $\hat P_\ff (W_\af (J,\sigma) \rtimes 
\Omega (J,\hat \sigma)) \hat P_\ff$, and this algebra has a basis indexed by
$W_\af (J,\sigma) \rtimes \Omega (J,\hat \sigma)$.
\end{thm}
\begin{proof}
The first isomorphism follows from \eqref{eq:5.19} and the construction of $\hat T_e$.~The 
support of $\hat T_e \in \mc H (G,P_\ff, \sigma)$ is $\hat P_{\ff,\sigma} / P_\ff =
\Omega_{\ff,\sigma}^0$, which by Lemma \ref{lem:3.5} is central in $W(J,\sigma)$.~By 
the multiplication relations in Theorem \ref{thm:2.4}, $\hat T_e$ commutes with each 
$T_{s_\alpha}$ and hence with $\mc H (W_\af (J,\sigma),q_\sigma)$. By Theorem \ref{thm:2.4},
we deduce
\begin{equation}
\hat T_e \mc H (G,P_\ff,\sigma) \hat T_e \cong \mc H (W_\af (J,\sigma),q_\sigma) \rtimes 
\hat T_e \C[ \Omega (J,\sigma) ,\mu_\sigma] \hat T_e . 
\end{equation}
Here $\hat T_e \C[ \Omega (J,\sigma) ,\mu_\sigma] \hat T_e$ is isomorphic to the subalgebra of
$\mc H (G,\hat P_\ff, \hat \sigma)$ supported on $\hat P_\ff \Omega (J,\hat \sigma) \hat P_\ff = \Omega (J,\hat \sigma) \hat P_\ff$, so it equals $\mc H (\Omega (J,\hat \sigma) \hat P_\ff, \hat P_\ff,\hat \sigma)$ 
and has a basis indexed by 
$\Omega (J,\hat \sigma) \hat P_\ff / \hat P_\ff = \Omega (J,\hat \sigma)$.~This shows the claims about the support and a basis of 
$\cH (G,\hat P_\ff,\hat \sigma)$.
For any $g \in \Omega (J,\hat \sigma) \hat P_\ff \subset G_\ff$, an element  
$\hat T_g \in \mc H (\Omega (J,\sigma) \hat P_\ff, \hat P_\ff,\hat \sigma)$ with 
support $\hat P_\ff g \hat P_\ff$ is determined by a nonzero element  
\begin{equation}\label{eq:5.38}
\hat T_g (g) \in \Hom_{\hat P_{L,\ff}}(\hat \sigma, g \cdot \hat \sigma)
\end{equation}
unique up to scalars, and $\hat T_g (g)^{-1} \in \Hom_{\hat P_{L,\ff}}(\hat \sigma, g^{-1} \cdot \hat \sigma)$ determines an element $\hat T_{g^{-1}}$ which is the inverse of $\hat T_g$. Furthermore,
$\hat T_g \hat T_h \in \C^\times \hat T_{g h}$ by uniqueness up to scalars. 
We do this for $g$ running through a set of representatives $\dot g$ for 
$\Omega (J,\hat \sigma)$. The formula
\begin{equation}\label{eq:5.39}
\hat T_{\dot g} \hat T_{\dot h} = \mu_{\hat \sigma} (\dot g,\dot h) \hat T_{\dot{gh}}
\end{equation}
defines a 2-cocycle $\mu_{\hat \sigma}$ on $\Omega (J,\hat \sigma)$. The cocycle relation 
follows from the associativity of 
$\mc H (\Omega (J,\sigma) \hat P_\ff, \hat P_\ff,\hat \sigma)$.
\end{proof}
Theorem \ref{thm:3.1} also tells us that the stabilizer of $\hat \sigma \in
\Irr (\hat P_{L,\ff})$ in $W(J,\sigma) \Omega_\ff^0 / \Omega_\ff^0$ is\label{i:120}
\begin{equation}\label{eq:5.21}
W(J,\hat \sigma) := W_\af (J,\sigma) \rtimes \Omega (J,\hat \sigma) .
\end{equation}

\subsection{The supercuspidal case} \

To better understand the isomorphisms in Theorem \ref{thm:3.1}, we first assume that
the associated Bernstein block of $G$ consists only of supercuspidal representations.
~This happens if and only if $\ff$ is a minimal facet of $\mc B (\mc G,F)$.~In this case 
$\mc L = \mc G$ and $W_\af (J,\sigma)$ is trivial.
~We will treat the general case in \S \ref{subsec:non-sc-case}.

\begin{cor}\label{cor:3.3}
Suppose that $\ff$ is a minimal facet of $\mc B (\mc G,F)$. Then
\[
\mc H (G,\hat P_\ff, \hat \sigma) \cong \C \big[ W(J,\hat \sigma) , \mu_{\hat \sigma} \big] .
\] 
\end{cor}
\begin{proof}
The minimality of $\ff$ implies that $\Delta_\af \setminus J$ contains at most one element
from each $F$-simple factor of $G$. Hence $\Delta_{\ff,\af}$ is empty and $|W_\af (J,\sigma)|
= 1$. By \eqref{eq:5.21}, $\Omega (J,\hat \sigma)$ equals $W(J,\hat \sigma)$. 
Now the isomorphism is clear by Theorem \ref{thm:3.1}.
\end{proof}
Note that in the special case where $\sigma$ is moreover regular, a more precise version 
of Corollary \ref{cor:3.3} was already known in \cite[Corollary 5.5]{Oha1}.

We now make the supercuspidal $G$-representations arising from $(\hat P_\ff,\hat \sigma)$ 
more explicit.~Let $\sigma'$ be an irreducible representation of $G_\ff$ whose restriction 
to $\hat P_\ff$ contains $\hat \sigma$. By \cite[Propositions 6.6 and 6.8]{MoPr2}, we know that
\label{i:121}
\begin{equation}\label{eq:5.14}
\tau := \ind_{G_\ff}^G (\sigma')
\end{equation}
is an irreducible supercuspidal $G$-representation, and that every object of 
$\Irr (G)_{(\hat P_\ff,\hat \sigma)}$ is of this form for some extension $\sigma'$ as above.

Let $G^1$\label{i:122} be the group generated by all compact subgroups of $G$, 
or equivalently the intersection of the kernels of all the unramified characters of $G$.

\begin{lem}\label{lem:3.4}
Suppose that $\ff$ is a minimal facet of $\mc B (\mc G,F)$.
Let $\tau_1$ be an irreducible subrepresentation of $\Res^G_{G^1} (\tau)$. Then
$\ind_{G^1}^G (\tau_1) \cong \ind_{\hat P_\ff}^G (\hat \sigma)$.
\end{lem}
\begin{proof}
By assumption, the image of $\ff$ in $\mc B (\mc G_\ad,F)$ is just one point, and by
construction $G^1$ acts trivially on $X_* (Z(G)) \otimes_\Z \R$.~Thus 
$G_\ff \cap G^1 = \hat P_\ff$. We claim that\label{i:123}
\begin{equation}\label{eq:5.30}
\tau_1 := \ind_{\hat P_\ff}^{G^1} (\hat \sigma) 
\end{equation}
is irreducible.~The intertwining set of $\hat \sigma \in \Irr (\hat P_\ff)$ is defined as
\[
\{ g \in G : \Hom_{\hat P_\ff \cap g \hat P_\ff g^{-1}} 
(\hat \sigma, g \cdot \hat \sigma) \neq 0 \} .
\]
It equals the support of $\mc H (G,\hat P_\ff,\hat \sigma)$, which by Theorem 
\ref{thm:2.4} is $G_{\ff,\hat \sigma}$.~Since $\hat P_\ff$ equals $G_{\ff, \hat \sigma} \cap G^1$,
the intertwining of $\hat \sigma \in \Irr ( \hat P_\ff)$ in $G^1$ equals $\hat P_\ff$. 
This implies the claimed irreducibility of $\tau_1$. By the transitivity of induction, we have
\begin{equation}\label{eq:3.6}
\ind_{G^1}^G (\tau_1) = \ind_{G^1}^G \ind_{\hat P_\ff}^{G^1} (\hat \sigma) = 
\ind_{\hat P_\ff}^G (\hat \sigma) .
\end{equation}
By Frobenius reciprocity, $\Hom_{G^1} (\tau_1, \tau) \cong \Hom_{\hat P_\ff}(\hat \sigma, \ind_{G_\ff}^G (\sigma') )
\subset \Hom_{\hat P_\ff}(\hat \sigma, \sigma' ) \neq 0$.~Hence this $\tau_1$ is indeed a subrepresentation of $\Res^G_{G^1} (\tau)$. By the 
irreducibility of $\tau$, every alternative $\tau_2$ for $\tau_1$ is isomorphic to 
$g \cdot \tau_1$ for some $g \in G$. In particular the choice of $\tau_1$ does not affect
$\ind_{G^1}^G (\tau_1)$.
\end{proof}
Since the supercuspidal Bernstein component $\Irr (G)_{(\hat P_\ff,\hat \sigma)} \cong \Irr \text{ - } \mc H (G,\hat P_\ff, \hat \sigma)$ has the structure of a complex torus, it is isomorphic to the space of irreducible
representations of a lattice.~This suggests that it is possible to get rid of the 2-cocycle 
$\mu_{\hat \sigma}$ in Corollary \ref{cor:3.3}. It turns out that is indeed the case, at 
the cost of passing to a smaller lattice. We write\label{i:124}
\begin{equation}\label{eq:5.22}
ZW (J,\hat \sigma) := \{ w \in W(J,\hat \sigma) : T_v T_w = T_w T_v 
\text{ for all } v \in W( J,\hat \sigma) \} .
\end{equation}
As a subgroup of a lattice, $ZW (J,\hat \sigma)$ is again a lattice.~We pick a basis 
$\mf B$, and for $z = \sum\limits_{b \in \mf B} n_b b \in ZW (J,\hat \sigma)$, we rescale  $T_z$ to 
$\prod\limits_{b \in \mf B} T_b^{n_b}$.~By \eqref{eq:5.22}, this is well-defined. Next we 
choose a set of representatives 
$\dot w$ for $W(J,\hat \sigma) / ZW (J,\hat \sigma)$, and we rescale $T_{\dot w z} = T_{\dot w} T_z = T_z T_{\dot w}$. This allows us to make $\mu_{\hat \sigma}$ 
factor through $\big( W(J,\hat \sigma) / ZW (J,\hat \sigma) \big)^2$. Together with the 
commutativity of $W(J,\hat \sigma) = \Omega (J,\hat \sigma)$, we obtain
\begin{equation}\label{eq:3.10}
Z \big( \C [W(J,\hat \sigma) , \mu_{\hat \sigma} ] \big) = \C [ZW (J,\hat \sigma)] .
\end{equation}
By \cite[\S 7.13]{Mor1}, it is easy to see that $ZW (J,\hat \sigma)$ contains the image in
$W(J,\hat \sigma)$ of the maximal central $F$-split torus of $G$.~This image has finite
index in the lattice $W(J,\hat \sigma)$, because it has the same rank by Lemma 
\ref{lem:3.2}(b). Thus $[W(J,\hat \sigma) : ZW (J,\hat \sigma)]$ is finite.

Let $\mathfrak{X}_{\nr} (G,\tau)$\label{i:125} be the stabilizer of 
$\tau \in \Irr (G)_{(\hat P_\ff,\hat \sigma)}$ 
in $\mathfrak{X}_{\nr} (G)$. There is a bijection 
\begin{equation}\label{eq:5.11}
\mathfrak{X}_{\nr} (G) / \mathfrak{X}_{\nr} (G,\tau) \isom \Irr (G)_{[G,\tau]} : 
\chi \mapsto \chi \otimes \tau
\end{equation}
Recall from \eqref{eq:5.30} that $\big( \ind_{\hat P_\ff}^{G^1}(\hat \sigma), 
\ind_{\hat P_\ff}^{G^1}(V_{\hat \sigma}) \big)$ is an irreducible $G^1$-subrepresentation of 
$\tau$. As in \cite[\S 2]{SolEnd}, we define\label{i:126}
\begin{align*}
& G_\tau^2 := \bigcap\nolimits_{\chi \in \mathfrak{X}_{\nr} (G,\tau)} \, \ker \chi ,\\
& G_\tau^3 := \{ g \in G : g \cdot \ind_{\hat P_\ff}^{G^1}(V_{\hat \sigma}) =
\ind_{\hat P_\ff}^{G^1}(V_{\hat \sigma}) \} ,\\
&  G_\tau^4 := \{ g \in G : g \cdot \ind_{\hat P_\ff}^{G^1}(\hat \sigma) \cong
\ind_{\hat P_\ff}^{G^1}(\hat \sigma) \} .
\end{align*}
With these notations, \cite[Lemma 10.1.b]{SolEnd} says that 
\begin{equation}\label{eq:5.4}
W(J,\hat \sigma) = G_\tau^4 / G^1 ,\qquad ZW(J,\hat \sigma) = G_\tau^2 / G^1 
\end{equation}
and $\C [G_\tau^3 / G^1]$ is a maximal commutative subalgebra of 
$\C [W(J,\hat \sigma),\mu_{\hat \sigma}]$. Let $\mc O (\mathfrak{X}_{\nr} (G)$\label{i:164} 
be the ring of regular functions on the complex algebraic variety $\mf X_\nr (G)$.
By \eqref{eq:5.11}, we obtain algebra isomorphisms
\begin{equation}\label{eq:5.40}
\C [ZW(J,\hat \sigma) \cong \C [G^2_\tau / G^1] \cong 
\mc O (\mathfrak{X}_{\nr} (G) / \mathfrak{X}_{\nr} (G,\tau)) \cong \mc O (\Irr (G)_{[G,\tau]}) ,
\end{equation}
determined entirely by the choice of $\tau$. We write \label{i:161} $CW (J,\hat \sigma) := 
G_\tau^3 / G^1$, such that there are finite index inclusions of lattices
\begin{equation}\label{eq:5.32}
ZW (J,\hat \sigma) \subset CW (J,\hat \sigma) \subset W(J,\hat \sigma) .
\end{equation}
By \cite[Lemma 1.6.3.1]{Roc}, we know that
\[
[W(J,\hat \sigma) : CW (J,\hat \sigma)] = [CW (J,\hat \sigma) : ZW (J,\hat \sigma)]
\]
equals the multiplicity of $\tau_1$ in $\Res^G_{G^1}(\tau)$.~For an open subset 
$U \subset \Irr (ZW(J,\hat \sigma))$, let $C^{an}(U)$\label{i:127} 
be the algebra of analytic functions on $U$.~The analytic localization of 
$\C[W(J,\hat \sigma),\mu_{\hat \sigma}]$ at $U$ is defined as
\begin{equation}\label{eq:5.34}
\C[W(J,\hat \sigma),\mu_{\hat \sigma}]^{an}_U := \C[W(J,\hat \sigma),\mu_{\hat \sigma}] 
\otimes_{\C [ZW(J,\hat \sigma)]} C^{an}(U) .
\end{equation}
The finite-length modules of this algebra are precisely those finite-length modules
of $\C[W(J,\hat \sigma),\mu_{\hat \sigma}]$, all whose $\C[ZW(J,\hat \sigma)]$-weights
belong to $U$, cf.~\cite[Proposition 4.3]{Opd}. Similarly, we can define the analytic 
localization of $\C[W(J,\hat \sigma),\mu_{\hat \sigma}]$ with respect to an open subset
$\tilde U$ of $\Irr (CW (J,\hat \sigma))$.~We denote this module by
$\C[W(J,\hat \sigma),\mu_{\hat \sigma}]^{an}_{\tilde U}$. If $\tilde U$ is the full
preimage of a subset $U \subset \Irr (ZW(J,\hat \sigma))$, then it acquires an algebra
structure via the natural isomorphism
\begin{equation}\label{eq:5.8}
\C[W(J,\hat \sigma),\mu_{\hat \sigma}]^{an}_{\tilde U} \cong
\C[W(J,\hat \sigma),\mu_{\hat \sigma}]^{an}_U 
\end{equation}

\begin{prop}\label{prop:5.4}
Assume that $L = G$. 
\enuma{
\item Suppose that the inverse image of $U$ in $\Irr (CW (J,\hat \sigma))$ is homeomorphic to a 
disjoint union of $d = [CW (J,\hat \sigma) : ZW(J,\hat \sigma)]$ copies of $U$. Then the algebras 
$\C [W(J,\hat \sigma) , \mu_{\hat \sigma} ]^{an}_U$ and $C^{an}(U)$ are Morita equivalent.
\item The algebras $\C [W(J,\hat \sigma),\mu_{\hat \sigma}]$ and $\C [ZW(J,\hat \sigma)]$
have equivalent categories of finite-length modules. The equivalence sends any
irreducible $\C [W(J,\hat \sigma),\mu_{\hat \sigma}]$-module to its central character.
}
\end{prop}
\begin{proof}
(a) Write the inverse image of $U$ in $\Irr (CW (J,\hat \sigma))$ as 
$U_1 \sqcup \cdots \sqcup U_d$, where each $U_i$ projects homeomorphically onto $U$. Then we have
\[
\C [CW (J,\hat \sigma)] \underset{\C [ZW(J,\hat \sigma)]}{\otimes} C^{an}(U) \cong 
C^{an}(U_1) \oplus \cdots \oplus C^{an}(U_d) ,
\]
and this is a subalgebra of $\C [W(J,\hat \sigma) , \mu_{\hat \sigma} ]^{an}_U$.~Then we have
\begin{equation}\label{eq:5.23}
\C [W(J,\hat \sigma) , \mu_{\hat \sigma} ]^{an}_U = \bigoplus\nolimits_{i,j=1}^d \,
1_{U_i} \C [W(J,\hat \sigma) , \mu_{\hat \sigma} ]^{an}_U 1_{U_j}, 
\end{equation}
where
$1_{U_i}$ denotes the indicator function of $U_i$. 
The commutator map 
\[
(w,v) \mapsto T_w T_v T_w^{-1} T_v^{-1} \in \C^\times
\]
induces a non-degenerate skew-symmetric bicharacter on $W(J,\hat \sigma) / ZW (J,\hat \sigma)$.
It is trivial on $(CW (J,\hat \sigma) / ZW(J,\hat \sigma))^2$, so it induces an isomorphism 
\begin{equation}\label{eq:5.33}
W(J,\hat \sigma) / CW (J,\hat \sigma) \isom \Irr (CW (J,\hat \sigma) / ZW(J,\hat \sigma)) .
\end{equation}
For each $i,j$, there is a unique character $\chi_{ij} \in \Irr (CW (J,\hat \sigma) / 
ZW(J,\hat \sigma))$ such that $U_i = \chi_{ij} \otimes U_j$.
Hence there exists a $w_{ij} \in W(J,\hat \sigma)$, unique up to $CW (J,\hat \sigma)$, such that 
$T_{w_{ij}} 1_{U_j} T_{w_{ij}}^{-1} \in \C 1_{U_i}$.
Now \eqref{eq:5.23} simplifies to
\[
\C [W(J,\hat \sigma) , \mu_{\hat \sigma} ]^{an}_U = \bigoplus\nolimits_{i,j=1}^d \,
1_{U_i} T_{w_{ij}} C^{an}(U_j) 1_{U_j}.
\]
This algebra is isomorphic to $M_d (\C) \otimes_\C C^{an}(U_1)$, hence Morita equivalent
to $C^{an}(U_1)$, which is isomorphic to $C^{an}(U)$.\\
(b) By part (a) and \cite[Proposition 4.3]{Opd}, the category of those finite-length 
$\C[W(J,\hat \sigma),\mu_{\hat \sigma}]$-modules all whose $\C[ZW(J,\hat \sigma)]$-weights 
belong to $U$, is equivalent to the analogous category for $\C[ZW(J,\hat \sigma)]$. 

We cover $\Irr (ZW(J,\hat \sigma))$ by a collection of open sets $U$ that satisfy the condition
of part (a).~This is possible because every sufficiently small open ball in 
$\Irr (ZW(J,\hat \sigma))$ has the required property.~Combining the previous observations for
all such $U$, we find the desired statement for all finite-length modules.

The explicit description of the map on irreducible modules follows from the construction in
part (a), i.e.~that preserves $C^{an}(U)$-weights and hence preserves 
$\C[ZW(J,\hat \sigma)]$-weights.
\end{proof}

We indicate a full subcategory of finite-length objects by a subscript fl. 
By \eqref{eq:3.4}, Corollary \ref{cor:3.3} and Proposition \ref{prop:5.4}, 
the categories\label{i:131}
\begin{equation}\label{eq:5.24}
\Rep_\fl (G)_{(\hat P_\ff, \hat \sigma)} ,
\quad \Mod_\fl \text{ - }\mc H (G,\hat P_\ff, \hat \sigma)
\quad \text{and} \quad \Rep_\fl (ZW (J,\hat \sigma)) 
\end{equation}
are equivalent. However, if $ZW(J,\hat \sigma) \neq W(J,\hat \sigma)$, then
it seems that
\eqref{eq:5.24} does not extend to representations of arbitrary length, because 
$\C[ZW (J,\hat \sigma)]$ and $\C [W(J,\hat \sigma), \mu_{\hat \sigma}]$ are not Morita equivalent.

\subsection{The non-supercuspidal case} \
\label{subsec:non-sc-case} 

We return to the case where the facet $\ff$ is not necessarily minimal.~Recall that the group
$W(G,L) = N_G (L) / L$ acts on $\Irr (L)$ and on the isomorphism classes in $\Rep (L)$.~Let 
$W(G,L)_{\hat \sigma}$\label{i:132} be the stabilizer of $\Rep (L)_{(P_{L,\ff},\hat \sigma)}$ 
(as in Theorem \ref{thm:3.6}) in $W(G,L)$.~In other words, $W(G,L)_{\hat \sigma}$ is the finite 
group attached to the Bernstein component $\Irr (L)_{(P_{L,\ff},\hat \sigma)}$ as in 
\eqref{eqn:finite-group-attached-to-Bernstein-component}.

\begin{lem}\label{lem:5.8}
\enuma{
\item The category $\Rep (L)_{(\hat P_{L,\ff}, \hat \sigma)}$ determines $(\hat P_{L,\ff},
\hat \sigma)$ up to $L$-conjugacy.
\item The natural map $W(J,\hat \sigma) / W_L (J,\hat \sigma) \to W(G,L)_{\hat \sigma}$
is an isomorphism.
}
\end{lem}
\begin{proof}
(a) By Lemma \ref{lem:2.2} (a), $\Rep (L)_{(\hat P_{L,\ff}, \hat \sigma)}$ determines the
$L$-conjugacy class of $(P_{L,\ff},\sigma)$.~The irreducible representations $\pi$ of
$\hat P_{L,\ff}$ such that $(\hat P_{L,\ff}, \pi)$ is a type for  
$\Rep (L)_{(\hat P_{L,\ff}, \hat \sigma)}$ are precisely those for which 
$\ind_{\hat P_{L,\ff}}^L \pi$ contains $\hat \sigma$. This happens if and only if
$g \cdot \pi \cong \hat \sigma$ for some $g \in L_\ff$, and in which case 
$g \cdot (\hat P_{L,\ff},\pi) \cong (\hat P_{L,\ff}, \hat \sigma)$. Hence
$L \cdot (\hat P_{L,\ff}, \hat \sigma)$ is uniquely determined by 
$\Rep (L)_{(\hat P_{L,\ff}, \hat \sigma)}$.

(b) Using part (a), this can be shown in the same way as Lemma \ref{lem:2.2} (b).
\end{proof}

For $L$ as in \eqref{eq:2.13}, Lemma \ref{lem:3.2} and Corollary \ref{cor:3.3} hold with $L$ 
instead of $G$. Let $W_L$\label{i:160} be the Iwahori--Weyl group of $(L,S)$ and abbreviate $CW_L (J,\hat \sigma) := L_\tau^3 / L^1$. By Lemma \ref{lem:3.2} (b), $CW_L (J,\hat \sigma)$ is canonically isomorphic to a lattice in 
$X_* (Z(L)) \otimes_\Z \R$.\label{i:138}

\begin{lem}\label{lem:5.9}
\enuma{
\item The affine Weyl group $W_\af (J,\sigma)$ is the semidirect product of a finite Weyl 
group $W(R_\sigma)$ with the normal subgroup of translations $T(J,\sigma)$. 
\item The group $W_L (J,\sigma) \cap W_\af$ equals $T(J,\sigma)$, and can be 
represented by elements of $Z^\circ (L)$.
}
\end{lem}
\begin{proof}
(a) This is an aspect of the general structure of affine Weyl groups, see for example 
\cite[Chapitre VI.2]{Bou}.~The reference also shows that the group of translations 
$T(J,\sigma)$ is generated by the finite root system from this setup.

(b) By Lemma \ref{lem:3.2} (b) and \eqref{eq:2.10}, $W_L (J,\sigma) \cap W_\af$ is a group of
translations. By part (a) for $W_\af$, the lattice of translations $T(J,\sigma)$
in $W_\af$ is generated by the elements $\alpha^\vee (\varpi_F^{-1})$ with $\alpha$ in a
finite root system $R_\sigma$. As $\alpha^\vee$ takes values in $S$, all translations in $W_\af$ 
can be represented by elements of $S$. We can also represent them 
in $X_* (S)$, if we identify it with $\{ t (\varpi_F^{-1}) : t \in X_* (S) \}$. Here it is 
convenient to use the inverse of a uniformizing element $\varpi_F$ of $\mf o_F$: this is 
compatible with \cite[Appendix A]{SolEnd} because $|\varpi_F^{-1}|_F > 1$.
If such an element $t(\varpi_F^{-1})$ belongs to $W_L (J,\sigma)$, then it translates 
$\mh A_S$ in a direction in $X_* (Z(L)) \otimes_\Z \R$. Hence it is orthogonal to $R (L,S)$,
which means that $t(\varpi_F^{-1}) \in Z^\circ (L)$. 

Conversely, the constructions of $S_{\ff,\af,\sigma}$ and $\Delta_{\ff,\af,\sigma}$ show that 
$R_\sigma$ consists of roots of $(G,Z(L)^\circ)$. There $\alpha^\vee (\varpi_F^{-1}) \in 
Z(L)^\circ$ for all $\alpha \in R_\sigma$, and $T(J,\sigma) \subset W_L (J,\sigma)$.
\end{proof}

By Theorem \ref{thm:3.6} and \cite[\S 8]{BuKu}, $\mc H (L,\hat P_{L,\ff},\hat \sigma)$ 
embeds in $\mc H (G,\hat P_\ff,\hat \sigma)$. We prefer to use the renormalized version
\begin{equation}\label{eq:3.1}
\mc H (L,\hat P_{L,\ff},\hat \sigma) \hookrightarrow \mc H (G,\hat P_\ff,\hat \sigma) 
\end{equation}
that respects parabolic induction, as in \cite[Condition 4.1 and Lemma 5.1]{SolComp}. 
The image of \eqref{eq:3.1}, however, does not depend on such a normalization.

\begin{lem}\label{lem:3.7}
\enuma{
\item Via \eqref{eq:3.1}, we have
\[
\mc H (L,\hat P_{L,\ff},\hat \sigma) \cap \mc H (W_\af (J,\sigma),q_\sigma) =
\C [T(J,\sigma)] = \C[ZW_L (J,\hat \sigma)] \cap \mc H (W_\af (J,\sigma),q_\sigma) .
\]
\item The conjugation action of $\Omega (J,\hat \sigma)$ on $\mc H (G,\hat P_\ff,\hat \sigma)$,
from Theorems \ref{thm:2.4} and \ref{thm:3.1}, stabilizes 
$\mc H (L,\hat P_{L,\ff},\hat \sigma)$ and $\C [ZW_L (J,\hat \sigma)]$. 
}
\end{lem}
\begin{proof}(a) 
By Lemma \ref{lem:5.9}, the first intersection is precisely the maximal commutative subalgebra 
$\C [T (J,\sigma)]$ of the Bernstein presentation of $\mc H (W_\af (J,\sigma),q_\sigma)$. 
By Theorem \ref{thm:3.1}, the 2-cocycle $\mu_{\hat \sigma}$ is trivial on $\C[T(J,\sigma)]$,
thus $\C[T(J,\sigma)]$ commutes with 
$\mc H (L,\hat P_{L,\ff},\hat \sigma)$ by Corollary \ref{cor:3.3} and Lemma \ref{lem:2.2}. 
Then by Proposition \ref{prop:5.4}, it is 
already contained in the image of $\C [ZW_L (J,\sigma)]$.\\
(b) By the remark after \eqref{eq:2.4}, the inverse image of $N_W (W_J)$ in $N_G (S)$
normalizes $L$. Therefore, conjugation by elements of $\Omega (J,\hat \sigma)$ (which is 
well-defined by Theorem \ref{thm:3.1}) stabilizes the image of \eqref{eq:3.1}. Hence
conjugation by such elements also stabilizes the center of 
$\mc H (L,\hat P_{L,\ff},\hat \sigma)$, which by \eqref{eq:3.10} is precisely 
$\C [ZW_L (J,\hat \sigma)]$.
\end{proof}

By Lemma \ref{lem:3.7} (a), $\mc H (G,\hat P_\ff, \hat \sigma)$ 
contains the affine Hecke algebra\label{i:133}
\begin{equation}\label{eq:5.17}
\begin{aligned}
\mc H (G,\hat P_\ff, \hat \sigma)^\circ & := 
\mc H (W_\af (J,\sigma), q_\sigma) \C [ZW_L (J,\hat \sigma)] \\
& \: = \mc H (W_\af (J,\sigma),q_\sigma) \underset{\C [T (J,\sigma)]}{\otimes} 
\C [ZW_L (J,\hat \sigma)] \\
& \: =  \mc H (W (R_\sigma),q_\sigma) \otimes_\C \C [ZW_L (J,\hat \sigma)] .
\end{aligned}
\end{equation}
By Lemma \ref{lem:3.7} (b), the action of $\Omega (J,\hat \sigma)$ on
$\mc H (G,\hat P_\ff, \hat \sigma)$ stabilizes $\mc H (G,\hat P_\ff, \hat \sigma)^\circ$.

We now introduce a new facet $\mf f_\sigma \subset \ff \cap \mc B (\mc G_\ad, F)$ as follows.
We use the notation $\ff = \prod_i \ff^i$, where $i$ runs through an indexing set for the
simple factors of $G$. For each simple factor $G^i$ of $G$ such that 
$\Delta_{\ff,\af,\sigma}$ contains elements from $G^i$, we denote by $\ff_\sigma^i$ the facet
of $\mc B (G^i)$ that is contained in $C_0$ and lies in the zero set of $\Delta_\af
\setminus \Delta_{\ff,\af, \sigma}$. For each simple factor $G^i$ of $G$ such that
$\Delta_{\ff,\af,\sigma}$ contains no elements from $G^i$, we set $\mf f_\sigma^i := \ff^i$. 
Finally, we define $\mf f_\sigma := \prod_i \mf f_\sigma^i$.

The finite Weyl group from Lemma \ref{lem:5.9}(a) arises as the $W_\af (J,\sigma)$-stabilizer
of a special vertex of $\mf f_\sigma$. For the simple factors $G^i$ that do not contribute to 
$W_\af (J,\sigma)$, this does not pose any condition on the vertex, so we make it more
explicit. If $G^i$ contributes to $W_\af (J,\sigma)$, let $y_\sigma^i$ be a special vertex of
$\ff_\sigma^i$; otherwise, let $y_\sigma^i$ be the barycentre of $\ff^i$. Then $y_\sigma = 
\prod_i y_\sigma^i$ is a point of the building $\mc B (\mc G_\ad,F)$. The set
\[
\Delta_{\af,\sigma,y_\sigma} = 
\{ \alpha \in \Delta_{\ff,\af,\sigma} : \alpha (y_\sigma) = 0 \}
\]
is a basis of a finite root system whose Weyl group $W_\af (J,\sigma)_{y_\sigma}$ is 
isomorphic to $W_\af (J,\sigma) / T(J,\sigma)$.~It follows that $\Delta_{\sigma,y_\sigma} := D (\Delta_{\af,\sigma,y_\sigma})$ 
is a basis for a root subsystem $R_\sigma \subset R(G,S)$ with Weyl group $W (R_\sigma) 
\cong W_\af (J,\sigma)_{y_\sigma}$.~This $R_\sigma$ can be identified with the root
system from the proof of Lemma \ref{lem:5.9}.
Let $R(G,S)^+$ be a positive system in $R(G,S)$ containing 
$\Delta_{\sigma,y_\sigma}$. Using $R(G,S)^+$, we can define standard parabolic subgroups of 
$\mc G$ containing $Z_{\mc G}(\mc S)$. Indeed, let $\mc Q \subset \mc G$ be the parabolic
$F$-subgroup with Levi factor $\mc L$ and $R(G,S)^+ \subset R(Q,S)$. By Lemma
\ref{lem:3.4} and \cite[Lemma 3.2 and Proposition 3.3]{Oha2}, there are canonical isomorphisms
\begin{equation}\label{eq:3.8}
\II_Q^G \big( \ind_{L^1}^L (\tau_1) \big) \cong \II_Q^G \big( \ind_{\hat P_{L,\ff}}^L 
(\hat \sigma) \big) \cong \ind_{\hat P_\ff}^G (\hat \sigma),
\end{equation}
which induce canonical algebra isomorphisms
\begin{equation}\label{eq:3.7}
\mc H (G,\hat P_\ff,\hat \sigma) \cong \End_G \big( \ind_{\hat P_\ff}^G 
(\hat \sigma) \big) \cong \End_G \big( \II_Q^G (\ind_{L^1}^L (\tau_1)) \big).
\end{equation}
The algebra $\End_G \big( \II_Q^G (\ind_{L^1}^L \tau_1) \big)$ in \eqref{eq:3.7} was studied in \cite{SolEnd}, 
and later compared with $\mc H (G,P_\ff, \sigma)$ and with $\mc H (G,\hat P_\ff,
\hat \sigma)$ in \cite{Oha2}.~Recall from Theorem \ref{thm:3.1} that $\mc H (G,P_\ff,
\sigma)$ and $\mc H (G,\hat P_\ff,\hat \sigma)$ have the same underlying affine Weyl
group and the same $q$-parameters.
The group $W(J,\sigma) / \Omega_\ff^0$ acts naturally on the finite root system 
$R_\sigma$ underlying the affine root system with basis $\Delta_{\ff,\af,\sigma}$, 
with the subgroup of translations $X_\sigma$ acting trivially. The group
\begin{equation}\label{eq:5.1}
W(G,L)_{\hat \sigma} \cong W(J,\hat \sigma) / W_L (J,\hat \sigma) 
\end{equation}
from Lemma \ref{lem:5.8} (b) acts naturally on $R_\sigma$, and this action comes
from the action of $W(\mc G,\mc S)$ on $R(G,S)$.
The Weyl group $W(R_\sigma) \cong W_\af (J,\sigma)_{y_\sigma}$ acts simply transitively 
on the collection of positive systems in $R_\sigma$. Let $H_{\Delta,\hat \sigma}$ 
be the stabilizer of $\Delta_{\sigma,y_\sigma}$ in $W(J,\hat \sigma)$ and let\label{i:134} 
$\Gamma_{\hat \sigma} := H_{\Delta,\hat \sigma} / W_L (J,\hat \sigma)$ be the stabilizer of $\Delta_{\sigma,y_\sigma}$ in \eqref{eq:5.1}. Then 
$H_{\Delta,\hat \sigma}$ is complementary to $W_\af (J,\sigma)_{y_\sigma}$ in 
$W(J,\hat \sigma)$, and by \cite[(3.2)]{SolEnd} we have
\begin{equation}\label{eq:3.9}
W(G,L)_{\hat \sigma} \cong W(R_\sigma) \rtimes \Gamma_{\hat \sigma} .
\end{equation}
Here $\Gamma_{\hat \sigma}$ is the stabilizer of the positive system $R_\sigma^+ = 
R(G,S)^+ \cap R_\sigma$ or equivalently of the basis 
$\Delta_\sigma = D (\Delta_{\sigma,y_\sigma})$ of $R_\sigma$.

Let $\mf s$ denote the inertial equivalence class of $[L,\tau]$ for $\Rep (G)$. 
Like $\ind_{\hat{P}_\ff}^G (\hat \sigma)$, the $G$-representation \label{i:135} 
$\Pi_{\mf s} := \II_Q^G (\ind_{L^1}^L (\tau))$ is a projective generator of $\Rep (G)_{[L,\tau]}$ (see for example \cite{Ren}).~In fact it 
is isomorphic to a finite direct sum of copies of 
$\II_Q^G (\ind_{L^1}^L (\tau_1)) \cong \ind_{\hat P_\ff}^G (\hat \sigma)$. 
Such an isomorphism can be constructed as follows. Write
\[
\tau = \bigoplus\nolimits_{l \in L / L_\tau^3} \, l \cdot \tau_1 ,
\]
where $L_\tau^3$ is the stabilizer of $V_{\tau_1} \subset V_\tau$ in $L$, and $l$ runs 
through a set of representatives for $L / L^\tau$. Translation by $l$ gives an isomorphism 
$\ind_{L^1}^L (l \cdot \tau_1) \isom \ind_{L^1}^L (\tau_1)$, which induces to an isomorphism
\begin{equation}\label{eq:5.2}
\Pi_{\mf s} \cong  
\bigoplus\nolimits_{l \in L / L_\tau^3} \, \II_Q^G (\ind_{L^1}^L (\tau_1)) .
\end{equation}
Via \eqref{eq:5.2} and \eqref{eq:3.7}, we embed $\mc H (G,\hat P_\ff,\hat \sigma)$
diagonally in $\End_G (\Pi_{\mf s})$. In \cite[Proposition 2.2]{SolEnd}, it is shown how 
$\C [\mathfrak{X}_{\nr} (L,\tau), \natural]$ embeds in $\End_G (\Pi_{\mf s})$ for a certain
2-cocycle $\natural$.~Combined with \cite[Lemma 2.1]{SolEnd}, one deduces that
$\Irr (L/L_\tau^3)$ is a subgroup of $\mathfrak{X}_{\nr} (L,\tau)$, maximal for the property 
that its group algebra embeds naturally in $\C [\mathfrak{X}_{\nr} (L,\tau), \natural]$. 
By \eqref{eq:5.2}, we have
\begin{equation}\label{eq:5.25}
\Pi_{\mf s}^{\Irr (L/L_\tau^3)} \cong \II_Q^G (\ind_{L^1}^L (\tau_1)) .
\end{equation}
On the other hand, the multiplication action of 
$\mc O (\mathfrak{X}_{\nr} (L)) \cong \C [L / L^1]$ on $\ind_{L^1}^L (\tau)$ gives embeddings
\[
\C [L/L^1] \hookrightarrow \End_L (\ind_{L^1}^L (\tau)) \hookrightarrow \End_G (\Pi_{\mf s}).
\]
The Weyl group of the root system \label{i:137}
\begin{equation}\label{eq:5.35}
R_{\sigma, \tau} = \{ \alpha \in R_\sigma : s_\alpha (\tau) \cong \tau \}
\end{equation}
stabilizes $\tau$. The stabilizer of $\tau$ in $W(R_\sigma) \rtimes \Gamma_{\hat \sigma}$ 
acts on $R_{\sigma,\tau}$, and can be written as
\begin{equation}\label{eq:5.36}
\big( W(R_\sigma) \rtimes \Gamma_{\hat \sigma} \big)_\tau = 
W (R_{\sigma,\tau}) \rtimes \Gamma_{\hat \sigma,\tau} ,
\end{equation}
where $\Gamma_{\hat \sigma,\tau}$\label{i:139} is the stabilizer of the set of 
positive roots in $R_{\sigma,\tau}$. Along the covering
\begin{equation}\label{eq:5.7}
\mathfrak{X}_{\nr} (L) \to \Irr (L)_{(\hat P_{L,\ff},\hat \sigma)} : 
\chi \mapsto \chi \otimes \tau ,
\end{equation}
every element of $W (R_\sigma) \rtimes \Gamma_{\hat \sigma}$ can be 
lifted to a diffeomorphism of $\mathfrak{X}_{\nr} (L)$,
and the elements that stabilize $\tau$ can be lifted to Lie group automorphisms of
$\mathfrak{X}_{\nr} (L)$. This gives rise to a finite group $W(L,\tau,\mathfrak{X}_{\nr} (L))$ of diffeomorphisms
of $\mathfrak{X}_{\nr} (L)$; see \cite[\S 3]{SolEnd}. It fits into a short exact sequence
\begin{equation}\label{eq:5.3}
1 \to \mathfrak{X}_{\nr} (L,\tau) \to W(L,\tau,\mathfrak{X}_{\nr} (L)) \to 
W(R_\sigma) \rtimes  \Gamma_{\hat \sigma} \to 1
\end{equation}
\label{i:140}and contains a subgroup canonically isomorphic to 
$\big( W(R_\sigma) \rtimes \Gamma_{\hat \sigma} \big)_\tau$.

\begin{prop}\label{prop:5.5}
\enuma{
\item We have the following identifications as vector spaces
\begin{multline*}
\End_G (\Pi_{\mf s}) = \bigoplus\nolimits_{l \in L / L_\tau^3} \, \C \{l\} \otimes 
\mc H (G,\hat P_\ff, \hat \sigma) \otimes \C [\Irr (L/L_\tau^3)] \\
= \mc O (\mathfrak{X}_{\nr} (L)) \otimes \mc H (W(R_\sigma), q_\sigma) \otimes 
\bigoplus\limits_{\gamma \in H_{\Delta,\hat \sigma} / W_L (J,\hat \sigma)} \C T_\gamma 
\otimes \C [\mathfrak{X}_{\nr} (L,\tau), \natural] .
\end{multline*}
\item The linear subspace $\mc O (\mathfrak{X}_{\nr} (L)) \otimes \mc H (W(R_\sigma), q_\sigma)$
is an affine Hecke algebra, and the conjugation action of $\mathfrak{X}_{\nr} (L,\tau) \subset
\C [\mathfrak{X}_{\nr} (L,\tau),\natural]^\times / \C^\times$ on it is given by translations on 
$\mc O (\mathfrak{X}_{\nr} (L))$.
}
\end{prop}
\begin{proof}
(a) By \eqref{eq:5.2}, we have 
$\End_G (\Pi_{\mf s}) \cong 
M_{[L : L_\tau^3]} (\C) \otimes_\C \mc H (G,\hat P_\ff,\hat \sigma)$. The elements $l \in L / L_\tau^3$ permute the different copies in \eqref{eq:5.2}, and by
\cite[\S 2]{SolEnd} so do the elements of $\Irr (L/L_\tau^3)$. 
Furthermore, by \cite[\S 5.1]{SolEnd},
\[
\bigoplus\nolimits_{l \in L / L_\tau^3} \, \C \{l\} \otimes 
\C [\mathfrak{X}_{\nr} (L,\tau),\natural] \; \subset \;
\mc O (\mathfrak{X}_{\nr} (L)) \otimes \C [\mathfrak{X}_{\nr} (L,\tau),\natural] 
\]
embeds in $\End_G (\Pi_{\mf s})$.
This establishes the first equality of vector spaces. By \eqref{eq:5.1}, \eqref{eq:3.9} and
Theorem \ref{thm:3.1}, we deduce that (as vector spaces)
\[
\mc H (G,\hat P_\ff,\hat \sigma) = \C [W_L (J,\hat \sigma)] \otimes 
\mc H (W(R_\sigma),q_\sigma) \otimes \bigoplus\nolimits_{\gamma \in H_{\Delta,\hat \sigma} / 
W_L (J,\hat \sigma)} \C T_\gamma .
\]
We also note that 
\[
\bigoplus\limits_{l \in L / L_\tau^3} \C \{l\} \otimes \C [CW_L (J,\hat \sigma)] \cong
\bigoplus\limits_{l \in L / L_\tau^3} \C \{l\} \otimes \C [L_\tau^3 / L^1] \cong
\C [L / L^1] = \mc O (\mathfrak{X}_{\nr} (L)) .
\]
It follows from \eqref{eq:5.33} that, also as vector spaces,
\begin{align*}
\C [ W_L (J,\hat \sigma) / CW_L (J,\hat \sigma) ] \otimes \C [\Irr (L/L_\tau^3)] & \cong 
\C [ \Irr (L_\tau^3 / L_\tau^2)] \otimes \C [\Irr (L / L_\tau^3)] \\ 
& \cong \C [\Irr (L / L_\tau^2)] \cong \C [\mathfrak{X}_{\nr} (L,\tau)] .
\end{align*}
These observations imply the second equality of vector spaces.\\
(b) The cross relations between $\mc O (\mathfrak{X}_{\nr} (L) / \mathfrak{X}_{\nr} (L,\tau))$
and $T_{s_\alpha}$ can be found for instance in \cite[Definition 1.11]{SolSurv}. These are
also multiplication relations in 
\[
\mc H (W_\af (J,\sigma),q_\sigma) \mc O (\mathfrak{X}_{\nr} (L) / \mathfrak{X}_{\nr} (L,\tau)),
\]
and they show that $T_{s_\alpha}$ (for $\alpha \in R_\sigma$) commutes with 
$\mc O (\mathfrak{X}_{\nr} (L) / \mathfrak{X}_{\nr} (L,\tau))^{s_\alpha}$. Comparing these with 
the multiplication relations
in $\End_G (\Pi_{\mf s})$ from \cite[Corollary 5.8]{SolEnd}, we deduce that the image of $T_{s_\alpha}$ in $\C (\mathfrak{X}_{\nr} (L)) \otimes_{\mc O (\mathfrak{X}_{\nr} (L))} 
\End_G (\Pi_{\mf s})$ lies in 
$\C (\mathfrak{X}_{\nr} (L)) \oplus \C (\mathfrak{X}_{\nr} (L)) \mc T_{s_\alpha}$, where $\mc T_{s_\alpha}$ is as in \cite[\S 5]{SolEnd}, and it acts on 
$\mc O (\mathfrak{X}_{\nr} (L))$ just 
like $s_\alpha$. The cross relations for $T_{s_\alpha}$ can be deduced from the
expression of $T_{s_\alpha}$ in terms of $\mc T_{s_\alpha}$ (see \cite[(6.25)]{SolEnd}).

Since the action of $W(R_\sigma)$ on $\mathfrak{X}_{\nr} (L) / \mathfrak{X}_{\nr} (L,\tau)$ 
lifts canonically to $\mathfrak{X}_{\nr} (L)$, the cross relations for $T_{s_\alpha}$ lift to 
$\mc O (\mathfrak{X}_{\nr} (L))$. Hence
$\mc O(\mathfrak{X}_{\nr} (L))$ and the $T_{s_\alpha}$ generate an affine Hecke algebra in the sense
of \cite[Definition 1.11]{SolSurv}.~The conjugation action of $\mathfrak{X}_{\nr} (L,\tau)$ is given 
by \cite[(5.16) and Corollary 5.18]{SolEnd}.
\end{proof}

The group $\Irr (L / L_\tau^3)$ is embedded in $\End_G (\Pi_{\mf s})$, 
hence has two commuting actions on that algebra: by left and right 
multiplications. It follows from \eqref{eq:5.25} that
\begin{equation}\label{eq:5.5}
\mc H (G,\hat P_\ff,\hat \sigma) = \End_G \big( \II_Q^G (\ind_{L^1}^L (\tau_1)) \big) =
\End_G (\Pi_{\mf s})^{\Irr (L/L_\tau^3) \times \Irr (L/L_\tau^3)} .
\end{equation}
Since $\mc O (\mathfrak{X}_{\nr} (L) / \Irr (L/L_\tau^3)) \cong \C [L_\tau^3 / L^1]$, 
this is consistent with Proposition \ref{prop:5.5}. 

For an open subset $U \subset \mathfrak{X}_{\nr} (L)$, we define
\begin{equation}
\End_G(\Pi_{\mf s})^{an}_U := \End_G (\Pi_{\mf s}) \otimes_{\mc O(\mathfrak{X}_{\nr} (L))} C^{an}(U) . 
\end{equation}
If $U$ is $W(L,\tau,\mathfrak{X}_{\nr} (L))$-stable, then this is an algebra, because it can be
realized as 
\begin{equation}\label{eq:5.6}
\End_G(\Pi_{\mf s})^{an}_U = \End_G (\Pi_{\mf s}) \otimes_{\mc O(\mathfrak{X}_{\nr} (L))^{
W(L,\tau,\mathfrak{X}_{\nr} (L))}} C^{an}(U)^{W(L,\tau,\mathfrak{X}_{\nr} (L))},
\end{equation}
and $\mc O(\mathfrak{X}_{\nr} (L))^{W(L,\tau,\mathfrak{X}_{\nr} (L))}$ is central in $\End_G(\Pi_{\mf s})$. 

Recall that $\C [L_\tau^3 / L^1] = \C [CW_L (J,\hat \sigma)]$ is a maximal commutative 
subalgebra of $\mc H (G,\hat P_\ff, \hat \sigma)$.
For $\cH (G,\hat P_\ff, \hat \sigma)$, analytic localization can be defined similarly as in
\eqref{eq:5.34}, and there is a variant with $\Irr (CW_L (J,\hat \sigma)) = 
\Irr (L^3_\tau / L^1)$ instead of $\Irr (ZW_L (J,\hat \sigma)) = \Irr (L^2_\tau / L^1)$. 
In the situation where we have a $W(L,\tau,\mathfrak{X}_{\nr} (L))$-stable $U$, these two versions of
analytic localization agree by \eqref{eq:5.8}, i.e.~we have 
\[
\cH (G,\hat P_\ff, \hat \sigma)^{an}_{U |_{L^3_\tau}} \cong  
\cH (G,\hat P_\ff, \hat \sigma)^{an}_{U |_{L^2_\tau}} .
\]
Let $U_\tau$ be a neighborhood of 1 in $\mathfrak{X}_{\nr} (L)$, which covers $\Irr (L)_{\mf s_L}$ 
via the operation $\otimes \tau$ from \eqref{eq:5.7}.
(In \cite{SolEnd}, this $U_\tau$ is called $U_u$.) 
We assume that $U_\tau$ satisfies \cite[Condition 6.3]{SolEnd}, which means that $U_\tau$ 
is sufficiently small and $w U_\tau \cap U_\tau=\varnothing$ whenever 
$w \in W(L,\tau,\mathfrak{X}_{\nr} (L))$ and 
$w 1 \neq 1$. We write $U = W(L,\tau,\mathfrak{X}_{\nr} (L)) U_\tau$, such that
\[
C^{an}(U) = \bigoplus\limits_{w \in W(L,\tau,\mathfrak{X}_{\nr} (L)) / 
(L,\tau,\mathfrak{X}_{\nr} (L))_1} 1_{w U_\tau} C^{an}(U) ,
\]
where $1_X$ denotes the indicator function of a set $X$.

\begin{lem}\label{lem:5.6}
For $w_1, w_2 \in W(L,\tau,\mathfrak{X}_{\nr} (L))$, there is a canonical isomorphism of vector spaces 
$1_{w_1 U_\tau} \End_G (\Pi_{[L,\tau]})_{U}^{an} 1_{w_2 U_\tau} \cong 
1_{w_1 U_\tau |_{L^3_\tau}} \mc H (G,\hat P_\ff, \hat \sigma)_{U |_{L_\tau^3}}^{an}
1_{w_2 U_\tau |_{L^3_\tau}}$.
\end{lem} 
\begin{proof}
Since $w U_\tau \cap U_\tau=\varnothing$ for $w \in W(L,\tau,\mathfrak{X}_{\nr} (L)) \setminus 
W(L,\tau,\mathfrak{X}_{\nr} (L))_1$, we have
\[
1_{\Irr (L/L_\tau^3) w_1 U_\tau}  \End_G (\Pi_{\mf s})_U^{an} 1_{\Irr (L/L_\tau^3) w_2 U_\tau}
= \bigoplus\limits_{\chi_1,\chi_2 \in \Irr (L/L_\tau^3)} \hspace{-1mm} 1_{\chi_1 w_1 U_\tau} 
\End_G (\Pi_{\mf s})_U^{an} 1_{\chi_2 w_2 U_\tau} .
\]
Here $\Irr (L/L_\tau^3)^2$ acts freely, so this space contains
\[
1_{w_1 U_\tau} \End_G (\Pi_{\mf s})_U^{an} 1_{w_2 U_\tau} \cong \big( 
1_{\Irr (L/L_\tau^3) w_1 U_\tau}  \End_G (\Pi_{\mf s})_U^{an} 1_{\Irr (L/L_\tau^3) w_2 U_\tau}
\big)^{\Irr (L/L_\tau^3)^2} .
\]
By \eqref{eq:5.6} and \eqref{eq:5.5}, we can rewrite the right hand side as
\begin{align*}
\!\! 1_{\Irr (L/L_\tau^3) w_1 U_\tau} & \big( \End_G (\Pi_{\mf s})^{\Irr (L/L_\tau^3)^2} \hspace{-7mm} 
\underset{\mc O (\mathfrak{X}_{\nr} (L)) / W(L,\tau,\mathfrak{X}_{\nr} (L))}{\otimes} \hspace{-7mm} 
C^{an} (U)^{W(L,\tau,\mathfrak{X}_{\nr} (L))} \big) 1_{\Irr (L/L_\tau^3) w_2 U_\tau} \\
& = 1_{\Irr (L/L_\tau^3) w_1 U_\tau} \mc H (G,\hat P_\ff, \hat \sigma)_{U |_{L_\tau^3}}^{an} 
1_{\Irr (L/L_\tau^3) w_2 U_\tau} \\
& = 1_{w_1 U_\tau |_{L^3_\tau}} \mc H (G,\hat P_\ff, \hat \sigma)_{U |_{L_\tau^3}}^{an} 
1_{w_2 U_\tau |_{L^3_\tau}} . \qedhere
\end{align*}
\end{proof}

\section{Hecke algebras for non-singular depth-zero Langlands parameters}
\label{sec:HeckeL}

\subsection{Preliminaries}\label{subsec:HeckeL-prelim} \

Consider an enhanced supercuspidal L-parameter $(\varphi,\rho)$ for a Levi subgroup $L$ of $G$.
Via \eqref{eq:6.27} and the natural isomorphism 
\[
\mf X_\nr (L) \cong ((Z(L^\vee)^{\mb I_F})_{\mb W_F} )^\circ
\]
from \cite[\S 3.3.1]{Hai}, $\mf X_\nr (L)$ acts on $\Phi_e (L)$.

The Bernstein component of $\Phi_e (L)$ containing $(\varphi,\rho)$ will be denoted \label{i:142}
\begin{equation}\label{eq:9.42}
\mf s_L^\vee := \mathfrak{X}_{\nr} (L) \cdot (\varphi,\rho) = 
\big\{ (z \varphi, \rho) : z \in ((Z(L^\vee)^{\mb I_F})_{\mb W_F} )^\circ \big\} .
\end{equation}
By $\mf s^\vee$,\label{i:413} we are referring to $\mf s_L^\vee$ considered as an inertial 
equivalence class for $\Phi_e (G)$. Recall the cuspidal support map Sc\label{i:171} 
for enhanced Langlands parameters from 
\cite[\S 7]{AMS1}, extended to the setting with rigid inner twists in \cite{DiSc}. 
It associates, to every enhanced L-parameter $(\psi,\epsilon)$ for $G$, a triple
$\mathrm{Sc}(\psi,\epsilon) = (L',\psi',\epsilon')$, where $L' \subset G$ is a 
Levi subgroup and $(\psi',\epsilon')$ is a cuspidal enhanced L-parameter for $L'$. The map 
Sc preserves $\psi |_{\mb I_F}$, so in particular sends the depth-zero parameters to
depth-zero parameters (but maybe of other groups). We write\label{i:144}
\begin{equation}
\Phi_e (G)^{\mf s^\vee} := \mr{Sc}^{-1} (\{L\} \times \mf s_L^\vee ) .
\end{equation}
By definition, this is a Bernstein component
of $\Phi_e (G)$. If $\mf s_L^\vee$ consists of depth-zero enhanced L-parameters, then 
$\Phi_e (G)^{\mf s^\vee} \subset \Phi_e^0 (G)$.~To $\Phi_e (G)^{\mf s^\vee}$, 
\cite[\S 3]{AMS3} associates a twisted affine Hecke algebra 
$\mc H (\mf s^\vee, q_F^{1/2})$\label{i:145}, which is a specialization of an algebra $\cH (\mf s^\vee, \mb z)$ with an invertible formal variable $\mb z$. 
However, \cite{AMS3} works for (normal, non-rigid) inner twists of $G$ and for enhancements of
L-parameters based on the component groups introduced in \cite{Art2}. We need to 
check that \cite{AMS3} also applies in the current setting.

\begin{lem}\label{lem:9.8}
The construction of $\cH (\mf s^\vee, \mb z)$ in \cite[\S 3]{AMS3} (for an arbitrary
Bernstein component $\Phi_e (G)^{\mf s^\vee}$) can be adapted
so that it also works in our setting with rigid inner twists of reductive $F$-groups
and component groups $\pi_0 (S_\varphi^+)$ for Langlands parameters $\varphi$. All the 
results in \cite[\S 3]{AMS3} remain valid for the adapted version
of the twisted affine Hecke algebra $\cH (\mf s^\vee, \mb z)$.
\end{lem}
\begin{proof}
Essentially, all the arguments in \cite[\S 3]{AMS3} rely on \cite[\S 1--5]{AMS1}, 
\cite[\S 2--4]{AMS2} and \cite[\S 1--2]{AMS3}, which apply to arbitrary (possibly
disconnected) complex reductive groups. The specific setup involving enhanced L-parameters
from \cite{Art2} and the associate groups only appear in the later sections of 
\cite{AMS1,AMS3}. A setting with slightly different enhanced L-parameters works
equally well in \cite{AMS1,AMS2,AMS3}, because the arguments with complex reductive
groups hardly change. Therefore it suffices to describe how the complex reductive
groups in \cite[\S 3]{AMS3} must be adapted. 

Firstly, consider Arthur's group $S_\varphi = Z^1_{\mc G_\Sc^\vee}(\varphi)$, which
is obtained by first taking the image of $Z_{G^\vee}(\varphi)$ in ${G^\vee}_\ad$ and then
the preimage of that in ${G^\vee}_\Sc$. In \cite{Art,AMS1,AMS3}, an enhancement of
$\varphi$ is an irreducible representation of $\pi_0 (S_\varphi)$. Secondly, the group
$G_\varphi := Z^1_{\mc G^\vee_\Sc}(\varphi |_{\mb W_F})$ has the property
$S_\varphi = Z_{G_\varphi}\big( \varphi (\SL_2 (\C)) \big)$ and is contained in the group
$J := Z^1_{\mc G^\vee_\Sc}(\varphi |_{\mb I_F})$. We make the following substitutions:
\begin{equation}\label{eq:9.mod}
\begin{array}{c|c}
\text{setting from \cite[\S 3]{AMS3}} & \text{setting with rigid inner twists} \\
\hline
S_\varphi = Z^1_{\mc G_\Sc^\vee}(\varphi) & S_\varphi^+ = Z_{\bar{G}^\vee}(\varphi) \\
G_\varphi = Z^1_{\mc G^\vee_\Sc}(\varphi |_{\mb W_F}) & Z_{\bar{G}^\vee}(\varphi (\mb W_F)) \\
J = Z^1_{\mc G^\vee_\Sc}(\varphi |_{\mb I_F}) & Z_{\bar{G}^\vee}(\varphi (\mb I_F)) \\
{G^\vee}_\Sc & \bar{G}^\vee \\
L^\vee_c = \text{preimage of } L^\vee \text{ in } {G^\vee}_\Sc
& \bar{L}^\vee 
\end{array}
\end{equation}
With these substitutions, \cite[Proposition 3.4 and Theorem 3.6]{AMS3} (which come
directly from \cite{AMS1}) hold, as also noted in \cite{DiSc}. The groups
$M$ and $T$ in \cite[\S 3]{AMS3} can be defined in essentially the same way, i.e.~$M := Z_{L^\vee}(\varphi (\mb W_F))$ and $T := Z(M)^\circ$. Then $Z(G^\vee)^{\mb W_F}
\subset T$ and hence \cite[Lemma 3.7]{AMS3} amounts to:
\[
\text{there is a finite covering } T \to \mf X_\nr ({}^L \mc G) = \mf X_\nr (G^\vee) .
\]
Therefore, we need to replace the group $G_{\varphi_b} \times \mf X_\nr (G^\vee) =
Z^1_{\mc G_\Sc^\vee}(\varphi_b |_{\mb W_F}) \times \mf X_\nr ({}^L \mc G)$
from \cite[(3.9)]{AMS3} by $Z_{\bar{G}^\vee}(\varphi_b (\mb W_F))$. Now all the 
remaining results in \cite[\S 3]{AMS3} hold in this new setting with the same proofs.
\end{proof}

We summarize the structure of $\cH (\mf s^\vee, q_F^{1/2})$ below:

\begin{itemize}
\item Since $\mf s_L^\vee$ carries a transitive action of torus 
$(Z(L^\vee)^{\mb I_F})^\circ_{\mb W_F}$ with finite stabilizers, the choice of a 
base point makes $\mf s_L^\vee$ into a complex algebraic torus.
\item The ring $\mc O (\mf s_L^\vee)$, of regular functions on the complex algebraic 
variety $\mf s_L^\vee$, is by definition a commutative subalgebra of 
$\mc H (\mf s^\vee,q_F^{1/2})$.
\end{itemize}
\begin{itemize}
\item The group $W(G,L) \cong W(G^\vee,L^\vee)^{\mb W_F}$ acts on $\Phi_e (L)$ and on
the set of Bernstein components of $\Phi_e (L)$. Let $W_{\mf s^\vee}$\label{i:150} 
denote the stabilizer of $\Phi_e (G)^{\mf s^\vee}$. 
\item There exists a root system \label{i:146}$R_{\mf s^\vee} \subset X^* (\mf s_L^\vee)$ 
on which $W_{\mf s^\vee}$ acts. The choice of a positive system $R_{\mf s^\vee}^+$ 
leads to a decomposition 
$W_{\mf s^\vee} = W(R_{\mf s^\vee}) \rtimes \Gamma_{\mf s^\vee}$, where 
\begin{equation}\Gamma_{\mf s^\vee} = \{ w \in W(G,L)^{\mf s^\vee} : 
w (R_{\mf s^\vee}^+) = R_{\mf s^\vee}^+ \}.
\end{equation}
\end{itemize}
\begin{itemize}
\item As vector spaces, we have
\begin{equation}\label{eq:9.1}
\mc H (\mf s^\vee, q_F^{1/2}) = 
\mc O (\mf s_L^\vee) \otimes \C [W(R_{\mf s^\vee})] \otimes \C [\Gamma_{\mf s^\vee}] ,
\end{equation}
where $\mf s_L^\vee$ is made into a complex algebraic torus by the choice of a basepoint.
\end{itemize}
\begin{itemize}
\item The subspace $\mc O (\mf s_L^\vee) \otimes \C [W(R_{\mf s^\vee})]$ is an affine Hecke
algebra $\mc H (\mf s^\vee, q_F^{1/2})^\circ$\label{i:147}
as in \cite[Definition 1.11]{SolSurv}, with complex torus $\mf s_L^\vee$, 
root system $R_{\mf s^\vee}$ and certain $q$-parameters $q_{\alpha^\vee}$\label{i:148}, 
$q^*_{\alpha^\vee}$ for $\alpha^\vee \in R_{\mf s^\vee}$. We take $q_F^{1/2}$ as the $q$-base 
(so that all the $\mb z_j$'s from \cite[\S 3.3]{AMS3} are specialized to $q_F^{1/2}$).
\item The group $\Gamma_{\mf s^\vee}$ acts naturally on $\mc H (\mf s^\vee, q_F^{1/2})^\circ$,
and $\mc H (\mf s^\vee,q_F^{1/2})$ is a twisted crossed product of
$\mc H (\mf s^\vee, q_F^{1/2})^\circ$ and $\Gamma_{\mf s^\vee}$. In particular 
$\C [\Gamma_{\mf s^\vee}]$ is embedded in $\mc H (\mf s^\vee, q_F^{1/2})$ as a twisted group
algebra $\C [\Gamma_{\mf s^\vee}, \natural_{\mf s^\vee}]$, for a certain 2-cocycle\label{i:149}
\begin{equation}
\natural_{\mf s^\vee} : \Gamma_{\mf s^\vee}^2 \to \C^\times.
\end{equation}
\end{itemize}
Note that $\mc H (\mf s^\vee,q_F^{1/2})$ is not exactly an instance of the twisted affine
Hecke algebras in \cite{AMS3}: those have formal variables as $q$-parameters, whereas our
$q$-parameters are real numbers. The quintessential property of $\mc H (\mf s^\vee, q_F^{1/2})$
is that its irreducible (left) modules are parametrized canonically by $\Phi_e (G)^{\mf s^\vee}$; see
\cite[Theorem 3.18.a]{AMS3}.

Recall from \eqref{eq:8.7} and Theorem \ref{thm:8.2} that we have an LLC for non-singular
supercuspidal representations of $L$, and that it is a bijection onto the appropriate set of 
enhanced L-parameters. By Theorem \ref{thm:8.1}, this LLC is 
$\mathfrak{X}_{\nr} (L)$-equivariant. Hence it induces a bijection
\begin{equation}\label{eq:9.49}
\begin{Bmatrix}\text{non-singular supercuspidal}\\
\text{Bernstein components in } \Irr^0 (L)\end{Bmatrix}\:\leftrightarrow\:
\begin{Bmatrix}\text{supercuspidal } \\
\text{Bernstein components in $\Phi_e^0 (L)$}\end{Bmatrix}.
\end{equation}
We write this bijection as $\mf s_L \mapsto \mf s_L^\vee$.

Let $\mf s$\label{i:172} be $\mf s_L$ viewed as an inertial equivalence class for $G$, and
let $\Rep (G)_{\mf s}$
be its corresponding Bernstein block of $\Rep (G)$. The associated Bernstein component of 
$\Irr (G)$ will be denoted $\Irr (G)_{\mf s}$, and the associated Bernstein component of 
$\Phi_e (G)$ will be denoted $\Phi_e (G)^{\mf s^\vee}$. Recall the group 
$W_{\mf s} := W(G,L)_{\hat \sigma}$ from Lemma \ref{lem:5.8}. By a similar argument as 
in Lemma \ref{lem:7.6} (a), replacing the stabilizer of $\theta$ by the stabilizer of 
$\Xo (L) \theta$ and $\rho$, it can be expressed as 
\begin{equation}\label{eq:9.30}
W_{\mf s} = W(G,L)_{jT,\mathfrak{X}_{\nr} (L) \theta, \rho} \cong 
W(N_G (L), jT)_{\mathfrak{X}_{\nr} (L) \theta,\rho}
\big/ W(L,jT)_{\mathfrak{X}_{\nr} (L) \theta, \rho} .
\end{equation}

\begin{lem}\label{lem:9.1}
The stabilizer $W_{\mf s}$ of $\Rep(L)_{\mf s_L}$ 
equals $W_{\mf s^\vee} = W(G^\vee,L^\vee)^{\mb W_F}_{\mf s^\vee}$.
\end{lem}
\begin{proof}
As observed before Lemma \ref{lem:5.8}, $W(G,L)_{\hat \sigma} = W_{\mf s}$ equals the 
stabilizer of $\Rep(L)_{\mf s_L}$ in $W(G,L)$. The equality $W_{\mf s} = W_{\mf s^\vee}$ 
follows directly from the $W(G,L)$-equivariance of the LLC in Theorem \ref{thm:8.2}.    
\end{proof}
Restricting the LLC from Theorem \ref{thm:8.2} to $\mf s_L^\vee$ and $\mf s_L$, 
we obtain a bijection
\begin{equation}\label{eq:9.2}
\Irr (L)_{\mf s_L} = \Irr (L)_{(\hat P_{L,\ff}, \hat \sigma)} 
\longrightarrow \Phi_e (L)^{\mf s_L^\vee}.
\end{equation}

\begin{lem}\label{lem:9.2}
The bijection \eqref{eq:9.2} induces an isomorphism of vector spaces
\[
\mc H (\mf s^\vee, q_F^{1/2}) \longrightarrow \mc H (G,\hat P_\ff, \hat \sigma)^\circ 
\otimes \C [\Gamma_{\hat \sigma}].
\]
\end{lem}
\begin{proof}
Firstly, pullback along \eqref{eq:9.2} defines an algebra isomorphism
\begin{equation}\label{eq:9.3}
\mc O (\mf s_L^\vee) \, \xrightarrow{\sim} \,
\mc O \big( \Irr (L)_{(\hat P_{L,\ff}, \hat \sigma)} \big) = \mc O (\Irr (L)_{\mf s_L}).
\end{equation}
By Corollary \ref{cor:3.3} and Proposition \ref{prop:5.4}, the right hand side is
\begin{equation}\label{eq:9.4}
\mc O \big (\Irr (ZW_L (J,\hat \sigma)) \big) \cong \C [ZW_L  (J,\hat \sigma)] .    
\end{equation}
Thus \eqref{eq:9.3} and \eqref{eq:9.4} give an isomorphism between the maximal
commutative subalgebras of $\mc H (\mf s^\vee, q_F^{1/2})$ and 
$\mc H (G, \hat P_\ff, \hat \sigma)^\circ$.~By Theorem \ref{thm:8.2} and Lemma \ref{lem:9.1},
this isomorphism intertwines the actions of $W_{\mf s} \cong W_{\mf s^\vee}$, 
so it extends to an algebra isomorphism
\begin{equation}\label{eq:9.5}
\mc O (\mf s_L^\vee) \rtimes W_{\mf s^\vee} \overset{\sim}{\longrightarrow} 
\C [ZW_L (J,\hat \sigma)] \rtimes W_{\mf s} .
\end{equation}
By \eqref{eq:5.17}, the basis elements $T_w$, for $w \in W_{\mf s} = 
W(R_\sigma) \rtimes \Gamma_{\hat \sigma}$, provide the following linear bijection 
(it is usually not an algebra homomorphism)
\begin{equation}\label{eq:9.6}
\C [ZW_L (J,\hat \sigma)] \rtimes W_{\mf s} = 
\C[ZW_L (J,\hat \sigma)] \otimes \C[W(R_\sigma)] \rtimes \Gamma_{\hat \sigma} 
\rightarrow \mc H (G, \hat P_\ff, \hat \sigma)^\circ \rtimes \Gamma_{\hat \sigma} .
\end{equation}
Similarly, by the construction \eqref{eq:9.1}, there is a linear bijection
\begin{equation}\label{eq:9.7}
\mc O (\mf s_L^\vee) \rtimes W_{\mf s^\vee} = (\mc O (\mf s_L^\vee) \rtimes W (R_{\mf s^\vee})) 
\rtimes \Gamma_{\mf s^\vee} \rightarrow \mc H (\mf s^\vee,q_F^{1/2}) = 
\mc H (\mf s^\vee,q_F^{1/2})^\circ \rtimes \C [\Gamma_{\mf s^\vee},\natural_{\mf s^\vee}] .
\end{equation}
To conclude, we compose the inverse of \eqref{eq:9.7} first with \eqref{eq:9.5}, then
with \eqref{eq:9.6}.
\end{proof}

\subsection{Comparison of \texorpdfstring{$q$}{q}-parameters} \ 
\label{par:q}

Despite the similarities between \eqref{eq:9.6} and \eqref{eq:9.7}, Lemmas \ref{lem:9.1}
and \ref{lem:9.2} do not yet establish isomorphisms
\begin{equation}\label{eq:9.8}
W(R_\sigma) \cong W(R_{\mf s^\vee}) \quad \text{and} \quad
\Gamma_{\hat \sigma} \cong \Gamma_{\mf s^\vee} .   
\end{equation}
To achieve \eqref{eq:9.8}, we need to compare the $q$-parameters of reflections in $W_{\mf s}$
and $W_{\mf s^\vee}$. More precisely, for \eqref{eq:9.8} we need to know which 
$q$-parameters are 1 and which are bigger than 1. By definition, a reflection in 
$\Gamma_{\hat \sigma}$ or $\Gamma_{\mf s^\vee}$ has $q$-parameter 1. We write
\begin{align*}
& \Omega' (\emptyset,\theta_\ff) := \Omega (\emptyset,\theta_\ff) \cap 
\langle W(\emptyset,\theta_\ff) \cap S_{\ff,\af} \text{ for } G_\sigma \rangle ,\text{ and }\\
& \langle W(\emptyset,\theta_\ff) \cap S_{\ff,\af} \text{ for } G_\sigma \rangle =
W_\af (\emptyset, \theta_\af) \rtimes \Omega' (\emptyset,\theta_\ff) .
\end{align*}
With these notations, Proposition \ref{prop:4.8} says that all information about the 
$q$-parameters for $\mc H (G,\hat P_\ff, \hat \sigma)$ is contained in the extended 
affine Hecke algebra
\begin{equation}\label{eq:9.9}
\mc H (W_\af (\emptyset,\theta_\ff), q_\theta) \rtimes \Omega' (\emptyset, \theta_\ff)
\; \subset \; \mc H (G_\sigma, P_{G_\sigma,\ff}, \theta_\ff) .
\end{equation}
Even better, Corollary \ref{cor:4.2} says that  we only need 
$\mc H (G_\alpha, P_{G_\alpha,\ff}, \theta_\ff)$ to determine $q_{\theta,\alpha}$ for 
$\alpha \in \Delta_{\ff,\sigma}$. The complex dual group of $G_\alpha$ has maximal
torus $T^\vee$. 

The maximal $F$-split subtorus $\mc T_s$ of $\mc T_\ff$ and of $\mc T$ corresponds to 
$T^\vee_{\mb W_F}$, which admits a finite covering from $T^{\vee,\mb W_F,\circ}$. 
By \eqref{eq:4.8}, the root system of $(G_\alpha^\vee,T^\vee)$ can be expressed as $R(\mc G_\alpha, \mc T)^\vee = \{ \beta^\vee \in R(G^\vee,\mc T^\vee) : 
\beta^\vee |_{T^{\vee,\mb W_F,\circ}} \in \R^\times \alpha^\vee \}$.~By construction \cite[Lemma 3.12]{AMS3}, $R_{\mf s^\vee}$ consists of certain integral 
multiples $m_\alpha \alpha^\vee$ of elements $\alpha^\vee \in R(\mc G_\alpha, \mc T)^\vee = 
R (G_\alpha^\vee, T^\vee)$. Furthermore, ${}^L G_\alpha = {}^L T G_\alpha^\vee = G_\alpha^\vee \rtimes {}^L j (\mb W_F)$.~The algebra $\mc H (W_\af (\emptyset,\theta_\ff),q_\theta) \rtimes \Omega' 
(\emptyset,\theta_\ff)$, which is given in terms of the Iwahori--Matsumoto presentation, can
also be written in terms of the Bernstein presentation, as in \cite[\S 3]{Lus-Gr} or
\cite[\S 1]{SolSurv}. This gives $q$-parameters $q^*_{\theta,\alpha}$ for all
$\alpha \in \Delta_{\ff,\sigma}$.

\begin{prop}\label{prop:9.3}
Let $\alpha \in \Delta_{\ff,\sigma}$. The parameter $q_{m_\alpha \alpha^\vee}$ for 
$s_{\alpha^\vee}$ in $\mc H (\mf s^\vee ,q_F^{1/2})$ is equal to $q_{\theta,\alpha}$. Furthermore,  
$q^*_{m_\alpha \alpha^\vee}$ is equal to $q^*_{\theta,\alpha}$.
\end{prop}
\begin{proof}
For the construction of $q_{m_\alpha \alpha^\vee}$ \cite[Proposition 3.14]{AMS3},
one first passes to the complex reductive group 
$J_\varphi := Z^1_{{G^\vee}_\Sc} (\varphi |_{\mb I_F})$. Next, for each 
$z \in \mathfrak{X}_{\nr} (L)$, as in \cite[\S3.1]{AMS3}, one constructs  
a graded Hecke algebra from $z \varphi$, $\rho$ and $G_{z\varphi} := Z^1_{{G^\vee}_\Sc}(z \varphi |_{\mb W_F})$.~This gives a family of parameters $\mb k_{\alpha^\vee, z \varphi}$, from which 
$q_{m_\alpha \alpha^\vee}$ and $q^*_{m_\alpha \alpha^\vee}$ are obtained (see \cite[\S 3.2]{AMS3}). 
One does not need the full $J_\varphi$, it suffices to consider a Levi subgroup of $J_\varphi^\circ$ containing 
$\varphi (\SL_2 (\C))$ and all the root subgroups from $\R^\times \alpha^\vee$.

Let $k_\alpha$ be as in Proposition \ref{prop:4.1}. Suppose that 
\begin{equation}\label{eq:9.10}
\theta_\ff \circ N_{k_\alpha / k_F} \circ \alpha^\vee \neq 1.    
\end{equation}
By Proposition \ref{prop:4.1}, $q_{\theta,\alpha} = 1$, which means that 
$q^*_{\theta,\alpha} = 1$ as well because $q^*_{\theta,\alpha} \in [1,q_{\theta,\alpha}]$. 
Let $F_\alpha / F$ be the unramified extension of $F$ with residue field $k_\alpha$. 
Via the local Langlands correspondence for tori, $(\theta_\ff, \mc T_\ff)$ corresponds to 
$\varphi |_{\mb I_F}$, because $\mc T_\ff (F)$ is the maximal compact subtorus of $\mc T (F)$. 
The condition \eqref{eq:9.10} is equivalent to $\alpha^\vee (\varphi (\mb I_{F_\alpha})) \neq 1$.~By construction, $\mb I_{F_\alpha}$ stabilizes every root subgroup of $G^\vee$ associated
to a root in $\R^\times \alpha^\vee$. By the description  of centralizers of semisimple elements 
(here of $\varphi (\mb I_{F_\alpha})$, a finite set) from \cite{Ste},
$\alpha^\vee (\varphi (\mb I_{F_\alpha})) \neq 1$ implies that $J_\varphi^\circ$ does not 
contain any representatives for 
$s_{\alpha^\vee}$. But then $\alpha^\vee$ does not correspond to a root for the graded 
Hecke algebras from \cite[\S 3.1]{AMS3}. Thus $s_{\alpha^\vee}$ only occurs in the 
R-groups/$\Gamma$-groups for those graded Hecke algebras. This implies that 
$\mb k_{\alpha,z \varphi} = 0$ for all $\chi \in \mathfrak{X}_{\nr} (L)$, and thus by
\cite[Proposition 3.14]{AMS3}, we have $q_{\alpha^\vee} = q^*_{\alpha^\vee} = 1$.

Suppose now that, in contrast to \eqref{eq:9.10}, $\theta_\ff \circ N_{k_\alpha / k_F} \circ \alpha^\vee = 1$.~For any lift $\beta^\vee \in R(G_\alpha^\vee,T^\vee)$, we have 
$\beta^\vee (\varphi (\mb I_{F,\beta^\vee})) = 1$. Set $U_{\mb I_F \beta^\vee} := 
\prod\nolimits_{\gamma \in \mb I_F / \mb I_{F,\beta^\vee}} \, U_{\gamma \beta^\vee}$.~The $\varphi (\mb I_F)$-invariants in this group can be identified with
\[
Z_{U_{\mb I_F \beta^\vee}} (\varphi (\mb I_F)) \cong 
Z_{U_{\beta^\vee}} (\varphi (\mb I_{F,\beta^\vee})) = U_{\beta^\vee} .
\]
Together with $Z_{U_{-\mb I_F \beta^\vee}} (\varphi (\mb I_F)) \cong U_{-\beta^\vee}$, this
allows us to construct a representative for $s_{\alpha^\vee}$ in $J_\varphi^\circ$. Starting
with $G_\alpha^\vee$ instead of $G^\vee$ gives a Levi subgroup $J_{\varphi,\alpha}^\circ$ of 
$J_\varphi^\circ$ containing $U_{\mb I_F ,\beta^\vee} \cap J_\varphi^\circ$. As explained at 
the start of the proof, this means that the parameters $q_{m_\alpha \alpha^\vee}$ and 
$q^*_{m_\alpha \alpha^\vee}$ can be computed just as well from ${}^L G_\alpha$.

In summary, on the $p$-adic side, Proposition \ref{prop:4.8} and Corollary \ref{cor:4.2}
reduce the computations of $q_{\theta,\alpha}$ and $q^*_{\theta,\alpha}$ to the Hecke
algebra $\mc H (G_\alpha, P_{G_\alpha,\ff}, \theta_\ff)$ for a Bernstein block in the 
principal series of a quasi-split reductive group $G_\alpha$.~On the Galois side, we 
reduced $q_{m_\alpha \alpha^\vee}$ and $q^*_{m_\alpha \alpha^\vee}$ to parameters for 
a Hecke algebra of the form $\mc H (\mf s^\vee,q_F^{1/2})$, computed from ${}^L G_\alpha$ 
instead of ${}^L G$. Thus we may apply the known results about principal series 
representations of quasi-split groups, where the desired equality of $q$-parameters follows 
from \cite[Lemma 5.2]{SolQS}.
\end{proof}
\noindent Recall from \cite[Proposition 6.9]{Mor1} that 
$R_\sigma = \{ \alpha \in \Delta_{\ff,\sigma} : q_\sigma (v(\alpha,J)) =
q_{\theta,\alpha} > 1 \}$. Similarly, by \cite[Proposition 3.14]{AMS3}, $R_{\mf s^\vee} = \{ m_\alpha \alpha^\vee : \alpha^\vee \in \Delta_{\ff,\sigma}^\vee, 
q_{m_\alpha \alpha^\vee} > 1 \}$.~Thus Proposition \ref{prop:9.3} produces a canonical bijection
\begin{equation}\label{eq:9.13}
R_\sigma \longleftrightarrow R_{\mf s^\vee},    
\end{equation}
which gives a group isomorphism
\begin{equation}\label{eq:9.14}
W(R_\sigma) \cong W (R_{\mf s^\vee}) . 
\end{equation}
Let $R_{\mf s^\vee}^+$ denote the image of $R_\sigma^+$ under \eqref{eq:9.13}, then it induces a group isomorphism
\begin{equation}\label{eq:9.15}
W_{\mf s} / W(R_\sigma) \cong \Gamma_{\hat \sigma} \cong \Gamma_{\mf s^\vee} \cong
W_{\mf s^\vee} / W(R_{\mf s^\vee}) .
\end{equation}
Recall the affine Hecke algebra $\mc H (G, \hat P_\ff, \hat \sigma )^\circ$ from 
\eqref{eq:5.17}.

\begin{prop}\label{prop:9.4}
Lemma \ref{lem:9.2} and Proposition \ref{prop:9.3} induce an algebra isomorphism
\[
\mc H (\mf s^\vee, q_F^{1/2})^\circ \overset{\sim}{\longrightarrow}  
\mc H (G, \hat P_\ff, \hat \sigma )^\circ .
\]
It is canonical up to: 
\begin{itemize}
\item inner automorphisms that fix $\C [ZW_L  (J,\hat \sigma)]$ pointwise;
\item for each short simple root $\alpha \in R_\sigma$ satisfying $q^*_{\theta,\alpha} = 1$,
$T_{s_\alpha}$ can be replaced with $h_\alpha^\vee \, T_{s_\alpha}$ where 
$h_\alpha^\vee \in R_\sigma^\vee \subset ZW_L  (J,\hat \sigma)$.
\end{itemize}
\end{prop}
\begin{proof}
On the maximal commutative subalgebras, this isomorphism is given by \eqref{eq:9.3} and
\eqref{eq:9.4}, which are canonical.~By construction, the bijection from Lemma \ref{lem:9.2} sends $T_{s_\alpha}$ to 
$T_{s_{\alpha^\vee}}$ whenever simple roots $\alpha$ and $\alpha^\vee$ match via Lemma 
\ref{lem:9.1} and \eqref{eq:9.13}.~By Proposition \ref{prop:9.3} and the multiplication rules in 
Iwahori--Hecke algebras, the linear map 
$\mc H (\mf s^\vee, q_F^{1/2})^\circ \rightarrow 
\mc H (G,\hat P_\ff, \hat \sigma )^\circ$ from Lemma \ref{lem:9.2} is in fact an algebra isomorphism. The non-canonicity of this isomorphism is limited to automorphisms of 
$\mc H (G,\hat P_\ff, \hat \sigma )^\circ$ (or equivalently of 
$\mc H (\mf s^\vee, q_F^{1/2})^\circ$) that respect the properties used in the above
construction, i.e. automorphisms of $\mc H (G,\hat P_\ff, \hat \sigma )^\circ$ 
which are the identity on $\mc O (\Irr (ZW_L  (J,\hat \sigma))) = \C [ZW_L  (J,\hat \sigma)]$. 
Such automorphisms were classified in \cite[Theorem 3.3 and its proof]{AMS4}. Indeed,
conjugation by any element of 
\[
\C [ZW_L (J,\hat \sigma)]^\times = \C^\times \times ZW_L (J,\hat \sigma)
\]
is possible, these are the relevant inner automorphisms of 
$\mc H (G,\hat P_\ff, \hat \sigma )^\circ$. Apart from this, there is at most one nontrivial
possibility for each irreducible component $R_{\sigma,i}$ of the root system $R_\sigma$. 
This occurs only when $R_{\sigma,i}$ has type $B_n$ and $q_{\theta,\alpha}^* = 1$ for the 
unique short simple root $\alpha \in R_{\sigma,i}$. Then there is an automorphism such that: (1) $T_{s_\alpha}$ is mapped to $h_\alpha^\vee T_{s_\alpha}$; and 
(2) $T_{s_\beta}$ is fixed for all other simple roots $\beta \in R_\sigma$. 
We remark that this automorphism could be inner, e.g.~when 
$h_\alpha^\vee / 2 \in ZW_L (J,\hat \sigma)$.
\end{proof}

\subsection{Comparison of 2-cocycles} \label{par:cocycle} \

We now study how the 2-cocycle $\natural_{\mf s^\vee}$ of $\Gamma_{\mf s^\vee}$ corresponds,
via the isomorphism \eqref{eq:9.15}, to a 2-cocycle of $\Gamma_{\hat \sigma}$
coming from $\End_G (\Pi_{\mf s})$. 

Since it is quite difficult to analyze $\natural_{\mf s^\vee}$ (inflated to $W_{\mf s^\vee}$) for 
elements that do not fix points of $\mf s_L^\vee$, we shall fix \label{i:151}
$(\varphi_b,\rho_b) \in \mf s_L^\vee$, and we restrict our attention to $W_{\mf s^\vee,\varphi_b}$
(the stabilizer of $(\varphi_b,\rho_b)$ in $W(G^\vee,L^\vee)^{\mb W_F}$).
To classify all irreducible representations of $\mc H (\mf s^\vee, q_F^{1/2})$, it suffices 
to consider the cases with $\varphi_b$ bounded; see \cite[\S 2]{AMS3}. 

Recall from \cite[\S 3.3.1]{Hai} that there is a natural isomorphism
\begin{equation}\label{eq:1.14}
\mf{X}_\nr (L) \cong H^1 \big( \mb W_F / \mb I_F, Z(L^\vee)^{\mb I_F} \big)^\circ 
\cong \big( Z(L^\vee)^{\mb I_F} \big)_{\mb W_F}^{\; \circ} 
\end{equation}
We write 
\[
\mathfrak{X}_{\nr} (L)^+ := \Hom (L,\R_{>0}) 
\; \subset \; \mathfrak{X}_{\nr} (L) \; \subset \; \Xo (L) 
\]
and we let $\mathfrak{X}_{\nr} (L^\vee)^+$\label{i:152}
be its image in $(Z(L^\vee)^{\mb I_F})_{\mb W_F}^{\;\circ} \subset \Xo (L^\vee)$ under \eqref{eq:1.14}. 
To analyze representations of $\mc H (\mf s^\vee, q_F^{1/2})$ with a central character in 
$W_{\mf s^\vee} \mathfrak{X}_{\nr} (L^\vee)^+ \varphi_b$, one can localize the algebra with 
respect to $\mathfrak{X}_{\nr} (L^\vee)^+ \varphi_b$. The proof of \cite[Theorem 3.18]{AMS3} 
shows that this 
localization can be described by a twisted graded Hecke algebra $\mh H (\varphi_b, v=1, 
\rho_b, \vec{\mb r})$, as in \cite[\S 4]{AMS2} and \cite[(3.9)]{AMS3}. This algebra
contains a twisted group algebra $\C [W_{\mf s^\vee, \varphi_b}, \natural_{\mf s^\vee}]$,
which enables us to study $\natural_{\mf s^\vee} |_{(W_{\mf s^\vee,\varphi_b})^2}$ via the 
description in \cite[Lemma 5.4]{AMS1} and \cite[(89)]{AMS2}, where 
$\C [W_{\mf s^\vee, \varphi_b}, \natural_{\mf s^\vee}]$ is obtained as the endomorphism
algebra of a certain equivariant local system determined by $(\varphi_b,\rho_b)$. 

We need to modify the setup in \cite{AMS1,AMS2,AMS3} from inner twists
of $p$-adic groups to rigid inner twists. The definition of the cuspidal support map for
enhanced L-parameters in this setting can be found in \cite[\S 7]{SolRamif} and
\cite{DiSc}; for the other parts of \cite{AMS1,AMS2,AMS3}, there is hardly any difference.
Let us work out the aforementioned local systems in our case. The enhancement $\rho_b$
can be viewed as a $S_{\varphi_b}^+$-equivariant local system on $\{0\}$ and (by pullback) on
\[
\mr{Lie} \big( Z(L^\vee)^{\mb W_F} \big) = \mr{Lie}(Z_{L^\vee}(\varphi_b)) = 
\mr{Lie}(S_{\varphi_b}^+) .
\]
Let $G_{\varphi_b}^{\vee,+}$ be $S_{\varphi_b}^+$ for $\varphi_b$ viewed as element of 
$\Phi (G)$. Then $S_{\varphi_b}^+$ is a quasi-Levi subgroup (i.e. the centralizer of the 
connected centre of a Levi subgroup) of $G_{\varphi_b}^{\vee,+}$.
We pick a parabolic subgroup $P^{\vee,\circ}$ of $(G_{\varphi_b}^{\vee,+})^\circ$ with
Levi factor $(S_{\varphi_b}^+)^\circ$, and we write $P^\vee := P^{\vee,\circ} S_{\varphi_b}^+$.
Consider the maps 
\begin{multline}
\{0\} \xleftarrow{f_1} \big\{ (x,g) \in \mr{Lie}(G_{\varphi_b}^{\vee,+}) \times 
G_{\varphi_b}^{\vee,+} : \Ad (g^{-1}) x \in \mr{Lie}(P^\vee) \big\} \xrightarrow{f_2} \\ 
\big\{ (x,g P^\vee) \in \mr{Lie}(G_{\varphi_b}^{\vee,+}) \times G_{\varphi_b}^{\vee,+} / P^\vee : 
\Ad (g^{-1}) x \in \mr{Lie}(P^\vee) \big\} \xrightarrow{f_3} \mr{Lie}(G_{\varphi_b}^{\vee,+})
\end{multline}
where $f_2 (x,g) = (x,gP^\vee)$ and $f_3 (x, gP^\vee) = x$.
Let $\dot{\rho_b}$ be the unique $G_{\varphi_b}^{\vee,+} $-equivariant local system on 
\[
\big\{ (x,g P^\vee) \in \mr{Lie}(G_{\varphi_b}^{\vee,+}) \times G_{\varphi_b}^{\vee,+} / 
P^\vee : \Ad (g^{-1}) x \in \mr{Lie}(P^\vee) \big\}
\]
such that $f_2^* \dot{\rho_b} = f_1^* \rho_b$. The map
\[
f_3 : \big\{ (x,g P^\vee) \in \mr{Lie}(G_{\varphi_b}^{\vee,+})_{\mathrm{rss}} \times G_{\varphi_b}^{\vee,+} / 
P^\vee : \Ad (g^{-1}) x \in \mr{Lie}(P^\vee) \big \} \to \mr{Lie}(G_{\varphi_b}^{\vee,+})_{\mathrm{rss}}
\]
restricted to regular semisimple elements\footnote{indicated by a subscript rss} is a fibration with 
fibre $N_{G_{\varphi_b}^{\vee,+}} (S_{\varphi_b}^+) / S_{\varphi_b}^+$. If $(\dot{\rho_b})_{\mathrm{rss}}$
denotes the restriction of $\dot{\rho_b}$ to the regular semisimple locus, then 
$f_{3,!} (\dot{\rho_b})_{\mathrm{rss}}$ is a local system on 
$\mr{Lie}(G_{\varphi_b}^{\vee,+})_{\mathrm{rss}}$. By \cite[Lemma 5.4]{AMS1}, we have
\begin{equation}\label{eq:9.16}
\C [W_{\mf s^\vee,\varphi_b}, \natural_{\mf s^\vee}] \cong \End (f_{3,!} (\dot{\rho_b})_{\mathrm{rss}}),
\end{equation}
where the endomorphisms are taken in the category of $G_{\varphi_b}^{\vee,+}$-equivariant
local systems on $\mr{Lie}(G_{\varphi_b}^{\vee,+})_{\mathrm{rss}}$. The proof of \cite[Lemma 5.4]{AMS1}
uses that of \cite[Proposition 4.5]{AMS1} and \cite[\S 2]{Lus-Int}. There it is
shown that $\End (f_{3,!} (\dot{\rho_b})_{\mathrm{rss}})$ is canonically a direct sum of 
one-dimensional linear subspaces $\mc A_w$, indexed by $w \in W_{\mf s^\vee,\varphi_b}$. By
\cite[(45)]{AMS1}, an element of $\mc A_w$ corresponds to a family $\mc A_{\tilde w}$ of morphisms of 
$S_{\varphi_b}^+$-equivariant local systems on $\mr{Lie}(S_{\varphi_b}^+)_{\mathrm{rss}}$ as follows:
\begin{equation}\label{eq:9.17} 
\mc A_{\tilde w} : \rho_b \to \tilde w^{-1} \cdot \rho_b \text{ for all } \tilde w \in 
N_{G_{\varphi_b}^{\vee,+}} (S_{\varphi_b}^+) \text{ representing } w \in W_{\mf s^\vee,\varphi_b},
\end{equation}
related by 
$\mc A_{\tilde w n} = \mc A_{\tilde w} \circ \text{(action of } n)$ for all $n \in S_{\varphi_b}^+$. The multiplication in\\ 
$\End (f_{3,!} (\dot{\rho_b})_{\mathrm{rss}})$ satisfies
$\mc A_{w_1} \cdot \mc A_{w_2} = \mc A_{w_1 w_2}$, so any choice of a nonzero 
element $A_w$ in each $\mc A_w$ determines a 2-cocycle $\natural_{\mf s^\vee}$
and an isomorphism \eqref{eq:9.16} by the relation
\begin{equation}\label{eq:9.18}
A_{w_1} A_{w_2} = \natural_{\mf s^\vee} (w_1,w_2) A_{w_1 w_2} .
\end{equation}
Multiplying $\varphi_b$ by $z \in \Xo (G^\vee)$, as in \eqref{eq:6.27}, is a
symmetry of the entire setup; in particular, one keeps the same $\rho_b$ and the same
$A_w$. Then \eqref{eq:9.18} shows that 
\begin{equation}\label{eq:9.57}
\natural_{z \mf s^\vee} \text{ can be chosen to be equal to } \natural_{\mf s^\vee} 
\text{ for } z \in \Xo (G^\vee) .
\end{equation}
Recall from \eqref{eq:6.3} that 
\begin{equation}\label{eq:9.29}
W_{\mf s^\vee,\varphi_b} \cong W(G^\vee,L^\vee)^{\mb W_F}_{\varphi_b,\rho_b} \cong
W(N_{G^\vee}(L^\vee),T^\vee)^{\mb W_F}_{\varphi_b,\rho_b} \big/ 
W(L^\vee,T^\vee)^{\mb W_F}_{\varphi_b,\rho_b} .
\end{equation}
Via the canonical bijection $\Irr (\mc E_\eta^{\varphi_T}, \mr{id}) \to \Irr (S_{\varphi_b}^+, \eta)$ from \eqref{eq:6.7}--\eqref{eq:6.8}, we can replace $\rho_b$ by a representation 
$\rho_\eta$ of $\mc E_\eta^{\varphi_T}$. The conjugation action of $N_{G^\vee}
(L^\vee),T^\vee)^{\mb W_F}_{\varphi_T,\eta}$ on $\mc E_\eta^{\varphi_T}$ is trivial 
on $T^{\vee,\mb W_F}$. Thus $w^{-1} \rho_\eta$ and $w^{-1} \rho_b$ are well-defined 
representations for $w \in W(N_{G^\vee}(L^\vee),T^\vee)^{\mb W_F}_{\varphi_b,\rho_b}$.

We now vary on \eqref{eq:9.17} by picking representatives 
$\tilde w \in N_{G_{\varphi_b}^{\vee,+}} (S_{\varphi_b}^+)$ and fixing
\begin{equation}\label{eq:9.32}
A_{\eta,\tilde w} : \rho_\eta \to \tilde w^{-1} \cdot \rho_\eta 
\end{equation}
for each $w \in W(N_{G^\vee}(L^\vee),T^\vee)^{\mb W_F}_{\varphi_b,\rho_b}$. We impose $A_{\eta, \tilde w t} = \mc A_{\eta, \tilde w} \rho_\eta (t) =
A_{\eta, \tilde w} \eta (t)$ for all $t \in \bar T^{\vee,+}$. In these terms, \eqref{eq:9.18} can be rewritten as
\begin{equation}\label{eq:9.20} 
A_{\eta,\tilde w_1} A_{\eta,\tilde w_2} = \natural_{\mf s^\vee} (w_1,w_2)
A_{\eta,\tilde w_1 \tilde w_2} = \natural_{\mf s^\vee} (w_1,w_2) 
A_{\eta, \widetilde{w_1 w_2}} \eta (\widetilde{w_1 w_2}^{-1} \tilde w_1 \tilde w_2) .
\end{equation}
Note that $w_1, w_2 \in W(N_{G^\vee}(L^\vee),T^\vee)^{\mb W_F}_{\varphi_b,\rho_b}$ in 
\eqref{eq:9.20}, in contrast to \eqref{eq:9.17} and \eqref{eq:9.18}. For suitable
choices of the $A_{\eta, \tilde w}$, the 2-cocycles $\natural_{\mf s^{\vee}}$ in
\eqref{eq:9.18} and \eqref{eq:9.20} coincide, while in general they are only cohomologous.
For book-keeping purposes, we introduce two further 2-cocycles of 
$W(N_{G^\vee}(L^\vee),T^\vee)^{\mb W_F}_{\varphi_b,\rho_b}$:
\begin{align*}
\natural_{\mf s^\vee,\rho_\eta}(w_1,w_2) := A_{\eta,\tilde w_1} A_{\eta,\tilde w_2} 
A_{\eta, \widetilde{w_1 w_2}}^{-1} \quad\text{and}\quad \natural_\eta (w_1,w_2) := \eta (\widetilde{w_1 w_2}^{-1} \tilde w_1 \tilde w_2) .
\end{align*}
We record that $\natural_\eta$ is the 2-cocycle associated to the extension obtained from
\[
1 \to \bar T^{\vee,+} \to N_{\bar G^\vee} (L^\vee,T^\vee)^+_{\eta,\varphi_b} \to 
W(G^\vee,T^\vee)^{\mb W_F}_{\eta,\varphi_b} \to 1
\]
by pushout along $\eta :  \bar T^{\vee,+} \to \C^\times$. 
Thus \eqref{eq:9.20} means $\natural_{\mf s^\vee,\rho_\eta} = 
\natural_{\mf s^\vee} \natural_\eta$, or equivalently,
\begin{equation}\label{eq:9.21}
\natural_{\mf s^\vee} = \natural_{\mf s^\vee,\rho_\eta} \natural_\eta^{-1} .
\end{equation}
Although, indeed, all these 2-cocycles depend on various choices of representatives, 
their cohomology classes are uniquely determined. 

Next we analyze the relevant 2-cocycles of $\Gamma_{\hat \sigma}$, which is hard to do for 
elements of $\Gamma_{\hat \sigma}$ that do not fix any object of $\Irr (L)_{\mf s_L}$.~Thus 
we focus on $\tau \in \Irr (L)_{\mf s_L}$ corresponding to the above $(\varphi_b,\rho_b)$ via
Theorem \ref{thm:8.2}.~By Proposition \ref{prop:8.7}, it is tempered and unitary.~By 
Theorem \ref{thm:8.2} and Lemma \ref{lem:9.1}, we obtain a canonical isomorphism
\begin{equation}\label{eq:9.33}
W_{\mf s,\tau} \cong W_{\mf s^\vee, \varphi_b} .     
\end{equation}
Recall from \cite[Theorem 6.11]{SolEnd} that a suitably localized version of 
$\End_G (\Pi_{\mf s})$ contains a twisted group algebra 
\begin{equation}\label{eq:9.34}
\C [\Gamma_{\hat \sigma, \tau}, \natural_\tau] .    
\end{equation}
We may inflate $\natural_\tau$\label{i:165} from $\Gamma_{\hat \sigma, \tau} \cong
W_{\mf s,\tau} / W(R_{\sigma,\tau})$ to a 2-cocycle of $W_{\mf s,\tau}$.~By \cite[(4.13) 
and proof of Proposition 5.12.a]{SolEnd}, $\natural_\tau$ can be constructed via intertwiners 
of $L$-representations
\begin{equation}\label{eq:9.22}
\dot w \cdot \tau \to \tau \text{ for } \dot w \in N_G (L) \text{ representing } 
w \in W_{\mf s,\tau}.
\end{equation}
This is quite similar to how $\mu_{\hat \sigma}$ is defined in 
\eqref{eq:5.38}--\eqref{eq:5.39}.~Unfortunately, in general (i.e.~when
$L^2_\tau \neq L^3_\tau \neq L^4_\tau$), it is difficult to formulate the link 
between $\natural_\tau$ and $\mu_{\hat \sigma}$ precisely. The 2-cocycle $\natural_\tau$ can be described further with the construction
of $\tau$ \`a la Deligne--Lusztig. Let $(jT,\theta,\rho)$ be the datum corresponding
to $(\varphi_b,\rho_b)$ via Theorem \ref{thm:8.2}. Recall from \eqref{eq:7.21} that 
$\tau = \kappa_{jT, \theta, \rho}^{L,\epsilon} =
\big( \rho \otimes \ind_{L_\ff}^L \mr{inf}_{\mc L_\ff (k_F)}^{L_\ff} (\pm 
\mc R_{j \mc T (k_F)}^{\mc L_\ff (k_F)} (\theta))^\epsilon \big)^{N_L (jT)_\theta}$. We also recall from \eqref{eq:9.30} that
\begin{equation}\label{eq:9.31}
W_{\mf s,\tau} \cong W(N_G (L),jT)_{\theta,\rho} / W(L,jT)_{\theta,\rho} .
\end{equation}
For $g \in N_G (L,jT)_{\theta,\rho} \subset G_\ff$ representing a 
$w \in W_{\mf s,\tau}$, we recall the isomorphism $g \cdot \kappa_{jT, \theta, \rho}^{L,\epsilon} \cong 
\kappa_{jT, \theta, g \cdot \rho}^{L,\epsilon}$ from \eqref{eq:7.28}--\eqref{eq:7.26}. It was canonical up to the choice of $\epsilon$,
but meanwhile $\epsilon$ has been fixed in Theorem \ref{thm:8.2}. Thus the choice of
an isomorphism as in \eqref{eq:9.22} boils down to the choice of an isomorphism of
$N_L (jT)_\theta$-representations 
\begin{equation}\label{eq:9.23}
g \cdot \rho \to \rho .     
\end{equation}
Recall the canonical bijection $\Irr (\mc E_\theta^{[x]},\mr{id}) \to \Irr (N_L (jT)_\theta, \theta)$ from \eqref{eq:7.19}--\eqref{eq:7.20}. We denote the preimage of $\rho$ by 
$\rho^{[x]} \in \Irr (\mc E_\theta^{[x]})$. Then \eqref{eq:9.23} is
equivalent to the choice of an isomorphism $B_g : g \cdot \rho^{[x]} \to \rho^{[x]}$ of 
$\mc E_\theta^{[x]}$-representations, for $g \in N_G (L,jT)_{\theta,\rho}$. We may assume that 
\begin{equation}\label{eq:9.26}
B_{gl} = B_g \circ \rho^{[x]}(l) \text{ for all } l \in N_L (jT)_\theta .
\end{equation}
In these terms, $\natural_\tau$ is given by
\begin{equation}\label{eq:9.25}
B_{g_1} B_{g_2} = \natural_\tau (w_1,w_2) B_{g_1 g_2}
\end{equation}
for $g_i$ representing $w_i \in W_{\mf s,\tau}$. For any $\chi \in \Xo (G)$, we have
\[
\Hom_{N_L (jT)_{\chi \otimes \theta}} (g \cdot (\chi \otimes \rho), \chi \otimes \rho) 
= \Hom_{N_L (jT)_\theta} (\chi \otimes g \cdot \rho, \chi \otimes \rho)
= \Hom_{N_L (jT)_\theta} (g \cdot \rho, \rho) .
\]
Hence we can use the same $B_g$ for $\chi \otimes \theta$ and for $\theta$. Knowing
this, \eqref{eq:9.25} shows that
\begin{equation}\label{eq:9.58}
\natural_{\chi \otimes \tau} \text{ can be chosen to be equal to } \natural_\tau 
\text{ for } \chi \in \Xo (G).
\end{equation}
Since the conjugation action of 
$jT$ on $\mc E_\theta^{[x]}$ is trivial, $g \cdot \rho^{[x]}$ is a well-defined
$\mc E_\theta^{[x]}$-representation for $g \in W(N_G (L),jT)_{\theta,\rho}$. 
We choose a set of representatives $\tilde w \in N_G (L,jT)_{\theta,\rho}$ for 
$W(N_G (L),jT)_{\theta,\rho}$, and we assume \eqref{eq:9.26} only for $l \in jT$.
(Thus we are implicitly inflating $\natural_\tau$ to $W(N_G (L),jT)_{\theta,\rho}$,
and we allow it to be replaced by a cohomologous 2-cocycle.)
Then for $w_1,w_2 \in W(N_G (L),jT)_{\theta,\rho}$, \eqref{eq:9.25}  becomes
\begin{equation}\label{eq:9.27}
\begin{aligned}
B_{\tilde{w_1}} B_{\tilde{w_2}} = \natural_\tau (w_1,w_2) B_{\tilde{w_1} \tilde{w_2}} 
& = \natural_\tau (w_1,w_2) B_{\widetilde{w_1 w_2}} \rho^{[x]} 
(\widetilde{w_1 w_2}^{-1} \tilde{w_1} \tilde{w_2}) \\
& = \natural_\tau (w_1,w_2) B_{\widetilde{w_1 w_2}} 
\theta (\widetilde{w_1 w_2}^{-1} \tilde{w_1} \tilde{w_2}) .
\end{aligned}
\end{equation}
Let us define two 2-cocycles of $W(N_G (L),jT)_{\theta,\rho}$ by
\begin{align*}
\natural_{\mf s, \rho^{[x]}} (w_1,w_2) := 
B_{\tilde{w_1}} B_{\tilde{w_2}} B_{\widetilde{w_1 w_2}}^{-1} \quad\text{and}\quad
\natural_\theta (w_1, w_2) := \theta (\widetilde{w_1 w_2}^{-1} \tilde{w_1} \tilde{w_2}) .
\end{align*}
Note that $\natural_\theta$ is the 2-cocycle associated to the extension 
$\mc E_{\theta,G}^{[x]}$ from \eqref{eq:7.35}--\eqref{eq:7.36}. 
Now \eqref{eq:9.27} gives  
$\natural_{\mf s, \rho^{[x]}} = \natural_\tau \natural_\theta$, or equivalently,
\begin{equation}\label{eq:9.28}
\natural_\tau = \natural_{\mf s, \rho^{[x]}} \natural_\theta^{-1} .    
\end{equation}
The natural isomorphism \eqref{eq:8.22} restricts to
\begin{equation}\label{eq:9.53}
W(N_{\mc G^\flat}(\mc L^\flat), \mc T^\flat)(F)_{[x],\theta} \cong
W(N_{G^\vee}(L^\vee), T^\vee )^{\mb W_F}_{\eta,\varphi_T} .
\end{equation}
This enables us to compare \eqref{eq:7.36} with \eqref{eq:6.26}. 

\begin{prop}\label{prop:9.5}
There exist isomorphisms of extensions of \eqref{eq:9.53} by $\C^\times$: 
\begin{equation*}
    \zeta_G^\rtimes : \mc E_{\theta,G}^{\rtimes [x]}  \xrightarrow{\sim} \mc E_{\eta,G}^{\rtimes \varphi_T},\quad \zeta_G^0 : \mc E_{\theta,G}^{0,[x]} \xrightarrow{\sim}  \mc E_{\eta,G}^{0,\varphi_T} \quad \text{and} \quad
B(\zeta_G^0, \zeta_G^\rtimes) :  \mc E_{\theta,G}^{[x]} \xrightarrow{\sim} \mc E_{\eta,G}^{\varphi_T},
\end{equation*}
which contain the similar isomorphisms without subscripts $G$ from \eqref{eq:8.19}. These 
isomorphisms do not change if we adjust both $\theta$ and $\varphi_T$ by an element of $\Xo (G)$.
\end{prop}
\begin{proof}
For $\zeta_G^\rtimes$, this follows from \cite[Proposition 8.1]{Kal4}, as in \eqref{eq:8.23}.
The isomorphism $\zeta_G^0$ exists because both source and target are split
by Propositions \ref{prop:7.3} and \ref{prop:6.6}. By Lemmas \ref{lem:7.2} and \ref{lem:6.5},
the Baer sum of $\zeta_G^0$ and $\zeta_G^\rtimes$ is the required isomorphism
$B(\zeta_G^0, \zeta_G^\rtimes)$. By \eqref{eq:9.57} and \eqref{eq:9.58}, we can make all the 
choices invariant under twisting by $\Xo (G) \cong \Xo (G^\vee)$.
\end{proof}

\noindent We are ready to complete the comparison of the 2-cocycles $\natural_{\mf s^\vee}$ and
$\natural_\tau$ on $W_{\mf s^\vee, \varphi_b} \cong W_{\mf s,\tau}$.~Recall that in the above
process we have already inflated these 2-cocycles to
\begin{equation}\label{eq:9.35}
W( N_{G^\vee}(L^\vee), T^\vee )^{\mb W_F}_{\eta,\varphi_b,\rho_b} \cong
W( N_{\mc G^\flat}(\mc L^\flat), \mc T^\flat) (F)_{[x],\theta,\rho} ,
\end{equation}
via \eqref{eq:9.29} and \eqref{eq:9.31}.

\begin{thm}\label{thm:9.6}
\enuma{
\item The following equalities hold in 
$H^2 \big( W( N_{G^\vee}(L^\vee), T^\vee )^{\mb W_F}_{\eta,\varphi_b,\rho_b}, \C^\times \big)$:
\[
\natural_{\mf s^\vee, \rho_\eta} = \natural_{\mf s,\rho^{[x]}} ,\quad
\natural_\eta = \natural_\theta ,\quad\text{and}\quad  \natural_{\mf s^\vee} = \natural_\tau .
\]
\item The 2-cocycles $\natural_{\mf s^\vee}$ and $\natural_\tau$ of
$W_{\mf s^\vee, \varphi_b} \cong W_{\mf s,\tau}$ are cohomologous.
}
\end{thm}
\begin{proof}
(a) The isomorphism $B(\zeta_G^0, \zeta_G^\rtimes) : \mc E_{\theta,G}^{[x]} \xrightarrow{\sim} 
\mc E_{\eta,G}^{\varphi_T}$ from Lemma \ref{lem:8.6}, translates $\rho_\eta$ into
$\rho^{[x]}$, because $\pi (\varphi_b,\rho_b) = \tau = \pi_{jT,\theta,\rho}^{L,\epsilon}$. By the $W( N_{G^\vee}(L^\vee), T^\vee )^{\mb W_F}_{\eta,\varphi_T}$-equivariance of
$B(\zeta_G^0, \zeta_G^\rtimes)$, the data for computing $\natural_{\mf s, \rho^{[x]}}$
match exactly with the data for computing $\natural_{\mf s^\vee, \rho_\eta}$. Hence
any choice of the $A_{\eta,\tilde w}$ in \eqref{eq:9.32} corresponds to a choice
of the $B_{\tilde w}$ in \eqref{eq:9.27}, and with these choices the 2-cocycles
$\natural_{\mf s^\vee, \rho_\eta}$ and $\natural_{\mf s, \rho^{[x]}}$ coincide.

The isomorphism $\mc E_{\eta,G}^{[x]} \xrightarrow{\sim} \mc E_{\eta,G}^{\varphi_T}$ from 
Proposition \ref{prop:9.5} shows that $\natural_\theta$ and $\natural_\eta$ are cohomologous.
~By \eqref{eq:9.28} and \eqref{eq:9.21}, we compute in  
$H^2 \big( W( N_{G^\vee}(L^\vee), T^\vee )^{\mb W_F}_{\eta,\varphi_b,\rho_b}, \C^\times \big)$:
\[
\natural_{\mf s^\vee} = \natural_{\mf s^\vee, \rho_\eta} \natural_\eta^{-1} =
\natural_{\mf s, \rho^{[x]}} \natural_\theta^{-1} = \natural_\tau .
\]
(b) By part (a), $\natural_{\mf s^\vee}$ and $\natural_\tau$ are cohomologous
2-cocycles of \eqref{eq:9.35}.~By construction, $\natural_{\mf s^\vee} \in 
H^2 \big( W( N_{G^\vee}(L^\vee), T^\vee )^{\mb W_F}_{\eta,\varphi_b,\rho_b}, \C^\times \big)$ arises by inflation from 
$\natural_{\mf s^\vee} \in H^2 (W_{\mf s^\vee,\varphi_b},\C^\times)$, and $\natural_\tau \in 
H^2 \big( W( N_{\mc G^\flat}(\mc L^\flat), \mc T^\flat) (F)_{[x],\theta,\rho} ,\C^\times \big)$ is inflated from $\natural_\tau \in H^2 (W_{\mf s,\tau}, \C^\times)$. Hence
$\natural_{\mf s^\vee}$ and $\natural_\tau$ are cohomologous 2-cocycles, via the
isomorphism \eqref{eq:9.33}.
\end{proof}
Recall that the root system $R_{\sigma,\tau}$ from \eqref{eq:5.35} consists
of roots $\alpha \in R_\sigma$ satisfying $s_\alpha (\tau) = \tau$. Via
\eqref{eq:9.13}, $R_{\sigma,\tau}$ corresponds to the root system\label{i:153}
\[
R_{\mf s^\vee,\varphi_b} := \{ \beta \in R_{\mf s^\vee} : s_\beta (\varphi_b) = \varphi_b \}.
\]
The set of positive roots $R_{\mf s^\vee,\varphi_b}^+ = R_{\mf s^\vee,\varphi_b} \cap R_{\mf s^\vee}^+$ gives rise to a decomposition\label{i:154}
\begin{equation}\label{eq:9.36}
W_{\mf s^\vee, \varphi_b} = W(R_{\mf s^\vee}, \varphi_b) \rtimes \Gamma_{\mf s^\vee,\varphi_b},
\quad\text{where}\quad\Gamma_{\mf s^\vee,\varphi_b} = 
\mr{Stab}_{W_{\mf s^\vee, \varphi_b}} (R_{\mf s^\vee,\varphi_b}^+) .
\end{equation}
This is compatible with \eqref{eq:5.36}, thus \eqref{eq:9.33} decomposes into
isomorphisms
\begin{equation}\label{eq:9.37}
W(R_{\sigma,\tau}) \cong W(R_{\mf s^\vee,\varphi_b}) \quad \text{and}
\quad \Gamma_{\hat \sigma,\tau} \cong \Gamma_{\mf s^\vee, \varphi_b} .
\end{equation}
Similar to \eqref{eq:9.34}, the 2-cocycle $\natural_{\mf s^\vee}$ is trivial
on $W(R_{\mf s^\vee,\varphi_b}) \subset W(R_{\mf s^\vee})$, so 
$\natural_{\mf s^\vee} |_{(W_{\mf s^\vee,\varphi_b})^2}$ factors through $(W_{\mf s^\vee,\varphi_b} / W(R_{\mf s^\vee,\varphi_b}) )^2 \cong
(\Gamma_{\mf s^\vee,\varphi_b} )^2$. As a direct consequence of Theorem \ref{thm:9.6} (b), the group isomorphisms 
\eqref{eq:9.37} can be lifted to algebra isomorphisms
\begin{equation}\label{eq:9.38}
\C [W_{\mf s^\vee,\varphi_b}, \natural_{\mf s^\vee}] \cong \C [W_{\mf s,\tau}, 
\natural_\tau] \quad \text{and} \quad \C [\Gamma_{\mf s^\vee,\varphi_b}, 
\natural_{\mf s^\vee}] \cong \C [\Gamma_{\hat \sigma,\tau}, \natural_\tau] .
\end{equation}
To a suitable localization of $\End_G (\Pi_{\mf s})$, in some sense
centering on $\tau$, one can associate a twisted graded Hecke algebra as
in \cite[\S 7]{SolEnd}, say 
$\mh H (\tilde{\mc R}_\tau, W_{\mf s,\tau}, k^\tau,\natural_\tau)$\label{i:155}.
Here $\tilde{\mc R}$ is a degenerate root datum involving the root system
$R_{\sigma,\tau}$ and the vector space \label{i:156}
\[
\mf t := \mr{Lie}(\mathfrak{X}_{\nr} (L)) = X^* (Z^\circ(L)) \otimes_\Z \C .
\]
By definition, $\mc O (\mf t)$ is a maximal commutative subalgebra,
$\C [W_{\mf s,\tau}, \natural_\tau]$ is a subalgebra and the multiplication map
\begin{equation}\label{eq:9.39}
\mc O (\mf t) \otimes \C [W_{\mf s,\tau}, \natural_\tau] \to 
\mh H (\tilde{\mc R}_\tau, W_{\mf s,\tau}, k^\tau,\natural_\tau)
\end{equation}
is a linear bijection.
Let $\mh H \big( \mf s^\vee, \varphi_b, \log (q_F^{1/2}) \big)$\label{i:158} 
be the twisted graded Hecke algebra obtained from $\mc H (\mf s^\vee, q_F^{1/2})$ 
via the reduction procedure from \cite{Lus-Gr} and \cite[\S 2.1]{SolAHA}, centred 
at $\varphi_b$. In the terminology of \cite[\S 3.1]{AMS3}, it can be written as
\begin{equation}\label{eq:9.41}
\mh H \big( \mf s^\vee, \varphi_b, \log (q_F^{1/2}) \big) =
\mh H (\varphi_b, \rho_b, \mb r) / (\mb r - \log (q_F^{1/2})) .
\end{equation}
Let $\mf l^\vee$\label{i:159} be the Lie algebra of $L^\vee$, so that
\[\mr{Lie} \big( {(Z(L^\vee)^{\mb I_F})_{\mb W_F}\!}^\circ \big) \cong Z(\mf l^\vee)^{\mb W_F}
= \big( X_* (Z^\circ (L^\vee)) \otimes_\Z \C \big)^{\mb W_F}. \]
By construction, $\mc O \big( Z(\mf l^\vee)^{\mb W_F} \big)$ is a maximal commutative 
subalgebra of \eqref{eq:9.41},\\ $\C [W_{\mf s^\vee,\varphi_b}, \natural_{\mf s^\vee}]$ 
is a subalgebra and the multiplication map
\begin{equation}\label{eq:9.40}
\mc O \big( Z(\mf l^\vee)^{\mb W_F} \big) \otimes \C [W_{\mf s^\vee,\varphi_b}, 
\natural_{\mf s^\vee}] \to \mh H \big( \mf s^\vee, \varphi_b, \log (q_F^{1/2}) \big)     
\end{equation}
is a linear bijection.

\begin{prop}\label{prop:9.7}
Proposition \ref{prop:9.4} and \eqref{eq:9.38} induce an algebra isomorphism
\[
\mh H \big( \mf s^\vee, \varphi_b, \log (q_F^{1/2}) \big) \overset{\sim}{\longrightarrow}
\mh H (\tilde{\mc R}_\tau, W_{\mf s,\tau}, k^\tau,\natural_\tau) .
\]
It is canonical up to:
\begin{itemize}
\item inner automorphisms that fix $\mc O (\mf t)$ pointwise;
\item twisting by characters of $W_{\mf s,\tau}$ that are trivial on the 
subgroup generated by the reflections $s_\alpha$ with $\alpha \in R_{\sigma,\tau}$
and $k^\tau_\alpha \neq 0$.
\end{itemize}
\end{prop}
\begin{proof}
Recall from \eqref{eq:5.7} and \eqref{eq:9.42} that there are finite coverings
\[
\begin{array}{lllllll}
\mathfrak{X}_{\nr} (L) & \to & \Irr (L)_{[L,\tau]} & : & 
\chi & \mapsto & \chi \otimes \tau ,\\
\big( Z(L^\vee)^{\mb I_F} \big)_{\mb W_F}^{\; \circ} & \to & \mf s_L^\vee & : &
z & \mapsto & (z \varphi_b , \rho_b).
\end{array}
\]
It follows that the tangent map of \eqref{eq:9.2} is a linear bijection
\begin{equation}\label{eq:9.45}
\mf t \to Z(\mf l^\vee)^{\mb W_F} ,
\end{equation}
which by Theorem \ref{thm:8.2} and Lemma \ref{lem:9.1} is equivariant for 
$W_{\mf s,\tau} \cong W_{\mf s^\vee,\varphi_b}$. In fact it comes from the
isomorphisms 
\[
X^* (Z(L)^\circ) \cong X^* (Z^\circ (\mc L))^{\mb W_F} \cong X_* (Z^\circ (L^\vee))^{\mb W_F}. 
\]
The map \eqref{eq:9.45} induces an algebra isomorphism
\begin{equation}\label{eq:9.43}
\mc O \big( Z(\mf l^\vee)^{\mb W_F} \big) \xrightarrow{\sim} \mc O (\mf t) .
\end{equation}
The map \eqref{eq:9.43} is induced just as well by \eqref{eq:9.3}, which is a part
of Proposition \ref{prop:9.4}. For the twisted group algebras in \eqref{eq:9.39}
and \eqref{eq:9.40}, we take the first isomorphism in \eqref{eq:9.38}. Note that
the restricted isomorphism $\C [W(R_{\mf s^\vee,\varphi_b})] \cong \C [W(R_{\sigma,\tau})]$ is also induced by Proposition \ref{prop:9.4}, via \cite{Lus-Gr} and \cite[\S 2.1]{SolAHA}. By \eqref{eq:9.43}, \eqref{eq:9.38}, \eqref{eq:9.39} 
and \eqref{eq:9.40}, we obtain a linear bijection
\begin{equation}\label{eq:9.46}
\mh H \big( \mf s^\vee, \varphi_b, \log (q_F^{1/2}) \big) \longrightarrow
\mh H (\tilde{\mc R}_\tau, W_{\mf s,\tau}, k^\tau,\natural_\tau) .
\end{equation}
To guarantee that this is an algebra isomorphism, it remains to check that the
parameters for the roots on both sides agree under the bijection 
$R_{\sigma,\tau} \leftrightarrow R_{\mf s^\vee, \varphi_b}$ from \eqref{eq:9.13}.
By \cite[Proposition 3.14.a]{AMS3}, the parameters of  
$\mh H \big( \mf s^\vee, \varphi_b, \log (q_F^{1/2}) \big)$ are obtained from the
parameters of $\mc H (\mf s^\vee, q_F^{1/2})^\circ$ via the method of
\cite[Theorems 2.5 and 2.11 and (2.19)]{AMS3}, or equivalently via 
\cite[Theorems 8.6 and 9.3]{Lus-Gr}.~In Proposition \ref{prop:9.3}, we showed that the 
parameters of $\mc H (\mf s^\vee, q_F^{1/2})^\circ$ match with those of 
$\mc H (G,\hat P_\ff,\hat \sigma)^\circ$. Since $\mc H (G,\hat P_\ff,\hat \sigma)^\circ
\subset \End_G (\Pi_{\mf s})$ by \eqref{eq:5.5}, all simple reflections in $W(R_\sigma)$ have the same
parameters ($q$ and $q^*$) in each of these three algebras. The parameters $k^\tau$ for
the roots in $\mh H (\tilde{\mc R}_\tau, W_{\mf s,\tau}, k^\tau,\natural_\tau)$
are defined in terms of the parameters for $\End_G (\Pi_{\mf s})$ in
\cite[\S 7 and (35)]{SolEnd}:
\begin{equation}\label{eq:9.47}
k^\tau_\alpha = \left\{ \begin{array}{cl}
\log (q_{\theta,\alpha})  & \text{if } X_\alpha (\tau) = 1 \\
\log (q_{\theta,\alpha}^*)  & \text{if } X_\alpha (\tau) = -1   
\end{array} \right. .
\end{equation}
By \cite[(95)]{SolEnd}, $q_{\theta,\alpha} q_{\theta,\alpha}^* = q_F^{\lambda (\alpha)}$ and $q_{\theta,\alpha} (q_{\theta,\alpha}^*)^{-1} = q_F^{\lambda^* (\alpha)}$. Thus \eqref{eq:9.47} gives
\begin{equation}\label{eq:9.44}
k^\tau_\alpha = 
\log (q_F^{1/2}) \big( \lambda (\alpha) + X_\alpha (\tau) \lambda^* (\alpha) \big). 
\end{equation}
If we set $\mb r = \log (q_F^{1/2})$ and replace $X_\alpha$ by $m_\alpha \alpha$ as
prescribed in \cite[Proposition 3.14]{AMS3}, then \eqref{eq:9.44} becomes 
\cite[(2.19)]{AMS3}. Hence $k^\tau_\alpha$ is also the parameter of $\alpha$ in 
$\mh H \big( \mf s^\vee, \varphi_b, \log (q_F^{1/2}) \big)$. The non-canonicity in \eqref{eq:9.46} comes from three sources: \\ 
(1) Algebra automorphisms of $\C [\Gamma_{\hat \sigma,\tau},\natural_\tau]$ that stabilize 
each line $\C \gamma$ for $\gamma \in \Gamma_{\hat \sigma,\tau}$. These are precisely 
the maps $\gamma \mapsto \chi (\gamma) \gamma$ where $\chi : \Gamma_{\hat \sigma,\tau} \to 
\C^\times$ is a character. On $\mh H (\tilde{\mc R}_\tau, W_{\mf s,\tau}, k^\tau,\natural_\tau)$, 
this means twisting by a character of $W_{\mf s,\tau} / W(R_{\sigma,\tau})$.\\
(2) The non-canonicity in Proposition \ref{prop:9.4}, in particular with respect to
inner automorphisms of $\mc H (G, \hat P_\ff, \hat \sigma)$ that restrict to the 
identity on $\C [ZW_L (J,\hat \sigma)]$. These account for inner automorphisms of
$\mh H (\tilde{\mc R}_\tau, W_{\mf s,\tau}, k^\tau,\natural_\tau)$ that are the 
identity on $\mc O (\mf t)$.\\
(3) Proposition \ref{prop:9.4} also allows for some adjustments for short simple roots
$\alpha \in R_{\sigma,\tau}$ satisfying $q_{\theta,\alpha}^* = 1$, i.e.~$T_{s_\alpha}$ may 
be replaced by $T_{s_\alpha} h_\alpha^\vee$ in $\mc H (G, \hat P_\ff, \hat \sigma)^\circ$. 
By \cite[(35)]{SolEnd}, this $h_\alpha^\vee$ corresponds to $X_\alpha$ in \eqref{eq:9.47}.
By \cite[Proposition 7.3 and its proof]{SolEnd}, we see that $T_{s_\alpha} \mapsto
T_{s_\alpha} h_\alpha^\vee$ translates to 
\begin{equation}\label{eq:9.48}
N_{s_\alpha} \mapsto X_\alpha (\tau) N_{s_\alpha} 
\text{ in } \mh H (\tilde{\mc R}_\tau, W_{\mf s,\tau}, k^\tau,\natural_\tau) .
\end{equation}
If $X_\alpha (\tau) = 1$, then this does nothing. On the other hand, if 
$X_\alpha (\tau) = -1$, then \eqref{eq:9.47} implies that $k^\tau_\alpha = 0$.
As mentioned in the proof of Proposition \ref{prop:9.4}, the operation 
$T_{s_\alpha} \mapsto T_{s_\alpha} h_\alpha^\vee$ fixes all generators $T_{s_\beta}$
where $s_\beta$ is not $W(R_\sigma)$-conjugate to $s_\alpha$. Hence \eqref{eq:9.48}
fixes all $N_{s_\beta}$ where $\beta \in R_{\sigma,\tau}$ and $k^\tau_\beta \neq 0$.
Consequently, \eqref{eq:9.48} gives rise to a character $\chi$ of $W_{\mf s,\tau}$ that 
is trivial on $\Gamma_{\hat \sigma,\tau}$ and on all such $s_\beta$, and the algebra
automorphism induced by \eqref{eq:9.48} is given by twisting by this $\chi$. \qedhere
\end{proof}
Let $R'_{\sigma,\tau}$ be the subset of $R_{\sigma,\tau}$ consisting of the
roots $\alpha$ with $k^\tau_\alpha \neq 0$. It is again a root system, with
positive roots $R^{'+}_{\sigma,\tau}$. This gives a decomposition
\[
W_{\mf s,\tau} = W(R'_{\sigma,\tau}) \rtimes \Gamma'_{\hat \sigma,\tau},
\]
where $\Gamma'_{\sigma,\tau}$ is the stabilizer of $R^{'+}_{\sigma,\tau}$. The
presentation of twisted graded Hecke algebras, as in \cite[Proposition 2.2]{AMS2},
shows that we can write
\begin{equation}\label{eq:9.55}
\mh H (\tilde{\mc R}_\tau, W_{\mf s,\tau}, k^\tau,\natural_\tau) =
\mh H (\tilde{\mc R}_\tau, W (R'_{\sigma,\tau}), k^\tau) \rtimes
\C [ \Gamma'_{\hat \sigma,\tau},\natural_\tau] .
\end{equation}
Proposition \ref{prop:9.7} allows us to transfer this decomposition to 
$\mh H \big( \mf s^\vee, \varphi_b, \log (q_F^{1/2}) \big)$. More precisely, let 
$R'_{\mf s^\vee,\varphi_b}$ be the subsystem of roots with nonzero parameters
and write
\[
W_{\mf s^\vee,\varphi_b} = W(R'_{\mf s^\vee,\varphi_b}) \rtimes \Gamma'_{\mf s^\vee,\varphi_b} .
\]
Let $\mh H \big( \mf s^\vee, \varphi_b, \log (q_F^{1/2}) \big)^\circ$ be the graded
Hecke algebra built from $Z(\mf l^\vee)^{\mb W_F}$, $R'_{\mf s^\vee,\varphi_b}$ and
the parameters for those roots in $\mh H \big( \mf s^\vee, \varphi_b, \log (q_F^{1/2}) \big)$.
Then
\begin{equation}\label{eq:9.56}
\mh H \big( \mf s^\vee, \varphi_b, \log (q_F^{1/2}) \big) =
\mh H \big( \mf s^\vee, \varphi_b, \log (q_F^{1/2}) \big)^\circ \rtimes 
\C[ \Gamma'_{\mf s^\vee,\varphi_b}, \natural_{\mf s^\vee}], 
\end{equation}
and Proposition \ref{prop:9.7} respects the decompositions \eqref{eq:9.55} and 
\eqref{eq:9.56}.

\section{Equivalences between module categories of Hecke algebras}
\label{sec:equiv}

Consider a type $(\hat P_\ff, \hat \sigma)$ for $G$ as in Theorem \ref{thm:3.6}, and recall
that it covers the type $(\hat P_{L,\ff}, \hat \sigma)$ for $L$. Let $\mf s$ be the
associated inertial equivalence class for $G$. By \cite{BuKu}, there is an equivalence of 
categories \label{i:35}
\begin{equation}\label{eq:10.1}
\Rep (G)_{\mf s} = \Rep (G)_{(\hat P_\ff, \hat \sigma)} 
\xrightarrow{\sim} 
\Mod \text{ - }\cH (G,\hat P_\ff, \hat \sigma) \; \text{given by} \;
\pi 
\mapsto  
\Hom_{\hat P_\ff} (\hat \sigma, \pi), 
\end{equation}
and likewise $\Rep (L)_{(\hat P_{L,\ff}, \hat \sigma)}  
\xrightarrow{\sim} 
\Mod \text{ - }\cH (L,\hat P_{L,\ff}, \hat \sigma)$ given by $\pi_L  
\mapsto 
\Hom_{\hat P_{L,\ff}} (\hat \sigma, \pi_L)$. Here Mod$\text{ - } \cH$ denotes the category of right $\mc H$-modules.~Recall from 
\eqref{eq:3.1} that $\cH (L,\hat P_{L,\ff}, \hat \sigma)$ embeds canonically
in $\cH (G,\hat P_\ff, \hat \sigma)$. By \cite[Lemma 4.1]{SolComp},
the supercuspidal support map $\Irr (G)_{\mf s} = \Irr (G)_{\hat P_\ff,\hat \sigma} \to 
\Irr (L)_{(\hat P_{L,\ff},\hat \sigma)} / W(G,L)_{\hat \sigma}$ translates via \eqref{eq:10.1} to the map 
\begin{equation}\label{eq:10.2}
\Irr \text{ - }\cH (G,\hat P_\ff, \hat \sigma) \longrightarrow
\Irr \text{ - }\cH (L,\hat P_{L,\ff}, \hat \sigma) / W(G,L)_{\hat \sigma},
\end{equation}
which sends an irreducible $\cH (G,\hat P_\ff, \hat \sigma)$-module $M$ to any 
irreducible $\cH (L,\hat P_{L,\ff}, \hat \sigma)$-subquotient of $M$. By Lemma 
\ref{lem:3.7} and \eqref{eq:5.1}, the map \eqref{eq:10.2} is well-defined.
The Bernstein presentation of $\cH (G,\hat P_\ff, \hat \sigma)^\circ$ shows that 
\eqref{eq:10.2} is essentially the central character map for 
$\cH (G,\hat P_\ff, \hat \sigma)$. 

Let $\tau \in \Irr (L)_{\mf s_L} = \Irr (L)_{(\hat P_{L,\ff},\hat \sigma)}$ be a unitary
non-singular supercuspidal representation of depth zero. Here we mean non-singularity
as in Section \ref{sec:DL}, based on a $F$-non-singular character of a torus and
slightly more restrictive than requiring $\sigma$ to be non-singular. Recall the group
$\mathfrak{X}_{\nr}^+ (L) = \Hom (L, \R_{>0})$
of positive unramified characters of $L$. Our LLC will run through the category $\Rep_\fl (G)_{\mathfrak{X}_{\nr}^+ (L) \tau}$, whose objects are all finite-length $G$-representations $\pi$ such that every irreducible
subquotient $\pi'$ of $\pi$ has supercuspidal support\footnote{Supercuspidal supports are 
only defined up to $G$-conjugacy, so strictly speaking we mean that Sc$(\pi')$ has a 
representative in $(L,\mathfrak{X}_{\nr}^+ (L) \tau)$.} in $(L,\mathfrak{X}_{\nr}^+ (L) \tau)$.
By convention, all our subcategories of $\Rep (G)$ will be full.

Let $(\mathfrak{X}_{\nr}^+ (L) \tau )_\cH$ be the subset of 
$\Irr \text{ - }\cH (L, \hat P_{L,\ff},\hat \sigma)$ corresponding to 
$\mathfrak{X}_{\nr}^+ (L) \tau$ via \eqref{eq:10.1} for $L$.~Define $\Mod_{\fl, (\mathfrak{X}_{\nr}^+ (L) \tau )_\cH} \text{ - }\mc H (G,\hat P_\ff, \hat \sigma)$ similarly (as $\Rep_\fl (G)_{\mathfrak{X}_{\nr}^+ (L) \tau}$), i.e.~its objects are the 
finite-length modules $M$ such that every irreducible subquotient of $M$ maps to 
$W(G,L)_{\hat \sigma} (\mathfrak{X}_{\nr}^+ (L) \tau )_\cH$ by \eqref{eq:10.2}. There is an equivalence of categories
\begin{equation}\label{eq:10.3}
\Rep (G)_{\mf s} 
\xrightarrow{\sim} 
\Mod \text{ - } \End_G (\Pi_{\mf s}),\quad 
\pi 
\mapsto 
\Hom_G (\Pi_{\mf s}, \pi).
\end{equation}
Here $\Pi_{\mf s} = \II_Q^G (\Pi_{\mf s}^L)$ for a progenerator $\Pi_{\mf s}^L$ of
$\Rep (L)_{\mf s_L}$, which gives the analogue of \eqref{eq:10.3} for $L$. Now 
$\mathfrak{X}_{\nr}^+ (L) \tau$ corresponds to a set of irreducible representations of 
$\End_L (\Pi_{\mf s}^L)$ that we denote $\mathfrak{X}_{\nr}^+ (L) \otimes \tau$. We define 
\begin{equation}\label{eq:10.27}
\Mod_{\fl, \mathfrak{X}_{\nr}^+ (L) \otimes \tau} \text{ - }\End_G (\Pi_{\mf s}) = 
\Mod_{\fl, W(G,L)_{\hat \sigma}(\mathfrak{X}_{\nr}^+ (L) \otimes \tau)} \text{ - }\End_G (\Pi_{\mf s}) 
\end{equation}
to be the category consisting of the finite-length modules $M$ such that every irreducible 
$\End_L (\Pi_{\mf s}^L)$-subquotient of $M$ belongs to 
$W(G,L)_{\hat \sigma} (\mathfrak{X}_{\nr}^+ (L) \otimes \tau)$. 

Recall the graded Hecke algebra $\mh H (\tilde{\mc R}_\tau, W_{\mf s,\tau}, 
k^\tau,\natural_\tau)$ from \eqref{eq:9.39}. We write\label{i:118}
\[
\mf t_\R = \mr{Lie}(\mathfrak{X}_{\nr}^+ (L)) = X^* (Z^\circ (L)) \otimes_\Z \R
\]
and let $\Mod_{\fl, \mf t_\R} \text{ - }\mh H (\tilde{\mc R}_\tau, W_{\mf s,\tau}, 
k^\tau,\natural_\tau)$ be the category whose objects are the finite-length modules $M$ such that, as an
$\mc O(\mf t)$-module, $M$ has all its irreducible subquotients in $\mf t_\R$.

\begin{prop}\label{prop:10.1}
The following categories are canonically equivalent:
\begin{enumerate}[(i)]
\item $\Rep_\fl (G)_{\mathfrak{X}_{\nr}^+ (L) \tau}$;
\item $\Mod_{\fl, (\mathfrak{X}_{\nr}^+ (L) \tau )_\cH} \text{ - }\cH (G,\hat P_\ff, \hat \sigma)$;
\item $\Mod_{\fl, \mathfrak{X}_{\nr}^+ (L) \otimes \tau} \text{ - }\End_G (\Pi_{\mf s})$;
\item $\Mod_{\fl, \mf t_\R} \text{ - }\mh H (\tilde{\mc R}_\tau, W_{\mf s,\tau}, 
k^\tau,\natural_\tau)$.
\end{enumerate}
These equivalences are compatible with parabolic induction and restriction.
\end{prop}
\begin{rem}\label{rem:10.8}
Here parabolic restriction from $\Rep (G)_{\mf s}$
to $\Rep (L)_{\mf s_L}$ means: Jacquet restriction with respect to
the parabolic subgroup of $G$ opposite to $(Q,L)$, followed by projection from
$\Rep (L)$ to the Bernstein block $\Rep (L)_{\mf s_L}$.
\end{rem}
\begin{proof}
The equivalence between (i) and (ii) follows directly from \eqref{eq:10.1}, 
\eqref{eq:10.2} and the definitions. It is compatible with parabolic induction
and restriction by \cite[Lemma 4.1]{SolComp}. The equivalences between (i), (iii) 
and (iv), as well as the compatibility with parabolic induction and restriction, 
follow from \cite[Corollary 8.1]{SolEnd}.
\end{proof}
\noindent In \eqref{eq:9.45}, the $\R$-linear subspace $\mf t_\R \subset \mf t$
corresponds to\label{i:163}
\[
Z(\mf l^\vee)^{\mb W_F}_\R := (X_* (Z^\circ (L^\vee)) \otimes_\Z \R )^{\mb W_F}
\; \subset Z(\mf l^\vee)^{\mb W_F}.
\]
We put
\begin{equation}\label{eq:10.16}
\mathfrak{X}_{\nr}^+ (L^\vee) := \exp \big( Z(\mf l^\vee)^{\mb W_F}_\R \big) \subset 
\big( {\big( (Z(L^\vee)^{\mb I_F} \big)_{\mb W_F}} \big)^{\!\circ} .
\end{equation}
Proposition \ref{prop:9.7} induces an equivalence of categories
\begin{equation}\label{eq:10.4}
\Mod_{\fl, \mf t_\R}\text{ - }\mh H (\tilde{\mc R}_\tau, W_{\mf s,\tau}, 
k^\tau,\natural_\tau) \cong \Mod_{\fl, Z(\mf l^\vee)^{\mb W_F}_\R}\text{ - }
\mh H \big( \mf s^\vee, \varphi_b, \log (q_F^{1/2}) \big) .
\end{equation}
By \eqref{eq:9.38} and \eqref{eq:9.43}, the isomorphism in Proposition 
\ref{prop:9.7} precisely matches the pa\-ra\-bolic subalgebras on both sides, so
\eqref{eq:10.4} commutes with parabolic induction and restriction.~Composing 
\eqref{eq:10.4} with (i)$\to$(iv) in Proposition \ref{prop:10.1}, 
we obtain an equivalence of categories 
\begin{equation}\label{eq:10.5}
\Rep_\fl (G)_{\mathfrak{X}_{\nr}^+ (L) \tau} \cong \Mod_{\fl, Z(\mf l^\vee)^{\mb W_F}_\R}\text{ - }
\mh H \big( \mf s^\vee, \varphi_b, \log (q_F^{1/2}) \big) ,
\end{equation}
which is again compatible with parabolic induction and restriction.
Given the algebras, the equivalences \eqref{eq:10.4} and \eqref{eq:10.5} are
canonical up to twisting by characters of $W_{\mf s,\tau}$, as described in
Proposition \ref{prop:9.7}. For the algebras in question, the only further 
choices are those of systems of positive roots, which are innocent.

However, there is another source of ambiguity: $\tau$ may be replaced by an 
$N_G (L)$-conjugate representation of $L$. Composing $\tau$ with conjugation
by elements of $L$ does not matter, so we are looking at $\bar w \cdot \tau$
with $\bar w \in N_G (L)$
representing $w \in W(G,L)$. Since the supercuspidal support of an irreducible 
$G$-representation is only defined up to $G$-conjugacy, 
$\Rep_\fl (G)_{\mathfrak{X}_{\nr}^+ (L) \tau}$ is equal to 
$\Rep_\fl (G)_{\mathfrak{X}_{\nr}^+ (L) \bar w \tau}$.

For elements of $W_{\mf s,\tau}$, this does not do anything. Therefore we may
adjust $w$ by an element of $W_{\mf s,\tau}$, and we may assume that 
\begin{equation}\label{eq:10.6}
w (R_{\sigma,\tau}^+) = R_{\bar w \sigma,\bar w \tau}^+ .
\end{equation}

\begin{prop}\label{prop:10.2}
Let $w \in W(G,L)$ be represented by $\bar w \in N_G (L)$ and satisfy
\eqref{eq:10.6}. Let $w^\vee$ be
the corresponding element of $W(G^\vee,L^\vee)^{\mb W_F}$. The diagram
\begin{equation}\label{diagram:prop:10.2}
    \begin{tikzcd}
        \Rep_\fl (G)_{\mathfrak{X}_{\nr}^+ (L) \tau} 
        \arrow[]{r}{\eqref{eq:10.5}}\arrow[]{d}{\mr{Ad}(\bar w)} & 
\Mod_{\fl, Z(\mf l^\vee)^{\mb W_F}_\R} -
\mh H \big( \mf s^\vee, \varphi_b, \log (q_F^{1/2}) \big)\arrow[]{d}{\mr{Ad}(w^\vee)} \\
\Rep_\fl (G)_{\mathfrak{X}_{\nr}^+ (L) \bar w \tau}  \arrow[]{r}{\eqref{eq:10.5}} & 
\Mod_{\fl, Z(\mf l^\vee)^{\mb W_F}_\R} -
\mh H \big( w^\vee \mf s^\vee, w^\vee \varphi_b, \log (q_F^{1/2}) \big) 
    \end{tikzcd}
\end{equation}
commutes, up to isomorphisms of representations of one algebra (resp.~group).
Here Ad$(w^\vee)$ is induced by the algebra isomorphism: for 
$f \in \mc O \big( Z(\mf l^\vee)^{\mb W_F} \big)$ and $v \in W_{\mf s^\vee,\varphi_b}$,
\begin{equation}\label{eq:10.8}
\mh H \big( \mf s^\vee, \varphi_b, \log (q_F^{1/2}) \big) 
\to  
\mh H \big( w^\vee \mf s^\vee, w^\vee \varphi_b, \log (q_F^{1/2}) \big),\; 
f N_v  
\mapsto 
(f \circ w^{\vee -1}) N_{w^\vee v w^{\vee -1}}.
\end{equation}
\end{prop}
\begin{proof}
Since we only have to consider $G$-representations up to isomorphism, the left 
hand side of the diagram reduces to the identity map. 
Condition \eqref{eq:10.6} implies that 
\begin{equation}\label{eq:10.19}
w^\vee (R^+_{\mf s^\vee, \varphi_b}) = R^+_{w^\vee \mf s^\vee, w^\vee \varphi_b}
\end{equation}
via \eqref{eq:9.14}. 
Therefore, \eqref{eq:10.8} is an algebra homomorphism (while bijectivity
is clear).

First we treat the case $w \in W (G,L)_{\hat \sigma}$. By \eqref{eq:9.30}, we can
represent $W(G,L)_{\hat \sigma}$ in $N_G (L,T)_{\hat \sigma}$, thus we may assume that
$\bar w \in N_G (L,T)_{\hat \sigma}$. We will use the notations for analytic 
localization as discussed around \eqref{eq:5.6}. In Proposition \ref{prop:10.1}.(iv), 
the identity Ad$(\bar w)$ on $\Rep_\fl (G)_{\mathfrak{X}_{\nr}^+ (L) \tau}$ corresponds to the 
composition of the canonical bijections 
\begin{equation}\label{eq:10.7}
\begin{aligned}
\Mod_{\fl, \mf t_\R}\text{ - }\mh H (\tilde{\mc R}_\tau, W_{\mf s,\tau}, 
k^\tau,\natural_\tau) & \to \Mod_{\fl, \mathfrak{X}_{\nr}^+ (L)} 1_{U_\tau} \text{ - }
\End_G (\Pi_{\mf s})^{an}_U 1_{U_\tau} \\
& \to \Mod_{\fl, W (L,\tau,\mathfrak{X}_{\nr} (L)) \mathfrak{X}_{\nr}^+ (L)}\text{ - }
\End_G (\Pi_{\mf s})^{an}_U \\
& \to \Mod_{\fl, w \mathfrak{X}_{\nr}^+ (L)}\text{ - }1_{w U_\tau} 
\End_G (\Pi_{\mf s})^{an}_U 1_{w U_\tau} \\
& \to \Mod_{\fl, \mf t_\R}\text{ - }\mh H (\tilde{\mc R}_{w \tau}, W_{w \mf s,w \tau}, 
k^{w\tau},\natural_{w\tau}) .
\end{aligned}
\end{equation}
The first and last maps in \eqref{eq:10.7} are induced by analytic localization,
see \cite[Lemma 7.2 and Proposition 7.3]{SolEnd}, so they do not change
anything on the level of modules up to isomorphism. The second and third maps in
\eqref{eq:10.7} follow from \cite[Lemmas 6.4 and 6.5]{SolEnd}. By the proof of 
\cite[Lemma 6.4]{SolEnd}, their composition 
\[
\Mod_{\fl, \mathfrak{X}_{\nr}^+ (L)} 1_{U_\tau} \End_G (\Pi_{\mf s})^{an}_U 1_{U_\tau} \to 
\Mod_{\fl, w \mathfrak{X}_{\nr}^+ (L)} 1_{w U_\tau} \End_G (\Pi_{\mf s})^{an}_U 1_{w U_\tau} 
\]
is given by $M \mapsto \mr{Ad}(\mc T_w) M$ with $\mc T_w$ as in \cite[\S 5.2]{SolEnd}.
By \cite[Proposition 7.3]{SolEnd} and the definition of the elements 
$A_r^\tau, A_v^\tau$ \cite[\S 6.1]{SolEnd} and $\mc T_v^\tau$ 
\cite[Lemma 6.10]{SolEnd}, 
\[
\mc T_w N_r^\tau N_v^\tau \mc T_w^{-1} = N_{w r w^{-1}}^{w \tau} N_{w v w^{-1}}^{w \tau} 
\in \mh H (\tilde{\mc R}_{w \tau}, W_{\mf s, w \tau}, k^{w \tau}, \natural_{w \tau})
\]
for all standard basis elements $N_r^\tau \in \C [\Gamma_{\mf s,\tau}, \natural_\tau]$
and $N_v^\tau \in \C [W(R_{\sigma,\tau})]$.\footnote{For $N_{w v w^{-1}}^{w \tau}$, this involves
a choice of normalization, but the freedom in that choice is equivalent to the freedom
we already had in defining $N_r^\tau$.} We conclude that the composition of the maps in
\eqref{eq:10.7} is given by push forward along the algebra isomorphism: 
for $f \in \mc O (\mf t)$ and $r v \in W_{\sigma,\tau}$,
\begin{equation}\label{eq:10.9}
\mh H (\tilde{\mc R}_\tau, W_{\mf s,\tau}, k^\tau,\natural_\tau)  
\to 
\mh H (\tilde{\mc R}_{w\tau}, W_{\mf s,w \tau}, k^{w\tau},\natural_{w\tau}),\;
f N_r N_v 
\mapsto 
(f \circ w^{-1}) N_{w r w^{-1}} N_{w v w^{-1}} 
.
\end{equation}
Next we need to transfer this along \eqref{eq:10.5} to the right-hand side of the diagram. 
Proposition \ref{prop:9.7} translates \eqref{eq:10.9} into
the algebra isomorphism \eqref{eq:10.8}, thus indeed the right-hand side of
the diagram is given by push-forward along \eqref{eq:10.9}. 

Let $w \in W(G,L) \setminus W(G,L)_{\hat \sigma}$.~Conjugation by
$\bar w$ induces an algebra isomorphism
\begin{equation}\label{eq:10.10}
\mr{Ad}(\bar w) : 
\cH (G,\hat P_\ff, \hat \sigma) 
\to 
\cH (G, \hat P_{\bar w \ff}, \bar w \hat \sigma),\; 
f 
\mapsto 
f \circ \mr{Ad}(\bar w)^{-1} = [g \mapsto f ({\bar w}^{-1} g \bar w)] . 
\end{equation}
It interacts with the left column of diagram \eqref{diagram:prop:10.2} as
\begin{equation*}
    \begin{tikzcd}
        \Rep_\fl (G)_{\mathfrak{X}_{\nr}^+ (L) \tau}  \arrow[]{r}{\eqref{eq:10.1}}\arrow[]{d}{\mr{Ad}(\bar w)} & 
\Mod_{\fl, (\mathfrak{X}_{\nr}^+ (L) \tau )_\cH} \text{ - }\cH (G,\hat P_\ff, \hat \sigma) \arrow[]{d}{\mr{Ad}(\bar w) }\\
\Rep_\fl (G)_{\mathfrak{X}_{\nr}^+ (L) \bar w \tau}  \arrow[]{r}{\eqref{eq:10.1}} & 
\Mod_{\fl, \bar w (\mathfrak{X}_{\nr}^+ (L) \tau )_\cH} \text{ - }\cH (G,\hat P_{\bar w \ff}, \bar w\hat \sigma)
    \end{tikzcd}
\end{equation*}
In terms of Theorem \ref{thm:3.1}, for a simple reflection $s_\alpha$, \eqref{eq:10.10} sends 
$T_{s_\alpha} \in \cH (G,\hat P_\ff, \hat \sigma)$ to 
$T_{s_{w (\alpha)}}$, where $s_{w(\alpha)}$ is a simple reflection in 
$W(w J, \bar w \hat \sigma)$ by \eqref{eq:10.6}. Similarly, Ad$(\bar w)$ sends a standard 
basis element $T_\gamma \in \cH (G,\hat P_\ff, \hat \sigma)$, where 
$\gamma \in \Omega (J,\hat \sigma)$, to 
$T_{w \gamma w^{-1}} \in \cH (G, \hat P_{\bar w \ff}, \bar w \hat \sigma)$.~(Note that
this imposes a normalization on $T_{w \gamma w^{-1}}$, just as we chose a normalization
of $T_\gamma$ in the proof of Theorem \ref{thm:3.1}.) It follows that on 
$\C [ L^3_\tau / L^1] \cong \mc O \big( \Irr (CW_L (J,\hat \sigma)) \big)$, embedded in 
$\cH (G,\hat P_\ff, \hat \sigma)$ via \eqref{eq:3.1}, Ad$(\bar w)$ restricts to
\begin{equation}\label{eq:10.12}
\mc O \big( \Irr (CW_L (J,\hat \sigma)) \big) \to 
\mc O \big( \Irr (CW_L (J,\bar w \hat \sigma)) \big) : f \mapsto f \circ w^{-1}.
\end{equation}
Recall from \eqref{eq:5.2} that $\Pi_{\mf s} \cong \ind_{\hat P_\ff}^G (\hat \sigma)^{[L :L^3_\tau]}$ and from \eqref{eq:3.7} that $\End_G (\Pi_{\mf s}) \cong M_{[L:L^3_\tau]}(\C) \otimes \cH (G,\hat P_\ff, \hat \sigma)$. In this way, \eqref{eq:10.10} induces an algebra isomorphism
\begin{equation}\label{eq:10.11}
\mr{Ad}(\bar w) : \End_G (\Pi_{\mf s}) \to \End_G (\Pi_{w \mf s}) .     
\end{equation}
Proposition \ref{prop:5.5} gives a more precise description of how $\cH (G,\hat P_\ff, \hat \sigma)$
is embedded in $\End_G (\Pi_{\mf s})$. Thus the property
Ad$(\bar w) T_v = T_{w v w^{-1}}$ for $w \in W(J,\hat \sigma)$ remains valid in 
\eqref{eq:10.11}. Upon analytic localization as in \eqref{eq:5.6}, Lemma \ref{lem:5.6}
shows that \eqref{eq:10.11} is already given by a localized version of \eqref{eq:10.10}.
Therefore, \eqref{eq:10.12} shows that the localized version of \eqref{eq:10.11} 
agrees with the algebra isomorphism \eqref{eq:10.9}, only with
$w \mf s$ instead of $\mf s$ on the right-hand side. We conclude as in the case
$w \in W(G,L)_{\hat \sigma}$, with the same argument as following \eqref{eq:10.9}.
\end{proof}

Proposition \ref{prop:10.2} allows us to combine the equivalences of
categories from \eqref{eq:10.4} and \eqref{eq:10.5} into the following cleaner statement.
\begin{thm}\label{thm:10.3}
There exist the following equivalences of categories
\begin{equation}\label{eqn:thm:10.3}
\Rep_\fl (G)_{\mf s} \; \cong \; \Mod_{\fl}\text{ - }\End_G (\Pi_{\mf s}) 
\; \cong \; \Mod_\fl\text{ - }\cH (\mf s^\vee, q_F^{1/2}) ,
\end{equation}
induced by Propositions \ref{prop:10.1} and \ref{prop:9.7}.~These equivalences are
compatible with parabolic induction and restriction
\footnote{Parabolic restriction in the sense of Remark \ref{rem:10.8}}.
\end{thm}
\begin{proof}
The first equivalence is just \eqref{eq:10.3} restricted to objects of finite length. 
By \cite[Corollary 8.1]{SolEnd}, this induces the equivalence
between (i) and (iv) in Proposition \ref{prop:10.1}. From another viewpoint, the
first equivalence in this theorem is obtained from (i)$\to$(iv) in Proposition 
\ref{prop:10.1} by taking the direct sum over all unitary representations
$\tau$ in $\Irr (L)_{\mf s_L} / W(G,L)_{\hat \sigma}$.

In \eqref{eq:10.5}, we can take the direct sum over all unitary representations
$\tau$ in $\Irr (L)_{\mf s_L}$, or equivalently over all bounded 
$(\varphi_b,\rho_b) \in \Phi_e (L)^{\mf s_L}$. The summands indexed by $\tau$ and $\tau'$ 
that differ by an element $w \in W(G,L)_{\hat \sigma} = W_{\mf s}$ are identified via Proposition 
\ref{prop:10.2}, and dividing out those relations recovers $\Rep_\fl (G)_{\mf s}$
from the left-hand side of \eqref{eq:10.5}. On the right-hand side of \eqref{eq:10.5}, 
we can reduce to a direct sum over $(\varphi_b,\rho_b)$ up to conjugation under
$W(G^\vee,L^\vee)^{\mb W_F}_{\mf s^\vee} = W_{\mf s^\vee}$, which brings us to
\begin{equation}\label{eq:10.13}
\bigoplus\nolimits_{(\varphi_b,\rho_b) \in \varphi_{e,bdd}^{\mf s_L}} \, \Mod_{\fl, 
Z(\mf l^\vee)^{\mb W_F}_\R}\text{ - }\mh H \big( \mf s^\vee, \varphi_b, \log (q_F^{1/2}) \big) 
\; \big/ W_{\mf s^\vee} ,
\end{equation}
where the subscript bdd stands for bounded. It was already shown in Proposition 
\ref{prop:10.1} and \eqref{eq:10.4} that all steps so far respect parabolic 
induction and restriction.

We claim that \eqref{eq:10.13} is equivalent to 
$\Mod_\fl\text{ - }\cH ( \mf s^\vee, q_F^{1/2})$ via an equivalence that respects parabolic
induction and restriction.
Finite-length modules $M$ for any algebra can be decomposed along central characters:
for each central character $\chi$, one takes $M_\chi$ to be the maximal submodule of $M$ 
such that all irreducible subquotients of $M_\chi$ admit central character 
$\chi$. In particular we have, in the notation of \eqref{eq:10.16},
\[
\Mod_\fl\text{ - }\cH (\mf s^\vee,q_F^{1/2}) \; \cong \; \bigoplus\nolimits_{(\varphi_b,\rho_b) 
\in \varphi_{e,bdd}^{\mf s_L} / W_{\mf s^\vee}} \,
\Mod_{\fl, \mathfrak{X}_{\nr}^+ (L^\vee) (\varphi_b,\rho_b)}\text{ - }\cH (\mf s^\vee, q_F^{1/2}) .
\]
For a suitable action of $W_{\mf s^\vee}$, the right-hand side can be rewritten as 
\begin{equation}\label{eq:10.14}
\bigoplus\nolimits_{(\varphi_b,\rho_b) \in \varphi_{e,bdd}^{\mf s_L}} \,
\Mod_{\fl, \mathfrak{X}_{\nr}^+ (L^\vee) (\varphi_b,\rho_b)}\text{ - }\cH (\mf s^\vee, q_F^{1/2}) 
\; \big/ W_{\mf s^\vee}.
\end{equation}
By construction $w^\vee \in W_{\mf s^\vee}$ 
acts trivially on summands indexed by $w^\vee$-fixed $(\varphi'_b,\rho_b)$. 
By \cite[Proposition 3.14.a, Theorem 3.18.a]{AMS3}, there is a canonical equivalence 
\begin{equation}\label{eq:10.44}
\Mod_{\fl, \mathfrak{X}_{\nr}^+ (L^\vee) (\varphi_b,\rho_b)}\text{ - }\cH (\mf s^\vee, q_F^{1/2})
\; \cong \; \Mod_{\fl, Z(\mf l^\vee)^{\mb W_F}_\R}\text{ - }
\mh H \big( \mf s^\vee, \varphi_b, \log (q_F^{1/2}) \big) .
\end{equation}
By \cite[Theorems 2.5.b and 2.11.b]{AMS3}, this equivalence commutes with parabolic
induction and restriction. For Hecke algebras, parabolic restriction is right adjoint
to parabolic induction (which is just Frobenius reciprocity for algebras).~By the
uniqueness of adjoint functors, \eqref{eq:10.44} also commutes with parabolic restriction.

Via \eqref{eq:10.44}, \eqref{eq:10.14} becomes
\begin{equation}\label{eq:10.15}
\bigoplus\nolimits_{(\varphi_b,\rho_b) \in \varphi_{e,bdd}^{\mf s_L}} \, \Mod_{\fl, 
Z(\mf l^\vee)^{\mb W_F}_\R}\text{ - }\mh H \big( \mf s^\vee, \varphi_b, \log (q_F^{1/2}) \big) 
\; \big/ W_{\mf s^\vee} .
\end{equation}
By \cite[\S 7]{Lus-Gr}, or by an argument analogous to the analysis of \eqref{eq:10.7} 
in the proof of Proposition \ref{prop:10.2}, we deduce that the action of $W_{\mf s^\vee}$ in 
\eqref{eq:10.15} reduces to the cases for which \eqref{eq:10.19} holds, where 
it is none other then Ad$(w^\vee)$ from Proposition \ref{prop:10.2}. This proves 
the claim we made after \eqref{eq:10.13}.
\end{proof}
\begin{rem}
We warn the reader that Theorem \ref{thm:10.3} does not imply that $\End_G (\Pi_{\mf s})$ and
$\cH (\mf s^\vee, q_F^{1/2})$ are Morita equivalent. We really need the restriction
to finite-length modules, because those can be decomposed along central characters.
The difficulties (or even obstructions) to extend such equivalences of categories to
representations of arbitrary length stem from \eqref{eq:5.24}. 
\end{rem}

Let $\mf B (G)_{ns}$ be the collection of inertial equivalence classes for $G$
whose supercuspidal representations are non-singular, and define the subset $\mf B (G)^0_{ns}$
by the additional condition that the supercuspidal representations have depth zero.
This $\mf B (G)^0_{ns}$ is a finite set because: $G$ has only finitely many conjugacy classes 
of Levi subgroups $L$; each such $L$ has only finitely many orbits of facets $\ff_L$ in its 
Bruhat--Tits building; and each of the groups $\hat P_{L,\ff}$ has only finitely many
irreducible representations that come from its finite reductive quotient. We write\label{i:117}
\begin{equation*}
\Rep^0 (G)_{ns} := \prod\nolimits_{\mf s \in \mf B (G)^0_{ns}} \, \Rep (G)_{\mf s}
\end{equation*}
for the category of $G$-representations whose cuspidal support consists of
non-singular depth-zero representations. Since the index set is finite,
the direct product is also a direct sum.
We recall that $\Xo (G)$ acts on $\Rep^0 (G)_{ns}$ by tensoring.

\begin{thm}\label{thm:10.7}
The equivalences \eqref{eqn:thm:10.3} induce equivalences of categories
\begin{equation}\label{eqn:thm:10.7}
\Rep^0_\fl (G)_{ns} \cong 
\bigoplus\limits_{\mf s \in \mf B (G)^0_{ns}} \Mod_\fl \text{ - }\End_G (\Pi_{\mf s}) \cong 
\bigoplus\limits_{\mf s \in \mf B (G)^0_{ns}} \Mod_\fl  \text{ - } \cH (\mf s^\vee, q_F^{1/2}) ,
\end{equation}
which are compatible with parabolic induction and restriction.

The group $\Xo (G) \cong \Xo (G^\vee)$ acts canonically 
on all three terms, and the equivalences are equivariant for these actions.
\end{thm}
\begin{proof}
The equivalences of categories follow directly from Theorem \ref{thm:10.3}. Next we decompose 
$\Rep^0 (G)_{ns}$ into $\Xo (G)$-stable pieces. Let $\Rep (G)_\ff$ be the 
sum of the categories $\Rep (G)_{(\hat P_\ff,\hat \sigma)}$, where $\hat \sigma$ runs over
all $F$-non-singular representations of $\hat P_\ff$. This category is stable
under twisting by $\Xo (G)$, because every $\chi \in \Xo (L)$ is trivial on $G_{\ff,0+} =
\ker (P_\ff \to \mc G^\circ_\ff (k_F))$. For an inertial equivalence class $\mf s = [L,\tau]_G$, 
we write $\mf s \prec \ff$ if $\Rep (G)_{\mf s}$ has the form $\Rep (G)_{(\hat P_\ff, 
\hat \sigma)}$ as in Theorem \ref{thm:3.6}, so that $\Rep (G)_\ff = 
\bigoplus\nolimits_{\mf s \prec \ff} \Rep (G)_{\mf s}$. In this context, 
Theorem \ref{thm:10.3} gives equivalences of categories
\begin{equation}\label{eq:10.29}
\Rep_\fl (G)_{(P_\ff,\sigma)} \xrightarrow{\sim} \bigoplus\nolimits_{\mf s \prec \ff} 
\Mod_{\fl}\text{ - }\End_G (\Pi_{\mf s}) \xrightarrow{\sim} \bigoplus\nolimits_{\mf s \prec \ff} 
\Mod_\fl\text{ - }\cH (\mf s^\vee, q_F^{1/2}) .
\end{equation}
Take $\chi \in \Xo (G)$. Let $\chi_u |\chi|$ be its polar decomposition, 
where $|\chi| \in \mathfrak{X}_{\nr}^+ (G) = \Hom (G,\R_{>0})$ and $\chi_u \in \Xo (G)$
has image in $S^1 \subset \C^\times$.~Then 
$[L, \chi_u \otimes \tau]_G = \chi_u \mf s = \chi \mf s$, and the construction of 
$\Pi_{\mf s} = \II_Q^G (\Pi_{\mf s}^L) = \II_Q^G (\ind_{L^1}^L \tau)$ shows that 
$\Pi_{\chi_u \mf s}$ is equal to $\chi \otimes \Pi_{\mf s} = 
\chi_u \otimes \Pi_{\mf s}$. The relation between $\End_G (\Pi_{\mf s})$ and
$\End_G (\Pi_{\chi_u \mf s})$ is best described with the following isomorphism from 
\cite[Corollary 5.8]{SolEnd}:
\begin{equation}\label{eq:10.30}
\End_G (\Pi_{\mf s}) \otimes_{\mc O (\mathfrak{X}_{\nr} (L))} \C (\mathfrak{X}_{\nr} (L)) \cong
\C (\mathfrak{X}_{\nr} (L)) \rtimes \C [W(L,\tau,\mathfrak{X}_{\nr} (L)), \natural_\tau],
\end{equation}
where $\C (\mathfrak{X}_{\nr} (L))$ denotes the field of
rational functions on $\mathfrak{X}_{\nr} (L)$, and $\mathfrak{X}_{\nr} (L)$ is identified with 
the family of $L$-representation $\{ z \otimes \tau : z \in \mathfrak{X}_{\nr} (L) \}$. The twisted
group algebra $\C [W(L,\tau,\mathfrak{X}_{\nr} (L)), \natural_\tau]$ is spanned by operators
$N_w$, which may have poles on $\mathfrak{X}_{\nr} (L)$. Since tensoring with $\chi$ is a 
symmetry of the entire setup, these operators $N_w$  can be constructed in exactly the same 
way for $\chi_u \mf s$. Then there is a canonical algebra isomorphism
\begin{equation}\label{eq:10.31}
\begin{split}
\End_G (\Pi_{\chi_u \mf s}) \otimes_{\mc O (\mathfrak{X}_{\nr} (L))} \C (\mathfrak{X}_{\nr} (L)) & \xrightarrow{\sim} 
\End_G (\Pi_{\mf s}) \otimes_{\mc O (\mathfrak{X}_{\nr} (L))} \C (\mathfrak{X}_{\nr} (L)) \\
f N_w & \mapsto 
(f \circ \otimes \chi) N_w, 
\end{split}
\end{equation}
where $f \in \C (\mathfrak{X}_{\nr} (L))$, $w \in W(L,\tau,\mathfrak{X}_{\nr} (L)) = 
W(L,\chi_u \otimes \tau,\mathfrak{X}_{\nr} (L))$ and $\otimes \chi$ 
is to be interpreted as the family $z \otimes \tau \mapsto |\chi| z \otimes \chi_u \tau$ for 
$z \in \mathfrak{X}_{\nr} (L)$. The poles of $N_w$'s on both sides match, thus 
\eqref{eq:10.31} restricts to an algebra isomorphism
\begin{equation}\label{eq:10.32}
\End_G (\Pi_{\chi_u \mf s}) \xrightarrow{\sim} \End_G (\Pi_{\mf s}) . 
\end{equation}
Pullback along \eqref{eq:10.32} is thought of as tensoring modules with $\chi$,
and this gives the action of $\Xo (G)$ on the middle term in \eqref{eq:10.29}.
For $\pi \in \Rep (G)_{\mf s}$, by \eqref{eq:10.3}, the first arrow in \eqref{eq:10.29} sends 
$\chi \otimes \tau $ to 
\[
\Hom_G (\Pi_{\chi_u \mf s}, \chi \otimes \pi) =
\Hom_G (\chi_u \otimes \Pi_{\mf s}, \chi \otimes \pi) = 
\Hom_G (\chi \otimes \Pi_{\mf s}, \chi \otimes \pi) .
\]
The right-hand side is the vector space $\Hom_G (\Pi_{\mf s}, \pi)$ with the
$\End_G (\Pi_{\mf s})$-module structure adjusted by \eqref{eq:10.31}, thus it is equal to $\chi \otimes \Hom_G (\Pi_{\mf s},\pi) \in \Mod \text{-} \End_G (\Pi_{\chi_u \mf s})$. 
This shows that the first equivalence in \eqref{eq:10.29} is $\Xo (G)$-equivariant.

The second arrow in \eqref{eq:10.29} will be treated in three steps. First we pass 
to $\Mod_{\fl,\mf t_\R}\text{ - }\mh H (\tilde{\mc R}^\tau, W_{\mf s,\tau}, k^\tau,
\natural_\tau)$, as in Proposition \ref{prop:10.1}. More precisely, we take
a direct sum of such categories, first over $\mf s \prec \ff$ and then
over unitary $\tau \in \Irr (L)_{\mf s_L}$ modulo $W_{\mf s}$. The equivalence
\begin{equation}\label{eq:10.33}
\bigoplus\nolimits_{\mf s} \, \Mod_\fl\text{ - }\End_G (\Pi_{\mf s}) \xrightarrow{\sim}
\bigoplus\nolimits_{\mf s,\tau} \, \Mod_{\fl,\mf t_\R}\text{ - }
\mh H (\tilde{\mc R}^\tau, W_{\mf s,\tau}, k^\tau, \natural_\tau)
\end{equation}
is described in the first two lines of \eqref{eq:10.7}; see \cite[(8.1)]{SolEnd}.
Here the steps in\-vol\-ving analytic localization are innocent, and it boils down to
the algebra isomorphism
\begin{equation}\label{eq:10.34}
\mh H (\tilde{\mc R}_\tau, W_{\mf s,\tau}, k^\tau,\natural_\tau )^{an}_{\log_\tau (U_\tau)} 
\xrightarrow{\sim} 1_{U_\tau} \End_G (\Pi_{\mf s})^{an}_U 1_{U_\tau}
\end{equation}
from \cite[Proposition 7.3]{SolEnd}, written in the notation of \eqref{eq:10.7}.
On $\C [W_{\mf s,\tau}, \natural_\tau]$, this isomorphism  is given
by the same formula (independent of twisting by $\chi$) for any $\tau$, and
it sends any $f \in \mc O (\mf t)$ to $f \circ \log_\tau$, where 
$\log_\tau (z \otimes \tau) := \log (z)$. Similar to $\otimes \chi$, we have the map 
\[
\otimes \log |\chi| : \log_\tau (U_\tau) \to 
\log_{\chi_u \tau} (\chi_u U_\tau) = \log_\tau (U_\tau).
\]
We claim that there is a commutative diagram of algebra isomorphisms 
\begin{equation}\label{eq:10.35}
    \begin{tikzcd}
\mh H (\tilde{\mc R}^{\chi_u \tau}, W_{\chi _u \mf s, \chi_u \tau}, k^{\chi_u \tau},
\natural_{\chi_u \tau})^{an}_{\log_{\chi_u \tau} (\chi_u U_\tau)}  \arrow[]{r}{}\arrow[]{d}{\eqref{eq:10.34}} &  
\mh H (\tilde{\mc R}_\tau, W_{\mf s,\tau}, k^\tau,\natural_\tau )^{an}_{\log_\tau (U_\tau)}\arrow[]{d}{\eqref{eq:10.34}}  \\
1_{\chi_u U_\tau} \End_G (\Pi_{\chi_u \mf s})^{an}_{\chi_u U} 1_{\chi_u U_\tau}  
\arrow[]{r}{\eqref{eq:10.31}} & 1_{U_\tau} \End_G (\Pi_{\mf s})^{an}_U 1_{U_\tau}
    \end{tikzcd} \hspace{-5mm}
\end{equation}
where the upper horizontal map is given by 
\begin{equation}\label{eq:10.36}
f N_w \mapsto (f \circ \otimes \log |\chi|) N_w \text{ for } f \in C^{an} (\log_{\chi_u \tau} 
(\chi_u U_\tau)),\;w \in W_{\mf s,\tau} = W_{\chi_u \mf s, \chi_u \tau} . \hspace{-3mm}
\end{equation}
Indeed, the only thing left to show is that the 2-cocycles $\natural_\tau$ and
$\natural_{\chi_u \tau}$ match, which is guaranteed by \eqref{eq:9.58} and 
Proposition \ref{prop:9.5}. We define the $\Xo (G)$-action as pullback along 
\eqref{eq:10.36}, and thus \eqref{eq:10.33} is $\Xo (G)$-equivariant. 

Next we consider the equivalence of categories 
\begin{equation}\label{eq:10.37}
\bigoplus\nolimits_{\mf s,\tau} \, \Mod_{\fl,\mf t_\R}\text{ - }\mh H (\tilde{\mc R}^\tau, 
W_{\mf s,\tau}, k^\tau, \natural_\tau) \xrightarrow{\sim} \bigoplus\nolimits_{\mf s,
\tau} \, \Mod_\fl\text{ - }\mh H \big( \mf s^\vee, \varphi_b ,\log (q_F^{1/2}) \big) ,
\end{equation} 
where $\mf s$ and $\tau$ run through the same set as in \eqref{eq:10.36}, and
$(\varphi_b,\rho_b)$ corresponds to $\tau$ via Theorem \ref{thm:8.2}. This follows from 
Proposition \ref{prop:9.7}. We claim that there is a 
commutative diagram of algebra isomorphisms
\begin{equation}\label{eq:10.38}
    \begin{tikzcd}
        \mh H (\tilde{\mc R}^{\chi_u \tau}, W_{\chi _u \mf s, \chi_u \tau}, k^{\chi_u \tau},
\natural_{\chi_u \tau}) \arrow[]{r}{\eqref{eq:10.36}}\arrow[]{d}{\text{Proposition \ref{prop:9.7}}} &  
\mh H (\tilde{\mc R}_\tau, W_{\mf s,\tau}, k^\tau,\natural_\tau )\arrow[]{d}{\text{Proposition \ref{prop:9.7}}} \\
\mh H \big( \chi_u \mf s^\vee, \chi_u \varphi_b, \log (q_F^{1/2}) \big) 
\arrow[]{r}{} & \mh H \big( \mf s^\vee, \varphi_b, \log (q_F^{1/2}) \big)
    \end{tikzcd}
\end{equation}
where the second row is given by
\begin{equation}\label{eq:10.39}
f N_{w^\vee} \mapsto (f \circ \otimes \log |\chi| ) N_{w^\vee} \;\text{for}\;
f \in \mc O (Z(\mf l^\vee)^{\mb W_F})\;\text{and}\; w^\vee \in W_{\mf s^\vee, \varphi_b} .
\end{equation}
Theorem \ref{thm:8.2} guarantees that $\chi_u \varphi_b$ and $\chi_u \mf s^\vee$
correspond to $\chi_u \tau$ and $\chi_u \mf s$, respectively, as desired.~By
$\chi_u, |\chi| \in \Xo (G^\vee)$ and \eqref{eq:9.57}, we
have that \eqref{eq:10.39} is an algebra isomorphism. Commutativity of the
diagram is clear from the formulas for the maps in question. This proves that \eqref{eq:10.37} 
is equivariant for $\Xo (G) \cong \Xo (G^\vee)$, where on the right we let $\Xo (G^\vee)$ 
act via pullback along \eqref{eq:10.39}. The equivalence of categories 
\begin{equation}\label{eq:10.40}
\bigoplus\nolimits_{\mf s,\tau} \, \Mod_\fl\text{ - }\mh H \big( \mf s^\vee, \varphi_b ,
\log (q_F^{1/2}) \big) \cong \bigoplus\nolimits_{\mf s \prec \ff} \,
\Mod_\fl\text{ - }\cH (\mf s^\vee, q_F^{1/2})
\end{equation}
was shown in the proof of Theorem \ref{thm:10.3}; see \eqref{eq:10.13}. It then boils
down to several applications of the following equivalence of categories
\begin{equation}\label{eq:10.41}
\Mod_{\fl,Z(\mf l^\vee)^{\mb W_F}_\R}\text{ - }\mh H \big( \mf s^\vee, \varphi_b ,
\log (q_F^{1/2}) \big) \cong 
\Mod_{\fl, \mathfrak{X}_{\nr}^+ (L^\vee) (\varphi_b,\rho_b)}\text{ - }\cH (\mf s^\vee, q_F^{1/2})
\end{equation}
from \cite[Proposition 3.14.a and Theorem 3.18.a]{AMS3}. This is analogous to
the equivalence of categories in \eqref{eq:10.33}. Using analytic localizations as
in \eqref{eq:10.35}, one can show that there is a commutative diagram 
\begin{equation}\label{eq:10.42}
\begin{tikzcd}
\Mod_{\fl,Z(\mf l^\vee)^{\mb W_F}_\R}\text{ - }\mh H \big( \mf s^\vee, \varphi_b ,
\log (q_F^{1/2}) \big) 
\arrow[]{r}{\eqref{eq:10.41}}\arrow[]{d}[swap]{\eqref{eq:10.39}} & 
\Mod_\fl\text{ - }\cH (\mf s^\vee, q_F^{1/2})\arrow[]{d}{} \\
\Mod_{\fl,Z(\mf l^\vee)^{\mb W_F}_\R}\text{ - }\mh H \big( \chi_u \mf s^\vee,
\chi_u \varphi_b ,\log (q_F^{1/2}) \big) 
\arrow[]{r}{\eqref{eq:10.41}}  & 
\Mod_\fl\text{ - }\cH (\chi_u \mf s^\vee, q_F^{1/2}) 
\end{tikzcd}
\end{equation}
where the right vertical arrow is pullback along the following algebra isomorphism from 
\cite[(19)]{SolLLCunip}: for $f \in \mc O (\chi_u \mf s_L^\vee)$ and 
$w \in W_{\mf s^\vee} = W_{\chi_u \mf s^\vee}$, 
\begin{equation}\label{eq:10.43}
\cH (\chi_u \mf s^\vee, q_F^{1/2}) 
\to 
\cH (\mf s^\vee, q_F^{1/2})\;\text{is given by } 
f N_{w^\vee} 
\mapsto 
(f \circ \otimes \chi) N_{w^\vee}. 
\end{equation}
Let $\chi \in \Xo (G^\vee)$ act on the right-hand side by pullback along 
\eqref{eq:10.43}, and thus \eqref{eq:10.40} is $\Xo (G^\vee)$-equivariant.
\end{proof}

\section{An LLC for non-singular depth-zero representations}
\label{sec:LLC}

\subsection{Construction} \

The right-hand sides of Proposition \ref{prop:10.2}, Theorems \ref{thm:10.3} 
and \ref{thm:10.7} concern Langlands parameters, but only cuspidal L-parameters for Levi 
subgroups of rigid inner twists of $G$. We also need to consider non-cuspidal enhanced 
L-parameters (see for example \cite{AMS3}) when parametrizing the irreducible modules 
(or the standard modules) of the relevant Hecke algebras. Note that \cite{AMS3} considers 
only left modules, in this article we will need to consider right modules at some point.
In preparation for this, we study how the cuspidal support map from \cite{AMS1}
behaves with respect to taking contragredients of enhancements of L-parameters.

\begin{lem}\label{lem:10.9}
Let $(\varphi,\rho) \in \Phi_e (G)$. \hspace{-2mm} 
Suppose that $\mr{Sc}(\varphi,\rho)$ is represented by $(L,\varphi_L,\rho_L)$.
\enuma{
\item $\mr{Sc}(\varphi,\rho^\vee)$ is represented by $(L,\varphi_L,\rho_L^\vee)$.
\item Consider inertial classes $\mf s_L^\vee = \mathfrak{X}_{\nr} (L) (\varphi_L,\rho_L)$ 
and $\mf s_L^{\vee op} = \mathfrak{X}_{\nr} (L) (\varphi_L,\rho_L^\vee)$
for $\Phi_e (L)$.~Let $\mf s^\vee$ and $\mf s^{\vee op}$ be the corresponding
inertial classes for $\Phi_e (G)$. Then
\[
\Phi_e (G)^{\mf s^{\vee op}}=
\big\{ (\varphi, \rho^\vee) : (\varphi,\rho) \in \Phi_e (G)^{\mf s^\vee} \big\}  .
\]
}   
\end{lem}
\begin{proof}
(a) By \cite[Definition 7.7]{AMS1}, the construction of Sc boils downs to cuspidal 
supports for local systems supported on unipotent orbits in complex reductive groups, 
for which the compatibility with contragredients follows from the characterization in 
\cite[Theorem 5.5.a]{AMS1} (see also \cite[Theorem 2.5.3]{DiSc}). 

(b) This follows from part (a).
\end{proof}

\noindent We now parametrize irreducible modules in Theorem \ref{thm:10.3} by
enhanced L-parameters.

\begin{thm}\label{thm:10.4}
There is a canonical bijection 
\begin{equation}\label{eqn:thm:10.4}
\Irr\;\text{-}\;\cH (\mf s^\vee,q_F^{1/2}) \;\cong\; \Phi_e (G)^{\mf s^\vee}.
\end{equation}
\end{thm}
\begin{proof}
We need to modify the bijection from \cite[Theorem 3.18.a]{AMS3}, which only concerns left 
modules, and adapt it for irreducible right modules. The equivalence between 
$\Mod_\fl \text{ - } \cH (\mf s^\vee, q_F^{1/2})$ and  \eqref{eq:10.13} established in 
the proof of Theorem \ref{thm:10.7} works both for left and for right modules.
Hence it suffices to modify \cite[Theorem 3.18.a]{AMS3} for 
\[
\Irr_{Z(\mf l^\vee)^{\mb W_F}_\R}\text{ - }\mh H \big( \mf s^\vee, \varphi_b, \log (q_F^{1/2}) 
\big) \; \subset \; \Mod_{\fl, Z(\mf l^\vee)^{\mb W_F}_\R}\text{ - }
\mh H \big( \mf s^\vee, \varphi_b, \log (q_F^{1/2}) \big).
\]
Then \cite[Theorem 3.18]{AMS3} reduces to \cite[Theorem 3.8]{AMS3}, and the set 
of enhanced Langlands parameters $\Phi_e (G)^{\mf s^\vee}$ reduces to those
with cuspidal support in $(L^\vee, \mathfrak{X}_{\nr}^+(L^\vee) (\varphi_b,\rho_b))$, 
denoted $\Phi_e (G)^{[\mathfrak{X}_{\nr}^+(L^\vee) (\varphi_b,\rho_b) ]}$.

For any $(\varphi,\rho) \in \Phi_e (G)^{[\mathfrak{X}_{\nr}^+(L^\vee) (\varphi_b,\rho_b) ]}$, 
there is a  unique $z \in \mathfrak{X}_{\nr}^+ (L^\vee)$ such that
$\varphi |_{\mb W_F} = z \varphi_b |_{\mb W_F}$. This $z$ has a unique logarithm\label{i:167}
\begin{equation}\label{eq:10.47}
t_\varphi := \log (z) = \log (\varphi (\Fr_F) \varphi_b (\Fr_F)^{-1}) \in Z(\mf l^\vee)^{\mb W_F}_\R .
\end{equation}
Let d$\varphi : \mf{sl}_2 (\C) \to \mr{Lie}(G^\vee)$ be the tangent map of
$\varphi |_{\SL_2 (\C)}$.~Write \label{i:168}
$N_\varphi := \textup{d}\varphi (\matje{0}{1}{0}{0})$. By
\cite[Theorem 3.8]{AMS3}, we can associate to $(\varphi,\rho)$ a left $\mh H \big( \mf s^\vee, \varphi_b, 
\log (q_F^{1/2}) \big)$-module $M \big( \varphi,\rho,\log(q_F^{1/2}) \big)$. By 
\cite[Proposition 3.3]{SolKL}, it can be expressed as
\begin{equation}\label{eq:10.18}
M \big( \varphi,\rho,\log(q_F^{1/2}) \big) \cong 
\mr{sgn}^* M_{N_\varphi, t_\varphi, -\log (q_F)/2,\rho},
\end{equation}
where sgn denotes the sign automorphism of the algebra $\mh H (\mf s^\vee, \varphi_b,
\mb r)$, with an indeterminate $\mb r$ instead of $\log (q_F^{1/2})$. We will modify this 
left module in several steps. By definition, $\mh H(\mf s^\vee, \varphi_b, \mb r) = \mc O (Z(\mf l^\vee)^{\mb W_F}) \otimes
\C [W_{\mf s^\vee, \varphi_b}, \natural_{\mf s^\vee}] \otimes \C [\mb r]$ as vector spaces; moreover, $\mr{sgn}|_{\mc O (Z(\mf l^\vee)^{\mb W_F})} = \mr{id}|_{\mc O (Z(\mf l^\vee)^{\mb W_F})}$, $\mr{sgn}(\mb r) = -\mb r$, and $\mr{sgn}(N_w) = \mr{sgn}(w) N_w$ for $w \in W_{\mf s^\vee,\varphi_b}$. The sign character of $W_{\mf s^\vee,\varphi_b}$ is
defined as $\det |_{X_* (Z(\mf l^\vee)^{\mb W_F})}$, which extends the sign character
of the Weyl group $W(R_{\mf s^\vee,\varphi_b})$. This constitutes a slight improvement,
already used in \cite[\S 6.2]{SolQS}, on an alternative sign character from 
\cite{AMS2,AMS3,SolKL}. As shown in \cite[after (6.16) and \S 7]{SolQS}, 
this minor modification does not affect any of the good properties established in
\cite{AMS2,AMS3,SolKL}. By \cite[Theorem 3.11]{AMS2} and \cite[Theorem 3.6]{AMS3}, 
$M \big( \varphi,\rho,\log(q_F^{1/2}) \big)$ is the unique irreducible quotient of the 
standard module 
\begin{equation}\label{eq:10.20}
E \big( \varphi,\rho,\log(q_F^{1/2}) \big) \cong 
\mr{sgn}^* E_{N_\varphi, t_\varphi, -\log (q_F)/2,\rho} .
\end{equation}
By \cite[Lemma 3.6.a]{AMS2} and \cite[(1.17)]{AMS3}, we can write
\begin{equation}\label{eq:10.21}
E \big( \varphi,\rho,\log(q_F^{1/2}) \big)  \cong \Hom_{\pi_0 (S_\varphi)} \big( \rho,
\mr{sgn}^* E_{N_\varphi, t_\varphi, -\log (q_F)/2} \big) .
\end{equation}
Since we have used a sign character different from previous literature, 
we hereby specify that we use the right-hand side formulas of \eqref{eq:10.18}, 
\eqref{eq:10.20} and \eqref{eq:10.21} to define the respective left-hand sides. These are still left modules, to obtain an analogue for right modules, we recall
from \cite[(5), (14)]{AMS2} the following canonical isomorphism for the opposite algebra of
$\mh H (\mf s^\vee, \varphi_b, \mb r )$: for 
$f \in \mc O ( Z(\mf l^\vee)^{\mb W_F}) \otimes \C [\mb r]$ and $w \in W_{\mf s^\vee,\varphi_b}$, 
\begin{equation}\label{eq:10.17}
\mh H (\mf s^\vee, \varphi_b, \mb r )^{op} 
\to 
\mh H (\mf s^{\vee op}, \varphi_b, \mb r) \;\text{ is given by } 
f N_w 
\mapsto 
N_w^{-1} f .
\end{equation}
Here $\mf s^{\vee op}$ comes from $\mf s_L^{\vee op}$
\footnote{This $\mf s_L^{\vee op}$ is associated to a rigid inner twist of $L$, not necessarily
to $L$, but at the moment we are only working in ${}^L G$ where it does not matter.} 
as in Lemma \ref{lem:10.9}. 
Applying the above discussions to $\mf s^{\vee op}$ instead of $\mf s^\vee$, we obtain the 
standard left $\mh H (\mf s^{\vee op}, \varphi_b, \mb r)$-module 
\begin{equation}\label{eq:10.26}
\begin{aligned}
E \big( \varphi,\rho^\vee,\log(q_F^{1/2}) \big) &= \Hom_{\pi_0 (S_\varphi^+)} \big( \rho^\vee,
\mr{sgn}^* E_{N_\varphi, t_\varphi, -\log (q_F)/2} \big) \\
& = \big( \rho \otimes \mr{sgn}^* E_{N_\varphi, t_\varphi, -\log (q_F)/2} \big)^{\pi_0 (S_\varphi^+)}.
\end{aligned}
\end{equation}
It has a unique irreducible quotient $M \big( \varphi,\rho^\vee,\log(q_F^{1/2}) \big)$.
Via \eqref{eq:10.17}, we can also view $E \big( \varphi,\rho^\vee,\log(q_F^{1/2}) \big)$ and
$M \big( \varphi,\rho^\vee,\log(q_F^{1/2}) \big)$ as right modules for
$\mh H (\mf s^\vee, \varphi_b, \mb r )$ or $\mh H \big( \mf s^\vee, \varphi_b, \log (q_F^{1/2}) 
\big)$. To emphasize this point of view, we shall add a superscript op and we replace $\rho^\vee$ 
by $\rho$ (in the notation only, so it remains \eqref{eq:10.26} as a vector space). 
This procedure does not change the $\mc O (Z(\mf l^\vee)^{\mb W_F})$-weights of 
$E \big( \varphi,\rho^\vee,\log(q_F^{1/2}) \big)$, so
\[
E \big( \varphi,\rho,\log(q_F^{1/2}) \big)^{op},\;
M \big( \varphi,\rho,\log(q_F^{1/2}) \big)^{op} \in 
\Mod_{\fl, Z(\mf l^\vee)^{\mb W_F}_\R}\text{ - }
\mh H \big( \mf s^\vee, \varphi_b, \log (q_F^{1/2}) \big).
\]
By \cite[Theorem 3.8]{AMS3}, $(\varphi,\rho) \mapsto M \big( \varphi,\rho,\log(q_F^{1/2}) \big)^{op}$ gives the desired bijection 
\begin{equation}\label{eq:10.22}
\Phi_e (G)^{[\mathfrak{X}_{\nr}^+(L^\vee) \varphi_b,\rho_b ]} \longrightarrow
\Irr_{Z(\mf l^\vee)^{\mb W_F}_\R}\text{ - }\mh H \big( \mf s^\vee, \varphi_b, \log (q_F^{1/2}) 
\big) . 
\end{equation}
Let $\bar E (\varphi,\rho,q_F^{1/2})^{op}$ and $\bar M (\varphi,\rho,q_F^{1/2})^{op}$ be
the corresponding modules obtained from the equivalence of \eqref{eq:10.13} with 
$\Mod_\fl\text{ - }\cH (\mf s^\vee,q_F^{1/2})$. Thus by \eqref{eq:10.22}, 
we obtain a map
\begin{align}\label{eq:10.23}
\Phi_e (G)^{\mf s^\vee} 
\rightarrow \Irr\text{ - }\cH (\mf s^\vee,q_F^{1/2})\;\text{ given by }\;
(\varphi,\rho) 
\longmapsto  \bar M(\varphi,\rho,q_F^{1/2})^{op}.
\end{align}
Again, by \cite[Theorem 3.18.a]{AMS3}, \eqref{eq:10.23}
inherits the bijectivity of \eqref{eq:10.22}.
\end{proof}

We remark that in \eqref{eq:10.26}, the second line fits better with the parametrization
of Deligne--Lusztig packets in \eqref{eq:7.21}.
By Theorems \ref{thm:10.3} and \ref{thm:10.4}, we obtain bijections
\begin{equation}\label{eq:10.24}
\Irr (G)_{\mf s} \longrightarrow \Irr \text{ - } \End_G (\Pi_{\mf s}) \longrightarrow
\Irr \text{ - } \cH (\mf s^\vee,q_F^{1/2}) \longleftarrow \Phi_e (G)^{\mf s^\vee} .
\end{equation}
On the appropriate subsets, we can describe these bijections more precisely using
Propositions \ref{prop:10.1} and \ref{prop:9.7}, i.e. we have 
\begin{multline}\label{eq:10.28}
\Irr (G)_{\mathfrak{X}_{\nr}^+ (L) \tau} \rightarrow \Irr_{\mathfrak{X}_{\nr}^+ (L) 
\otimes \tau} \text{ - } \End_G (\Pi_{\mf s}) \rightarrow \Irr_{\mf t_\R} \text{ - } 
\mh H (\tilde{\mc R}_\tau, W_{\mf s,\tau}, k^\tau,\natural_\tau) \rightarrow \\ 
\Irr_{Z(\mf l^\vee)^{\mb W_F}_\R} \text{ - } \mh H \big( \mf s^\vee, \varphi_b, \log (q_F^{1/2}) 
\big) \leftarrow \Irr_{\mathfrak{X}_{\nr}^+ (L^\vee)} \text{ - } \cH (\mf s^\vee,q_F^{1/2}) 
\leftarrow \Phi_e (G)^{[\mathfrak{X}_{\nr}^+ (L^\vee) (\varphi_b,\rho_b) ]} .     
\end{multline}
We abbreviate the bijection between the outer sides of \eqref{eq:10.24} or
\eqref{eq:10.28} as\label{i:169}
\begin{equation}\label{eq:10.25}
\Irr (G)_{\mf s} 
\longleftrightarrow 
\Phi_e (G)^{\mf s^\vee} \text{ given by }
\pi 
\mapsto 
(\varphi_\pi, \rho_\pi) \text{ and }
\pi (\varphi,\rho) 
\text{\reflectbox{$\mapsto$}} 
(\varphi,\rho). 
\end{equation}
In the proof of Theorem \ref{thm:10.4}, we constructed a standard module
\begin{equation}\label{eq:10.45}
\bar E (\varphi,\rho,q_F^{1/2})^{op} \in \Mod_\fl\text{ - }\cH (\mf s^\vee,q_F^{1/2}).
\end{equation}
Let \label{i:170}$\pi^{st}(\varphi,\rho) \in \Rep_\fl^0 (G)_{ns}$ be its image 
via Theorem \ref{thm:10.3}. 

\begin{lem}\label{lem:10.5}
In \eqref{eq:10.25}, the map $\pi \mapsto \varphi_\pi$ is canonical.
\end{lem}
\begin{proof}
The remarks after \eqref{eq:10.5} and Proposition \ref{prop:10.2} show that
the non-canonicity in the construction of \eqref{eq:10.25} comes from four sources:\\
(i) On the supercuspidal level, i.e.~for $\Irr (L)_{\mf s_L}$, where the non-canonicity 
only comes from the enhancements.\\
(ii) Choices of systems of positive roots in the construction of the various algebras. But since all
positive systems in a finite root system are associate under the Weyl group,
these choices do not affect the L-parameters up to conjugacy.\\
(iii) Twisting by characters of $W_{\mf s,\tau}$ that are trivial on the 
subgroup generated by the reflections $s_\alpha$ where $\alpha \in R_{\sigma,\tau}$ 
and $k^\tau_\alpha \neq 0$, as in Proposition \ref{prop:9.7}.~By \eqref{eq:9.55} and
\eqref{eq:9.56}, this can be translated to a character twist of the twisted group algebra
part of $\mh H \big( \mf s^\vee, \varphi_b, \log (q_F^{1/2}) \big)$.~By the construction
of $E \big( \varphi,\rho,\log (q_F^{1/2}) \big)$ (as in the proof of Theorem \ref{thm:10.4}) 
and \cite[Lemma 3.18]{AMS2}, the twisted
group algebra in a twisted graded Hecke algebra only affects the enhancements
of the parameters for the irreducible or standard modules. Hence these character
twists only affect $\rho$, not $\varphi$.\\
(iv) Normalizations of the various 2-cocycles. Choices have to be made both on the 
supercuspidal level (see especially Lemmas \ref{lem:8.5} and \ref{lem:8.6}) and on the 
non-supercuspidal level (namely in Proposition \ref{prop:9.5}). Often we do not know a 
natural choice. For the same reason as in point (iii), this affects the
enhancements $\rho_\pi$ but not the L-parameters $\varphi_\pi$. \qedhere
\end{proof}

Combining \eqref{eq:10.25} over blocks gives the following. 
\begin{thm}\label{thm:10.6}
The equivalences \eqref{eqn:thm:10.7} and \eqref{eqn:thm:10.4} induce a bijection between
\begin{itemize}
\item the set $\Irr^0 (G)_{ns}$ of irreducible depth-zero $G$-representations 
with non-singular cuspidal support; and
\item the set $\Phi^0_e (G)_{ns}$ of depth-zero parameters in $\Phi_e (G)$, whose 
cuspidal support is supercuspidal, i.e.~trivial on $\SL_2 (\C)$.
\end{itemize}    
For any L-parameter $\varphi \in \Phi (G)$, the set of non-singular 
depth-zero representations in the L-packet $\Pi_\varphi (G) = \{ \pi \in \Irr (G) : 
\varphi_\pi = \varphi \}$ is determined canonically.
\end{thm}
\begin{proof}
By \eqref{eq:9.49}, if we take the union over
all Bernstein blocks $\Rep (G)_{\mf s}$ of the indicated kind, then $\mf s^\vee$
runs precisely once through all inertial equivalence classes of the indicated kind
for $\Phi_e (G)$. Hence the union of \eqref{eq:10.25} gives the required bijection. 
The statement about the L-packets follows from Lemma \ref{lem:10.5}.
\end{proof}
\begin{rem}We warn the reader that, for $(\varphi,\rho) \in \Phi_e^0 (G)_{ns}$, maybe not all 
enhancements $\rho'$ of $\varphi$ lead to cuspidal supports that are trivial on $\SL_2 (\C)$. 
Hence the L-packet $\Pi_\varphi (G)$ need not consist entirely of non-singular depth-zero 
representations; its other members fall outside the scope of this paper.
\end{rem}

 \subsection{Properties} \
\label{par:prop}

We now show that our local Langlands correspondence for non-singular 
depth-zero representations enjoys several nice 
properties, including those desired by Borel \cite[\S 10]{Bor}.
Recall from \eqref{eq:6.27} that $\Xo (G^\vee)$ acts naturally on $\Phi_e^0 (G)_{ns}$. 
In the following, we label the bijection in Theorem \ref{thm:10.6} as 
\begin{equation}\label{eqn:thm:10.6}
    \Irr^0 (G)_{ns}\longleftrightarrow \Phi^0_e (G)_{ns}
\end{equation}

\begin{lem}\label{lem:11.6}
The map \eqref{eqn:thm:10.6} is equivariant for the canonical actions of 
$\Xo (G) \cong \Xo (G^\vee)$. Similarly, the map $(\varphi,\rho) \mapsto 
\pi^{st}(\varphi,\rho)$ from \eqref{eq:10.45} is $\Xo (G)$-equivariant.
\end{lem}
\begin{proof}
By Theorem \ref{thm:10.7}, it suffices to prove that the maps
\[
\Phi_e^0 (G)_{ns}  
\rightarrow 
\bigsqcup\nolimits_{\mf s \in 
\mf B (G)^0_{ns}} 
\Rep_\fl\text{ - }\cH (\mf s^\vee, q_F^{1/2})\] 
given by  
$(\varphi,\rho) 
\mapsto  
\bar M (\varphi,\rho,q_F^{1/2})^{op}$ and 
$(\varphi,\rho) 
\mapsto 
\bar E (\varphi,\rho,q_F^{1/2})^{op}$ are equivariant for the $\Xo (G^\vee)$-actions 
from \eqref{eq:6.27}, 
\eqref{eq:10.42}, \eqref{eq:10.43}.~This follows from \cite[Lemma 2.2]{SolLLCunip}.
\end{proof}

Cuspidality for enhanced L-parameters was introduced in \cite[\S 6]{AMS1}, 
generalizing earlier works of Lusztig. 

\begin{prop}\label{prop:11.1}
In \eqref{eqn:thm:10.6}, $\pi$ is supercuspidal if and only if $(\varphi_\pi,
\rho_\pi)$ is cuspidal. In this case, \eqref{eqn:thm:10.6} coincides with
\eqref{eq:9.7} and with \cite{Kal2,Kal3}.
\end{prop}
\begin{proof}
Since $\pi \in \Rep (G)_{(\hat P_\ff, \hat \sigma)}$ arises via parabolic
induction from $\Rep (L)_{(\hat P_{L,\ff}, \hat \sigma)}$, we know that $\pi$ is supercuspidal
if and only if $L = G$. On the other hand, any $(\varphi, \rho)$ in Theorem 
\ref{thm:10.6} has cuspidal support in $\Phi_e (L)$, for some Levi subgroup 
$L \subset G$, so it is cuspidal if and only if $L = G$. 

Suppose now that $L=G$. Then $\cH (\mf s^\vee, q_F^{1/2})$ reduces to
$\mc O (\mf s_L^\vee)$, while the isomorphism of twisted graded Hecke algebras in 
Proposition \ref{prop:9.7} reduces to 
\[
\mh H \big( \mf s^\vee, \varphi_b , \log (q_F^{1/2}) \big) = 
\mc O \big( Z(\mf l^\vee)^{\mb W_F} \big) \longrightarrow
\mh H (\tilde{\mc R}_\tau, W_{\mf s,\tau}, k^\tau, \natural_\tau) = \mc O (\mf t) .
\]
This isomorphism is simply \eqref{eq:9.43}, which is induced by the tangent map of
$\Irr (L)_{\mf s_L} \! \to \Phi_e (L)^{\mf s_L^\vee}$ at $\tau$.~By
\cite[(2.25)]{SolEnd}, we have $\End_L (\Pi_{\mf s}^L) \cong 
\mc O(\mathfrak{X}_{\nr} (L)) \rtimes \C [\mathfrak{X}_{\nr} (L,\tau),\natural_\tau]$. Consider the sequence from \eqref{eq:10.28}:
\begin{multline}\label{eq:11.1}
\Irr (L)_{\mathfrak{X}_{\nr}^+ (L) \tau} \rightarrow 
\Irr_{\mathfrak{X}_{\nr}^+ (L) \otimes \tau} \text{-} 
\End_L (\Pi_{\mf s}^L) \rightarrow \Irr_{\mf t_\R} \text{-} \mc O(\mf t) 
\rightarrow \\
\Irr_{Z(\mf l^\vee)^{\mb W_F}_\R} \text{-} \mc O \big( Z(\mf l^\vee)^{\mb W_F}\big)
\rightarrow \Irr_{\mathfrak{X}_{\nr}^+(L^\vee) \varphi_b} \text{-} \mc O (\mf s_L^\vee) 
\rightarrow \Phi_e (L)^{[\mathfrak{X}_{\nr}^+ (L^\vee) \varphi_b,\rho_b ]}
\end{multline}
We start on the right-hand side with $(z \varphi_b,\rho_b)$ for any
$z \in \mathfrak{X}_{\nr}^+ (L^\vee) \cong \mathfrak{X}_{\nr} (L)$.~Its image in 
$\Irr_{\mathfrak{X}_{\nr}^+(L^\vee) \varphi_b} \text{-} \mc O (\mf s_L^\vee)$ is again 
$(z \varphi_b,\rho_b)$.~In $\Irr_{Z(\mf l^\vee)^{\mb W_F}_\R} \text{-} \mc O \big( Z(\mf l^\vee)^{\mb W_F}\big)
\cong Z(\mf l^\vee)^{\mb W_F}_\R$, this becomes $\log (z)$ and in 
$\Irr_{\mf t_\R} \text{-} \mc O(\mf t) \cong \mf t_\R$,
it also maps to $\log (z)$. From there to
\[
\Irr_{\mathfrak{X}_{\nr}^+ (L) \otimes \tau} \text{-} \End_L (\Pi_{\mf s}^L) =
\Irr_{\mathfrak{X}_{\nr}^+ (L) \otimes \tau} \text{-}
\mc O(\mathfrak{X}_{\nr} (L)) \rtimes \C [\mathfrak{X}_{\nr} (L,\tau),\natural_\tau], 
\]
we apply \cite[(8.1) and Corollary 8.1]{SolEnd}. By \cite[Lemmas 6.4 and 6.5
and Proposition 7.3]{SolEnd}, $\log (z)$ is mapped to
$\ind_{\mc O(\mathfrak{X}_{\nr} (L))}^{\End_L (\Pi_{\mf s}^L)} (z)$. Since $\tau$ 
is our basepoint, this irreducible module corresponds to $z \otimes \tau \in 
\Irr (\End_L (\Pi_{\mf s}^L))$ in the notation of \eqref{eq:10.27}. In the
conventions of Proposition \ref{prop:10.1}, the leftmost bijection in \eqref{eq:11.1} 
sends $z \otimes \tau$ to $z \otimes \tau$, but now as an element of $\Irr(L)_{\mf s_L}$.
Thus \eqref{eq:11.1} is just $z \otimes \tau \mapsto (z \varphi_b ,\rho_b)$, which
by Theorem \ref{thm:8.2} agrees with \eqref{eq:9.7}.
\end{proof}

Let $\mf{Lev}(G)$ be a set of representatives for the Levi subgroups of $G$ (i.e.~the
$F$-Levi subgroups of $\mc G$) modulo $G$-conjugation.~Then $\mf{Lev}(G)$ also 
represents the $G^\vee$-conjugacy classes of $G$-relevant L-Levi subgroups of
${}^L G$ by \cite[Corollary 1.3]{SolLLCunip}.

\begin{lem}\label{lem:11.2}
The cuspidal support maps and \eqref{eqn:thm:10.6} form a commutative diagram
\[
\begin{array}{ccc}
\Irr^0 (G)_{ns} & \longleftrightarrow & \Phi_e^0 (G)_{ns} \\
\downarrow \mr{Sc} & & \downarrow \mr{Sc} \\
\bigsqcup_{L \in \mf{Lev}(G)} \Irr^0_\cusp (L)_{ns} / W(G,L) &
\longleftrightarrow & \bigsqcup_{L \in \mf{Lev}(G)}
\Phi^0_\cusp (L)_{ns} / W(G^\vee,L^\vee)^{\mb W_F}
\end{array} .
\]
\end{lem}
\begin{proof}
By Proposition \ref{prop:11.1} and Theorem \ref{thm:8.2}, the maps from Theorem
\ref{thm:10.6} on the cuspidal level are equivariant for $W(G,L) \cong
W(G^\vee,L^\vee)^{\mb W_F}$. In particular, the bottom line of the diagram
is well-defined and bijective. 

Suppose that $(\varphi,\rho) \in \Phi_e (G)^{\mf s^\vee} \subset \Phi_e^0 (G)_{ns}$ 
has cuspidal support $(L,\varphi_L ,\rho_L)$. Recall from \cite[Lemma 2.3]{AMS3} that 
\begin{equation}\label{eq:11.6}
Z \big( \cH (\mf s^\vee, q_F^{1/2}) \big) = \mc O (\mf s_L^\vee / W_{\mf s^\vee}) =  
\mc O (\mf s_L^\vee)^{W_{\mf s^\vee}} 
= \cH (\mf s_L^\vee, q_F^{1/2})^{W_{\mf s^\vee}}. 
\end{equation}
By \cite[Theorem 3.18.a]{AMS3}, the left $\cH (\mf s^{\vee op}, q_F^{1/2})$-module 
$M(\varphi,\rho^\vee,q_F^{1/2})$ admits central character 
$W_{\mf s^{\vee op}} (\varphi_L,\rho^\vee)$.
~Then the construction of $M(\varphi,\rho,q_F^{1/2})^{op}$ in the proof of Theorem 
\ref{thm:10.4} shows that it admits the same central character
$W_{\mf s^{\vee op}} (\varphi_L,\rho^\vee)$.~Changing the notation from
$(\mf s^{\vee op}, \rho^\vee)$ to $(\mf s^\vee,\rho)$ means that this 
central character must now be written as $W_{\mf s^\vee} (\varphi_L,\rho_L)$. 
This and \eqref{eq:11.6} imply that $M(\varphi,\rho,q_F^{1/2})^{op}$ is a constituent of
\[
\ind_{\cH (\mf s_L^\vee, q_F^{1/2})}^{\cH (\mf s^\vee, q_F^{1/2})}
(\varphi_L,\rho_L) = \ind_{\cH (\mf s_L^\vee, q_F^{1/2})}^{
\cH (\mf s^\vee, q_F^{1/2})} M (\varphi_L,\rho_L, q_F^{1/2})^{op} .
\]
By the compatibility with parabolic induction in Theorem \ref{thm:10.3}, we know that
$\pi (\varphi,\rho)$ is a constituent of $\II_Q^G \pi (\varphi_L,\rho_L)$. By Proposition
\ref{prop:11.1},  $\pi (\varphi_L,\rho_L) \in \Irr (L)$ is supercuspidal,
thus it represents the cuspidal support of $\pi (\varphi,\rho)$.
\end{proof}

In the following, we shall use the notion of temperedness for modules of twisted 
graded Hecke algebras as in \cite[Definition 9.2]{SolEnd}.

\begin{lem}\label{lem:11.3}
In \eqref{eqn:thm:10.6}, $\pi \in \Irr^0 (G)_{ns}$ is tempered if and only if
$\varphi_\pi \in \Phi^0 (G)$ is bounded.    
\end{lem}
\begin{proof}
We keep track of what happens to temperedness in \eqref{eq:10.28}, starting with 
$\pi (\varphi,\rho) \in \Irr (G)_{\mf s}$ on the left. By 
\cite[Proposition 9.5.a]{SolEnd}, $\pi (\varphi,\rho)$ is tempered if and only if $\bar M (\varphi,\rho) \in \Irr_{\mf t_\R} \text{-} \mh H (\mc R^\tau, W_{\mf s,\tau},
k^\tau, \natural_\tau)$ is tempered. The isomorphism from Proposition \ref{prop:9.7} respects the 
maximal commutative subalgebras and the sets of positive roots for these
twisted graded Hecke algebras, so it preserves temperedness. By
\cite[Theorem 3.8.b and 3.18.b]{AMS3}, the last two maps in \eqref{eq:10.28}
match tempered irreducible modules with bounded enhanced L-parameters.
\end{proof}

Recall that a $G$-representation is called essentially square-integrable if its
restriction to $G_\der$ is square-integrable. The corresponding notion for
modules of Hecke algebras is \textit{essentially discrete series}; see for example
\cite[Definition 9.2]{SolEnd}. 

\begin{lem}\label{lem:11.4}
In \eqref{eqn:thm:10.6}, $\pi \in \Irr^0 (G)_{ns}$ is essentially 
square-integrable if and only if $\varphi_\pi \in \Phi^0 (G)$ is discrete. 
\end{lem}
\begin{proof}
Again we trace through \eqref{eq:10.28}, starting with $\pi (\varphi,\rho) \in 
\Irr (G)_{\mathfrak{X}_{\nr}^+ (L)\tau}$. By \cite[Proposition 9.5.b,c]{SolEnd}, 
$\pi (\varphi,\rho)$ is essentially square-integrable if and only if 
\begin{equation}\label{eq:11.2}
\text{rk } R_{\sigma,\tau} =\;\text{rk }\Sigma (G,L),
\end{equation}
where $\Sigma (G,L)$ denotes the set of nonzero weights of the maximal $F$-split 
subtorus of $Z^\circ (L)$ acting on Lie($G$).~As with temperedness, the algebra 
isomorphism from Proposition \ref{prop:9.7} preserves ``essentially discrete 
series".~Condition \eqref{eq:11.2} is equivalent to the condition that
\[
\text{rk } R_{\sigma,\tau} = \dim_\C \mathfrak{X}_{\nr} (L) - 
\dim_\C \mathfrak{X}_{\nr} (G) = \dim \mf t - \text{rk } X^* (Z^0 (G)), 
\]
which is then equivalent to the condition that 
\begin{equation}\label{eq:11.3}
\text{rk } R_{\mf s^\vee,\varphi_b} =
\dim_\C Z(\mf l^\vee)^{\mb W_F} - \dim_\C Z(\mf g^\vee)^{\mb W_F} .
\end{equation}
By \cite[Lemma 3.7]{AMS3}, we can express the right-hand side of 
\eqref{eq:11.3} as $\dim_\C (T)$, where $T$ is as in \cite[\S 3.1]{AMS3}. 
Thus $\pi (\varphi,\rho)$ is essentially square-integrable if
and only if $M \big( \varphi,\rho,\log (q_F^{1/2}) \big) \in \Irr \text{-} 
\mh H \big( \mf s^\vee, \varphi_b, \log (q_F^{1/2}) \big)$ is essentially discrete series and rk $R_{\mf s^\vee,\varphi_b}=\dim_\C (T)$. 
By \cite[Theorem 3.18.c]{AMS3}, the combination
of the latter two conditions is equivalent to the discreteness of $\varphi$.
\end{proof}

Recall from \cite{Lan,Bor} that every $\varphi \in \Phi (G)$ canonically determines 
a character $\chi_\varphi$ of $Z(G)$. To better utilize the cuspidal support map for 
enhanced L-parameters in the proof of the following Proposition \ref{prop:11.5}, we
will also need the perspective of L-parameters as Weil--Deligne morphisms
\begin{equation}\label{eqn:infini-psi}
\psi :\mb W_F \ltimes \C \to {}^L G . 
\end{equation}
More explicitly, given $\varphi : \mb W_F \times \SL_2 (\C) \to {}^L G$, 
we define $\psi$ by
\begin{equation}
\psi (w,z) := \varphi \big( w, \matje{\|w\|^{1/2}}{0}{0}{\|w\|^{-1/2}} 
\matje{1}{z}{0}{1} \big) .
\end{equation}
It is well-known that $\psi$ determines $\varphi$ up to $G^\vee$-conjugacy, 
see for instance \cite[Proposition 2.2]{GrRe}.

\begin{prop}\label{prop:11.5}
In \eqref{eqn:thm:10.6}, the central character of $\pi$ is $\chi_{\varphi_\pi}$.    
\end{prop}
\begin{proof}
Let $(\varphi,\rho) \in \Phi^0_e (G)_{ns}$ and let $(L,\varphi_L,\rho_L)$ be (a 
representative of) its cuspidal support. By Lemma \ref{lem:11.2}, 
$\pi (\varphi_L,\rho_L) \in \Irr^0_\cusp (L)$ represents the cuspidal support
of $\pi = \pi (\varphi,\rho)$. Then $\pi (\varphi,\rho)$ is a subquotient of
$\II_Q^G \pi (\varphi_L,\rho_L)$; moreover, $\pi (\varphi,\rho)$ and $\pi (\varphi_L,\rho_L)$ 
admit the same $Z(G)$-character, i.e.~the restriction to $Z(G)$ of the $Z(L)$-character of
$\pi (\varphi_L,\rho_L)$, which by Lemma \ref{lem:8.8} is equal to  
\begin{equation}\label{eq:11.4} 
\chi_{\varphi_L} |_{Z(G)} = \theta |_{Z(G)} .    
\end{equation}
The construction of $\chi_\varphi$ from \cite{Lan,Bor} is recalled just above Lemma
\ref{lem:8.8}. Let $\mc G \to \tilde{\mc G}$ be an embedding such that $\mc G_\der = 
\tilde{\mc G}_\der$ and $Z(\tilde{\mc G})$ is connected. Let $\tilde \varphi \in \Phi 
(\tilde G)$ be a lift of $\varphi$. The image $\tilde \varphi_z \in \varphi (Z(\tilde G))$ 
of $\tilde \varphi$ determines a character $\chi_{\tilde \varphi}$ of $Z(\tilde G)$, and
by definition $\chi_\varphi = \chi_{\tilde \varphi} |_{Z(G)}$.

We now consider $\psi$ as in \eqref{eqn:infini-psi}. Similarly, we
define $\psi_L$ and $\tilde \psi$, in terms of $\varphi_L$ and $\tilde \varphi$.
By \cite[Definition 7.7 and (108)]{AMS1}, $\psi |_{\mb W_F} =
\psi_L |_{\mb W_F}$, thus $\psi$ and $\psi_L$ differ only on the unipotent elements
$u_\varphi := \psi (1,1)$ and $u_{\varphi_L} := \psi_L (1,1)$. The lift $\tilde \psi$ of
$\psi$ gives rise to a lift $\tilde{\psi}_L : \mb W_F \ltimes \C \to {}^L G$ of
$\psi_L$, defined by
\[
\tilde{\psi}_L |_{\mb W_F} = \tilde \psi |_{\mb W_F} \quad\text{and}\quad 
\tilde{\psi}_L |_\C = \psi_L |_\C .
\]
The map ${}^L \tilde G \to {}^L Z(\tilde G)$ dual to $Z(\tilde{\mc G}) \to 
\tilde{\mc G}$ divides ${\tilde{G}^\vee}_\der$ out, in particular $\tilde \psi (\C)$ 
and $\tilde{\psi}_L (\C)$ belong to its kernel. Hence $\tilde{\psi}_z \in \varphi 
(Z(\tilde G))$ is equal to the image $\tilde{\psi}_{L,z}$ of $\tilde{\psi}_L$ in
$\varphi (Z(\tilde G))$. In other words, $\tilde{\psi}_z$ and $\tilde{\psi}_{L,z}$
determine the same character of $Z(\tilde G)$.

Consider the maximal torus $\tilde{\mc T}  = \mc T Z(\tilde G)$ of $\tilde G$.
Since $\varphi_L$ and $\psi_L$ have image in ${}^L T$, $\tilde{\psi}_L$ has image
in ${}^L \tilde T$. By the functoriality of the LLC for tori \cite{Yu}, $\tilde{\psi}_L$
determines a character of $\tilde T$ that extends the character $\theta$ of
$T$ determined by $\psi_L$ or $\varphi_L$. Hence the character $\chi_{\tilde \varphi}$
of $Z(\tilde G)$ determined by $\tilde{\psi}_z = \tilde{\psi}_{L,z}$ (or 
equivalently by $\tilde{\varphi}_z$) extends $\theta |_{Z(G))}$. We conclude that
$\chi_\varphi$ is equal to $\theta |_{Z(G)}$, which by \eqref{eq:11.4} is also the
$Z(G)$-character of $\pi (\varphi,\rho)$.
\end{proof}

Next we investigate the compatibility between our LLC and parabolic induction.
Let $\mc P = \mc M \mc U$ be a parabolic $F$-subgroup of $\mc G$, with
unipotent radical $\mc U$ and Levi factor $\mc M$. Suppose that $\varphi \in 
\Phi^0 (G)$ factors through ${}^L M$. Then we can compare representations of
$G$ and of $M$ associated to enhancements of $\varphi$, via normalized parabolic
induction. However, there is an obstruction to doing this nicely, given by
a function $\epsilon$ from \cite[\S 1.16]{Lus-Cusp3} in a setting with graded 
Hecke algebras. As in the discussion before \cite[Lemma 3.19]{AMS3}, $\epsilon$ can be 
interpreted as a function 
\[
\epsilon (\varphi,q_F^{1/2}) := \epsilon (t_\varphi,-\log (q_F)/2),
\]
where $t_\varphi$ is as in \eqref{eq:10.47}, computed in a setting from 
$\mh H (\mf s^\vee, \varphi_b, \mb r)$.

\begin{lem}\label{lem:11.7}
Let $(\varphi,\rho^M) \in \Phi_e^0 (M)_{ns}$ and let $\pi^M (\varphi,\rho^M) \in
\Irr (M)$ be its image under \eqref{eqn:thm:10.6} for $M$. 
\enuma{
\item Suppose that $\epsilon (\varphi, q_F^{1/2}) \neq 0$. There is a canonical 
isomorphism
\[
\II_P^G \pi^{M,st}(\varphi,\rho^M) \cong \bigoplus\nolimits_\rho \, \Hom_{S_\varphi^{M+}}
(\rho, \rho^M) \otimes \pi^{st}(\varphi,\rho) ,
\]
where the sum runs through all $\rho \in \Irr \big( \pi_0 (S_\varphi^+) \big)$ such that 
Sc$(\varphi,\rho)$ and Sc$(\varphi,\rho^M)$ are $G^\vee$-conjugate.

The multiplicity of $\pi (\varphi,\rho)$ in $\II_P^G \pi^{M,st}(\varphi,\rho^M)$ is
$\dim \Hom_{S_\varphi^{M+}} (\rho, \rho^M)$, and $\pi (\varphi,\rho)$ already appears 
this many times as a quotient of $\II_P^G \pi^M (\varphi,\rho^M)$.
\item Suppose that $\varphi$ is bounded. Then $\epsilon (\varphi,q_F^{1/2}) \neq 0$,
$\pi^M (\varphi,\rho^M) = \pi^{M,st}(\varphi,\rho^M)$ and $\pi (\varphi,\rho) = 
\pi^{st}(\varphi,\rho)$, thus 
\[
\II_P^G \pi^M (\varphi,\rho^M) \cong \bigoplus\nolimits_\rho \, \Hom_{S_\varphi^{M+}}
(\rho, \rho^M) \otimes \pi (\varphi,\rho) .
\]
}
\end{lem}
\begin{proof}
(a) \cite[Lemma 3.19]{AMS3} gives this for left $\cH (\mf s^\vee,
q_F^{1/2})$-modules, only with $\Hom_{S_\varphi^{M+}}(\rho^M, \rho)$ instead of 
$\Hom_{S_\varphi^{M+}}(\rho, \rho^M)$.~The constructions in Theorem 
\ref{thm:10.4} translate this to our right $\cH (\mf s^\vee, q_F^{1/2})$-modules.
Then the isomorphism becomes
\begin{align*}
\ind_{\cH (\mf s_M^\vee, q_F^{1/2})}^{\cH (\mf s^\vee, q_F^{1/2})}  \bar E_M
(\varphi,\rho^M ,q_F^{1/2})^{op} & \cong 
\bigoplus\nolimits_\rho \, \Hom_{S_\varphi^{M+}} ( \rho^{M\vee}, \rho^\vee) \otimes
\bar E (\varphi,\rho ,q_F^{1/2})^{op} \\
& \cong \bigoplus\nolimits_\rho \, \Hom_{S_\varphi^{M+}} ( \rho, \rho^M) \otimes
\bar E (\varphi,\rho ,q_F^{1/2})^{op} .
\end{align*}
Theorem~\ref{thm:10.3} allows us to transfer statements from
$\Mod_\fl\text{-}\cH (\mf s^\vee, q_F^{1/2})$ to $\Rep_\fl^0 (G)_{ns}$.

(b) The boundedness of $\varphi$ implies that $t_\varphi = 0$. By \cite[Lemma B.3]{SolKL}
\footnote{The element $t_\varphi$ is called $\sigma_0$ in \cite{SolKL}}, 
we know that $\epsilon (t_\varphi,-\log (q_F)/2) \neq 0$. The equalities between irreducible
and standard $\mh H (\mf s^\vee,\varphi_b,q_F^{1/2})$-modules come from
\cite[Proposition B.4.a]{SolKL}, and via Theorem \ref{thm:10.3} they can be carried 
over to the corresponding results about $G$-representations.
\end{proof}

Next we verify compatibility of Theorem \ref{thm:10.6} with the Langlands 
classification for $p$-adic groups as in \cite{Kon,Ren}. We briefly recall the statement. For every $\pi \in \Irr (G)$, there exists a triple $(P,\tau,\nu)$, unique up to 
$G$-conjugation, such that: 
\begin{itemize}
\item $P = MU$ is a parabolic subgroup of $G$;
\item $\tau \in \Irr (M)$ is tempered;
\item the unramified character $\nu \in \mathfrak{X}_{\nr}^+ (M)$ is strictly positive
with respect to $P$;
\item $\pi$ is the unique irreducible quotient of the standard representation
$\II_P^G (\tau \otimes \nu)$.
\end{itemize}
These constructions provide bijections between:
\begin{itemize}
\item the set of triples $(P,\tau,\nu)$ as above, up to $G$-conjugation,
\item the set of standard $G$-representations, up to isomorphism,
\item $\Irr (G)$.
\end{itemize}
Note that in the above setting, $\pi$, $\II_P^G (\tau \otimes \nu)$ and 
$\tau \otimes \nu$ have the same cuspidal support. Hence $\pi$ lies in
$\Irr^0 (G)_{ns}$ if and only if $\tau \otimes \nu$ lie in $\Irr^0 (M)_{ns}$.

Similarly, there is a Langlands classification for L-parameters in \cite{SiZi}.
For every $\varphi \in \Phi (G)$, there exists a parabolic subgroup $P = M U$ of $G$,
such that $\varphi$ factors through ${}^L M$ and can be written as $\varphi = z \varphi_b$,
where $\varphi_b \in \Phi (M)$ is bounded and $z \in \mathfrak{X}_{\nr}^+ (M)$ is strictly 
positive with respect to $P$. This gives a bijection between $\Phi (G)$ and such 
triples $(P,\varphi_b,z)$, up to $G^\vee$-conjugacy (see \cite[Theorem 4.6]{SiZi}). 
The strict positivity of $z$ implies that 
\begin{equation}\label{eq:11.5}
Z_{G^\vee}(\varphi (\mb W_F)) = Z_{M^\vee}(\varphi (\mb W_F)) \text{ and }
S_\varphi^+ = S_\varphi^{M+} .
\end{equation}
For any enhancement $\rho \in \Irr \big( \pi_0 (S_\varphi^+) \big)$, the construction of the 
cuspidal support of $(\varphi,\rho)$ reduces to a construction in $Z_{G^\vee}(\varphi (\mb W_F))$,
with $\varphi |_{\SL_2 (\C)}$ and $\rho$ as input (see for example \cite[\S 7]{AMS1}). 
Therefore, $(\varphi,\rho) \in \Phi_e (M)$ has the same cuspidal support as $(\varphi,\rho) 
\in \Phi_e (G)$. In particular, $(\varphi,\rho) \in \Phi^0_e (G)_{ns}$ if and only if
$(\varphi,\rho) \in \Phi^0_e (M)_{ns}$.

Now, let $(\varphi,\rho) \in \Phi^0_e (G)_{ns}$ and write $\varphi = z \varphi_b \in \Phi (M)$ 
as in \cite[Theorem 4.6]{SiZi}. 
\begin{prop}\label{prop:11.8}
\enuma{
\item $\pi^{st}(\varphi,\rho)$ is isomorphic to
$\II_P^G \pi^M (\varphi,\rho) = \II_P^G \big( z \otimes \pi^M (\varphi_b,\rho) \big)$.
\item $\pi^{st}(\varphi,\rho)$ is a standard $G$-representation in the sense of Langlands,
and $\pi (\varphi,\rho)$ is its unique irreducible quotient.
\item Every standard representation in $\Rep^0 (G)_{ns}$ arises in this way.
}
\end{prop}
\begin{proof}
(a) Since $z \in (Z(M^\vee)^{\mb I_F})_{\mb W_F}$, we have $S_\varphi^+=S_{\varphi_b}^+$ 
and $\rho$ can be viewed as an enhancement of $\varphi_b \in \Phi (M)$. 
By Lemmas \ref{lem:11.6} and \ref{lem:11.7} (b), we have 
\[
\pi^M (\varphi,\rho) = z \otimes \pi^M (\varphi_b, \rho) = z \otimes \pi^{M,st} (\varphi_b,\rho)
= \pi^{M,st} (\varphi,\rho) .
\]
By \cite[Lemma B.3]{SolKL}, $\epsilon (\varphi,q_F^{1/2}) = 
\epsilon (t_\varphi, -\log (q_F)/2)$ is nonzero. By \eqref{eq:11.5} and 
Lemma \ref{lem:11.7} (a), we have $\pi^{st}(\varphi,\rho) \cong \II_P^G \pi^M (\varphi,\rho) = \II_P^G \pi (\varphi,\rho)$.

(b) By Lemma \ref{lem:11.3}, $\pi^M (\varphi_b,\rho) \in \Irr (M)$ is tempered,
and by construction, $z \in \mathfrak{X}_{\nr}^+ (M)$ is strictly positive with respect to $P$.
Hence $\pi^{st}(\varphi,\rho) \cong \II_P^G (z \otimes \pi^M (\varphi_b,\rho))$ is a
standard $G$-representation. By \cite[Proposition B.4.c]{SolKL}, 
$M \big( \varphi,\rho^\vee,\log (q_F^{1/2}) \big)$ is the unique irreducible quotient of
$E \big( \varphi,\rho^\vee, \log (q_F^{1/2}) \big)$, in the category of left modules for
$\mh H \big( \mf s^{\vee op},\varphi_b, \log (q_F^{1/2}) \big)$.
By \eqref{eq:10.17}, \eqref{eq:10.26} and \eqref{eq:10.23}, the same holds for 
\[
M(\varphi,\rho,q_F^{1/2})^{op} \text{ and } E (\varphi,\rho,q_F^{1/2})^{op} \in
\Mod_\fl\text{ - }\cH (\mf s^\vee,q_F^{1/2}).
\]
Then by Theorem \ref{thm:10.3}, $\pi (\varphi,\rho)$ is the unique
irreducible quotient of $\pi^{st} (\varphi,\rho)$.

(c) Let $\II_P^G (\tau \otimes \nu)$ be a standard representation in $\Rep^0 (G)_{ns}$.
Then $\tau \in \Irr^0 (M)_{ns}$ is tempered, thus by Lemma \ref{lem:11.3}, 
$\varphi_\tau \in \Phi (M)$ is bounded. The strict positivity of $\nu$ agrees with the notion 
in \cite{SiZi}, because the same parabolic subgroup $P \subset G$ is used on both sides.
Hence $(P,\varphi_\tau,\nu)$ is a triple as in \cite[Theorem 4.6]{SiZi}. By part (a), 
we have $\pi^{st} (\nu \varphi_\tau,\rho_\tau) \cong \II_P^G \big( \nu \otimes 
\pi (\varphi_\tau,\rho_\tau )\big) = \II_P^G (\nu \otimes \tau)$. \qedhere
\end{proof}

Finally, we deduce a corollary of our results on the $p$-adic Kazhdan--Lusztig conjecture from
\cite[Conjecture 8.11]{Vog}. It expresses the multiplicity of an irreducible representation
in a standard representation as the multiplicity of a local system in a perverse sheaf, 
both arising from enhanced L-parameters.

\begin{cor}\label{cor:11.9}
The $p$-adic Kazhdan--Lusztig conjecture holds for $\Rep^0 (G)_{ns}$.
\end{cor}
\begin{proof}
This follows from Theorems \ref{thm:10.7} and \ref{thm:10.6} in combination with
\cite[\S 5]{SolKL} (in particular the proof of \cite[Theorem 5.4]{SolKL}).
\end{proof}

\appendix

\section{Splittings of some extensions on the \texorpdfstring{$p$}{p}-adic side} 
\label{sec:A}

In \S\ref{par:cocycle}, we will need generalizations of some results in \S 
\ref{par:2.split}. The extensions \eqref{eq:7.22}, \eqref{eq:7.19} and \eqref{eq:7.24} can
also be constructed with $N_G (L)$ instead of $L$, giving
\begin{align}\label{eq:7.35}
\begin{split}
&1 \to T \to N_{G^\flat}(L^\flat, T^\flat)_{\theta,[x]} \to 
W (N_{\mc G^\flat}(\mc L^\flat), \mc T^\flat) (F)_{\theta,[x]} \to 1 , \\
&1 \to T \to N_G (L, j T)_\theta \to 
W (N_{\mc G^\flat}(\mc L^\flat), \mc T^\flat) (F)_{\theta,[x]} \to 1 , \\
&1 \to T \to (\mc T^\flat \! \rtimes \! W (N_{\mc G^\flat}(\mc L^\flat),\mc T^\flat))_x (F)_\theta \to 
W (N_{\mc G^\flat}(\mc L^\flat), \mc T^\flat) (F)_{\theta,[x]} \to 1 .
\end{split}
\end{align}
Pushout along $\theta : T \to \C^\times$ gives extensions containing \eqref{eq:7.23}, 
\eqref{eq:7.20} and \eqref{eq:7.25} resp.~
\begin{equation}\label{eq:7.36}
\begin{array}{ccccccccc}
1 & \to & \C^\times & \to & \mc E_{\theta,G}^{0,[x]} & \to &  
W (N_{\mc G^\flat}(\mc L^\flat), \mc T^\flat) (F)_{\theta,[x]} & \to & 1 , \\
1 & \to & \C^\times & \to & \mc E_{\theta,G}^{[x]} & \to &  
W (N_{\mc G^\flat}(\mc L^\flat), \mc T^\flat) (F)_{\theta,[x]} & \to & 1 , \\
1 & \to & \C^\times & \to & \mc E_{\theta,G}^{\rtimes [x]} & \to &  
W (N_{\mc G^\flat}(\mc L^\flat), \mc T^\flat) (F)_{\theta,[x]} & \to & 1 .
\end{array}
\end{equation}
Arguments in \cite[\S 8]{Kal4} are written for extensions of finite groups by tori, they also apply to \eqref{eq:7.35} and \eqref{eq:7.36}. 
Thus, similar as in Lemma \ref{lem:7.10}, we conclude:

\begin{lem}\label{lem:7.2}~In \eqref{eq:7.35} and \eqref{eq:7.36}, 
the middle extension is the Baer sum of the other two.
\end{lem}
The first extensions in \eqref{eq:7.35} and \eqref{eq:7.36} have the following analogues without
restricting to the stabilizer of $[x] \in H^1 (\mc E, \mc Z \to \mc T)$:\label{i:66}
\begin{align}\label{eqn:extension-EthetaG0}
\begin{split}
&1 
\to 
T 
\to 
N_{G^\flat}(L^\flat, T^\flat)_\theta 
\to 
W (N_{\mc G^\flat}(\mc L^\flat), \mc T^\flat) (F)_\theta  
\to 
1 , \\
&1 
\to 
\C^\times 
\to 
\mc E_{\theta,G}^0 
\to 
W (N_{\mc G^\flat}(\mc L^\flat), \mc T^\flat) (F)_\theta 
\to 
1 .
\end{split}
\end{align}
We have the following analogue of Proposition \ref{prop:7.9}.
\begin{prop}\label{prop:7.3}
The extension $\mc E_{\theta,G}^0$ from \eqref{eqn:extension-EthetaG0} splits.
\end{prop}
\begin{proof}
As in Proposition \ref{prop:7.9}, it suffices to
construct a setwise splitting of \eqref{eq:7.40}, which upon pushout along
$\theta_\ff : \mc T_\ff (k_F) \to \C^\times$ becomes a groupwise splitting of
\begin{equation}\label{eq:7.52}
1 \to \C^\times \to \mc E_{\theta_\ff, \mc G_y^\circ}^0 \to 
W(N_{\mc G_y^\circ}(\mc L_y^\circ), \mc T_\ff) (k_F)_{\theta_\ff} \to 1 .
\end{equation}
As the notation $\mc E_{\theta_\ff, \mc G_y^\circ}^0$ suggests, this extension
is the analogue of $\mc E_{\theta,G}^0$ for the finite reductive groups
$\mc G_y^\circ (k_F)$, $\mc L_y^\circ (k_F)$ and $\mc T_\ff (k_F)$.
By a standard argument analogous to that in the proof of Proposition 
\ref{prop:7.9}, one reduces further to the cases where $\mc G_y^\circ$ is
simply connected and absolutely simple. From the proof of Proposition \ref{prop:7.9}
we recall the groups $\mc L_{y,i} \subset \mc L_{y,i}^+$ and
$Z_{\mc G_y^\circ}(\mc L_{y,\der}^\circ)^\circ \subset 
Z_{\mc G_y^\circ}(\mc L_{y,\der}^\circ)^+$, 
the embedding \eqref{eq:7.45} and the extension
\begin{equation}\label{eq:A.46}
1 \to \mc T_{\ff,i}(k_F) \to N_{\mc L_{y,i}^+}(\mc T_{\ff,i})(k_F)_{\theta_\ff} \to
W (\mc L_{y,i}^+, \mc T_{\ff,i})(k_F)_{\theta_\ff} \to 1,
\end{equation}
We write the pushout of the extension \eqref{eq:A.46} along $\theta_\ff$ as
\begin{equation}\label{eq:7.58}
1 \to \C^\times \to \mc E_{\theta_\ff, \mc L_{y,i}^+}^0 \to 
W(\mc L_{y,i}^+, \mc T_{\ff,i})(k_F)_{\theta_\ff} \to 1 .
\end{equation}
Let $W(Z_{\mc G_y^\circ}(\mc L_{y,\der}^\circ)^+,
Z(\mc L_y^\circ)^\circ)_{\theta_\ff}$ be the image of 
$W(N_{\mc G_y^\circ} (\mc L_y^\circ), \mc T_\ff)_{\theta_\ff}$ under
\eqref{eq:7.45} followed by projection onto the first factor.~Similar to \eqref{eq:7.40}
and \eqref{eq:A.46}, we construct 
\begin{equation}\label{eq:7.53}
1 \to Z(\mc L_y^\circ)^\circ (k_F) \to 
N_{Z_{\mc G_y^\circ}(\mc L_{y,\der}^\circ)^+}(\mc Z (\mc L_y^\circ)^\circ)
(k_F)_{\theta_\ff} \to W(Z_{\mc G_y^\circ}(\mc L_{y,\der}^\circ)^+,
Z(\mc L_y^\circ)^\circ)_{\theta_\ff} \to 1,
\end{equation}
whose pushout along $\theta_\ff$ gives
\begin{equation}\label{eq:7.59}
1 \to \C^\times \to \mc E_{\theta_\ff, Z_{\mc G_y^\circ}(\mc L_{y,\der}^\circ)^+}^0 
\to W(Z_{\mc G_y^\circ}(\mc L_{y,\der}^\circ)^+, 
\mc T_{\ff,i})(k_F)_{\theta_\ff} \to 1 .    
\end{equation}
We claim that it suffices to construct setwise splittings of \eqref{eq:A.46}
(for each $i$) and of \eqref{eq:7.53}, such that upon pushout along 
$\theta_\ff$ they become groupwise splittings of \eqref{eq:7.58} and \eqref{eq:7.59}. 
More precisely, in this case $\theta_{\ff,i}$ extends to 
$N_{\mc L_{y,i}^+}(\mc T_\ff)(k_F)_{\theta_\ff}$
and hence to $N_{\mc L_{y,i} \rtimes \Gamma}(\mc T_{\ff,i})(k_F)_{\theta_\ff}$.
Doing this for all $i$ gives us an extension of $\theta_\ff$ from 
$(\mc T_\ff \cap \mc L_{y,\der}^\circ)(k_F)$ to 
$\prod\nolimits_i \, N_{\mc L_{y,i} \rtimes \Gamma}(\mc T_{\ff,i})(k_F)_{\theta_\ff}$. 
Similarly, $\theta_{\ff,Z} := \theta_\ff |_{Z(\mc L_y^\circ)^\circ (k_F)}$ 
extends to $N_{Z_{\mc G_y^\circ}(\mc L_{y,\der}^\circ) \rtimes \Gamma}
(\mc Z (\mc L_y^\circ)^\circ) (k_F)_{\theta_\ff}$. These extensions combine 
to a character $\theta_\ff^+$ of 
\begin{equation}\label{eq:7.54}
\frac{N_{Z_{\mc G_y^\circ}(\mc L_{y,\der}^\circ) \rtimes \Gamma}
(\mc Z (\mc L_y^\circ)^\circ) (k_F)_{\theta_\ff} \times \prod\nolimits_i \, 
N_{\mc L_{y,i} \rtimes \Gamma}(\mc T_{\ff,i})(k_F)_{\theta_\ff}}{\{ 
(z,z^{-1}) : z \in (Z(\mc L_y^\circ)^\circ \cap \mc L_{y,\der}^\circ)(k_F)\}}, 
\end{equation}
which extends $\theta_\ff$.~Since $N_{\mc G_y^\circ} (\mc L_y^\circ, \mc T_\ff)
(k_F)_{\theta_\ff}$ embeds in \eqref{eq:7.54}, we obtain an 
extension of $\theta_\ff$ in \eqref{eq:7.40}.~This gives a groupwise 
splitting of \eqref{eq:7.52} and hence our claim.

By \eqref{eq:7.44}, $W(\mc G_y^\circ,\mc T_\ff)_{\theta_\ff}$ 
is a semidirect product of a normal subgroup and a complementary part that embeds in
$X^* (\mc T_\ff) / \Z R (\mc G_y^\circ,\mc T_\ff)$. By the simplicity of 
$\mc G_y^\circ$, that complement is cyclic or isomorphic to the Klein four group.
The group $W(\mc L_y^\circ,T_\ff )_{\theta_\ff}$ embeds in
$W(\mc G_y^\circ,T_\ff)_{\theta_\ff}$. By the non-singularity of $\theta_\ff$, we know that 
$W(\mc L_y^\circ,T_\ff )_{\theta_\ff}$ intersects the reflection
subgroup of $W(\mc G_y^\circ,T_\ff)_{\theta_\ff}$ trivially. Hence 
\begin{equation}\label{eq:7.55}
W(\mc L_{y,i},T_{\ff,i} )_{\theta_\ff} \text{ embeds in }
X^* (\mc T_\ff) / Z R (\mc G_y^\circ,T_\ff).
\end{equation}
By the simplicity of $\mc G_y^\circ$, the right hand side of \eqref{eq:7.55} is either cyclic or 
isomorphic to the Klein four group.

By the classification of irreducible root systems and parabolic subsystems, 
one deduces that each $\mc L_{y,i}$ is either simple or a direct product of
simply connected groups of type $A_{n-1}$. This leads to three cases of 
\eqref{eq:A.46} and \eqref{eq:7.58}, treated below.\\

\noindent
\textit{Case A. $\mc L_{y,i}$ is a product of $d > 1$ simple factors
$\mc M_i$ of type $A_{n-1}$, permuted transitively by 
$N_{\mc G_y^\circ}(\mc L_y^\circ) (k_F)_{\theta_\ff} \times \mb W_F$.} \\
All these $\mc M_i$'s are isomorphic to $\SL_n$. Let $\Fr^{d'}$ be the 
smallest power of $\Fr$ that stabilizes all the $\mc M_i$'s (or equivalently
one of them). Then $\Fr^{d'}$ acts on $\mc M_i$ as raising to the $q_F^{d'}$-th power 
composed with an elliptic element $F_A$ of $W(A_{n-1}) \cong S_n$
(by classification and by the ellipticity of $\mc T_\ff$). Every elliptic 
element of $S_n$ is an $n$-cycle, and by adjusting the coordinates in
$\mc M_i \cong \SL_n$ we can achieve that $F_A \in \GL_n$ is the product
of the matrix of the permutation $(1 \, 2 \ldots n)$ with a scalar matrix.

By \eqref{eq:7.55}, the group $W(\mc L_{y,i}, \mc T_{\ff,i})_{\theta_\ff}$
comes from at most two simple factors $\mc M_i$ of $\mc L_{y,i}$.~By the transitive 
action on the set of simple factors of $\mc L_{y,i}$, each such $\mc M_i$ 
contributes the same number of elements to $W(\mc L_{y,i}, 
\mc T_{\ff,i})_{\theta_\ff}$. If $\mc L_{y,i}$ has two simple factors which both 
contribute, then \eqref{eq:7.55} is 
not cyclic, which implies that $\mc G_y^\circ$ has type $D_{2n}$. But in this case, the
elements of $W(\mc G_y^\circ,\mc T_\ff)_{\theta_\ff}$ outside its reflection subgroup 
do not come from any type-$A$ Levi subgroup of $\mc G_y^\circ$. Indeed, the Weyl 
group of a maximal type-$A$ Levi subgroup of $\mc G_y^\circ$ is $S_{2n}$,
but any expression of an element of $W(\mc G_y^\circ,\mc T_\ff)_{\theta_\ff}$ 
outside its reflection subgroup contains some of the sign reflections in $W(D_{2n}) \cong S_{2n} \ltimes \{\pm 1\}^{2n}$. This shows that $W(\mc L_{y,i}, \mc T_{\ff,i})_{\theta_\ff}$ is trivial.

Consider two simple factors $\mc M_1, \mc M_2$ of $\mc L_{y,i}$.~The Dynkin diagram 
of $\mc M_1 \times \mc M_2$ embeds into a connected type-$A$ sub-diagram $J$ of the 
Dynkin diagram of $\mc G_y^\circ$. In the Weyl group generated by $J$, 
we can find an element of $N_{\mc G_y^\circ}(\mc L_{y,i})$ that exchanges $\mc M_1$ 
with $\mc M_2$ and stabilizes the other $\mc M_i$'s.
Hence $W (\mc L_{y,i}^+,\mc T_{\ff,i})_{\theta_\ff}$ surjects onto $S_d$ (viewed as
the permutation group of the set of simple factors $\mc M_i$). 

We take a closer look at the reflections in this $S_d$.~It suffices to consider
the transposition $s_{12}$ that exchanges $\mc M_1$ and $\mc M_2$.~The Levi subgroup of 
$\mc G_y^\circ$ generated by $J$ is simple of type $A$, and it contains finite covers of 
$\mc M_1$ and $\mc M_2$ as Levi subgroups.~In terms of the coordinates of 
$\mc M_i \cong \SL_n$, $s_{12}$ is the product of $n$ commuting reflections $s_\alpha$, 
permuted cyclically by $F_A$ and each sending one coordinate for $\mc M_1$ to one 
coordinate for $\mc M_2$.~Since $s_\alpha$ fixes $\theta_\ff$ and the two coordinates 
exchanged by $\alpha$ have no relations (like they would in $\SL_2$), we have $\langle \alpha^\vee , \theta_\ff \rangle = 0$. Let $\tilde s_\alpha$ be a Tits lift of $s_\alpha$ with respect to some pinning
\cite[\S 4]{Tits}.~This element is canonical up to the image of $\alpha^\vee$, and 
$\langle \alpha^\vee , \theta_\ff \rangle = 0$ implies that $\tilde s_\alpha$ becomes 
canonical after pushout along $\theta_\ff$.~Set 
\[
\tilde s_{\Fr^{d'm} (\alpha)} := \Fr^{d'm} \tilde s_\alpha \Fr^{-d'm} \quad \text{and} 
\quad \tilde s_{12} := \prod\nolimits_{m=0}^{n-1} \, \tilde s_{\Fr^{d'm} (\alpha)} .
\]
Then $\tilde s_{12}$ is fixed by $\Fr^{d'}$ and 
\begin{equation}\label{eq:7.TL}
{\tilde s_{12}}^2 = \prod\nolimits_{m=0}^{n-1} \, \Fr^{d'm} (\alpha^\vee (-1)) 
\; \in \; \ker (\theta_\ff) \cap Z(\mc L_{y,i})^{\Fr^{d'}},
\end{equation}
so ${\tilde s_{12}}^2$ becomes trivial upon pushout along $\theta_\ff$.
The same construction can be carried out for any simple reflection $s_{i, \,i+1}$ 
in $S_d$. By the length-multiplicativity of Tits lifts \cite[Proposition 9.3.2]{Spr} these 
$\tilde s_{i, \, i+1}$ give rise to a map 
\begin{equation}\label{eq:7.57}
S_d \to \big( \mc L_{y,i}^+ \big)^{\Fr^{d'}} ,
\end{equation}
which by \eqref{eq:7.TL} yields a homomorphism $S_d \to \mc E^0_{\theta_\ff, \mc L_{y,i}^+}$.

Recall that $d'$ is at most the order of the action of $\Fr$ on $\mc G_y^\circ$, 
so it is at most 3. By the assumptions of case A, every simple factor of $\mc L_{y,i}$   
lies in a $\Fr$-orbit consisting of precisely $d'$ many of the $\mc M_i$'s. The centralizer 
of $\Fr$ in $S_d$ is a semidirect product of two subgroups:
\begin{itemize}
\item[(i)] the normal subgroup $N (S_d,\Fr)$ generated by the cycles of $\Fr$,\\ 
e.g.~$\langle (123), (456) \rangle$ if $\Fr$ acts by $(123)(456)$;
\item[(ii)] a subgroup $\Gamma (S_d,\Fr) \cong S_{d/d'}$ that permutes the various 
$\Fr$-orbits of $\mc M_i$'s (where the coordinates of each $\mc M_i$ are ordered 
in a cycle given by $F_A$), e.g.~$(14)(25)(36)$ if $\Fr$ acts by $(123)(456)$. 
\end{itemize}

We order the $\mc M_i$'s so that each $\Fr$-orbit forms one string of 
consecutive entries. Then the set of those $\tilde s_\alpha$ (as above) 
such that $s_\alpha$ only permutes one $\Fr$-orbit, can be constructed in a
$\Fr$-equivariant way. This allows us to make the image of $N(S_d,\Fr)$ under 
\eqref{eq:7.57} $\Fr$-invariant. 

A transposition $h$ in $\Gamma (S_d,\Fr)$ is formed of a product of $d' n$ commuting
reflections, each of the form $\Fr^m s_\alpha \Fr^{-m} = s_{\Fr^m (\alpha)}$.
Analogous to the construction of $\tilde s_\alpha$ above, we can construct
a $\Fr$-invariant lifting $\tilde h$ of $h$. Using these $\tilde h$, we can make the
image of $\Gamma (S_d,\Fr)$ under \eqref{eq:7.57} $\Fr$-invariant. 

This provides the desired lift of $S_d^\Fr$, which gives a splitting of
\eqref{eq:7.58}.\\

\noindent \textit{Case B. $\mc L_{y,i}$ is $\Fr$-stable and
$W(\mc L_{y,i}, T_{\ff,i})_{\theta_\ff} = \{1\}$}. \\
Now $W (\mc L_{y,i}^+, \mc T_{\ff,i})_{\theta_\ff}$ injects into the group of 
diagram automorphisms of $\mc L_{y,i}$. If $W (\mc L_{y,i}^+, 
\mc T_{\ff,i})_{\theta_\ff}$ is cyclic, then there exists a splitting of 
\eqref{eq:7.58}. This group can only be non-cyclic if $\mc L_{y,i}$ has type 
$D_4$ and $W(N_{\mc G_y^\circ}(\mc L_{y,i}),T_\ff)_{\theta_\ff}$
surjects onto the group of outer automorphisms Out$(D_4) \cong S_3$. 

We represent $\theta_{\ff,i}$ by an element $\tilde \theta_{\ff,i}$ of the
fundamental alcove for $W_\af (D_4)$ in $X^* (\mc T_{\ff,i}) \otimes_\Z \R$. Then the 
stabilizer of $\tilde \theta_{\ff,i}$ in $W(D_4) \rtimes \mr{Out}(D_4)$ is isomorphic 
to the stabilizer of $\tilde \theta_{\ff,i}$ in $X^* (\mc T_{\ff,i}) \rtimes 
\mr{Aut}(W(D_4))$. The aforementioned shape of $W (\mc L_{y,i}^+, 
\mc T_{\ff,i})_{\theta_\ff}$ implies that it is generated by elements that stabilize 
the fundamental alcove.~The regularity of $\theta_{\ff,i}$ gives the regularity 
of $\tilde \theta_{\ff,i}$, in the sense that it does not lie on any hyperplane in 
the affine Coxeter complex.~Thus $W (\mc L_{y,i}^+, 
\mc T_{\ff,i})_{\theta_\ff}$ is isomorphic to the group Out$(D_4)$ of diagram 
automorphisms, which we view as a subgroup of $W(D_4) \rtimes \mr{Out}(D_4)$. 

Next we show that $W (\mc L_{y,i}^+, \mc T_{\ff,i})(k_F)_{\theta_\ff}$ is cyclic 
or that \eqref{eq:7.58} splits. If $\Fr$ acts on $\mc L_{y,i}$ by an outer 
automorphism, then Out$(D_4)^\Fr$ is cyclic of order 2 or 3.~Suppose $\Fr$ acts on 
$\mc L_{y,i}$ by an inner automorphism.~Its image in $W(D_4)$ 
is elliptic because $\mc T_{\ff,i}$ is elliptic.~The elliptic elements in $W(D_4) \cong S_4 \ltimes \{\pm 1\}^4$ are easily classified:
\begin{itemize}
\item[(i)] a product of two disjoint 2-cycles in $S_4$ times one sign change in both 
cycles, e.g. $(12)(34) \epsilon_1 \epsilon_4$,
\item[(ii)] a 3-cycle in $S_4$ times two sign changes, of which one outside the 3-cycle,
e.g. $(123) \epsilon_1 \epsilon_4$,
\item[(iii)] the central element $-1 = \epsilon_1 \epsilon_2 \epsilon_3 \epsilon_4$.
\end{itemize}
Elliptic elements of the first two kinds do not commute with the whole group 
Out$(D_4)$.~For such $\Fr$, $W (\mc L_{y,i}^+, \mc T_{\ff,i})(k_F)_{\theta_\ff}$ 
is a proper subgroup of Out$(D_4)$, and hence cyclic. Suppose that $\Fr$ acts via a lift of $-1 \in W(D_4)$ to $\mc L_{y,i}^+$.~The character lattice 
\[
X_* (\mc T_{\ff,i}) = \{ x \in \Z^4 : x_1 + x_2 + x_3 + x_4 \text{ is even} \} 
\]
is spanned by the standard basis 
$\{e_1 - e_2, e_2 - e_3, e_3 - e_4, e_3 + e_4\}$ of the root system $D_4$.
Here $e_2 - e_3$ is the central node in the Dynkin diagram of type $D_4$.~The given 
$\Fr$-action implies that $\mc T_{\ff,i}(k_F)$ is a direct product
of 4 copies of the unitary group $U_1 (k_F)$, where the cocharacter lattice of
each copy is spanned by one of the simple roots.~Accordingly, the character
$\theta_{\ff,i}$ has four coordinates, each a character of $U_1 (k_F)$.~Since 
$\theta_{\ff,i}$ is stable under Out$(D_4)$, its coordinates associated
to the simple roots $e_1 - e_2$, $e_3 - e_4$ and $e_3 + e_4$ are equal.~The 
cocharacters $\Z (e_2 - e_3)$ are fixed by Out$(D_4)$, so we may ignore them in our 
analysis.~Then we are in a situation like Case A, with three simply connected
groups of type $A_1$ permuted transitively by $W(\mc L_{y,i}^+,
\mc T_{\ff,i})_{\theta_\ff}$. Same as in Case A, we produce
liftings $\tilde s_\alpha$ and $\tilde s_{12}$, which combine to a splitting of
\eqref{eq:A.46}.\\

\noindent \textit{Case C. $\mc L_{y,i}$ is $\mb W_F$-stable and
$W(\mc L_{y,i}, \mc T_{\ff,i})_{\theta_\ff} \neq \{1\}$}. \\
Cases I--V in the proof of Proposition \ref{prop:7.9} show that
\[
W (\mc L_{y,i}^+, \mc T_{\ff,i})_{\theta_\ff} \cong 
W(\mc L_{y,i}, \mc T_{\ff,i})_{\theta_\ff} \times N, 
\]
where $N = \{1\}$ unless $\mc L_{\ff,i}$ has type $D_n$ (case IV) in which case $N$ may have 
two elements. In case IV, the proof of Proposition \ref{prop:7.9} already
gives a $W(\mc L_{y,i}^+,\mc T_{\ff,i})_{\theta_\ff}$-equivariant splitting
on $W(\mc L_{y,i}, \mc T_{\ff,i})(k_F)_{\theta_\ff}$ in \eqref{eq:A.46}. 
It remains to find a splitting for $N$, when its $k_F$-points have order two. 
Since $N$ is cyclic, its generator $\epsilon_n$ can always be represented 
by an element of $N_{\mc L_{y,i}^+}(\mc T_{\ff,i})(k_F)_{\theta_\ff}$
whose order reduces to two upon pushout along $\theta_\ff$. \\

Now that we have treated the case of \eqref{eq:A.46}, we next treat \eqref{eq:7.53} and 
\eqref{eq:7.59}.~Since the $\Fr$-action on $\mc T_\ff$ came from 
$\mr{Aut}(\mc L_y^\circ)$, $\Fr$ acts on 
$Z (\mc L_y^\circ)$ and on $Z_{\mc G_y^\circ} (\mc L_{y,\der}^\circ)^+$ just
via the automorphism used to define $\mc G_y^\circ$ as a $k_F$-group.~In particular,
$Z(\mc L_y^\circ)^\circ$ is a maximal, maximally $k_F$-split torus of
$Z_{\mc G_y^\circ} (\mc L_{y,\der}^\circ)^\circ$.~Let $\cH$ be the simply
connected cover of $Z_{\mc G_y^\circ} (\mc L_{y,\der}^\circ)^\circ$ and set
$\cH^+ := \cH \rtimes \Gamma_Z$.~The $\Fr$-action on 
$Z_{\mc G_y^\circ} (\mc L_{y,\der}^\circ)^+$ lifts canonically to $\cH^+$,
making it a $k_F$-group.~Let $\mc T_\cH$ be the inverse image of
$Z(\mc L_y^\circ)^\circ$ in $\cH$.~We choose a $\Fr$-stable Borel subgroup
$\mc B = \mc T_\cH \mc U_\cH$ of $\cH$ and enhance $(\mc T_\cH, \mc B)$ to a 
$\Fr$-stable pinning of $\cH$.~Without changing the group $\cH^+ (k_F)$, we may 
replace $\Gamma_Z$ by the isomorphic ($\Fr$-stable) subgroup of $\mr{Aut}(\cH)$ 
that stabilizes the chosen pinning. Via the canonical map $\cH \to Z_{\mc G_y^\circ} (\mc L_{y,\der}^\circ)^\circ$,
we may pull back $\theta_\ff$ to a character $\theta_\cH$ of $\mc T_\cH$.
We can obtain \eqref{eq:7.59} from the extension
\begin{equation}\label{eq:7.60}
1 \to \mc T_\cH (k_F) \to N_{\cH^+} (\mc T_\cH) (k_F)_{\theta_\cH}
\to W(\cH^+, \mc T_\cH)(k_F)_{\theta_\cH} \to 1 
\end{equation}
in two steps, i.e.~first we push out along $\theta_\cH$
to produce the extension
\begin{equation}\label{eq:7.61}
1 \to \C^\times \to \mc E_{\theta_\cH, \cH^+}^0 \to  
W(\cH^+, \mc T_\cH)(k_F)_{\theta_\cH} \to 1 ,
\end{equation}
which we then pull back along
\[
W(Z_{\mc G_y^\circ} (\mc L_{y,\der}^\circ)^+, Z(\mc L_y^\circ)^\circ)
(k_F)_{\theta_\ff} \to W(\cH^+, \mc T^\cH)(k_F)_{\theta_\cH}.
\]
Hence we may replace \eqref{eq:7.53} by \eqref{eq:7.60}, and we need to find
a splitting of \eqref{eq:7.61}. Note that $\theta_\cH$ can be an arbitrary
character of $\mc T_\cH (k_F)$, the non-singularity of $\theta_\ff$ for
$\mc L_y^\circ$ does not impose restrictions on  $\theta_\ff$ for $Z(\mc L_y^\circ)^\circ$.

By \cite[\S 2.7]{Kal3}, the $\cH (k_F)$-endomorphism algebra of
$\ind_{\mc B (k_F)}^{\cH (k_F)} (\theta_\cH)$ is the twisted group algebra
$\C [W(\cH, \mc T_\cH)(k_F)_{\theta_\cH}, \natural_\cH]$ associated to the
extension $\mc E_{\theta_\cH,\cH}^0$.~Let $\xi_\cH : \mc U (k_F) \to \C^\times$ 
be a non-degenerate character.~By adjointness of parabolic induction and 
restriction for finite reductive groups \cite[Proposition 3.1.10]{GeMa}, we compute
\begin{equation}\label{eq:7.63}
\begin{aligned}
\Hom_{\mc U (k_F)} \big( \xi_\cH, \ind_{\mc B (k_F)}^{\cH (k_F)} (\theta_\cH) \big) 
& \cong \Hom_{\cH (k_F)} (\ind_{\mc U (k_F)}^{\cH (k_F)}(\xi_\cH), 
\ind_{\mc B (k_F)}^{\cH (k_F)} (\theta_\cH) \big) \\
& \cong \Hom_{\mc T_\cH (k_F)} \big( *\!R_{\mc B}^\cH 
\ind_{\mc U (k_F)}^{\cH (k_F)}(\xi_\cH), \theta_\cH \big).
\end{aligned}
\end{equation}
By \cite[Theorem 2.9]{DLM}, the right hand side of \eqref{eq:7.63} simplifies to
\[
\Hom_{\mc T_\cH (k_F)} \big( \ind_{\{e\}}^{\mc T_\cH (k_F)}(\xi_\cH), \theta_\cH \big)
\cong \Hom_{\{e\}} (\xi_\cH, \theta_\cH) \cong \C .
\]
Therefore, $\xi_\cH$ appears in $\ind_{\mc B (k_F)}^{\cH (k_F)} (\theta_\cH)$
with multiplicity one. We fix a nonzero vector $v_\xi \in 
\ind_{\mc B (k_F)}^{\cH (k_F)} (\theta_\cH)$ on which $\mc U (k_F)$ acts via
the character $\xi_\cH$. Every element of 
\[
\End_{\cH (k_F)} (\ind_{\mc B (k_F)}^{\cH (k_F)} (\theta_\cH)) \cong 
\C [W(\cH, \mc T_\cH)(k_F)_{\theta_\cH}, \natural_\cH]
\]
sends $v_\xi$ to the $\xi_\cH$-weight space of $\mc U (k_F)$, which is
$\C v_\xi$. Thus for every $w \in W(\mc H, \mc T_\cH)(k_F)_{\theta_\cH}$,
there exists a unique lift $\tilde w \in \mc E_{\theta_\cH, \cH}^0 \subset 
\End_{\cH (k_F)} (\ind_{\mc B (k_F)}^{\cH (k_F)} (\theta_\cH))$ that fixes $v_\xi$. The collection of these $\tilde w$ gives a group
homomorphism
\[
W(\mc H, \mc T_\cH)(k_F)_{\theta_\cH} \to \mc E_{\theta_\cH, \cH}^0 
\subset \mc E_{\theta_\cH, \cH^+}^0 
\]
that splits a part of \eqref{eq:7.61}. The only elements of $\Gamma_Z$ that appear in \eqref{eq:7.60} are those that
stabilize the orbit $W(\cH, \mc T_\cH)(k_F) \theta_\cH$ and are fixed by $\Fr$. Therefore,
we may assume without loss of generality that 
$\Gamma_Z = \Gamma^\Fr_{Z,W(\cH, \mc T_\cH)(k_F) \theta_\cH}$.

Since $\Gamma_Z$ stabilizes the pinning, the Whittaker datum $(\mc U (k_F),\xi_\cH)$
can be chosen so that it is fixed by $\Gamma_Z$. For example, if the pinning gives
isomorphisms $\mc U_\alpha \cong \mc G_a$ for simple roots $\alpha$, defined over
a finite field extension $k'$ of $k_F$, then we can take $\xi_\cH ((x_\alpha)_\alpha) = \xi' \big( \sum\nolimits_\alpha \, x_\alpha \big)$ where $x_\alpha \in k'$, for a nontrivial additive character $\xi' : k' \to \C^\times$.~Then $v_\xi$ is 
also an element of $\gamma \cdot \ind_{\mc B (k_F)}^{\cH (k_F)} (\theta_\cH)$ 
on which $\mc U (k_F)$ acts by the character $\xi_\cH$. 

By \cite[\S 2.7]{Kal3}, the $\cH^+ (k_F)$-endomorphism algebra of
$\pi_{\cH^+} := \ind_{\mc B (k_F)}^{\cH^+ (k_F)} (\theta_\cH)$ is isomorphic
to the twisted group algebra $\C [W(\cH^+, \mc T_\cH)(k_F)_{\theta_\cH}, 
\natural_\cH]$ associated to the extension \eqref{eq:7.61}. Moreover
\begin{equation}\label{eq:7.62}
\pi_{\cH^+} \cong \bigoplus\nolimits_{\gamma \in \Gamma_Z} \, \gamma \cdot \ind_{
\mc B (k_F)}^{\cH (k_F)} (\theta_\cH) \text{ as representations of } \cH (k_F) ,
\end{equation}
where $\gamma \cdot \ind_{\mc B (k_F)}^{\cH (k_F)} (\theta_\cH)$ is identified with
the subset of $\pi_{\cH^+}$ consisting of functions supported on $\gamma \cH (k_F)$.
Hence the $\xi_\cH$-weight space of $\mc U (k_F)$ in $\pi_{\cH^+}$ is
$\bigoplus\nolimits_{\gamma \in \Gamma_Z} \, \pi_{\cH^+}(\gamma) \C v_\xi$.
The elements of $\mc E_{\theta_\cH, \cH^+}^0 \subset \C [W(\cH^+, \mc T_\cH)(k_F)_{\theta_\cH}, 
\natural_\cH] \cong \End_{\cH^+ (k_F)} (\pi_{\cH^+})$ permute the terms in the decomposition \eqref{eq:7.62}, according to their 
images in $\Gamma_Z$.~Take $w \in W(\mc H^+, \mc T_\cH)(k_F)_{\theta_\cH}$ 
and write it as $w = \gamma w_0 \text{ with } \gamma \in \Gamma_Z \text{ and }
w_0 \in W(\mc H, \mc T_\cH)(k_F)_{\theta_\cH}$. Let $\tilde w \in \mc E_{\theta_\cH, \cH^+}^0$ be the unique lift of $w$ that
sends $v_\xi \in \ind_{\mc B (k_F)}^{\cH (k_F)} (\theta_\cH)$ to $v_\xi$ as an
element of $\gamma \cdot \ind_{\mc B (k_F)}^{\cH (k_F)} (\theta_\cH)$ in
\eqref{eq:7.62}. Consider 
\[
\pi_{\cH^+}(\gamma^{-1}) \tilde w \in \Hom_{\cH (k_F)} \big(
\ind_{\mc B (k_F)}^{\cH (k_F)} (\theta_\cH), 
\gamma^{-1} \cdot \ind_{\mc B (k_F)}^{\cH (k_F)} (\theta_\cH) \big) .
\]
All maps of this form fix $v_\xi$, thus the composition of two such maps also
fixes $v_\xi$. This implies that $\tilde w \tilde v = \widetilde{w v}$ for all
$w ,v \in W(\mc H^+, \mc T_\cH)(k_F)_{\theta_\cH}$, thus providing the 
required splitting of \eqref{eq:7.61}. 
\end{proof} 

\section{Splittings of some extensions on the Galois side}
\label{sec:B}

For applications in \S\ref{par:cocycle}, some parts of \S\ref{par:3.split} need 
to be generalized to larger groups. In the extensions \eqref{eq:6.14}, \eqref{eq:6.7} 
and \eqref{eq:6.11}, we can replace $L^\vee$ by $N_{G^\vee}(L^\vee)$, giving 
\begin{equation}\label{eq:6.25}
\begin{array}{c@{\;\to\;}c@{\;\to\;}c@{\;\to\;}c@{\;\to\;}c}
1 & \bar T^{\vee,+} & (N_{G^\vee}(L^\vee, T^\vee)^+)_{\varphi_T,\eta} & 
W (N_{G^\vee}(L^\vee), T^\vee)^{\mb W_F}_{\varphi_T, \eta} & 1 , \\
1 & \bar T^{\vee,+} & \text{preimage of } Z_{N_{G^\vee}(L^\vee)}(\varphi )_\eta \text { in } 
\bar G^\vee & W (N_{G^\vee}(L^\vee), T^\vee)^{\mb W_F}_{\varphi_T, \eta} & 1 , \\
1 & \bar T^{\vee,+} & \big( \bar T^{\vee,+} \rtimes W (N_{G^\vee}(L^\vee),T^\vee
)^{\mb W_F}_\eta \big)^{\varphi_T (\mb W_F)} & 
W (N_{G^\vee}(L^\vee), T^\vee)^{\mb W_F}_{\varphi_T, \eta} & 1 .
\end{array}
\end{equation}
Analogous to \eqref{eq:7.35}--\eqref{eq:7.36}, pushout along $\eta : \bar T^{\vee,+} \to
\C^\times$ gives three new extensions, which contain \eqref{eq:6.15},
\eqref{eq:6.8} and \eqref{eq:6.11}, respectively:
\begin{equation}\label{eq:6.26}
\begin{array}{ccccccccc}
1 & \to & \C^\times & \to & \mc E_{\eta,G}^{0,\varphi_T} & \to & 
W (N_{G^\vee}(L^\vee), T^\vee)^{\mb W_F}_{\varphi_T, \eta} & \to & 1 , \\
1 & \to & \C^\times & \to & \mc E_{\eta,G}^{\varphi_T} & \to & 
W (N_{G^\vee}(L^\vee), T^\vee)^{\mb W_F}_{\varphi_T, \eta} & \to & 1 , \\
1 & \to & \C^\times & \to & \mc E_{\eta,G}^{\rtimes \varphi_T} & \to &  
W (N_{G^\vee}(L^\vee), T^\vee)^{\mb W_F}_{\varphi_T, \eta} & \to & 1 .
\end{array}
\end{equation}
As in Lemmas \ref{lem:7.10} and \ref{lem:7.2}, the proof of Lemma \ref{lem:6.3} only uses 
\cite[\S 8]{Kal4} and hence applies to the above extensions as well. 
Therefore we have the following.
\begin{lem}\label{lem:6.5}
In \eqref{eq:6.25} and \eqref{eq:6.26}, the middle extension is the Baer sum of the 
other two, in the category of $\tilde N_{\varphi,\eta}$-groups.
\end{lem}

The analogue of Proposition \ref{prop:6.4} for these extensions is the following. 
\begin{prop}\label{prop:6.6}
The extension $\mc E_{\eta,G}^{0,\varphi_T}$ splits.
\end{prop}
\begin{proof}
The first part of the (long) proof consists of several simplifying steps, which allow us to reduce 
certain concrete cases to the case of simple groups, which will be treated by explicit arguments. As at the start of the proof of Proposition \ref{prop:6.4},
we can reduce to the case where $\mb I_F$ acts trivially on $G^\vee$, and thus we
may replace $\varphi$ by the single element $\varphi (\imath_F)$. Since 
\[
W(N_{G^\vee}(L^\vee),T^\vee)_{\varphi_T} \subset W(N_{G^\vee}(L^\vee),T^\vee)_{\varphi (\imath_F)},
\]
the required equivariance is automatic once we have constructed a splitting
in this simplified setting.

As before, $\eta$ factors 
through $(T^\vee / Z(G^\vee))^{\mb W_F}$, and the first extension in \eqref{eq:6.25} 
can be obtained by push out the following along $\eta$
\begin{equation}\label{eq:6.31}
1 \to (T^\vee / Z(G^\vee))^{\mb W_F} \!\to (N_{G^\vee}(L^\vee,T^\vee) / Z(G^\vee) 
)^{\mb W_F}_{\varphi (\imath_F),\eta} \!\to 
W(N_{G^\vee}(L^\vee),T^\vee)^{\mb W_F}_{\varphi (\imath_F), \eta} \!\to 1 .
\end{equation}
The extension \eqref{eq:6.31} is the direct product of analogous extensions for the $F$-simple 
factors of $G$.~Therefore we may assume that $G$ is $F$-simple and simply connected. 

Now $G^\vee$ is the direct product of its simple factors, and they are permuted
transitively by $\mb W_F$.~Hence we can replace $G^\vee$ by one of its simple factors
and $\mb W_F$ by the subgroup stabilizing that factor, without changing
\eqref{eq:6.31}. This allows us to reduce 
further to the cases where $G^\vee$ is simple. 

By \cite{Ste}, similar to the discussion for finite reductive groups near \eqref{eq:7.44}, 
we know that $W(G^\vee,T^\vee)_{\varphi (\imath_F)}$ is a semidirect
product of a normal subgroup and a complementary part that embeds into
$X_* (T^\vee) / Z R (G^\vee,T^\vee)$. By the simplicity of $G^\vee$,
this complement is either cyclic or isomorphic to the Klein four-group.~The 
group $W(L^\vee,T^\vee )_{\varphi (\imath_F)}$ embeds into
$W(G^\vee,T^\vee)_{\varphi (\imath_F)}$.~By the non-singularity of $\theta_\ff$,
$W(L^\vee,T^\vee )_{\varphi (\imath_F)}$ intersects the reflection
subgroup of $W(G^\vee,T^\vee)_{\varphi (\imath_F)}$ trivially. Hence 
\begin{equation}\label{eq:6.40}
W(L_i^\vee,T^\vee )_{\varphi (\imath_F)} \quad \text{embeds in} \quad
X_* (T^\vee) / \Z R^\vee (G^\vee,T^\vee).
\end{equation}
By the simplicity of $G^\vee$, the right-hand side of \eqref{eq:6.40} is either cyclic or 
isomorphic to the Klein four-group. Recall from \eqref{eq:6.32} that $\eta$ factors through $(T^\vee / Z(L^\vee))^{\mb W_F}$.
Hence \eqref{eq:6.31} can be replaced by 
\begin{equation}\label{eq:6.33}
1 \to (T^\vee / Z(L^\vee))^{\mb W_F} \!\to (N_{G^\vee}(L^\vee,T^\vee) / Z(L^\vee) 
)^{\mb W_F}_{\varphi (\imath_F),\eta} \!\to 
W(N_{G^\vee}(L^\vee),T^\vee)^{\mb W_F}_{\varphi (\imath_F),\eta} \!\to 1 ,
\end{equation}
and we need to show that this extension splits upon pushout along $\eta$. We decompose
\begin{equation}\label{eq:6.34}
X_* (T^\vee) \otimes_\Z \R = X_* (Z(L^\vee)^\circ) \otimes_\Z \R \: \oplus \: 
\bigoplus\nolimits_i \, X_* (L_i^\vee \cap T^\vee) \otimes_\Z \R ,
\end{equation}
where $L_i^\vee$ runs through the 
$W(N_{G^\vee}(L^\vee),T^\vee)^{\mb W_F}_{\varphi (\imath_F),\eta} \times \mb W_F$-orbits of 
simple factors of $L^\vee$.~Accordingly, the character $\eta$ decomposes as a product 
of characters $\eta_i$. Let $P$ denote the image of 
$W(N_{G^\vee} (L^\vee),T^\vee)^{\mb W_F}_{\varphi (\imath_F),\eta}$ in the orthogonal group 
of $X_* (Z(L^\vee)^\circ) \otimes_\Z \R$, and write $T_i^\vee:=T^\vee \cap L_i^\vee$. 
The decomposition \eqref{eq:6.34} gives rise to 
an embedding of \eqref{eq:6.33} in the product of the extensions $1 \to 1 \to P \to P \to 1$ and 
\begin{multline}\label{eq:6.35}
1 \to (T_i^\vee / Z(L_i^\vee))^{\mb W_F} \to 
\big(N_{G^\vee}(L^\vee,T^\vee) \big/ Z_{G^\vee}(L_i^\vee) \cap N_{G^\vee}
(L^\vee,T^\vee) \big)^{\mb W_F}_{\varphi (\imath_F),\eta_i}\\ 
\to W^{\mb W_F}_{i,\eta_i} 
\to 1 ,
\end{multline}
where $i$ runs through the same index set as in \eqref{eq:6.34}, and
\[
W_i :=  W(N_{G^\vee}(L^\vee),T^\vee)_{\varphi (\imath_F)} /
W ( Z_{G^\vee}(L_i^\vee) \cap N_{G^\vee}(L^\vee), T^\vee)_{\varphi (\imath_F)} . 
\]
It suffices to construct, for each $i$, a splitting of \eqref{eq:6.35} which becomes
a group homomorphism upon pushout along $\eta_i$:
\begin{equation}\label{eq:6.43}
1 \to \C^\times \to \mc E_{i,\eta,G}^{0,\varphi_T} \to W_{i,\eta_i}^{\mb W_F} \to 1.    
\end{equation}
Alternatively, we may directly construct a groupwise splitting of \eqref{eq:6.43}, i.e.~in 
this case $\eta_i$ extends to
$(N_{G^\vee}(L_i^\vee,T_i^\vee) / Z(L_i^\vee) )^{\mb W_F}_{\varphi (\imath_F),\eta}$, and hence 
to $(N_{G^\vee}(L_i^\vee,T_i^\vee) )^{\mb W_F}_{\varphi (\imath_F),\eta}$. 
Thus $\eta = \prod_i \eta_i$ extends to
$(N_{G^\vee}(L^\vee,T^\vee) )^{\mb W_F}_{\varphi (\imath_F),\eta}$ and we are done.

By classification of irreducible root systems and parabolic subsystems, $L_i^\vee / Z(L_i^\vee)$ 
is simple or is a product of adjoint groups of type $A$. 
This leads to three cases (A,B,C), treated separately below, in which we construct
such a splitting \eqref{eq:6.35}. \\

\textit{Case A. $L_i^\vee$ is a product of $d> 1$ type $A_{n-1}$ 
simple factors $M_i$, which are permuted transitively by 
$N_{G^\vee}(L^\vee)^{\mb W_F}_{\varphi (\imath_F)} \times \mb W_F$.} \\
All $M_i$'s are isomorphic to $\PGL_n (\C)$.~Recall that the $\mb W_F$-action is given via 
elements of $N_{L^\vee} (T^\vee) \rtimes \mb W_F$ and that $\mb I_F$ acts trivially.~Let 
$\Fr_F^{d'}$ be the smallest power of $\Fr_F$ that stabilizes any (or equivalently all) of 
the $M_i$'s.~Then $\Fr_F^{d'}$ acts as an elliptic element $F_A$ on each $M_i$.~Every elliptic 
element in $S_n$ is an $n$-cycle.~By adjusting the coordinates in each $M_i \cong \PGL_n (\C)$, 
we can make $F_A$ the product of the matrix of permutation $(1\,2\ldots n)$ with a scalar matrix.

By \eqref{eq:6.40}, the group $W(L_i^\vee, T_i^\vee)_{\varphi (\imath_F)}$
arises from at most two simple factors $M_i$ of $L_i^\vee$.~By the transitive action
on the set of simple factors of $L_i^\vee$, each such $M_i$ contributes the same
number of elements to $W(L_i^\vee, T_i^\vee)_{\varphi (\imath_F)}$.~If
$L_i^\vee$ has two simple factors that both contribute, then \eqref{eq:6.40} is 
not cyclic, thus $G^\vee$ has type $D_{2n}$.~But in this case, the
elements of $W(G^\vee,T^\vee)_{\varphi (\imath_F)}$ modulo its reflection subgroup do not
come from any type-$A$ Levi subgroup of $G^\vee$.~More precisely, when one expresses elements 
of $W(G^\vee,T^\vee)_{\varphi (\imath_F)}$ outside its reflection subgroup in terms of 
simple reflections, one must use some of the sign reflections in $W(D_{2n}) \cong S_{2n} \ltimes \{\pm 1\}^{2n}$. 
However, the Weyl group of a maximal type-$A$ Levi subgroup of $G^\vee$ is $S_{2n}$. 
This shows that $W(L_i^\vee, T_i^\vee)_{\varphi (\imath_F)}$ is trivial.

Consider two simple factors $M_1, M_2$ of $L_i^\vee$.~The Dynkin diagram of $M_1 \times
M_2$ embeds into a connected type-$A$ subdiagram $J$ of the Dynkin diagram of $G^\vee$.~In 
the Weyl group generated by $J$, we can find an element of 
$N_{G^\vee}(L^\vee)$ that exchanges $M_1$ and $M_2$ and stabilizes the other $M_i$'s.~Hence 
$W (N_{G^\vee}(L^\vee),T^\vee)_{\varphi (\imath_F)}$ surjects onto $S_d$, the latter viewed as
the permutation group on the set of simple factors $M_i$. We now take a closer look at the reflections in this $S_d$.~It suffices to consider
the transposition $s_{12}$ that exchanges $M_1$ and $M_2$.~The Levi subgroup of 
$G^\vee$ generated by $J$ is simple of type $A$, and it contains finite covers of 
$M_1$ and $M_2$ as Levi subgroups.~In terms of the coordinates of 
$M_i \cong \PGL_n (\C)$, 
$s_{12}$ is the product of $n$ commuting reflections $s_\alpha$, permuted cyclically
by $F_A$ and each sending one coordinate for $M_1$ to one coordinate for $M_2$.~Since 
$s_\alpha$ fixes $\varphi (\imath_F)$ and the two coordinates exchanged by $\alpha$ have 
no relations (like they would in $\PGL_2 (\C)$), we have $\alpha^\vee (\varphi (\imath_F)) = 1$ and $\alpha^\vee (\C^\times) \subset Z_{G^\vee}(\varphi (\imath_F))$. Let $\tilde s_\alpha$ be a Tits lift of $s_\alpha$ with respect to some pinning \cite{Tits}.
We set 
\[
\tilde s_{F_A^m (\alpha)} := F_A^m \tilde s_\alpha F_A^{-m} \quad \text{and} \quad
\tilde s_{12} := \prod\nolimits_{m=0}^{n-1} \, \tilde s_{F_A^m (\alpha)} .
\]
Then $\tilde s_{12}$ is fixed by $F_A$ and 
\[
\tilde s_{12}^2 = 
\prod\nolimits_{m=0}^{n-1} \, (F_A^m \alpha)^\vee (-1) \in Z(L^\vee)^{F_A}.
\]
The same construction can be carried out for any simple reflection $s_{i, i+1}$ 
in $S_d$. By the length multiplicativity of Tits lifts \cite[Proposition 9.3.2]{Spr},
this extends to a group homomorphism 
\begin{equation}\label{eq:6.41}
S_d \to \big( N_{G^\vee}(L^\vee,T^\vee) / Z_{G^\vee}(L_i^\vee) \big)^{F_A} .
\end{equation}
Recall that $F_A = \Fr_F^{d'}$ for some $d' \geq 1$.~This $d'$ is at most the order 
of the action of $\Fr_F$ on $G^\vee$, thus it is at most 3. By the assumptions in case A, 
every simple factor of $L_i^\vee$ lies in a $\Fr_F$-orbit consisting of precisely 
$d'$ of the $M_i$'s.~The centralizer of $\Fr_F$ in $S_d$ is a semidirect product of
two subgroups:
\begin{itemize}
\item[(i)] the normal subgroup $N (S_d,\Fr_F)$ generated by the cycles of $\Fr_F$, and
\item[(ii)] a subgroup $\Gamma (S_d,\Fr_F) \cong S_{d/d'}$ that permutes the various 
$\Fr_F$-orbits of $M_i$'s (where the coordinates of each $M_i$ are ordered 
in a cycle given by $F_A$). 
\end{itemize}
We order the $M_i$'s so that each $\Fr_F$-orbit forms one string of 
consecutive entries.~Then the set of those $\tilde s_\alpha$ (as above) 
such that $s_\alpha$ only permutes one $\Fr_F$-orbit, can be constructed in a
$\Fr_F$-equivariant way. This allows us to make the image of $N(S_d,\Fr_F)$ under 
\eqref{eq:6.41} $\Fr_F$-invariant. 

A transposition $\gamma$ in $\Gamma (S_d,\Fr_F)$ is given by a product of $d' n$
commuting reflections, each of the form 
$\Fr_F^m s_\alpha \Fr_F^{-m} = s_{\Fr_F^m (\alpha)}$.
Analogous to the construction of $\tilde s_\alpha$ above, we can construct
a $\Fr_F$-invariant lifting $\tilde \gamma$ of $\gamma$. Using these $\tilde \gamma$,
we can make the image of $\Gamma (S_d,\Fr_F)$ under \eqref{eq:6.41} $\Fr_F$-invariant. This provides the desired lift of $S_d^{\Fr_F}$, which gives a desired splitting of
\eqref{eq:6.35}. \\

\textit{Case B. $L_i^\vee$ is $\mb W_F$-stable and
$W(L_i^\vee, T_i^\vee)_{\varphi (\imath_F)} = \{1\}$.} \\
$W_i$ injects into the group of diagram automorphisms of $L_i^\vee$.~If $W_i$ is cyclic, 
then there exists a splitting of \eqref{eq:6.43} along
$\eta_i$; and $W_i$ can only be non-cyclic if $L_i^\vee$ has type $D_4$ and 
$W(N_{G^\vee}(L^\vee),T^\vee)_{\varphi (\imath_F)}$ surjects onto the group of outer
automorphisms Out$(D_4) \cong S_3$. 

We represent $\varphi (\imath_F)$ by an element $\varphi_a$ of the fundamental alcove for
$W_\af (D_4)$ in $X_* (T^\vee) \otimes_\Z \R$.~Then the stabilizer of $\varphi (\imath_F)$
in $W(D_4) \rtimes \mr{Out}(D_4)$ is isomorphic to the stabilizer of $\varphi_a$ in 
$X_* (T^\vee) \rtimes \mr{Aut}(W(D_4))$.~The structure of $W_i$ implies that 
it is generated by elements that stabilize the fundamental alcove.~The regularity of 
$\theta_\ff$ gives the regularity of $\varphi_a$, in the sense that it does not lie
on any hyperplane of the affine Coxeter complex.~Thus $W_i$ is isomorphic
to the group Out$(D_4)$ of diagram automorphisms, which we view as a subgroup of
$W(D_4) \rtimes \mr{Out}(D_4)$. 

We next show that $W_i^{\Fr_F}$ is cyclic or that \eqref{eq:6.43} splits. If $\Fr_F$ acts 
on $L_i^\vee$ by an outer automorphism, then Out $(D_4)^{\Fr_F}$ is cyclic of order 2 or 3. 
Since $\Fr_F$ acts on $L^\vee$ by an inner automorphism, its image in $W(D_4)$ is elliptic 
because $T^\vee$ is elliptic. The elliptic elements in 
$W(D_4) \cong S_4 \ltimes \{\pm 1\}^4$
are easily classified:
\begin{itemize}
\item[(i)] a product of two disjoint 2-cycles in $S_4$ times one sign change in both 
cycles, e.g. $(12)(34) \epsilon_1 \epsilon_4$; 
\item[(ii)] a 3-cycle in $S_4$ times two sign changes, of which one outside the 3-cycle,
e.g. $(123) \epsilon_1 \epsilon_4$; 
\item[(iii)] the central element $-1 = \epsilon_1 \epsilon_2 \epsilon_3 \epsilon_4$.
\end{itemize}
Elliptic elements of the first two kinds do not commute with the whole group 
Out$(D_4)$.~For such $\Fr_F$, $W_i^{\Fr_F}$ is a proper subgroup of Out$(D_4)$ and hence cyclic. 

Suppose now that $\Fr_F$ acts via a lift of $-1 \in W(D_4)$ to $N_{L^\vee}(T^\vee)$.~In this case, we need to use $\eta_i$.~The group $T_i^{\vee,\Fr_F}$
consists of the elements of order $\leq 2$ in $T_i^\vee$. Since we divided out
$Z(L^\vee)$, we may take $L_i^\vee = \mathrm{PSO}_8 (\C)$ with $T_i^\vee$ being the maximal
torus on the diagonal, i.e.
\[
T_i^\vee = \{ \mr{diag}(t_1,t_2,t_3,t_4,t_4^{-1}, t_3^{-1}, t_2^{-1}, t_1^{-1})
: t_i \in \C^\times \}. 
\] 
We write an arbitrary $t \in T_i^\vee$ as $t=(t_1,t_2,t_3,t_4)$.~In this notation, 
$T_i^{\vee,\Fr_F} \cong \{ \pm 1\}^4$ is generated by
\[
(-1,1,1,1), (1,-1,1,1), (1,1,-1,1), (1,1,1,-1) \text{ and } 
(\sqrt{-1},\sqrt{-1},\sqrt{-1},\sqrt{-1}). 
\]
The element $\epsilon_4 \in \mr{Out}(D_4)$ stabilizes the character $\eta_i$ of
$T_i^{\vee,\Fr_F}$ if and only if 
\begin{align*}
\eta_i \big( \epsilon_4 (\sqrt{-1},\sqrt{-1},\sqrt{-1},\sqrt{-1}) \big) & 
= 
\eta_i (\sqrt{-1},\sqrt{-1},\sqrt{-1},-\sqrt{-1}) \\
& 
= \eta_i (\sqrt{-1},\sqrt{-1},\sqrt{-1},\sqrt{-1}) \: \eta_i (1,1,1,-1)
\end{align*}
equals $\eta_i (\sqrt{-1},\sqrt{-1},\sqrt{-1},\sqrt{-1})$,
or equivalently $\eta_i (1,1,1,-1) = 1$.~The element $\epsilon_4$ can be
lifted to an element of the subgroup of $\mathrm{PO}_8 (\C)$ that changes only the
fourth and fifth coordinates (corresponding to only the fourth coordinate of
$T_i^\vee$).~Since $-1 \in W(D_4)$ acts on the coordinates
of $T^\vee_i$ separately, this lift of $\epsilon_4$ can be adjusted by some element 
$(1,1,1,t_4) \in T_i^\vee$ to make it into a $\Fr_F$-invariant lift, say $a_4$. Then 
\[
a_4^2 \in \{ (1,1,1,t_4) : t_4 \in \C^\times \}^{\Fr_F} =
\langle (1,1,1,-1) \rangle \subset \ker \eta_i.
\]
By construction, $a_4$ is canonical up to $\{ (1,1,1,t_4) : t_4 \in \C^\times \}^{\Fr_F}$, 
thus upon pushout along $\eta_i$ it becomes unique. Via conjugation by order 3 elements of Out$(D_4)$, we also obtain canonical
lifts of the other order 2 elements of Out$(D_4)$. By their canonicity, these
lifts combine to a splitting of \eqref{eq:6.43}.\\

\textit{Case C. $L_i^\vee$ is $\mb W_F$-stable and
$W(L_i^\vee, T_i^\vee)_{\varphi (\imath_F)} \neq \{1\}$.} \\
Cases I--V in the proof of Proposition \ref{prop:7.9} show that
\[
W_i \cong W(L_i^\vee, T_i^\vee)_{\varphi (\imath_F)} \times N, 
\]
where $N = \{1\}$ unless $L_i^\vee$ has type $D_n$ (case IV), in which case $N$ may have 
two elements. In case IV, the proof of Proposition \ref{prop:7.9} 
provides a $W(N_{G^\vee}(L^\vee),T^\vee)_{\varphi (\imath_F)}$-equivariant splitting
on $W(L_i^\vee, T^\vee \cap L_i^\vee)_{\varphi (\imath_F)}$ in \eqref{eq:6.43}. 
It remains to construct a
splitting for $N$, which has order two. Since it is cyclic, its generator 
$\epsilon_n$ can be represented by an element of 
\[
\big( N_{G^\vee}(L^\vee,T^\vee) \big/ Z_{G^\vee}(L_i^\vee) \cap N_{G^\vee}
(L^\vee,T^\vee) \big)^{\mb W_F}_{\varphi (\imath_F)}, 
\]
whose order reduces to two upon pushout along $\eta$. 
\end{proof}

\section*{Index of notations}
\begin{multicols}{2}
$\natural_{\mf s^\vee}$\qquad \iref{i:149}

$\natural_\tau$\qquad \iref{i:165}

$\mh A_S$\qquad \iref{i:102}

$\mc B (\mc L,F)$\qquad \iref{i:15}

$\mc B (\mc L_\ad,F)$\qquad \iref{i:11}

$C^{an}(U)$\qquad \iref{i:127} 

$CW (J,\hat \sigma)$\qquad \iref{i:161} 

$\Gamma_{\hat \sigma}$\qquad \iref{i:134} 

$\Gamma_{\hat \sigma,\tau}$\qquad \iref{i:139} 

$\Gamma_{\mf s^\vee}$\qquad \iref{i:146}

$\Gamma_{\mf s^\vee,\varphi_b}$\qquad \iref{i:154}

$D$\qquad \iref{i:96}

$\Delta_\af$\qquad \iref{i:97} 

$\Delta_{\ff,\af}$\qquad \iref{i:105}

$\eta$\qquad \iref{i:81}

$\mc E_\eta^{\varphi_T}$\qquad \iref{i:85} 

$\mc E_\eta^{\rtimes \varphi_T}$\qquad \iref{i:86} 

$\mc E_\theta^0$\qquad \iref{i:57} 

$\mc E_\theta^{[x]}$\qquad \iref{i:55}

$\mc E_\theta^{\ltimes [x]}$\qquad \iref{i:58}

$\mc E_{\theta,G}^0$\qquad \iref{i:66}

$F$\qquad \iref{i:1}

$F_s$\qquad \iref{i:44}

$\ff$\qquad \iref{i:18}

$\ff_L$\qquad \iref{i:17}

$\Fr$\qquad \iref{i:45}

$\Fr_F$\qquad \iref{i:88} 

$\mc G$\qquad \iref{i:3}

$G$\qquad \iref{i:4}

$\mc G^\flat$\qquad \iref{i:53}

$G^1$\qquad \iref{i:122}

$G^\vee$\qquad \iref{i:67}

$\mc G_\alpha$\qquad \iref{i:115}

$\mc G_\der$\qquad \iref{i:59}

$\mc G_\ff$\qquad \iref{i:24}

$G_\ff$\qquad \iref{i:23}

$G_{\ff,0+}$\qquad \iref{i:Uf}

$G_{\ff,\sigma}$ \qquad \iref{i:174}
 
$\mc G^\circ_\ff (k_F)$\qquad \iref{i:21} 

$\mc G_{\mf f_\alpha}^\circ (k_F)$\qquad \iref{i:111}

$\mc G_\sigma$\qquad \iref{i:114}

$\mc G_\Sc$\qquad \iref{i:60}

$G_\tau^2$\qquad \iref{i:126}

$G_\tau^3$\qquad \iref{i:126}

$\mc G_y^\circ$ \qquad \iref{i:46}

$\mc H (G,P_\ff,\sigma)$\qquad \iref{i:128}

$\mc H (G,\hat P_\ff, \hat \sigma)^\circ$\qquad \iref{i:133}

$\mh H (\tilde{\mc R}_\tau, W_{\mf s,\tau}, k^\tau,\natural_\tau)$\qquad \iref{i:155}

$\mh H \big( \mf s^\vee, \varphi_b, \log (q_F^{1/2}) \big)$\qquad \iref{i:158} 

$\mc H (\mf s^\vee, q_F^{1/2})$\qquad \iref{i:145}

$\mc H (\mf s^\vee, q_F^{1/2})^\circ$\qquad \iref{i:147} 

$\theta$\qquad \iref{i:27}

$\theta_\ff$\qquad \iref{i:28}

$\imath_F$\qquad \iref{i:87}

$\mb I_F$\qquad \iref{i:69}

$\Irr (\mc E_\theta^{[x]},\mr{id})$\qquad \iref{i:56}

$\Irr (G)_{\mf s}$\qquad \iref{i:172}

$\Irr (L)$\qquad \iref{i:36}

$\Irr^0 (L)$\qquad \iref{i:33}

$\Irr (N_L (T)_\theta, \theta)$\qquad \iref{i:9}

$\Irr (S_\varphi^+,\eta)$\qquad \iref{i:82} 

$J$\qquad \iref{i:116}

$j$\qquad \iref{i:48}

$j_0$\qquad \iref{i:49}

$k_F$\qquad \iref{i:2}

$\kappa_{(T,\theta)}^{L,\epsilon}$\qquad \iref{i:38}

$\kappa_{T,\theta,\rho}^{L,\epsilon}$\qquad \iref{i:63}

$\mc L$\qquad \iref{i:51}

$\mc L^\flat$\qquad \iref{i:50}

$\mf l^\vee$\qquad \iref{i:159}

$\mc L_\alpha$\qquad \iref{i:112} 

$\mc L_\ad$\qquad \iref{i:10}

${}^L G = {}^L \mc G$\qquad \iref{i:68}

${}^L j$\qquad \iref{i:74}

${}^L T$\qquad \iref{i:75}

$\mc L_{y,i}$\qquad \iref{i:64}

$\mc L_{y,i}^+$\qquad \iref{i:65}

$\Mod \text{ - } \cH (G,\hat P_\ff, \hat \sigma)$\qquad \iref{i:141}

$\Mod_\fl \text{ - } \mc H (G,\hat P_\ff, \hat \sigma)$\qquad \iref{i:131}

$N_G (L,T)$ \qquad \iref{i:NGLT}

$N_\varphi$\qquad \iref{i:168}

$\tilde N_\varphi$\qquad \iref{i:84}

$N_{\mc L_\ff (k_F)} (\mc T)_\theta$\qquad \iref{i:31}

$N_W (J)$\qquad \iref{i:103}

$\mf o_F$\qquad \iref{i:20}

$\mc O (\mathfrak{X}_{\nr} (G))$\qquad \iref{i:164} 

$\mb P_F$\qquad \iref{i:69}

$P_\ff$\qquad \iref{i:19}

$\hat P_\ff$\qquad \iref{i:22}

$P_{L,\ff}$\qquad \iref{i:25}

$\Pi_\varphi$\qquad \iref{i:90}

$\Pi_{\varphi,\eta}$\qquad \iref{i:89}

$\Pi_\varphi (L')$\qquad \iref{i:91}

$\pi (\varphi,\rho)$\qquad \iref{i:169}

$\pi^{st}(\varphi,\rho)$\qquad \iref{i:170}

$\Pi (L,T,\theta)$\qquad \iref{i:32}

$\Pi_{\mf s}$\qquad \iref{i:135}

$q_{\alpha^\vee}$\qquad \iref{i:148}

$q_F$\qquad \iref{i:166}

$q_{\sigma,\alpha}$\qquad \iref{i:109}

$q_{\theta,\alpha}$\qquad \iref{i:113}

$\Rep_\fl (G)$\qquad \iref{i:131}

$\Rep^0 (G)_{ns}$\qquad \iref{i:117}

$\Rep (G)_{\mf s}$\qquad \iref{i:35}

$\Rep (G)_{(P_\ff,\sigma)}$\qquad \iref{i:177}

$\Rep (G)_{(\hat P_\ff,\hat \sigma)}$\qquad \iref{i:130}

$\Rep (L)$\qquad \iref{i:34}

$R(G,S)$\qquad \iref{i:94} 

$R_\sigma$\qquad \iref{i:138}

$R_{\sigma, \tau}$\qquad \iref{i:137}

$R_{\mf s^\vee}$\qquad \iref{i:146}

$R_{\mf s^\vee,\varphi_b}$\qquad \iref{i:153}

$\mc R_{\mc T (k_F)}^{\mc L_\ff (k_F)} (\theta)$\qquad \iref{i:29}

$\pm \mc R_{\mc T (k_F)}^{\mc L_\ff (k_F)} (\theta)$\qquad \iref{i:30} 

$\rho_b$\qquad \iref{i:151}

$\rho_\pi$\qquad \iref{i:169}

$\mc S$\qquad \iref{i:93}

$\mf s$\qquad \iref{i:172}

$\mf s^\vee$\qquad \iref{i:413} 

$S_\af$\qquad \iref{i:98}

Sc \qquad \iref{i:171}

$S_{\ff,\af}$\qquad \iref{i:106}

$S_{\ff,\af,\sigma}$\qquad \iref{i:109}

$S_\varphi^+$\qquad \iref{i:72}

$\mf s_L^\vee$\qquad \iref{i:142}

$\sigma$\qquad \iref{i:107} 

$\hat \sigma$\qquad \iref{i:129}

$\mc T$\qquad \iref{i:16}

$\mf t$\qquad \iref{i:156}

$\mf t_\R$\qquad \iref{i:118}

$T^\flat$\qquad \iref{i:52}

$\mc T_s$\qquad \iref{i:95}

$T^\vee$\qquad \iref{i:73}

$T_\Sc$\qquad \iref{i:87}

$\overline{T}^{\vee,+}$\qquad \iref{i:80}

$\mc T_\ff$\qquad \iref{i:26}

$\mc T_{\ff,i}$\qquad \iref{i:61}

$t_\varphi$\qquad \iref{i:167}

$\tau$\qquad \iref{i:121}

$\tau_1$\qquad \iref{i:123}

$\Phi (G)$\qquad \iref{i:70}

$\Phi^0_e (G)$\qquad \iref{i:71}

$\Phi_e (G)^{\mf s^\vee}$\qquad \iref{i:144}

$\varphi$\qquad \iref{i:77}

$\varphi_b$\qquad \iref{i:151}

$\varphi_\pi$\qquad \iref{i:169}

$\varphi_T$\qquad \iref{i:76}

$v (\alpha,J)$\qquad \iref{i:104}

$W$\qquad \iref{i:100}

$W_\af$\qquad \iref{i:99} 

$W_\af (J)$ \qquad \iref{i:106}

$W_\af (J,\sigma)$\qquad \iref{i:110}

$\mb W_F$\qquad \iref{i:6}

$W_J$ \qquad \iref{i:116}

$W(G,L)$\qquad \iref{i:39}

$W(G^\vee,L^\vee)^{\mb W_F}$\qquad \iref{i:78} 

$W(G,L)_{\hat \sigma}$\qquad \iref{i:132}

$W(G,L)_{(T,\theta)}$\qquad \iref{i:40}

$W(G,L)_{(T,\Xo (L) \theta)}$\qquad \iref{i:43}

$W(J,\sigma)$\qquad \iref{i:108}

$W(J,\hat \sigma)$\qquad \iref{i:120}

$W_L$\qquad \iref{i:160}

$W_L (J,\sigma) $\qquad \iref{i:160}

$W(\mc L,\mc T)$\qquad \iref{i:5}

$W(L,\tau,\mathfrak{X}_{\nr} (L))$\qquad \iref{i:140}

$W(N_G (L),T)$\qquad \iref{i:41}

$W(N_{G^\vee}(L^\vee),T^\vee )^{\mb W_F}$\qquad \iref{i:79} 

$W_{\mf s^\vee}$\qquad \iref{i:150} 

$[x]$\qquad \iref{i:54}

$\mf{X}_\nr (L)$\qquad \iref{i:42} 

$\mf{X}_\nr (L)^+$\qquad \iref{i:152}

$\mf{X}_\nr (L^\vee)^+$\qquad \iref{i:152}

$\Xo (G)$\qquad \iref{i:Xo}

$\Xo (G^\vee)$\qquad \iref{i:Xov}

$\mf{X}_\nr (G,\tau)$\qquad \iref{i:125}

$X_* (Z^\circ (L))$\qquad \iref{i:14}

$X^* (\mc T_\ff)$\qquad \iref{i:47}

$\chi_\varphi$\qquad \iref{i:92} 

$y$\qquad \iref{i:46}

$\mc Z$\qquad \iref{i:7}

$T_\cpt$ \qquad \iref{i:cpt}

$Z(\mc L)$\qquad \iref{i:12}

$Z^\circ (\mc L)$\qquad \iref{i:13}

$Z(\mf l^\vee)^{\mb W_F}_\R$\qquad \iref{i:163}

$ZW (J,\hat \sigma)$\qquad \iref{i:124}

$\Omega$\qquad \iref{i:101}

$\Omega_{\ff,\sigma}$ \qquad \iref{i:174}

$\Omega (J)$ \qquad \iref{i:175}

$\Omega (J,\sigma)$\qquad \iref{i:109}

$\Omega (J,\hat \sigma)$\qquad \iref{i:119} 
\end{multicols}

\end{document}